\documentclass{svmono}
\usepackage{amsmath}
\usepackage{amssymb,eucal,latexsym,xy,graphicx,mathrsfs}

\usepackage{multind}
\makeindex{s}
\makeindex{i}
\makeindex{c}

\xyoption{all}
\xyoption{ps}

\spdefaulttheorem{nclaim}{Claim}{\itshape}{}

\spnewtheorem{lem}{Lemma}[section]{\bfseries}{\itshape}
\spnewtheorem{thm}[lem]{Theorem}{\bfseries}{\itshape}
\spnewtheorem{prop}[lem]{Proposition}{\bfseries}{\itshape}
\spnewtheorem{cor}[lem]{Corollary}{\bfseries}{\itshape}
\spnewtheorem{defn}[lem]{Definition}{\bfseries}{}
\spnewtheorem{exple}[lem]{Example}{\bfseries}{}
\spnewtheorem{remk}[lem]{Remark}{\bfseries}{}
\spnewtheorem{notation}[lem]{Notation}{\bfseries}{}

\spnewtheorem*{snote}{Note}{\itshape}{}

\renewcommand{\thesection}{\thechapter-\arabic{section}}

\numberwithin{equation}{section}

\newcommand{\pup}[1]{\textup{(}{#1}\textup{)}}

\newcommand{\tui}{\textup{(i)}}
\newcommand{\tuii}{\textup{(ii)}}
\newcommand{\tuiii}{\textup{(iii)}}
\newcommand{\tuiv}{\textup{(iv)}}
\newcommand{\tuv}{\textup{(v)}}

\newcommand{\tua}{\textup{(a)}}
\newcommand{\tub}{\textup{(b)}}

\newcommand{\dzero}{\dot{0}}
\newcommand{\done}{\dot{1}}
\newcommand{\da}{{\dot{a}}}
\newcommand{\db}{{\dot{b}}}

\newcommand{\dx}{{\dot{x}}}
\newcommand{\du}{{\dot{u}}}
\newcommand{\dA}{{\dot{A}}}

\newcommand{\bu}{{\boldsymbol{u}}}

\newcommand{\bx}{{\boldsymbol{x}}}
\newcommand{\by}{{\boldsymbol{y}}}

\newcommand{\tA}{\tilde{A}}
\newcommand{\tB}{\tilde{B}}

\newcommand{\tX}{\tilde{X}}
\newcommand{\tU}{\tilde{U}}
\newcommand{\tV}{\tilde{V}}

\newcommand{\tf}{\tilde{f}}
\newcommand{\tg}{\tilde{g}}

\newcommand{\tu}{\tilde{u}}
\newcommand{\tv}{\tilde{v}}

\newcommand{\xf}{\boldsymbol{f}}
\newcommand{\xg}{\boldsymbol{g}}

\newcommand{\xu}{\boldsymbol{u}}
\newcommand{\xv}{\boldsymbol{v}}
\newcommand{\xw}{\boldsymbol{w}}

\newcommand{\bA}{{\boldsymbol{A}}}
\newcommand{\bB}{{\boldsymbol{B}}}
\newcommand{\bC}{{\boldsymbol{C}}}
\newcommand{\bD}{{\boldsymbol{D}}}

\newcommand{\bF}{{\boldsymbol{F}}}
\newcommand{\bI}{{\boldsymbol{I}}}
\newcommand{\bL}{{\boldsymbol{L}}}
\newcommand{\bM}{{\boldsymbol{M}}}
\newcommand{\bN}{{\boldsymbol{N}}}

\newcommand{\bS}{{\boldsymbol{S}}}
\newcommand{\bU}{{\boldsymbol{U}}}
\newcommand{\bV}{{\boldsymbol{V}}}

\newcommand{\bX}{{\boldsymbol{X}}}
\newcommand{\bY}{{\boldsymbol{Y}}}
\newcommand{\bZ}{{\boldsymbol{Z}}}

\newcommand{\fra}{{\mathfrak{a}}}
\newcommand{\frb}{{\mathfrak{b}}}

\newcommand{\vx}{\mathsf{x}}
\newcommand{\vy}{\mathsf{y}}
\newcommand{\vz}{\mathsf{z}}

\newcommand{\vu}{\mathsf{u}}

\newcommand{\vE}{\mathsf{E}}
\newcommand{\vF}{\mathsf{F}}

\newcommand{\vR}{\mathsf{R}}

\newcommand{\vX}{\mathsf{X}}
\newcommand{\vY}{\mathsf{Y}}

\newcommand{\ga}{\alpha}
\newcommand{\gb}{\beta}
\newcommand{\gc}{\gamma}
\newcommand{\gd}{\delta}
\newcommand{\gf}{\varphi}
\newcommand{\gi}{\iota}
\newcommand{\gy}{\psi}
\newcommand{\gk}{\kappa}
\newcommand{\gl}{\lambda}
\newcommand{\gm}{\mu}

\newcommand{\gp}{\pi}
\newcommand{\gq}{\theta}
\newcommand{\gr}{\rho}
\newcommand{\vgr}{\varrho}
\newcommand{\gs}{\sigma}
\newcommand{\gt}{\tau}
\newcommand{\gx}{\xi}
\newcommand{\gh}{\eta}

\newcommand{\go}{\omega}
\newcommand{\gos}{\omega\setminus\set{0}}

\newcommand{\tgb}{{\tilde{\gb}}}
\newcommand{\tbgb}{{\boldsymbol{\tilde{\gb}}}}
\newcommand{\bchi}{\boldsymbol{\chi}}

\newcommand{\bga}{{\boldsymbol{\ga}}}
\newcommand{\bgb}{{\boldsymbol{\gb}}}
\newcommand{\bgc}{{\boldsymbol{\gc}}}

\newcommand{\bgf}{{\boldsymbol{\gf}}}

\newcommand{\bgp}{{\boldsymbol{\gp}}}
\newcommand{\bgq}{{\boldsymbol{\gq}}}

\newcommand{\bgx}{{\boldsymbol{\gx}}}
\newcommand{\bgh}{{\boldsymbol{\gh}}}

\newcommand{\cA}{\mathcal{A}}
\newcommand{\cB}{\mathcal{B}}
\newcommand{\cC}{\mathcal{C}}
\newcommand{\cD}{\mathcal{D}}

\newcommand{\scC}{\mathscr{C}}
\newcommand{\scL}{\mathscr{L}}
\newcommand{\scT}{\mathscr{T}}
\newcommand{\cK}{\mathcal{K}}

\newcommand{\cP}{\mathcal{P}}
\newcommand{\cR}{\mathcal{R}}
\newcommand{\cS}{\mathcal{S}}
\newcommand{\cI}{\mathcal{I}}

\newcommand{\cU}{\mathcal{U}}
\newcommand{\cV}{\mathcal{V}}

\newcommand{\bveA}[2]
{{\Vert{#1}\nobreak=\nobreak{#2}\Vert}_{\bA}^{\cV}}
\newcommand{\bveAb}[2]
{{\Vert{#1}\nobreak=\nobreak{#2}\Vert}_{\oll{\bA}}^{\cV}}
\newcommand{\bveAi}[2]
{{\Vert{#1}\nobreak=\nobreak{#2}\Vert}_{\bA_i}^{\cV}}
\newcommand{\bveB}[2]
{{\Vert{#1}\nobreak=\nobreak{#2}\Vert}_{\bB}^{\cV}}
\newcommand{\bveX}[2]
{{\Vert{#1}\nobreak=\nobreak{#2}\Vert}_{\bX}^{\cV}}
\newcommand{\bvrA}[1]
{{\Vert{#1}\Vert}_{\bA}^{\cV}}
\newcommand{\bvrAb}[1]
{{\Vert{#1}\Vert}_{\oll{\bA}}^{\cV}}
\newcommand{\bvrAi}[1]
{{\Vert{#1}\Vert}_{\bA_i}^{\cV}}
\newcommand{\bvrB}[1]
{{\Vert{#1}\Vert}_{\bB}^{\cV}}
\newcommand{\bvrX}[1]
{{\Vert{#1}\Vert}_{\bX}^{\cV}}
\newcommand{\conj}{\mathbin{\bigwedge\mkern-15mu\bigwedge}}

\newcommand{\les}{\leqslant}
\DeclareMathOperator{\kur}{kur}
\DeclareMathOperator{\card}{card}
\DeclareMathOperator{\cf}{cf}
\DeclareMathOperator{\crit}{crit}
\DeclareMathOperator{\critr}{crit_r}
\DeclareMathOperator{\Lift}{Lift}
\DeclareMathOperator{\Ob}{Ob}
\DeclareMathOperator{\Mor}{Mor}
\DeclareMathOperator{\Min}{Min}
\DeclareMathOperator{\Max}{Max}
\DeclareMathOperator{\dom}{dom}
\DeclareMathOperator{\rng}{rng}
\DeclareMathOperator{\Con}{Con}
\DeclareMathOperator{\Res}{Res}
\DeclareMathOperator{\Conc}{Con_c}
\DeclareMathOperator{\Concr}{Con_{c,r}}
\newcommand{\ConV}{\operatorname{\Con^{\cV}}}
\newcommand{\ConcV}{\operatorname{\Con_{\mathrm{c}}^{\cV}}}
\newcommand{\ConA}{\operatorname{\Con^{\cA}}}
\newcommand{\ConcA}{\operatorname{\Con_{\mathrm{c}}^{\cA}}}

\newcommand{\ConcB}{\operatorname{\Con_{\mathrm{c}}^{\cB}}}
\DeclareMathOperator{\Id}{Id}
\DeclareMathOperator{\NId}{NId}
\DeclareMathOperator{\Idc}{Id_c}
\DeclareMathOperator{\Ids}{Id_s}
\DeclareMathOperator{\J}{J}

\DeclareMathOperator{\At}{At}

\DeclareMathOperator{\Ult}{Ult}
\DeclareMathOperator{\Clop}{Clop}
\DeclareMathOperator{\ari}{ar}
\newcommand{\Pow}{\mathfrak{P}}
\newcommand{\todot}{\overset{\boldsymbol{.}}{\rightarrow}}
\newcommand{\Todot}{\overset{\boldsymbol{.}\ }{\Rightarrow}}

\newcommand{\Bool}{\mathbf{Bool}}
\newcommand{\BTop}{\mathbf{BTop}}
\newcommand{\Var}{\operatorname{\mathbf{Var}}}

\newcommand{\CLOS}{\mathrm{CLOS}}
\newcommand{\CLOSr}{\mathrm{CLOS}^{\mathrm{r}}}
\newcommand{\PROD}{\mathrm{PROD}}
\newcommand{\PRES}{\mathrm{PRES}}
\newcommand{\PROJ}{\mathrm{PROJ}}
\newcommand{\CONT}{\mathrm{CONT}}
\newcommand{\LS}{\mathrm{LS}}
\newcommand{\LSb}{\mathrm{LS}^{\mathrm{b}}}
\newcommand{\LSr}{\mathrm{LS}^{\mathrm{r}}}

\newcommand{\utr}{\trianglelefteq}
\newcommand{\res}{\mathbin{\restriction}}
\newcommand{\norm}[1]{\Vert{#1}\Vert}

\newcommand{\xA}{\mathbf{A}}
\newcommand{\xB}{\mathbf{B}}

\newcommand{\xD}{\mathbf{D}}
\newcommand{\xE}{\mathbf{E}}
\newcommand{\xF}{\mathbf{F}}

\newcommand{\xS}{\mathbf{S}}

\newcommand{\LL}{\mathbb{L}}
\newcommand{\VV}{\mathbb{V}}
\newcommand{\FF}{\mathbb{F}}
\newcommand{\ZZ}{\mathbb{Z}}

\newcommand{\jirr}{join-ir\-re\-duc\-i\-ble}
\newcommand{\mirr}{meet-ir\-re\-duc\-i\-ble}
\newcommand{\Mirr}{Meet-ir\-re\-duc\-i\-ble}

\newcommand{\cm}{commutative monoid}
\newcommand{\gqv}{generalized quasivariety}
\newcommand{\gqvs}{generalized quasivarieties}

\newcommand{\eps}{\varepsilon}

\newcommand{\es}{\varnothing}
\newcommand{\into}{\hookrightarrow}
\newcommand{\onto}{\twoheadrightarrow}
\newcommand{\mono}{\rightarrowtail}

\newcommand{\famm}[2]{\left(#1\mid#2\right)}
\newcommand{\set}[1]{\{#1\}}
\newcommand{\setm}[2]{\set{#1\mid#2}}

\newcommand{\seq}[1]{\langle{#1}\rangle}

\newcommand{\SET}{\mathbf{Set}}
\newcommand{\REG}{\mathbf{Reg}}
\newcommand{\SCML}{\mathbf{SCML}}
\newcommand{\MOD}{\mathbf{Mod}}
\newcommand{\SEM}{\mathbf{Sem}_{\vee,0}}
\newcommand{\MIND}{\mathbf{MInd}}
\newcommand{\MALG}{\mathbf{MAlg}}
\newcommand{\METR}{\mathbf{Metr}}

\newcommand{\ol}[1]{\overline{#1}}
\newcommand{\oll}[1]{\,\overline{\!#1}}
\newcommand{\zero}{\mathbf{0}}
\newcommand{\one}{\mathbf{1}}
\newcommand{\two}{\mathbf{2}}
\newcommand{\dnw}{\mathbin{\downarrow}}
\newcommand{\ddnw}{\mathbin{\downdownarrows}}
\newcommand{\uupw}{\mathbin{\upuparrows}}
\newcommand{\upw}{\mathbin{\uparrow}}
\newcommand{\Dnw}{\mathbin{\Downarrow}}
\newcommand{\Upw}{\mathbin{\Uparrow}}
\newcommand{\sor}{\mathbin{\triangledown}}
\newcommand{\Sor}{\mathbin{\bigtriangledown}}
\newcommand{\id}{\mathrm{id}}
\newcommand{\jz}{$(\vee,0)$}
\newcommand{\jzu}{$(\vee,0,1)$}

\newcommand{\jzs}{\jz-semi\-lat\-tice}
\newcommand{\jzus}{\jzu-semi\-lat\-tice}
\newcommand{\jh}{join-ho\-mo\-mor\-phism}
\newcommand{\jzh}{\jz-ho\-mo\-mor\-phism}

\newcommand{\jzuh}{\jzu-ho\-mo\-mor\-phism}
\newcommand{\jze}{\jz-em\-bed\-ding}
\newcommand{\jzue}{\jzu-em\-bed\-ding}

\newcommand{\js}{join-sem\-i\-lat\-tice}
\newcommand{\ajs}{al\-most join-sem\-i\-lat\-tice}
\newcommand{\pjs}{pseu\-do join-sem\-i\-lat\-tice}
\newcommand{\fin}{\mathrm{fin}}
\newcommand{\op}{\mathrm{op}}

\DeclareMathOperator{\Lg}{Lg}
\DeclareMathOperator{\Cst}{Cst}
\DeclareMathOperator{\Op}{Op}
\DeclareMathOperator{\Rel}{Rel}
\newcommand{\rB}{\mathrm{B}}
\newcommand{\rF}{\mathrm{F}}

\DeclareMathOperator{\Ker}{Ker}
\DeclareMathOperator{\Mat}{M}
\DeclareMathOperator{\Sub}{Sub}
\newcommand{\FL}{\operatorname{\rF_{\mathbf{L}}}}

\begin{document}

\title{{}From objects to diagrams for ranges of functors}

\author{Pierre Gillibert and Friedrich Wehrung\\
\\
\normalsize \textbf{Address} (Gillibert):\\
Charles University in Prague\\
Faculty of Mathematics and Physics\\
Department of Algebra\\
Sokolovsk\'a 83\\
186 00 Praha\\
Czech Republic\\
\textbf{e-mail} (Gillibert): \texttt{gilliber@karlin.mff.cuni.cz}, \texttt{pgillibert@yahoo.fr}\\
\textbf{URL} (Gillibert): \texttt{http://www.math.unicaen.fr/\~{}giliberp/}\\
\\
\normalsize \textbf{Address} (Wehrung):\\
LMNO, CNRS UMR 6139\\
D\'epartement de Math\'ematiques, BP 5186\\
Universit\'e de Caen, Campus 2\\
14032 Caen cedex\\
France\\
\textbf{e-mail} (Wehrung): \texttt{wehrung@math.unicaen.fr}, \texttt{fwehrung@yahoo.fr}\\
\textbf{URL} (Wehrung): \texttt{http://www.math.unicaen.fr/\~{}wehrung/}}

\thanks{The first author was partially supported by the institutional grant MSM 0021620839}

\maketitle
\date{\today}



\tableofcontents

\foreword{
The aim of the present work is to introduce a general method, applicable to various fields of mathematics, that enables us to gather information on the range of a functor~$\Phi$, thus making it possible to solve previously intractable representation problems with respect to~$\Phi$. This method is especially effective in case the problems in question are ``cardinality-sensitive'', that is, an analogue of the cardinality function turns out to play a crucial role in the description of the members of the range of~$\Phi$.

Let us first give a few examples of such problems. The first three belong to the field of universal algebra, the fourth to the field of ring theory (nonstable K-theory of rings).

\begin{description}
\item[\textbf{Context~1.}] The classical Gr\"atzer-Schmidt\index{c}{Gr\"atzer, G.}\index{c}{Schmidt, E.\,T.} Theorem, in universal algebra, states that every \jzs\ is isomorphic to the compact congruence lattice of some algebra\index{i}{algebra!universal}. Can this result be extended to \emph{diagrams} of \jzs s?

\item[\textbf{Context~2.}] For a member~$\bA$ of a quasivariety~$\cA$ of algebraic systems, we denote by~$\ConcA\bA$\index{s}{compcongV@$\ConcV$ functor} the \jzs\ of all compact elements of the lattice of all congruences of~$\bA$ with quotient in~$\cA$; further, we denote by~$\Concr\cA$\index{s}{compcongVcr@$\Concr\cV$} the class of all isomorphic copies of~$\ConcA\bA$ where $\bA\in\cA$. For quasivarieties~$\cA$ and~$\cB$ of algebraic systems, we denote by~$\critr(\cA;\cB)$\index{s}{critrAB@$\critr(\cA;\cB)$} (\emph{relative critical point}\index{i}{critical point!relative} between~$\cA$ and~$\cB$) the least possible cardinality, if it exists, of a member of $(\Concr\cA)\setminus(\Concr\cB)$, and~$\infty$ otherwise. What are the possible values of $\critr(\cA;\cB)$, say for~$\cA$ and~$\cB$ both with finite language?

\item[\textbf{Context~3.}] Let~$\cV$ be a nondistributive variety\index{i}{distributive!non-${}_{-}$ variety} of lattices and let~$F$ be the free lattice in~$\cV$ on~$\aleph_1$\index{s}{aleph0@$\aleph_{\ga}$} generators. Does~$F$ have a congruence-permutable, congruence-preserving extension\index{i}{congruence-preserving extension}?

\item[\textbf{Context 4.}] Let $E$ be an exchange ring. Is there a (von~Neumann) regular ring, or a C*-algebra of real rank zero, $R$ with the same nonstable K-theory as~$E$?
\end{description}

It turns out that each of these problems can be reduced to a category-theoretical problem of the following general kind.

Let $\cA$, $\cB$, $\cS$ be categories, let $\Phi\colon\cA\to\cS$ and $\Psi\colon\cB\to\cS$ be functors. We are also given a subcategory~$\cS^\Rightarrow$\index{s}{RightarrowCat@$\cS^\Rightarrow$} of~$\cS$, of which the arrows will be called \emph{double arrows}\index{i}{double arrow} and written $f\colon X\Rightarrow Y$\index{s}{AtorightarrowB@$f\colon A\Rightarrow B$}. We assume that for ``many'' \emph{objects}~$A$ of~$\cA$, there are an object~$B$ of~$\cB$ and a double arrow $\chi\colon\Psi(B)\Rightarrow\Phi(A)$. We also need to assume that our categorical data forms a so-called \emph{larder}\index{i}{larder}. In such a case, we establish that under certain combinatorial assumptions on a poset~$P$, for ``many'' \emph{diagrams}~$\overrightarrow{A}=\famm{A_p,\ga_p^q}{p\leq q\text{ in }P}$ from~$\cA$, a similar conclusion holds at the diagram $\Phi\overrightarrow{A}$,} that is, there are a $P$-indexed diagram~$\overrightarrow{B}$ from~$\cB$ and a double arrow $\overrightarrow{\chi}\colon\Psi\overrightarrow{B}\Rightarrow\Phi\overrightarrow{A}$ from~$\cS^P$. The combinatorial assumptions on~$P$ imply that every principal ideal\index{i}{ideal!of a poset}\index{i}{ideal!principal ${}_{-}$, of a poset} of~$P$ is a \js\ and the set of all upper bounds of any finite subset is a finitely generated\index{i}{finitely generated!upper subset} upper subset\index{i}{upper subset}.

We argue by concentrating all the relevant properties of the diagram~$\overrightarrow{A}$ into a \emph{condensate} of~$\overrightarrow{A}$, which is a special kind of directed colimit of finite products of the~$A_p$ for $p\in P$. Our main result, the \emph{Condensate Lifting Lemma} (CLL)\index{i}{Condensate Lifting Lemma (CLL)}, reduces the liftability of a diagram to the liftability of a condensate, modulo a list of elementary verifications of categorical nature. The impact of CLL on the four problems above can be summarized as follows:

\begin{description}
\item[\textbf{Context~1.}] The Gr\"atzer-Schmidt\index{c}{Gr\"atzer, G.}\index{c}{Schmidt, E.\,T.} Theorem can be extended to any diagram of \jzs s and \jzh s indexed by a finite poset (resp., assuming a proper class of Erd\H os cardinals\index{i}{Erd\H{o}s cardinal}\index{c}{Erdos@Erd\H os, P.}, an arbitrary poset), lifting\index{i}{diagram!lifting} with algebras of variable similarity type.

\item[\textbf{Context~2.}] We prove that in a host of situations, either $\critr(\cA;\cB)<\nobreak\aleph_\go$\index{s}{aleph0@$\aleph_{\ga}$} or $\critr(\cA;\cB)=\infty$\index{s}{critrAB@$\critr(\cA;\cB)$}. This holds, in particular, if~$\cA$ is a locally finite\index{i}{quasivariety!locally finite} quasivariety with finitely many relations while~$\cB$ is a finitely generated, congruence-modular variety with finite similarity type\index{i}{variety!finitely generated}\index{i}{variety!congruence-modular} of algebras\index{i}{algebra!universal} of finite type (e.g., \emph{groups}, \emph{lattices}, \emph{modules over a finite ring}).

\item[\textbf{Context~3.}] The free $\cV$-lattice~$F$ has no congruence-permutable, con\-gru\-ence-preserving extension\index{i}{congruence-preserving extension} in any similarity type containing the lattice type. Due to earlier work by Gr\"atzer, Lakser, and Wehrung~\cite{GLWe}\index{c}{Gr\"atzer, G.}\index{c}{Lakser, H.}\index{c}{Wehrung, F.}, if~$\cV$ is locally finite\index{i}{variety!locally finite}, then the cardinality~$\aleph_1$\index{s}{aleph0@$\aleph_{\ga}$} is optimal.

\item[\textbf{Context~4.}] By using the results of the present work, the second author proves in Wehrung~\cite{VLift}\index{c}{Wehrung, F.} that the answer is no (both for regular rings and for C*-algebras of real rank zero), with a counterexample of cardinality~$\aleph_3$\index{s}{aleph0@$\aleph_{\ga}$}.

\end{description}

We also pave the way to solutions of further beforehand intractable open problems:
\begin{description}
\item[---] the determination of all the possible critical points\index{i}{critical point} between finitely generated\index{i}{variety!finitely generated} varieties of lattices (Gillibert~\cite{Gill3})\index{c}{Gillibert, P.}, then between varieties~$\cA$ and~$\cB$ of algebras such~$\cA$ is locally finite while~$\cB$ is finitely generated with finite similarity type and omits tame congruence theory types~$\mathbf{1}$ and~$\mathbf{5}$ (Gillibert~\cite{Gill5}).

\item[---] the problem whether every lattice of cardinality~$\aleph_1$\index{s}{aleph0@$\aleph_{\ga}$}, in the variety generated by~$\bM_3$, has a congruence $m$-permutable, congruence-preserving extension\index{i}{congruence-preserving extension} for some positive integer~$m$ (Gillibert~\cite{Gill4})\index{c}{Gillibert, P.};

\item[---] a 1962 problem by J\'onsson\index{c}{Jonsson@J\'onsson, B.} about coordinatizability of sectionally complemented\index{i}{lattice!sectionally complemented} modular\index{i}{lattice!modular} lattices without unit (Wehrung~\cite{Banasch2}\index{c}{Wehrung, F.}).

\end{description}

\chapter{Background}\label{Ch:Bckgrnd}

\section{Introduction}\label{S:Intro}
The present work originates in a collection of attempts to solving various open problems, in different topics in mathematics (mainly, but not restricted to, universal algebra and lattice theory), all related by a common feature:

\begin{quote}\normalsize\em
How large can the range of a functor be? If it is large on objects, then is it also large on diagrams?
\end{quote}

``Largeness'' is meant here in the sense of containment (not cardinality): for example, largeness of a class~$\cC$ of sets could mean that~$\cC$ contains, as an element, every finite set; or every countable set; or every cartesian product of~$17$ sets; and so on. Our primary interests are in \emph{lattice theory}, \emph{universal algebra}, and \emph{ring theory} (\emph{nonstable K-theory}). Our further aims include, but are not restricted to, \emph{group theory} and \emph{module theory}. One of our main tools is \emph{infinite combinatorics}, articulated around the new notion of a \emph{$\gl$-lifter}.\index{i}{lifter ($\gl$-)}

Our main goal is to introduce a new method, of categorical nature, originally designed to solve functorial representation problems in the above-mentioned topics. Our main result, which we call the \emph{Condensate Lifting Lemma}, CLL\index{i}{Condensate Lifting Lemma (CLL)} in abbreviation (Lemma~\ref{L:CLL}), as well as its main precursor the \emph{Armature Lemma} (Lemma~\ref{L:Armature}), are theorems of \emph{category theory}. They enable us to solve several until now seemingly intractable open problems, outside the field of category theory, by reducing them to the verification of a list of mostly elementary categorical properties.

The largest part of this work belongs to the field of category theory. Nevertheless, its prime intent is strongly oriented towards outside applications, including the aforementioned topics. Hence, unlike most works in category theory, these notes will show at some places more detail in the statements and proofs of the purely categorical results.

\subsection{The search for functorial solutions to certain representation problems}\label{Su:FunctSol}
Roughly speaking, the kind of problem that the present book aims to help solving is the following. We are given categories~$\cB$ and~$\cS$ together with a functor $\Psi\colon\cB\to\cS$. We are trying to describe the members of the (categorical) range of~$\Psi$, that is, the objects of~$\cS$ that are isomorphic to~$\Psi(B)$ for some object~$B$ of~$\cB$. Our methods are likely to shed some light on this problem, even in some cases solving it completely, mainly in case the answer turns out to be \emph{cardinality-sensitive}. At this point, cardinality is meant somewhat heuristically, involving an appropriate notion of ``size'' for objects of~$\cS$.

Yet we need to be more specific in the formulation of our general problem, thus inevitably adding some complexity to our statements. This requires a slight recasting of our problem. We are now given categories~$\cA$, $\cB$, $\cS$ together with functors $\Phi\colon\cA\to\cS$ and $\Psi\colon\cB\to\cS$, as illustrated on the left hand side of Figure~\ref{Fig:Prelarder}.
\begin{figure}[htb]
 \[
 \def\labelstyle{\displaystyle}
 \xymatrix{
 & \cS & & & & \cS &\\
 \cA\ar[ur]^{\Phi} & & \cB\ar[ul]_{\Psi} & & \cA\ar[ur]^{\Phi}\ar[rr]^{\Gamma}
  & & \cB\ar[ul]_{\Psi}
 }
 \]
\caption{A few categories and functors}
\label{Fig:Prelarder}
\end{figure}
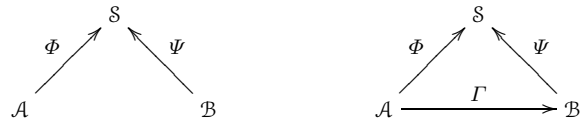

We assume that for ``many'' objects~$A$ of~$\cA$, there exists an object~$B$ of~$\cB$ such that $\Phi(A)\cong\Psi(B)$. We ask to what extent the assignment $(A\mapsto B)$ can be made \emph{functorial}. Ideally, we would be able to prove the existence of a functor $\Gamma\colon\cA\to\cB$ such that~$\Phi$ and~$\Psi\circ\Gamma$ are isomorphic functors---then we say that~$\Gamma$ \emph{lifts~$\Phi$ with respect to~$\Psi$}\index{i}{diagram!lifting}, as illustrated on the right hand side of Figure~\ref{Fig:Prelarder}. (Two functors are \emph{isomorphic} if there is a natural transformation from one to the other whose components are all isomorphisms).

This highly desirable goal is seldom reached. Here are a few instances where this is nevertheless the case (and then this never happens for trivial reasons), and some other instances where it is still not known whether it is the case. In both Examples~\ref{Ex:Schmidt81} and~\ref{Ex:Ruzi04}, $\cS$ is the category of all distributive\index{i}{distributive!semilattice} \jzs s with \jzh s, $\cA$ is the subcategory of~$\cS$ consisting of all distributive\index{i}{lattice!distributive}\index{i}{distributive!lattice|seeonly{lattice, distributive}} \emph{lattices} with zero and $0$-preserving \emph{lattice embeddings}, and~$\Phi$ is the inclusion functor from~$\cA$ into~$\cS$.

While Example~\ref{Ex:Schmidt81} is in lattice theory, Examples~\ref{Ex:Ruzi04} and~\ref{Ex:nsKth} are in ring theory, and Example~\ref{Ex:Lampe} is in universal algebra. Examples~\ref{Ex:RTWe} and~\ref{Ex:PosetMeas} involve $S$-valued versions of a discrete set and a poset, respectively, where~$S$ is a given \jzs.

\begin{exple}\label{Ex:Schmidt81}
$\cB$ is the category of all lattices with lattice homomorphisms, and $\Psi$ is the natural extension to a functor of the assignment that to every lattice~$L$ assigns the \jzs\ $\Conc L$\index{s}{compcon1@$\Conc\bA$, $\Conc f$} of all compact (i.e., finitely generated) congruences of~$L$. The statement $(\forall D\in\cA)(\exists L\in\cB)(D\cong\Conc L)$ was first established by Schmidt~\cite{Schm81}\index{c}{Schmidt, E.\,T.}. The stronger statement that the assignment $(D\mapsto L)$ can be made functorial was established by Pudl\'ak~\cite{Pudl85}\index{c}{Pudl\'ak, P.}. That is, there exists a functor~$\Gamma$, from distributive\index{i}{lattice!distributive} $0$-lattices with $0$-lattice embeddings to lattices and lattice embeddings, such that $\Conc\Gamma(D)\cong D$\index{s}{compcon1@$\Conc\bA$, $\Conc f$} naturally on all distributive\index{i}{lattice!distributive} $0$-lattices~$D$. Furthermore, in Pudl\'ak~\cite{Pudl85}\index{c}{Pudl\'ak, P.}, $\Gamma(D)$ is finite atomistic in case~$D$ is finite.
\end{exple}

Pudl\'ak's approach was motivated by the search for the solution of the \emph{Congruence Lattice Problem} CLP\index{i}{Congruence Lattice Problem (CLP)}, posed by Dilworth in the forties, which asked whether every distributive\index{i}{distributive!semilattice} \jzs\ is isomorphic to~$\Conc L$\index{s}{compcon1@$\Conc\bA$, $\Conc f$} for some lattice~$L$. Pudl\'ak asked in~\cite{Pudl85}\index{c}{Pudl\'ak, P.} the stronger question whether CLP\index{i}{Congruence Lattice Problem (CLP)} could have a functorial answer. This approach proved extremely fruitful, although it gave rise to unpredictable developments. For instance, Pudl\'ak's question was finally answered, in the negative, in T\r{u}ma and Wehrung~\cite{Bowtie}\index{c}{Tuma@T\r{u}ma, J.}\index{c}{Wehrung, F.}, but this was not of much help for the full negative solution of CLP\index{i}{Congruence Lattice Problem (CLP)} in Wehrung~\cite{CLP}\index{c}{Wehrung, F.}, which required even much more work.

\begin{exple}\label{Ex:Ruzi04}
An algebra\index{i}{algebra!over a field}~$R$ over a field~$\FF$ is
\begin{itemize}
\item\emph{matricial}\index{i}{algebra!matricial|ii} if it is isomorphic to a finite product of full matrix rings over~$\FF$;

\item\emph{locally matricial}\index{i}{algebra!locally matricial|ii} if it is a directed colimit of matricial algebras\index{i}{algebra!matricial}.
\end{itemize}
Denote by $\Idc R$\index{s}{idcomp@$\Idc R$, $R$ ring|ii} the \jzs\ of all compact (i.e., finitely generated) two-sided ideals\index{i}{ideal (in a ring)!two-sided} of a ring~$R$. Then~$\Idc$\index{s}{idcomp@$\Idc R$, $R$ ring} extends naturally to a functor from rings with ring homomorphisms to \jzs s with \jzh s.

Let $\cB$ be the category of locally matricial $\FF$-algebras\index{i}{algebra!locally matricial} and let~$\Psi$ be the~$\Idc$\index{s}{idcomp@$\Idc R$, $R$ ring} functor from~$\cB$ to~$\cS$. R\r{u}\v{z}i\v{c}ka proved in~\cite{Ruzi1}\index{c}{Ruzicka@R\r{u}\v{z}i\v{c}ka, P.} that for every distributive\index{i}{lattice!distributive} $0$-lattice~$D$, there exists a locally matricial $\FF$-algebra\index{i}{algebra!locally matricial}~$R$ such that $\Idc R\cong D$\index{s}{idcomp@$\Idc R$, $R$ ring}. Later on, R\r{u}\v{z}i\v{c}ka extended his result in~\cite{Ruzi2}\index{c}{Ruzicka@R\r{u}\v{z}i\v{c}ka, P.} by proving that the assignment $(D\mapsto R)$ can be made functorial, from distributive $0$-lattices with $0$-preserving lattice embeddings to locally matricial rings (over a given field) and ring embeddings.
\end{exple}

\begin{exple}\label{Ex:nsKth}
For a (not necessarily unital) ring~$R$ and a positive integer~$n$, we denote by $\Mat_n(R)$\index{s}{Mat@$\Mat_n(R)$, $\Mat_{\infty}(R)$|ii} the ring of all $n\times n$ square matrices with entries from~$R$, and we identity $\Mat_n(R)$ with a subring of $\Mat_{n+1}(R)$ \emph{via} the embedding $x\mapsto\begin{pmatrix}x&0\\ 0&0\end{pmatrix}$. Setting $\Mat_\infty(R):=\bigcup_{n>0}\Mat_n(R)$\index{s}{Mat@$\Mat_n(R)$, $\Mat_{\infty}(R)$}, we say that idempotent matrices $a,b\in\Mat_\infty(R)$ are \emph{equivalent}, in notation $a\sim b$, if there are $x,y\in\Mat_\infty(R)$ such that $a=xy$ and $b=yx$\index{i}{equivalence of idempotents|ii}. The relation~$\sim$ is an equivalence relation, and we denote by~$[a]$ the $\sim$-equivalence class of an idempotent matrix $a\in\Mat_\infty(R)$. Equivalence classes can be added, \emph{via} the rule
 \[
 [a]+[b]:=\left[\begin{pmatrix}a&0\\ 0&b\end{pmatrix}\right]\,,
 \]
for all idempotent matrices $a,b\in\Mat_\infty(R)$. This way, we get a \cm,\index{s}{VV@$\VV$ functor|ii}
 \[
 \VV(R):=\setm{[a]}{a\in\Mat_\infty(R)\text{ idempotent}}\,,
 \]
that encodes the so-called \emph{nonstable K-theory}\index{i}{nonstable K- (or K$_0$-) theory|ii} (or, sometimes, the \emph{nonstable K$_0$-theory}) of~$R$. It is easy to see that~$\VV$ extends to a \emph{functor}, from rings with ring homomorphisms to \cm s with monoid homomorphisms, \emph{via} the formula $\VV(f)([a])=[f(a)]$ for any homomorphism $f\colon R\to S$ of rings and any idempotent $a\in\Mat_\infty(R)$. The monoid~$\VV(R)$ is \emph{conical}\index{i}{monoid!conical|ii}, that is, it satifies the following statement:
 \[
 (\forall\vx,\vy)(\vx+\vy=0\Rightarrow\vx=\vy=0)\,,
 \]
For example, if~$R$ is a division ring, then~$\VV(R)$ is isomorphic to the monoid $\ZZ^+:=\set{0,1,2,\dots}$\index{s}{ZZplus@$\ZZ^+$|ii} of all non-negative integers, while if~$R$ is the endomorphism ring of a vector space of infinite dimension~$\aleph_\ga$\index{s}{aleph0@$\aleph_{\ga}$} (over an arbitrary division ring), then
$\VV(R)\cong\ZZ^+\cup\setm{\aleph_\xi}{\xi\leq\ga}$.

Every conical \cm\ is isomorphic to $\VV(R)$ for some hereditary ring~$R$: this is proved in Theorems~6.2 and~6.4 of Bergman~\cite{Berg74}\index{c}{Bergman, G.\,M.} for the finitely generated case, and in Bergman and Dicks \cite[page~315]{BeDi78}\index{c}{Bergman, G.\,M.}\index{c}{Dicks, W.} for the general case with order-unit. The general, non-unital case is proved in Ara and Goodearl \cite[Proposition~4.4]{ArGo11}\index{c}{Ara, P.}\index{c}{Goodearl, K.\,R.}. In light of this result, it is natural to ask whether this solution can be made \emph{functorial}, that is, whether there exist a functor~$\Gamma$, from the category of conical \cm s with monoid homomorphisms to the category of rings and ring homomorphisms, such that $\VV\circ{\Gamma}$ is isomorphic to the identity.

This problem is still open. Variants are discussed in Wehrung~\cite{VLift}\index{c}{Wehrung, F.}.
\end{exple}

\begin{exple}\label{Ex:Lampe}
Define both~$\cA$ and~$\cS$ as the category of all \jzus s with \jzue s, $\Phi$ as the identity functor on~$\cA$, $\cB$ as the category of all \emph{groupoids} (a groupoid\index{i}{groupoid|ii} is a nonempty set endowed with a binary operation), and~$\Psi$ as the~$\Conc$\index{s}{compcon1@$\Conc\bA$, $\Conc f$} functor from~$\cB$ to~$\cS$. Lampe proves in~\cite{Lamp82}\index{c}{Lampe, W.\,A.} that for each \jzus\ $S$ there exists a groupoid\index{i}{groupoid}~$G$ such that~$\Conc G\cong S$\index{s}{compcon1@$\Conc\bA$, $\Conc f$}. However, it is not known whether the assignment $(S\mapsto G)$ can be made functorial, from \jzus s with \jzue s to groupoids\index{i}{groupoid} and their embeddings.
\end{exple}

\begin{exple}\label{Ex:RTWe}
For a \jzs\ $S$, an \emph{$S$-valued distance}\index{i}{distance!semilattice-valued|ii} on a set~$\Omega$ is a map $\gd\colon\Omega\times\Omega\to S$ such that
\begin{align*}
\gd(x,x)&=0\,,\\
\gd(x,y)&=\gd(y,x)\,,\\
\gd(x,z)&\leq\gd(x,y)\vee\gd(y,z)&&(\text{\emph{triangular inequality}}),
\end{align*}
for all $x,y,z\in\Omega$. Furthermore, for a positive integer~$n$, we say that~$\gd$ is a \emph{V-distance of type~$n$}\index{i}{distance!semilattice-valued V-|ii}\index{i}{distance!semilattice-valued V-${}_{-}$ of type $n$|ii}\index{i}{V-distance|seeonly{distance, semilattice-valued~V-}} if for all $x,y\in\Omega$ and all $\ga_0,\ga_1\in S$, $\gd(x,y)\leq\ga_0\vee\ga_1$ implies the existence of $z_0,\dots,z_{n+1}\in\Omega$ such that $z_0=x$, $z_{n+1}=y$, and $\gd(z_i,z_{i+1})\leq\ga_{i\,\mathrm{mod}\,2}$ for each $i\leq n$.

It is not hard to modify the proof in J\'onsson~\cite{Jons53}\index{c}{Jonsson@J\'onsson, B.} to obtain that every \jzs\ with modular\index{i}{lattice!modular} ideal lattice is the range of a V-distance\index{i}{distance!semilattice-valued V-} of type~2\index{i}{distance!semilattice-valued V-${}_{-}$ of type $n$} on some set. On the other hand, it is proved in R\r{u}\v{z}i\v{c}ka, T\r{u}ma, and Wehrung~\cite{RTW}\index{c}{Ruzicka@R\r{u}\v{z}i\v{c}ka, P.}\index{c}{Tuma@T\r{u}ma, J.}\index{c}{Wehrung, F.} that this result can be made functorial on \emph{distributive}\index{i}{distributive!semilattice} semilattices. It is also proved in that paper that there exists a distributive\index{i}{distributive!semilattice} \jzus, of cardinality~$\aleph_2$\index{s}{aleph0@$\aleph_{\ga}$}, that is not generated by the range of any V-distance\index{i}{distance!semilattice-valued V-} of type~1\index{i}{distance!semilattice-valued V-${}_{-}$ of type $n$} on any set.
\end{exple}

\begin{exple}\label{Ex:PosetMeas}
For a \jzs\ $S$, an \emph{$S$-valued measure}\index{i}{measure (semilattice-valued)|ii} on a poset~$P$ is a map $\gm\colon P\times P\to S$ such that
 \[
 \gm(x,y)=0\text{ in case }x\leq y\,,\quad\text{and}\quad
 \gm(x,z)\leq\gm(x,y)\vee\gm(y,z)\,,
 \]
for all $x,y,z\in\Omega$.

It is proved in Wehrung~\cite{PosetMeas}\index{c}{Wehrung, F.} that every distributive\index{i}{distributive!semilattice} \jzs~$S$ is join-generated by the range of an $S$-valued measure on some poset~$P$, which is also a meet-semilattice with zero, such that for all $x\leq y$ in~$P$ and all $\ga_0,\ga_1\in S$, if $\gm(y,x)\leq\ga_0\vee\ga_1$, then there are a positive integer $n$ and
a decomposition $x=z_0\leq z_1\leq\cdots\leq z_n=y$ such that $\gm(z_{i+1},z_i)\leq\ga_{i\,\mathrm{mod}\,2}$ for each $i<n$.

However, although the assignment $(S\mapsto(P,\gm))$ can be made ``functorial on lattice-indexed diagrams'', we do not know whether it can be made functorial (starting with \jzs s with \emph{\jze s}).
\end{exple}

\subsection{Partially functorial solutions to representation problems}\label{Su:PartFunctSol}

Let us consider again the settings of Section~\ref{Su:FunctSol}, that is, categories~$\cA$, $\cB$, and~$\cS$ together with functors $\Phi\colon\cA\to\cS$ and $\Psi\colon\cB\to\cS$. The most commonly encountered situation is that for all ``not too large'', but not all, objects~$A$ of~$\cA$ there exists an object~$B$ of~$\cB$ such that $\Phi(A)\cong\Psi(B)$. It turns out that the extent of ``not too large'' is often related to combinatorial properties of the functors~$\Phi$ and~$\Psi$. More precisely, a large part of the present book aims at explaining a formerly mysterious relation between the two following statements:

\begin{itemize}
\item For each object~$A$ of~$\cA$ of ``cardinality'' at most~$\aleph_n$\index{s}{aleph0@$\aleph_{\ga}$}, there exists an object~$B$ of~$\cB$ such that $\Phi(A)\cong\Psi(B)$.

\item For every diagram~$\overrightarrow{A}$ of ``finite'' objects in~$\cA$, indexed by $\set{0,1}^{n+1}$, there are a $\set{0,1}^{n+1}$-indexed diagram~$\overrightarrow{B}$ in~$\cB$ such that $\Phi\overrightarrow{A}\cong\Psi\overrightarrow{B}$.
\end{itemize}

\begin{exple}\label{Ex:PermCong}
In R\r{u}\v{z}i\v{c}ka, T\r{u}ma, and Wehrung~\cite{RTW}\index{c}{Ruzicka@R\r{u}\v{z}i\v{c}ka, P.}\index{c}{Tuma@T\r{u}ma, J.}\index{c}{Wehrung, F.} the question whether every distributive\index{i}{distributive!semilattice} \jzs\ is isomorphic to~$\Conc\bA$\index{s}{compcon1@$\Conc\bA$, $\Conc f$} for some con\-gru\-ence-permutable algebra\index{i}{algebra!congruence-permutable}~$\bA$ is solved in the negative. For instance, if~$\FL(X)$\index{s}{FreeLatt@$\FL(X)$|ii} denotes the free lattice on a set~$X$, then $\Conc\FL(\aleph_2)$\index{s}{compcon1@$\Conc\bA$, $\Conc f$}\index{s}{FreeLatt@$\FL(X)$}\index{s}{aleph0@$\aleph_{\ga}$} is not isomorphic to~$\Conc\bA$\index{s}{compcon1@$\Conc\bA$, $\Conc f$}, for any congruence-permutable algebra\index{i}{algebra!congruence-permutable}~$\bA$. In particular, the \jzs\ $\Conc\FL(\aleph_2)$\index{s}{compcon1@$\Conc\bA$, $\Conc f$}\index{s}{aleph0@$\aleph_{\ga}$}\index{s}{FreeLatt@$\FL(X)$} is neither isomorphic to the finitely generated normal subgroup \jzs\ of any group, nor to the finitely generated submodule lattice of a module. In these results the cardinality~$\aleph_2$\index{s}{aleph0@$\aleph_{\ga}$} is optimal.

On the other hand, every $\set{0,1}^2$-indexed diagram (we say \emph{square}\index{i}{square (shape of a diagram)}) of finite distributive\index{i}{distributive!semilattice} \jzs s can be lifted\index{i}{diagram!lifted}, with respect to the~$\Conc$\index{s}{compcon1@$\Conc\bA$, $\Conc f$} functor, by a square\index{i}{square (shape of a diagram)} of finite relatively complemented\index{i}{lattice!relatively complemented} (thus congruence-permutable\index{i}{algebra!congruence-permutable}) lattices, see T\r{u}ma~\cite{Tuma}\index{c}{Tuma@T\r{u}ma, J.} and Gr\"atzer, Lakser, and Wehrung~\cite{GLWe}\index{c}{Gr\"atzer, G.}\index{c}{Lakser, H.}\index{c}{Wehrung, F.}. It is proved in R\r{u}\v{z}i\v{c}ka, T\r{u}ma, and Wehrung~\cite{RTW}\index{c}{Ruzicka@R\r{u}\v{z}i\v{c}ka, P.}\index{c}{Tuma@T\r{u}ma, J.}\index{c}{Wehrung, F.} that this result does not extend to \emph{cubes}\index{i}{cube (shape of a diagram)|ii}, that is, diagrams indexed by $\set{0,1}^3$.
\end{exple}

\begin{exple}\label{Ex:IdLattRR}
It is proved in Plo\v{s}\v{c}ica, T\r{u}ma, and Wehrung~\cite{PTW}\index{c}{Plo\v{s}\v{c}ica, M.}\index{c}{Tuma@T\r{u}ma, J.}\index{c}{Wehrung, F.} that there is no (von Neumann) regular ring\index{i}{regular ring}~$R$ such that the ideal lattice of~$R$ is isomorphic to $\Con\FL(\aleph_2)$\index{s}{aleph0@$\aleph_{\ga}$}\index{s}{conA@$\Con\bA$}\index{s}{FreeLatt@$\FL(X)$}. On the other hand, it is proved in Wehrung~\cite{WReg}\index{c}{Wehrung, F.} that every square\index{i}{square (shape of a diagram)} of finite Boolean semilattices and \jzh s can be lifted\index{i}{diagram!lifted}, with respect to the~$\Idc$\index{s}{idcomp@$\Idc R$, $R$ ring} functor, by a square\index{i}{square (shape of a diagram)} of regular rings\index{i}{regular ring}. As all rings have permutable congruences, it follows from~\cite[Corollary~7.3]{RTW}\index{c}{Ruzicka@R\r{u}\v{z}i\v{c}ka, P.}\index{c}{Tuma@T\r{u}ma, J.}\index{c}{Wehrung, F.} that this result cannot be extended to the cube\index{i}{cube (shape of a diagram)} denoted there by~$\cD_{\mathrm{ac}}$\index{s}{Dac@$\cD_{\mathrm{ac}}$}.
\end{exple}

\begin{exple}\label{Ex:Momega}
Denote by~$\LL(R)$\index{s}{LLR@$\LL(R)$, $\LL(f)$} the lattice of all principal right ideals\index{i}{ideal (in a ring)!right} of a regular ring\index{i}{regular ring}~$R$. The assignment $(R\mapsto\LL(R))$ can be naturally extended to a functor, from regular rings\index{i}{regular ring} with ring homomorphisms to sectionally complemented\index{i}{lattice!sectionally complemented} modular\index{i}{lattice!modular} lattices with $0$-lattice homomorphisms (see Chapter~\ref{Ch:RegRngLard} for details). A lattice is \emph{coordinatizable}\index{i}{lattice!coordinatizable|ii} if it is isomorphic to~$\LL(R)$\index{s}{LLR@$\LL(R)$, $\LL(f)$} for some regular ring\index{i}{regular ring}~$R$.

Denote by $\bM_\go$\index{s}{Momeg@$\bM_\omega$|ii} the lattice of length two with~$\go$ atoms~$a_n$, with $n<\go$. The assignment $(a_n\mapsto a_{n+1})$ defines an endomorphism~$\gf$ of~$\bM_\go$. The second author proves in~\cite{CXCoord}\index{c}{Wehrung, F.} that
\begin{itemize}
\item one cannot have rings~$R$ and~$S$, a unital ring homomorphism $f\colon R\to S$, and a natural equivalence between the diagrams $\LL(f)\colon\LL(R)\to\LL(S)$\index{s}{LLR@$\LL(R)$, $\LL(f)$} and $\gf\colon M_\go\to M_\go$;

\item there exists a non-coordinatizable\index{i}{lattice!coordinatizable} $2$-distributive\index{i}{lattice!$2$-distributive} complemented\index{i}{lattice!complemented} modular\index{i}{lattice!modular} lattice, of cardinality~$\aleph_1$\index{s}{aleph0@$\aleph_{\ga}$}, containing a copy of~$\bM_\go$ with the same zero and the same unit (we say \emph{spanning}\index{i}{spanning (sublattice)|ii}).
\end{itemize}

The cardinality~$\aleph_1$\index{s}{aleph0@$\aleph_{\ga}$} is optimal in the second point above, as the second author proved, by methods extending those of Wehrung~\cite{CXCoord}\index{c}{Wehrung, F.} (and mentioned there without proof), that \emph{Every countable $2$-distributive\index{i}{lattice!$2$-distributive} complemented\index{i}{lattice!complemented} modular\index{i}{lattice!modular} lattice with a spanning\index{i}{spanning (sublattice)}~$\bM_\go$\index{s}{Momeg@$\bM_\omega$} is coordinatizable}\index{i}{lattice!coordinatizable}.
\end{exple}

\begin{exple}\label{Ex:NonstKTh}
In nonstable K$_0$-theory\index{i}{nonstable K- (or K$_0$-) theory}, three particular classes of rings enjoy a special importance: namely, the \emph{\pup{von Neumann} regular rings}\index{i}{regular ring}, the \emph{C*-algebras of real rank zero}\index{i}{C*-algebra}, and the \emph{exchange rings}\index{i}{exchange ring}. Every regular ring is an exchange ring, the converse fails; and a C*-algebra has real rank zero if{f} it is an exchange ring.

The second author finds in Wehrung~\cite{VLift}\index{c}{Wehrung, F.} a diagram, indexed by the powerset $\set{0,1}^3$ of a three-element set, that can be lifted both by a commutative diagram of C*-algebras of real rank one and by a commutative diagram of exchange rings, but that cannot be lifted by any commutative diagram of either regular rings or C*-algebras of real rank zero. This leads, in the same paper, to the construction of a dimension group with order-unit whose positive cone can be represented as~$\VV(A)$ for a C*-algebra~$A$ of real rank one, and also as~$\VV(R)$ for an exchange ring~$R$, but never as~$\VV(B)$ for either a C*-algebra of real rank zero or a regular ring~$B$. Due to the cube $\set{0,1}^3$ having order-dimension three, the cardinality of this counterexample jumps up to~$\aleph_3$\index{s}{aleph0@$\aleph_{\ga}$}. It is conceivable, although yet unknown, that refining the methods used could yield a counterexample of cardinality~$\aleph_2$\index{s}{aleph0@$\aleph_{\ga}$}. On the other hand, it is well-known that positive cones of dimension groups of cardinality~$\aleph_1$\index{s}{aleph0@$\aleph_{\ga}$} do not separate the nonstable K$_0$-theories\index{i}{nonstable K- (or K$_0$-) theory} of exchange rings\index{i}{exchange ring}, C*-algebras\index{i}{C*-algebra} of real rank zero, and regular rings\index{i}{regular ring}.
\end{exple}

\begin{exple}\label{Ex:CLPaleph2}
The original solution to CLP\index{i}{Congruence Lattice Problem (CLP)} (see Wehrung~\cite{CLP}\index{c}{Wehrung, F.}) produces a distributive\index{i}{distributive!semilattice} \jzus\ of cardinality~$\aleph_{\go+1}$\index{s}{aleph0@$\aleph_{\ga}$} that is not isomorphic to the compact congruence semilattice of any lattice. The bound is improved to the optimal one, namely~$\aleph_2$\index{s}{aleph0@$\aleph_{\ga}$}, by R\r{u}\v{z}i\v{c}ka\index{c}{Ruzicka@R\r{u}\v{z}i\v{c}ka, P.} in~\cite{Ruzi08}.

On the other hand, every square\index{i}{square (shape of a diagram)} of finite \jzs s with \jzh s can be lifted\index{i}{diagram!lifted} with respect to the~$\Conc$\index{s}{compcon1@$\Conc\bA$, $\Conc f$} functor on lattices, see Gr\"atzer, Lakser, and Wehrung~\cite{GLWe}\index{c}{Gr\"atzer, G.}\index{c}{Lakser, H.}\index{c}{Wehrung, F.}. It is not known whether this result can be extended to $\set{0,1}^3$-indexed diagrams.
\end{exple}

Many of the combinatorial patterns encountered in Examples~\ref{Ex:PermCong}--\ref{Ex:CLPaleph2} appear in the following notion, first formulated within Problem~5 in the survey paper T\r{u}ma and Wehrung~\cite{CLPSurv}\index{c}{Tuma@T\r{u}ma, J.}\index{c}{Wehrung, F.} and then extensively studied by the first author in \cite{GillTh,Gill1,Gill2,Gill3,Gill4}\index{c}{Gillibert, P.}.

\begin{defn}\label{D:CritPtVar}
For varieties~$\cA$ and~$\cB$ of algebras\index{i}{algebra!universal} \pup{not necessarily on the same similarity type}, we define
\begin{itemize}
\item $\Conc\cA$\index{s}{compcon2@$\Conc\cA$|ii} is the class of all \jzs s that are isomorphic to~$\Conc\bA$\index{s}{compcon1@$\Conc\bA$, $\Conc f$}, for some $\bA\in\cA$;

\item $\crit(\cA;\cB)$\index{s}{critAB@$\crit(\cA;\cB)$|ii}, the \emph{critical point}\index{i}{critical point|ii} of~$\cA$ and~$\cB$, is the least possible cardinality of a \jzs\ in $(\Conc\cA)\setminus(\Conc\cB)$\index{s}{compcon2@$\Conc\cA$} if $\Conc\cA\not\subseteq\Conc\cB$, $\infty$ otherwise.
\end{itemize}
\end{defn}

The following result is proved by the first author in \cite[Corollary~7.13]{Gill1}\index{c}{Gillibert, P.}.

\begin{thm}[Gillibert]\label{T:critptalephn}
Let~$\cA$ and~$\cB$ be varieties of algebras with~$\cA$ locally finite\index{i}{variety!locally finite} and~$\cB$ finitely generated\index{i}{variety!finitely generated} congruence-distributive\index{i}{variety!congruence-distributive}. Then $\Conc\cA\not\subseteq\Conc\cB$\index{s}{compcon2@$\Conc\cA$} implies that $\crit(\cA;\cB)<\aleph_\go$\index{s}{critAB@$\crit(\cA;\cB)$}\index{s}{aleph0@$\aleph_{\ga}$}.
\end{thm}

The results of the present book make it possible to extend Theorem~\ref{T:critptalephn} considerably in Gillibert~\cite{Gill5}\index{c}{Gillibert, P.}, yielding the following result (proved there in the more general context of \emph{quasivarieties} and \emph{relative congruence lattices}).

\begin{thm}[Gillibert]\label{T:critptaleph2}
Let~$\cA$ and~$\cB$ be locally finite\index{i}{variety!locally finite} varieties of algebras such that for each $\bA\in\cA$ there are only finitely many \pup{up to isomorphism} $\bB\in\cB$ such that $\Conc\bA\cong\Conc\bB$, and every such~$\bB$ is finite. Then $\Conc\cA\not\subseteq\Conc\cB$\index{s}{compcon2@$\Conc\cA$} implies that $\crit(\cA;\cB)\leq\aleph_2$\index{s}{aleph0@$\aleph_{\ga}$}\index{s}{critAB@$\crit(\cA;\cB)$}.
\end{thm}

Due to known examples with varieties of lattices, the bound~$\aleph_2$\index{s}{aleph0@$\aleph_{\ga}$} in Theorem~\ref{T:critptaleph2} is sharp. If~$\cB$ is, in addition, finitely generated with finite similarity type, then, due to results from \cite{HoMK88}\index{c}{Hobby, D.}\index{c}{McKenzie, R.\,N.}, the condition of Theorem~\ref{T:critptaleph2} holds in case~$\cB$ omits tame congruence theory types~$\mathbf{1}$ and~$\mathbf{5}$, which in turns holds if~$\cB$ satisfies a nontrivial congruence lattice identity. In particular, this holds in case~$\cB$ is a finitely generated variety of groups, lattices, loops, or modules over a finite ring.

As to the present writing, the only known possibilities for the critical point\index{i}{critical point} between two varieties of algebras, on finite similarity types, are either finite, $\aleph_0$, $\aleph_1$, $\aleph_2$\index{s}{aleph0@$\aleph_{\ga}$}, or~$\infty$ (cf. Problem~\ref{Pb:DichotCritPt}). The proofs of Theorems~\ref{T:critptalephn} and~\ref{T:critptaleph2} involve a deep analysis of the relationship between liftability of \emph{objects} and liftability of \emph{diagrams} with respect to the~$\Conc$\index{s}{compcon1@$\Conc\bA$, $\Conc f$} functor on algebras.

\subsection{Contents of the book}\label{Su:Contents}
The main result of this work, the Condensate Lifting Lemma (CLL, Lemma~\ref{L:CLL})\index{i}{Condensate Lifting Lemma (CLL)}, is a complex, unfriendly-looking categorical statement, that most readers might, at first sight, discard as very unlikely to have any application to any previously formulated problem. Such is CLL's\index{i}{Condensate Lifting Lemma (CLL)} primary precursor, the \emph{Armature Lemma} (Lemma~\ref{L:Armature})\index{i}{Armature Lemma}. The formulation of CLL's\index{i}{Condensate Lifting Lemma (CLL)} secondary precursor, the \emph{Buttress Lemma} (Lemma~\ref{L:Buttress})\index{i}{Buttress Lemma}, is even more technical, although its proof is easy and its intuitive content can be very roughly described as a diagram version of the L\"owenheim-Skolem Theorem\index{i}{Lowenheim@L\"owenheim-Skolem Theorem} of model theory.

Therefore, although most of the technical difficulty of our book lies in thirty pages of a relatively easy preparation in category theory, plus a combined six-page proof of CLL\index{i}{Condensate Lifting Lemma (CLL)} together with its precursors (the Armature Lemma\index{i}{Armature Lemma}---Lemma~\ref{L:Armature}, and the Buttress Lemma\index{i}{Buttress Lemma}---Lemma~\ref{L:Buttress}), together with two sections on infinite combinatorics, a large part of our book will consist of defining and using contexts of possible applications of CLL\index{i}{Condensate Lifting Lemma (CLL)}.

Our largest such ``applications'' chapter is Chapter~\ref{Ch:FirstOrd2Lard}, that deals with first-order structures. We are having in mind the Gr\"atzer-Schmidt\index{c}{Gr\"atzer, G.}\index{c}{Schmidt, E.\,T.} Theorem~\cite{GrSc62}\index{c}{Gr\"atzer, G.}\index{c}{Schmidt, E.\,T.}, that says that every \jzs\ is isomorphic to~$\Conc\bA$\index{s}{compcon1@$\Conc\bA$, $\Conc f$} for some algebra\index{i}{algebra!universal}~$\bA$, with no possibility of assigning a specified similarity type to the algebra~$\bA$ (this caveat being due to the main result of Freese, Lampe, and Taylor~\cite{FLT}\index{c}{Freese, R.}\index{c}{Lampe, W.\,A.}\index{c}{Taylor, W.}). Hence we consider the category of all first-order structures as a whole, with the existence of a homomorphism from~$\bA$ to~$\bB$ requiring the language (i.e., similarity type) of~$\bA$ to be contained in the language of~$\bB$. We denote by~$\MIND$\index{s}{Mind@$\MIND$} the category thus formed, and we call it the category of all \emph{monotone-indexed structures}. As applications of CLL\index{i}{Condensate Lifting Lemma (CLL)} in that context, we point the following:

\begin{itemize}
\item\textup{(Extending the Gr\"atzer-Schmidt\index{c}{Gr\"atzer, G.}\index{c}{Schmidt, E.\,T.} Theorem to poset-indexed diagrams of \jzs s and \jzh s, Theorem~\ref{T:MindConcLift})} {\em Every diagram of \jzs s and \jzh s, indexed by a finite poset~$P$, can be lifted\index{i}{diagram!lifted}, with respect to the~$\Conc$\index{s}{compcon1@$\Conc\bA$, $\Conc f$} functor, by a diagram of algebras\index{i}{algebra!universal} \pup{with variable similarity types}}. In the presence of large cardinals \pup{e.g., existence of a proper class of Erd\H os cardinals}\index{i}{Erd\H{o}s cardinal}\index{c}{Erdos@Erd\H os, P.}, the finiteness assumption on~$P$ can be removed. We emphasize that it sounds reasonable that a closer scrutiny of the proof of the Gr\"atzer-Schmidt\index{c}{Gr\"atzer, G.}\index{c}{Schmidt, E.\,T.} Theorem could conceivably lead to a direct functorial proof of that result. However, our proof that the result for objects implies the result for diagrams requires no knowledge about the details of the proof of the Gr\"atzer-Schmidt\index{c}{Gr\"atzer, G.}\index{c}{Schmidt, E.\,T.} Theorem.

\item Theorem~\textup{\ref{T:critptalephn}} \pup{about critical points\index{i}{critical point} being either~$\infty$ or less than~$\aleph_\go$}\index{s}{aleph0@$\aleph_{\ga}$} can be, in a host of situations, extended to quasivarieties of algebraic systems\index{i}{algebraic system} and the relative congruence lattice functor \textup{(cf. Theorem~\ref{T:DichotCritPt})}. By using the results of \emph{commutator theory for congruence-modular varieties}~\cite{FrMK87}\index{c}{Freese, R.}\index{c}{McKenzie, R.\,N.}\index{i}{commutator theory}, we prove that this applies to~$\cA$ being a \emph{locally finite quasivariety with finitely many relation symbols}\index{i}{quasivariety!locally finite} and~$\cB$ a \emph{finitely generated, \pup{say} congruence-modular variety of algebras with finite type}\index{i}{variety!locally finite}\index{i}{variety!congruence-modular} (cf. Theorem~\ref{T:DichotNonType15}). In particular, $\cB$ can be a finitely generated variety of \emph{groups}, \emph{modules} (over a finite ring), \emph{loops}, \emph{lattices}, and so on. Actually, by using the results of \cite{HoMK88}\index{c}{Hobby, D.}\index{c}{McKenzie, R.\,N.}, we can even formulate our result in the more general context of varieties avoiding the \emph{tame congruence theory}\index{i}{tame congruence theory} types~$\mathbf{1}$ and~$\mathbf{5}$ (instead of just congruence-modular); on the other hand it is not so easy to find such examples which are not congruence-modular (Polin's variety is such an example). However, even for groups, modules, or lattices, the result of Theorem~\ref{T:DichotNonType15} is highly non-trivial.
\end{itemize}

The main part of Chapter~\ref{Ch:FirstOrd2Lard} consists of checking, one after another, the many conditions that need to be verified for our various applications of~CLL\index{i}{Condensate Lifting Lemma (CLL)}, as well as for new potential ones. Most of these verifications are elementary, with the possible exception of a condition, called the ``L\"owenheim-Skolem Condition''\index{i}{Lowenheim@L\"owenheim-Skolem Condition} and denoted by either $(\LS_\gm(B))$\index{s}{LS@$(\LS_\gm(B))$} (for larders)\index{i}{larder} or $(\LSr_{\gm}(B))$\index{s}{LSr@$(\LSr_\gm(B))$} (for right larders)\index{i}{larder!right}.

Another application chapter of CLL\index{i}{Condensate Lifting Lemma (CLL)} is Chapter~\ref{Ch:CongPres}. Its main purpose is to solve the problem, until now open, whether every lattice of cardinality~$\aleph_1$\index{s}{aleph0@$\aleph_{\ga}$} has a congruence-permutable\index{i}{algebra!congruence-permutable}, congruence-preserving extension. (By definition, an algebra\index{i}{algebra!universal}~$\bB$ is a \emph{congruence-preserving extension}\index{i}{congruence-preserving extension|ii} of a subalgebra~$\bA$ if every congruence of~$\bA$ extends to a unique congruence of~$\bB$.) The solution turns out to be negative:

\begin{itemize}\em
\item Let $\cV$ be a nondistributive\index{i}{distributive!non-${}_{-}$ variety} variety of lattices. Then the free lattice \pup{resp. free bounded lattice} on~$\aleph_1$\index{s}{aleph0@$\aleph_{\ga}$} generators within~$\cV$ has no congruence-per\-mut\-a\-ble, congruence-preserving extension\index{i}{congruence-preserving extension} \textup{(cf. Corollary~\textup{\ref{C:NoCPCP}})}.
\end{itemize}

Once again, the largest part of Chapter~\ref{Ch:CongPres} consists of checking, one after another, the various conditions that need to be verified for our application of~CLL\index{i}{Condensate Lifting Lemma (CLL)}. Most of these verifications are elementary. Once they are performed, a diagram counterexample, described in Section~\ref{S:UnliftMetr}, can be turned to an object counterexample.

Our final application chapter is Chapter~\ref{Ch:RegRngLard}, which deals with the context of (von~Neumann) regular rings\index{i}{regular ring}. The functor in question is the one, denoted here by~$\LL$, that sends every regular ring\index{i}{regular ring}~$R$ to the lattice~$\LL(R)$\index{s}{LLR@$\LL(R)$, $\LL(f)$} of all its principal right ideals\index{i}{ideal (in a ring)!right}. Aside from paving the road for further work on coordinatization, one of the goals of Chapter~\ref{Ch:RegRngLard} is to provide ``black box''-like tools. One of these tools enables the second author to prove in~\cite{Banasch2}\index{c}{Wehrung, F.} the following statement, thus solving a problem of J\'onsson stated in his 1962 paper~\cite{Jons62}\index{c}{Jonsson@J\'onsson, B.}:

\begin{itemize}\em
\item There exists a sectionally complemented\index{i}{lattice!sectionally complemented} modular\index{i}{lattice!modular} lattice, with a large $4$-frame, of cardinality~$\aleph_1$\index{s}{aleph0@$\aleph_{\ga}$}, that is not coordinatizable\index{i}{lattice!coordinatizable}.
\end{itemize}

Let us now go back to the categorical framework underlying CLL\index{i}{Condensate Lifting Lemma (CLL)}. The context in which this result lives is very much related to the one of the monograph by Ad\'amek and Rosick\'y~\cite{AdRo}\index{c}{Ad\'amek, J.}\index{c}{Rosick\'y, J.}. The starting assumptions of CLL\index{i}{Condensate Lifting Lemma (CLL)} are categories~$\cA$, $\cB$, and~$\cS$, together with functors~$\Phi\colon\cA\to\cS$ and $\Psi\colon\cB\to\cS$. We are also given a poset (=partially ordered set)~$P$ and a diagram
 \[
 \overrightarrow{A}=\famm{A_p,\ga_p^q}{p\leq q\text{ in }P}
 \]
from~$\cA$. The \emph{raison d'\^etre} of CLL\index{i}{Condensate Lifting Lemma (CLL)} is to construct a certain object~$A$ of~$\cA$ such that any lifting\index{i}{diagram!lifting}, with respect to the functor~$\Psi$, of the \emph{object}~$\Phi(A)$ creates a lifting\index{i}{diagram!lifting}, with respect to~$\Psi$, of the \emph{diagram} $\Phi\overrightarrow{A}$.

Of course, this can be done only under additional conditions, on the categories~$\cA$, $\cB$, $\cS$, on the functors~$\Phi$ and~$\Psi$, but also on the poset~$P$.

The object~$A$, denoted in the statement of CLL\index{i}{Condensate Lifting Lemma (CLL)} by the notation~$\xF(X)\otimes\overrightarrow{A}$\index{s}{FxX@$\xF(X)$}\index{s}{otimAS@$\bA\otimes\overrightarrow{S}$, $\gf\otimes\overrightarrow{S}$}, is a so-called \emph{condensate}\index{i}{condensate} of the diagram~$\overrightarrow{A}$. It concentrates, in one object, enough properties of the diagram~$\overrightarrow{A}$ to imply our statement on liftability. It is the directed colimit of a suitable diagram of finite products of the~$A_p$, for $p\in P$ (cf. Section~\ref{S:AtensS}). The notion of \emph{lifting}\index{i}{diagram!lifting} itself requires a modification, \emph{via} the introduction of a suitable subcategory, denoted by~$\cS^\Rightarrow$\index{s}{RightarrowCat@$\cS^\Rightarrow$}, of~$\cS$, whose arrows are called \emph{double arrows}\index{i}{double arrow}. While, in applications such as Theorems~\ref{T:MindConcLift} and~\ref{T:DichotCritPt}, double arrows can always be reduced to single arrows (isomorphisms) \emph{via} a condition called \emph{projectability} (cf. Section~\ref{S:ProjWit}), double arrows cannot always be eliminated, the most prominent example being Corollary~\ref{C:NoCPCP} (about lattices of cardinality~$\aleph_1$\index{s}{aleph0@$\aleph_{\ga}$} without congruence-permutable\index{i}{algebra!congruence-permutable} congruence-preserving extensions\index{i}{congruence-preserving extension}).

The construction and basic properties of the condensate\index{i}{condensate} $\xF(X)\otimes\overrightarrow{A}$\index{s}{FxX@$\xF(X)$}\index{s}{otimAS@$\bA\otimes\overrightarrow{S}$, $\gf\otimes\overrightarrow{S}$} are explained in Chapter~\ref{Ch:PscaledBAs}. This requires a slight expansion of the category of Boolean algebras\index{i}{algebra!Boolean}. This expansion involves a poset parameter~$P$, and it is the dual of the notion of what we shall call a \emph{$P$-normed Boolean space}\index{i}{normed (Boolean) space}\index{i}{PnormBS@$P$-normed Boolean space|seeonly{Boolean space, normed}} (cf. Definition~\ref{D:Pnorm}). By definition, a $P$-normed\index{i}{normed (Boolean) space} space is a topological space~$X$, endowed with a map~$\nu$ (the ``norm'') from~$X$ to the set~$\Id P$\index{s}{IdP@$\Id P$, $P$ poset} of \emph{ideals}\index{i}{ideal!of a poset} (nonempty, directed lower subsets\index{i}{lower subset}) of~$P$ such that $\setm{x\in X}{p\in\nu(x)}$ is open for each $p\in P$.

However, the statement of CLL\index{i}{Condensate Lifting Lemma (CLL)} involves directed \emph{colimits}, not directed \emph{limits}. Hence we chose to formulate most of the results of Chapter~\ref{Ch:PscaledBAs} in the context of the dual objects of $P$-normed\index{i}{normed (Boolean) space} Boolean spaces, which we shall call \emph{$P$-scaled Boolean algebras}\index{i}{algebra!Pscaled Boolean@$P$-scaled Boolean} (Definition~\ref{D:BoolP}). The duality between the category~$\BTop_P$\index{s}{BTop@$\BTop_P$} of $P$-normed\index{i}{normed (Boolean) space} Boolean spaces and the category~$\Bool_P$\index{s}{BoolP@$\Bool_P$} of $P$-scaled Boolean algebras\index{i}{algebra!Pscaled Boolean@$P$-scaled Boolean} is (easily) established in Section~\ref{S:Pnorm}. All the other results of Chapter~\ref{Ch:PscaledBAs} are formulated ``algebraically'', that is, within the category~$\Bool_P$\index{s}{BoolP@$\Bool_P$}. Furthermore, the introduction of the ``free $P$-scaled Boolean algebras''\index{i}{algebra!Pscaled Boolean@$P$-scaled Boolean} (on certain generators and relations) $\xF(X)$\index{s}{FxX@$\xF(X)$} (cf. Section~\ref{S:twoxF}) is far more natural in the algebraic context~$\Bool_P$\index{s}{BoolP@$\Bool_P$} than in the topological one~$\BTop_P$\index{s}{BTop@$\BTop_P$}. Chapter~\ref{Ch:PscaledBAs} is still to be considered only as an introduction of the CLL\index{i}{Condensate Lifting Lemma (CLL)} framework.

The statement of CLL\index{i}{Condensate Lifting Lemma (CLL)} involves three categorical add-ons to the initial data $\cA$, $\cB$, $\cS$, $\Phi$, and~$\Psi$. One of them we already discussed is the subcategory~$\cS^\Rightarrow$\index{s}{RightarrowCat@$\cS^\Rightarrow$} of ``double arrows''\index{i}{double arrow} in~$\cS$. The other two are the class~$\cA^\dagger$ (resp., $\cB^\dagger$) of all ``small'' objects of~$\cA$ (resp., $\cB$). In case the cardinal parameter~$\gl$ is equal to~$\aleph_0$\index{s}{aleph0@$\aleph_{\ga}$}, then~$\cA^\dagger$ (resp., $\cB^\dagger$) may be thought of as the ``finite'' objects of~$\cA$ (resp.~$\cB$).

The statement of CLL\index{i}{Condensate Lifting Lemma (CLL)} also involves a condition on the poset~$P$. This condition is stated as~$P$ having a ``$\gl$-lifter''\index{i}{lifter ($\gl$-)}. This condition, introduced in Definition~\ref{D:Lifter}, has combinatorial nature, inspired by Kuratowski's\index{c}{Kuratowski, C.} Free Set Theorem~\cite{Kura51}.

There are posets, even finite, for which this is forbidden---that is, there are no lifters\index{i}{lifter ($\gl$-)}. Actually, the finite posets for which there are lifters are exactly the finite disjoint unions of finite posets with zero in which every principal ideal\index{i}{ideal!of a poset}\index{i}{ideal!principal ${}_{-}$, of a poset} is a \js\ (cf. Corollary~\ref{C:LiftFinPos}). An example of a finite poset for which this does not hold is represented on the right hand side of Figure~\ref{Fig:posets}, page~\pageref{Fig:posets}. For infinite posets, the situation is more complicated: even for the dual chain of $\go+1=\set{0,1,2,\dots}\cup\set{\go}$\index{s}{omega1op@$(\omega+1)^{\op}$}, which is trivially a poset with zero in which every principal ideal\index{i}{ideal!of a poset}\index{i}{ideal!principal ${}_{-}$, of a poset} is a \js, there is no $(2^{\aleph_0})^+$-lifter\index{s}{aleph0@$\aleph_{\ga}$}\index{i}{lifter ($\gl$-)} (cf. Corollary~\ref{C:LiftnotWFrestr}).

Nevertheless our theory can be developed quite far, even for infinite posets. Three classes of posets emerge: the \emph{\pjs s}\index{i}{pseudo join-semilattice}, the \emph{supported posets}\index{i}{poset!supported}, and the \emph{\ajs s}\index{i}{almost join-semilattice} (Definition~\ref{D:PJS}). \emph{Intriguingly, some of our definitions are closely related to definitions used in domain theory\index{i}{continuous domain}\index{i}{continuous domain}, see, in particular, \textup{\cite[Chapter~4]{AbJu}}}\index{c}{Abramsky, S.}\index{c}{Jung, A.}. Every \ajs\index{i}{almost join-semilattice}\ is supported\index{i}{poset!supported}, and every supported\index{i}{poset!supported} poset is a \pjs\index{i}{pseudo join-semilattice}; none of the converse implications hold. Every poset with a $\gl$-lifter\index{i}{lifter ($\gl$-)} (where~$\gl$ is an infinite cardinal) is the disjoint union of finitely many \ajs s\index{i}{almost join-semilattice} with zero (Proposition~\ref{P:NoBowTie}). In the finite case, or even in the infinite case provided~$P$ is \emph{lower finite}\index{i}{poset!lower finite} (meaning that every principal ideal\index{i}{ideal!of a poset}\index{i}{ideal!principal ${}_{-}$, of a poset} of~$P$ is finite) and in the presence of a suitable large cardinal assumption, the converse holds (Corollary~\ref{C:CharLift}). Some of the infinite combinatorial aspects of lifters, involving our $(\gk,{<}\gl)\leadsto P$\index{s}{arr0x@$(\gk,{<}\gl)\leadsto P$} notation (cf. Definition~\ref{D:InfCombP}), are of independent interest and are developed further in our paper~\cite{GiWe1}\index{c}{Gillibert, P.}\index{c}{Wehrung, F.}.

Part of the conclusion of CLL\index{i}{Condensate Lifting Lemma (CLL)}, namely the one about going from ``object representation'' to ``diagram representation'', can be reached even for diagrams indexed by posets without lifters\index{i}{lifter ($\gl$-)} (but still \ajs s\index{i}{almost join-semilattice}), at the expense of losing CLL's\index{i}{Condensate Lifting Lemma (CLL)} condensate\index{i}{condensate} $\xF(X)\otimes\overrightarrow{A}$\index{s}{FxX@$\xF(X)$}\index{s}{otimAS@$\bA\otimes\overrightarrow{S}$, $\gf\otimes\overrightarrow{S}$}. This result is presented in Corollary~\ref{C:CLLnoLF}. This corollary involves the large cardinal axiom denoted, using the notation from the Erd\H{o}s, Hajnal, M\'at\'e, and Rado monograph~\cite{EHMR}\index{c}{Erd\H{o}s, P.}\index{c}{Hajnal, A.}\index{c}{Mate@M\'at\'e, A.}\index{c}{Rado, R.}, by $(\gk,{<}\go,\gl)\rightarrow\gl$\index{s}{arr0free@$(\gk,{<}\go,\gl)\rightarrow\gr$}. It is, for example, involved in the proof of Theorem~\ref{T:MindConcLift}.

\subsection{How not to read the book}\label{S:NottoRead}
In the present section we would like to give some hints about how to use the present work as a toolbox, ideally a list of ``black box principles'' that would enable the reader who, although open-minded, is not necessarily in the mood to go through all the details of our book, to solve his or her own problem. As our work involves a fair amount of topics as distant as category theory, universal algebra, or infinite combinatorics, such qualms are not unlikely to occur.

In order for our work to have any relevance to our reader's problem, it is quite likely that this problem would still involve categories~$\cA$, $\cB$, $\cS$ together with functors $\Phi\colon\cA\to\cS$ and $\Psi\colon\cB\to\cS$. In addition, the reader's problem might also involve suitable choices of a subcategory~$\cS^\Rightarrow$\index{s}{RightarrowCat@$\cS^\Rightarrow$} of~$\cS$ (the ``double arrows''\index{i}{double arrow} of~$\cS$) together with classes of ``small objects'' $\cA^\dagger\subseteq\cA$, $\cB^\dagger\subseteq\cB$.

Then, the reader may need to relate liftability, with respect to the functor~$\Psi$ (and possibly the class of all double arrows\index{i}{double arrow} of~$\cS$), of \emph{objects} of the form~$\Phi(A)$, and of \emph{diagrams} of the form~$\Phi\overrightarrow{A}$.

A typical case is the following. The reader may have found, through previous research, a diagram~$\overrightarrow{A}$ from the category~$\cA$, indexed by a poset~$P$, such that there are no $P$-indexed diagram~$\overrightarrow{B}$ from~$\cB$ and no double arrow\index{i}{double arrow} $\Psi\overrightarrow{B}\Rightarrow\Phi\overrightarrow{A}$. This is, for example, the case for the family of square\index{i}{square (shape of a diagram)} diagrams of Lemma~\ref{L:UnliftSqMetr} (about CPCP-retracts\index{i}{CPCP-retract}) or a certain $\go_1$-indexed diagram of sectionally complemented\index{i}{lattice!sectionally complemented} modular\index{i}{lattice!modular} lattices in~\cite{Banasch2}\index{c}{Wehrung, F.}.

A preliminary question that needs to be addressed is the following. \emph{Is the poset~$P$ an \ajs}\index{i}{almost join-semilattice} (cf. Definition~\ref{D:PJS}), or, at least, can the problem be reduced to the case where~$P$ is an \ajs\index{i}{almost join-semilattice}\ (like in Theorem~\ref{T:MindConcLift})?

If not, then, as to the present writing, there is no form of CLL\index{i}{Condensate Lifting Lemma (CLL)} that can be used (cf. Problem~\ref{Pb:CLLBowTie} in Chapter~\ref{Ch:Discussion}).

If yes, then this looks like a case where CLL\index{i}{Condensate Lifting Lemma (CLL)} (or its pair of ancestors, the Armature Lemma\index{i}{Armature Lemma} and the Buttress Lemma\index{i}{Buttress Lemma}) could be applied. The hard core of the present book lies in the proof of CLL\index{i}{Condensate Lifting Lemma (CLL)} (Lemma~\ref{L:CLL}), however the reader may not need to go through the details thereof.

The setting up of the cardinal parameters~$\gl$ and $\gk:=\card X$ depends of the nature of the problem. The infinite combinatorial aspects, which can be looked up in Section~\ref{S:Pos2Lift} and~\ref{S:LiftRetr}, are articulated around the notion of a \emph{$\gl$-lifter}\index{i}{lifter ($\gl$-)}. The existence of a lifter is an essential assumption in the statement of both CLL and the Armature Lemma. A more amenable variant of the existence of a lifter, the $(\gk,{<}\gl)\leadsto P$\index{s}{arr0x@$(\gk,{<}\gl)\leadsto P$} relation, is introduced in our paper Gillibert and Wehrung~\cite{GiWe1}\index{c}{Gillibert, P.}\index{c}{Wehrung, F.} and recalled in Definition~\ref{D:InfCombP}.

As most of the difficulty underlying the formulation of CLL\index{i}{Condensate Lifting Lemma (CLL)} (Lemma~\ref{L:CLL}) is related to the definition of a larder\index{i}{larder}, we split that definition in two parts, namely \emph{left larders}\index{i}{larder!left} and \emph{right larders}\index{i}{larder!right} (cf. Section~\ref{S:LeftRightL}). Roughly speaking, the left larder corresponds to the left hand side arrow $\Phi\colon\cA\to\cS$ of the left part of the diagram of Figure~\ref{Fig:Prelarder}, while the right larder corresponds to the right hand side arrow $\Psi\colon\cB\to\cS$.

In all cases encountered so far, left larderhood\index{i}{larder!left} never led to any difficulty in verification. The situation is different for right larders\index{i}{larder!right}, that involve a deeper analysis of the structures involved. In such applications as the one in Wehrung~\cite{Banasch2}\index{c}{Wehrung, F.}, it is important to keep some control on how the $P$-scaled Boolean algebra\index{i}{algebra!Pscaled Boolean@$P$-scaled Boolean}~$\xF(X)$\index{s}{FxX@$\xF(X)$} and the condensate\index{i}{condensate} $\xF(X)\otimes\overrightarrow{A}$\index{s}{FxX@$\xF(X)$}\index{s}{otimAS@$\bA\otimes\overrightarrow{S}$, $\gf\otimes\overrightarrow{S}$} are created. For such notions, the relevant part of our work to look up is Chapter~\ref{Ch:PscaledBAs}.

In writing Chapter~\ref{Ch:FirstOrd2Lard}, we made the bet that many right larders\index{i}{larder!right} would involve \emph{first-order structures} and \emph{congruence lattices} (with relations possibly allowed), and so we included in that chapter some detail about most of the basic tools required for checking those larders\index{i}{larder}. In that chapter, the functor~$\Psi$ has to be thought of as the relative compact congruence semilattice functor\index{i}{relative!compact congruence semilattice functor} (cf. Definition~\ref{D:RelCritPoint}) in a given quasivariety, or even in a \emph{\gqv}\index{i}{generalized quasivariety} (cf. Definition~\ref{D:GQV}). While some results in that chapter are already present in references such as Gorbunov's monograph~\cite{Gorb}\index{c}{Gorbunov, V.\,A.}, the references for some other items can be quite hard to trace---such as the description of directed colimits in~$\MIND$\index{s}{Mind@$\MIND$} (Proposition~\ref{P:DirColimMIND}) or the relative congruence lattice functor on a \gqv\index{i}{generalized quasivariety}\ (Section~\ref{S:ConFunct}). Further results of Chapter~\ref{Ch:FirstOrd2Lard} also appear in print here for the first time (to our knowledge)---such as the preservation of all small directed colimits by the~$\ConcV$\index{s}{compcongV@$\ConcV$ functor} functor (Theorem~\ref{T:ConcVPresDirColim}). Most results in Chapter~\ref{Ch:FirstOrd2Lard} are articulated around Theorem~\ref{T:MindConcLift} (diagram version of the Gr\"atzer-Schmidt\index{c}{Gr\"atzer, G.}\index{c}{Schmidt, E.\,T.} Theorem), Theorem~\ref{T:RelCritalephn} (estimates of relative critical points \emph{via} combinatorial properties of indexing posets), and Theorem~\ref{T:DichotCritPt} (Dichotomy Theorem\index{i}{Dichotomy Theorem} for relative critical points between quasivarieties).

Chapter~\ref{Ch:RegRngLard} gives another class of right larders\index{i}{larder!right}, this time arising from the~$\LL$\index{s}{LLR@$\LL(R)$, $\LL(f)$} functor on regular rings\index{i}{regular ring} (cf. Example~\ref{Ex:Momega}). We also wrote that chapter in our ``toolbox'' spirit. Furthermore, although we believe that all of the basic results presented there already appeared in print somewhere in the case of unital regular rings\index{i}{regular ring}, this does not seem to be the case for non-unital rings, in particular Lemma~\ref{L:xRyxRNeutr} and thus Proposition~\ref{P:NIdequiv2Id} ($\Id R\cong\NId\LL(R)$)\index{s}{LLR@$\LL(R)$, $\LL(f)$}\index{s}{IdR@$\Id R$, $R$ ring}.

It is certainly Chapter~\ref{Ch:CongPres} that involves the most unusual kind of larder\index{i}{larder}, tailored to solve a specific problem about congruence-permutable\index{i}{algebra!congruence-permutable}, congruence-preserving extensions\index{i}{congruence-preserving extension} of lattices, thus perhaps providing the best illustration of the versatility of our larder tools. Indeed, $\Psi$ can be described there as the forgetful functor from a category of pairs~$(B^*,B)$ to the first component~$B^*$.

We end the present section with a list of all the places in our work stating larderhood\index{i}{larder} of either a structure or a class of structures.

\begin{itemize}
\item Proposition~\ref{P:MalgMetrLeftLard} (left larder, from algebras\index{i}{algebra!universal} and semilattice-metric spaces\index{i}{semilattice-metric!space}).

\item Theorem~\ref{T:MindConcLift} (Claim~\ref{Cl:GrSchleftlard} for a left larder, Claim~\ref{Cl:GrSchrightlard} for a right $\gl$-larder, for regular uncountable~$\gl$, from the compact congruence lattice functor on the category $\MALG_1$\index{s}{Malg1@$\MALG_1$} of all unary algebras).

\item Theorem~\ref{T:1stordLardCtble} ($\aleph_0$-larder\index{s}{aleph0@$\aleph_{\ga}$}, from the relative compact congruence semilattice functor\index{i}{relative!compact congruence semilattice functor} on a congruence-proper\index{i}{congruence-proper} and locally finite quasivariety\index{i}{quasivariety!locally finite}).

\item Theorem~\ref{T:Algorightlard} ($\gl$-larder, for any uncountable cardinal~$\gl$, from the relative compact congruence semilattice functor\index{i}{relative!compact congruence semilattice functor} on a quasivariety on a $\gl$-small language).

\item Proposition~\ref{P:MetrLarder} (right $\aleph_0$-larder\index{s}{aleph0@$\aleph_{\ga}$}, from a forgetful functor from semi\-lat\-tice-metric covers to semilattice-metric spaces\index{i}{semilattice-metric!space}).

\item Theorem~\ref{T:RR2Lard} (right $\gl$-larder, from the~$\LL$\index{s}{LLR@$\LL(R)$, $\LL(f)$} functor on a class of regular rings\index{i}{regular ring} with mild directed colimits-related assumptions and closure under homomorphic images).
\end{itemize}

The (trivial) statement that brings together left and right larders into larders\index{i}{larder} is Proposition~\ref{P:LR2Larder}.

\section{Basic concepts}\label{S:Basic}
\subsection{Set theory}\label{Su:BasicSet}
We shall use basic set-theoretical notation and terminology about ordinals and cardinals. 
We denote by $\dom f$\index{s}{dom f@$\dom f$|ii} (resp., $\rng f$\index{s}{rng f@$\rng f$|ii}) the domain (resp., the range) of a function~$f$, and by~$f``(X)$\index{s}{fiiimX@$f``(X)$, $f``X$|ii}, or~$f``X$ (resp., $f^{-1}X$)\index{s}{fiiivX@$f^{-1}X$|ii} the image (resp., inverse image) of a set~$X$ under~$f$. Cardinals are initial ordinals. We denote by~$\cf(\ga)$\index{s}{cofin@$\cf(\ga)$|ii} the cofinality\index{i}{cofinality} of an ordinal~$\ga$. We denote by $\go:=\set{0,1,2,\dots}$\index{s}{omega@$\omega$|ii} the first limit ordinal, mostly denoted by~$\aleph_0$ in case it is viewed as a cardinal. We denote by~$\gk^+$\index{s}{kappa+@$\gk^+$, $\gk^{+\ga}$|ii} the successor cardinal of a cardinal~$\gk$, and we define~$\gk^{+\ga}$, for an ordinal~$\ga$, by~$\gk^{+0}:=\gk$, $\gk^{+(\ga+1)}:=(\gk^{+\ga})^+$, and $\gk^{+\gl}:=\sup_{\ga<\gl}\gk^{+\ga}$ for every limit ordinal~$\gl$. We set $\aleph_\ga:=(\aleph_0)^{+\ga}$\index{s}{aleph0@$\aleph_{\ga}$|ii}, for each ordinal~$\ga$. We denote by~$\Pow(X)$\index{s}{PowX@$\Pow(X)$|ii} the powerset of a set~$X$, and we put\index{s}{la0@$[X]^\gk$|ii}\index{s}{la1@$[X]^{<\gk}$|ii}\index{s}{la2@$[X]^{\les\gk}$|ii}
 \begin{align*}
 [X]^\gk&:=\setm{Y\in\Pow(X)}{\card Y=\gk}\,,\\
 [X]^{<\gk}&:=\setm{Y\in\Pow(X)}{\card Y<\gk}\,,\\
 [X]^{{\les}\gk}&:=\setm{Y\in\Pow(X)}{\card Y\leq\gk}\,,
 \end{align*}
for every cardinal~$\gk$. A set~$X$ is \emph{$\gk$-small}\index{i}{kapsmall@$\gk$-small!set|ii} if $\card X<\nobreak\gk$. We shall often add an extra largest element~$\infty$ to the class of all cardinals, and we shall say, by convention, ``small'' (resp., ``finite'') instead of ``$\infty$-small'' (resp., ``$\aleph_0$-small\index{s}{aleph0@$\aleph_{\ga}$}'').

\subsection{Stone duality for Boolean algebras}\label{Su:Stone}
\index{i}{Stone duality}
A topological space is \emph{Boolean} if it is compact Hausdorff and every open subset is a union of clopen (i.e., closed open) subsets. We denote by $\Clop X$\index{s}{ClopX@$\Clop X$, $\Clop\bX$, $\Clop f$|ii} the Boolean algebra\index{i}{algebra!Boolean} of clopen subsets of a topological space~$X$. We denote by $\Ult B$\index{s}{UltB@$\Ult B$, $\Ult\bB$, $\Ult\gf$|ii} the Boolean space of all ultrafilters of a Boolean algebra\index{i}{algebra!Boolean}~$B$ (a topological space is \emph{Boolean} if it is compact Hausdorff and it has a basis consisting of clopen sets). The pair $(\Ult,\Clop)$ can be extended to the well-known \emph{Stone duality}\index{i}{Stone duality} between the category~$\Bool$\index{s}{Bool@$\Bool$|ii} of Boolean algebras\index{i}{algebra!Boolean} with homomorphisms of Boolean algebras\index{i}{algebra!Boolean} and the category~$\BTop$\index{s}{Btop@$\BTop$|ii} of Boolean spaces with continuous maps, in the following way. For a homomorphism~$\gf\colon A\to B$ of Boolean algebras\index{i}{algebra!Boolean}, we put\index{s}{UltB@$\Ult B$, $\Ult\bB$, $\Ult\gf$|ii}
 \begin{equation}\label{Eq:DefUltf}
 \Ult\gf\colon\Ult B\to\Ult A\,,\quad\frb\mapsto\gf^{-1}\frb\,.
 \end{equation}
For a continuous map $f\colon X\to Y$ between Boolean spaces, we put\index{s}{ClopX@$\Clop X$, $\Clop\bX$, $\Clop f$|ii}
 \begin{equation}\label{Eq:DefClopf}
 \Clop f\colon\Clop Y\to\Clop X\,,\quad V\mapsto f^{-1}V\,.
  \end{equation}
We denote by $\At A$\index{s}{AtA@$\At A$|ii} the set of \emph{atoms} of a Boolean algebra\index{i}{algebra!Boolean}~$A$.

\subsection{Partially ordered sets (posets) and lattices}
\label{Su:Posets}
All our posets will be nonempty. A poset~$P$ is the \emph{disjoint union} of a family $\famm{P_i}{i\in I}$ of sub-posets if $P=\bigcup\famm{P_i}{i\in I}$ and any element of~$P_i$ is incomparable with any element of~$P_j$, for all distinct indices~$i,j\in I$.
For posets~$P$ and~$Q$, a map $f\colon P\to Q$ is \emph{isotone}\index{i}{isotone|ii} (resp., \emph{antitone}\index{i}{antitone|ii}) if $x\leq y$ implies that $f(x)\leq f(y)$ (resp., $f(x)\geq f(y)$), for all $x,y\in P$. For $x,y\in P$, let $x\prec y$\index{s}{prec@$x\prec y$|ii} hold if $x<y$ and there is no~$z\in P$ such that $x<z<y$. We also say that~$y$ is an \emph{upper cover}\index{i}{cover!upper|ii} of~$x$ and that~$x$ is a \emph{lower cover}\index{i}{cover!lower|ii} of~$y$.

We denote by~$0_P$ the least element of~$P$ if it exists, and by~$\Min P$\index{s}{MinP@$\Min P$|ii} (resp., $\Max P$\index{s}{MaxP@$\Max P$|ii}) the set of all minimal (resp., maximal) elements of~$P$.
An element~$p$ in a poset~$P$ is \emph{\jirr}\index{i}{irreducible!join-|ii} if $p=\bigvee X$ implies that $p\in X$, for every (possibly empty) finite subset~$X$ of~$P$; we denote by $\J(P)$\index{s}{JP@$\J(P)$, $P$ poset|ii}
the set of all \jirr\ elements of~$P$, endowed with the induced partial ordering. \emph{\Mirr}\index{i}{irreducible!meet-|ii} elements are defined dually. We set\index{s}{arrd@$Q\dnw X$|ii}\index{s}{arru@$Q\upw X$|ii}\index{s}{arrdd@$Q\ddnw X$|ii}\index{s}{arruu@$Q\uupw X$|ii}\index{s}{arrD@$Q\Dnw X$|ii}\index{s}{arrU@$Q\Upw X$|ii}
 \begin{align*}
 Q\dnw X&:=\setm{q\in Q}{(\exists x\in X)(q\leq x)}\,,&
 Q\upw X&:=\setm{q\in Q}{(\exists x\in X)(q\geq x)}\,,\\
 Q\ddnw X&:=\setm{q\in Q}{(\exists x\in X)(q<x)}\,,&
 Q\uupw X&:=\setm{q\in Q}{(\exists x\in X)(q>x)}\,,\\
 Q\Dnw X&:=\setm{q\in Q}{(\forall x\in X)(q\leq x)}\,,&
 Q\Upw X&:=\setm{q\in Q}{(\forall x\in X)(q\geq x)}\,,
 \end{align*}
for all subsets~$Q$ and~$X$ of~$P$; in case $X=\set{a}$ is a singleton, then we shall write $Q\dnw a$ instead of~$Q\dnw\set{a}$, and so on. A subset~$Q$ of~$P$ is a \emph{lower subset of~$P$}\index{i}{lower subset|ii} (\emph{upper subset of~$P$}\index{i}{upper subset|ii}, \emph{cofinal in~$P$}\index{i}{subset!cofinal|ii}, respectively) if~$P\dnw Q=Q$ ($P\upw Q=Q$, $P\dnw Q=P$, respectively). A lower subset~$Q$ of~$P$ is \emph{finitely generated}\index{i}{finitely generated!lower subset|ii} if $Q=P\dnw X$ for some finite subset~$X$ of~$P$. The dual definition holds for finitely generated\index{i}{finitely generated!upper subset|ii} upper subsets of~$P$.

Joins (=suprema) and meets (=infima) in posets are denoted by~$\vee$ and~$\wedge$, respectively.

We say that the poset~$P$ is
\begin{itemize}
\item \emph{lower finite}\index{i}{poset!lower finite|ii} if $P\dnw a$ is finite for each $a\in P$;

\item \emph{well-founded}\index{i}{poset!well-founded|ii} if every nonempty subset of~$P$ has a minimal element (equivalently, $P$ has no strictly decreasing $\go$-sequence);

\item \emph{monotone $\gs$-complete}\index{i}{poset!monotone $\gs$-complete|ii} if every increasing sequence (indexed by the set~$\go$ of all natural numbers) of elements of~$P$ has a join;

\item \emph{directed}\index{i}{poset!directed|ii} if every finite subset of~$P$ has an upper bound in~$P$.
\end{itemize}

An \emph{ideal}\index{i}{ideal!of a poset|ii} of~$P$ is a nonempty, directed\index{i}{poset!directed}, lower subset\index{i}{lower subset} of~$P$; we denote by $\Id P$\index{s}{IdP@$\Id P$, $P$ poset|ii} the set of all ideals of~$P$, ordered by containment. Observe that $\Id P$\index{s}{IdP@$\Id P$, $P$ poset} need be neither a meet- nor a \js. A \emph{principal ideal}\index{i}{ideal!principal ${}_{-}$, of a poset|ii} of~$P$ is a subset of~$P$ of the form $P\dnw x$, for $x\in P$.

For a cardinal~$\gl$, we say that the poset~$P$ is 
\begin{itemize}
\item \emph{lower $\gl$-small}\index{i}{poset!lower $\gl$-small|ii} if $\card(P\dnw a)<\gl$ for each $a\in P$;

\item \emph{$\gl$-directed}\index{i}{poset!directedl@$\gl$-directed|ii} if every $X\subseteq P$ such that $\card X<\gl$ has an upper bound in~$P$.
\end{itemize}

\begin{defn}\label{D:glClosedCof}
For a monotone $\gs$-complete\index{i}{poset!monotone $\gs$-complete} poset~$P$, a subset $X\subseteq P$ is \emph{$\gs$-closed cofinal}\index{i}{subset!$\gs$-closed cofinal|ii} if the following conditions hold:
\begin{description}
\item[\tui] $X$ is cofinal in~$P$;

\item[\tuii] the least upper bound of any increasing sequence of elements of~$X$ belongs to~$X$.
\end{description}
\end{defn}

We will sometimes abbreviate the statement~(ii) above by saying that~$X$ is \emph{$\gs$-closed}.

(This notion is named after the corresponding classical notion of closed cofinal subsets in uncountable regular cardinals.)

\begin{prop}\label{P:ClCofClos}
Let~$P$ be a monotone $\gs$-complete\index{i}{poset!monotone $\gs$-complete} poset. Then the intersection of any at most countable collection of $\gs$-closed cofinal\index{i}{subset!$\gs$-closed cofinal} subsets of~$P$ is $\gs$-closed cofinal.
\end{prop}

\begin{proof}
Let $\famm{X_n}{n<\go}$ be a sequence of $\gs$-closed cofinal\index{i}{subset!$\gs$-closed cofinal} subsets of~$P$. As the intersection $X:=\bigcap\famm{X_n}{n<\go}$ is obviously $\gs$-closed, it remains to prove that it is cofinal. By reindexing the~$X_n$s, we may assume that the set $\setm{n<\go}{X_n=X_m}$ is infinite for each $m<\go$. Let $p_0\in P$. If~$p_n$ is already constructed, pick $p_{n+1}\in X_n$ such that $p_n\leq p_{n+1}$. The supremum $\bigvee\famm{p_n}{n<\go}$ lies above~$p_0$ and it belongs to~$X$.
\qed\end{proof}

We shall denote by $P^{\op}$\index{s}{Pop@$P^{\mathrm{op}}$|ii} the \emph{dual} of a poset~$P$, that is, the poset with the same underlying set as~$P$ and opposite order.

An element $a$ in a lattice~$L$ is \emph{compact}\index{i}{compact!element in a lattice|ii} if for every subset~$X$ of~$L$ such that~$\bigvee X$ exists and $a\leq\bigvee X$, there exists a finite subset~$Y$ of~$X$ such that $a\leq\bigvee Y$. We say that~$L$ is \emph{algebraic}\index{i}{lattice!algebraic|ii}\index{i}{algebraic lattice|seeonly{lattice, algebraic}} if it is complete and every element of~$L$ is a join of compact elements. A subset of a lattice~$L$ is \emph{algebraic}\index{i}{algebraic subset|ii} (see Gorbunov~\cite{Gorb}\index{c}{Gorbunov, V.\,A.}) if it is closed under arbitrary meets and arbitrary nonempty directed joins; in particular, any algebraic subset\index{i}{algebraic subset} of an algebraic lattice is an algebraic lattice\index{i}{lattice!algebraic} under the induced ordering (see \cite[Section~1.3]{Gorb}\index{c}{Gorbunov, V.\,A.}).

A lattice~$L$ with zero is \emph{sectionally complemented}\index{i}{lattice!sectionally complemented|ii} if for all $x\leq y$ in~$L$ there exists $z\in L$ such that $x\vee z=y$ while $x\wedge z=0$ (abbreviation: $y=x\oplus z$\index{s}{oplus@$y=x\oplus z$|ii}). Elements~$x$ and~$y$ in a lattice~$L$ are \emph{perspective}, in notation $x\sim y$\index{s}{sim@$x\sim y$|ii}, if there exists $z\in L$ such that $x\vee z=y\vee z$ while $x\wedge z=y\wedge z$. In case~$L$ is sectionally complemented\index{i}{lattice!sectionally complemented} and modular\index{i}{lattice!modular}, we may assume that $x\wedge z=y\wedge z=0$ while $x\vee z=y\vee z=x\vee y$. An ideal\index{i}{ideal!of a poset}~$I$ in a lattice is \emph{neutral}\index{i}{ideal!neutral|ii} if the sublattice of the ideal lattice of~$L$ generated by $\set{I,X,Y}$ is distributive\index{i}{lattice!distributive}, for all ideals~$X$ and~$Y$ of~$L$. If this holds, then~$I$ is a \emph{distributive ideal}\index{i}{distributive!ideal|ii}\index{i}{ideal!distributive|seeonly{distributive, ideal}} of~$L$, that is, the binary relation~$\equiv_I$ on~$L$ defined by
 \[
 x\equiv_Iy\ \Leftrightarrow\ (\exists u\in I)(x\vee u=y\vee u)\,,\text{for all }
 x,y\in L\,,
 \]
is a lattice congruence of~$L$. Then we denote by~$L/I$ the quotient lattice~$L/{\equiv_I}$, and we denote by $x/I:=x/{\equiv_I}$ the $\equiv_I$-equivalence class of~$x$, for each $x\in L$. In case~$L$ is sectionally complemented\index{i}{lattice!sectionally complemented} and modular\index{i}{lattice!modular}, $I$ is distributive\index{i}{distributive!ideal} if{f} it is neutral\index{i}{ideal!neutral}, if{f} $x\sim y$ and $x\in I$ implies that $y\in I$, for all $x,y\in L$ (cf. \cite[Theorem~III.13.20]{Birk79}\index{c}{Birkhoff, G.}).

For further unexplained notions in lattice theory we refer to the monographs by Gr\"atzer~\cite{GLT2,LTF}\index{c}{Gr\"atzer, G.}.

\subsection{Category theory}\label{Su:Categories}
Our categorical background will be mainly borrowed from Mac
Lane~\cite{McLa}\index{c}{Mac Lane, S.}, Ad\'amek and Rosick\'y \cite{AdRo}\index{c}{Ad\'amek, J.}\index{c}{Rosick\'y, J.}, Johnstone~\cite{John86}\index{c}{Johnstone, P.\,T.}, and from the second author's earlier papers \cite{Ultra,RetrLift}\index{c}{Wehrung, F.}.
We shall often use special symbols for special sorts of arrows in a given category~$\cC$:
\begin{itemize}
\item $f\colon A\mono B$\index{s}{AtomonoB@$f\colon A\mono B$|ii} for \emph{monomorphisms}, or (for concrete categories) \emph{one-to-one maps};

\item $f\colon A\onto B$\index{s}{AtoonB@$f\colon A\onto B$|ii} for \emph{epimorphisms}, or (for concrete categories) \emph{surjective maps};

\item $f\colon A\into B$\index{s}{AtoinB@$f\colon A\into B$|ii} for \emph{embeddings}, that is, usually, special classes of monomorphisms preserving some additional structure (e.g., \emph{order-embeddings});

\item $f\colon A\Rightarrow B$\index{s}{AtorightarrowB@$f\colon A\Rightarrow B$|ii} for \emph{double arrows}\index{i}{double arrow}, that is, morphisms in a distinguished subcategory~$\cC^\Rightarrow$\index{s}{RightarrowCat@$\cS^\Rightarrow$} of~$\cC$ (see the end of the present section);

\item $f\colon\xA\todot\xB$ for \emph{natural transformations}\index{s}{AtorightarrowdotB@$f\colon\xA\todot\xB$|ii} between functors;

\item $f\colon\xA\Todot\xB$\index{s}{AtoRightarrowdotB@$f\colon\xA\Todot\xB$|ii} for natural transformations consisting of families of double arrows.
\end{itemize}

We denote by $\Ob\cA$\index{s}{ObA@$\Ob\cA$|ii} (resp., $\Mor\cA$\index{s}{MorA@$\Mor\cA$|ii}) the class of objects (resp., morphisms) of a category~$\cA$. We say that~$\cA$ \emph{has small hom-sets} (resp., \emph{has $\gk$-small hom-sets}, where~$\gk$ is a cardinal) if the class of all morphisms from~$X$ to~$Y$ is a set (resp., a set with less than~$\gk$ elements), for any objects~$X$ and~$Y$ of~$\cA$.

A morphism $f\colon A\to B$ in a category~$\cC$ \emph{factors through
$g\colon B'\to B$} (or, if~$g$ is understood, \emph{factors through~$B'$}) if there exists $h\colon A\to B'$ such that $f=g\circ h$. For an object~$C$ of~$\cC$ and a subcategory~$\cC'$ of~$\cC$, the \emph{comma category}\index{i}{comma category|ii}\index{s}{comma@$\cC'\dnw C$|ii} $\cC'\dnw C$\index{s}{arrz@$\cC'\dnw C$|ii} is the category whose objects are the morphisms of the form $x\colon X\to C$ in~$\cC$ where~$X$ is an object of~$\cC'$, and where the morphisms from $x\colon X\to C$ to $y\colon Y\to C$ are the morphisms $f\colon X\to Y$ in~$\cC'$ such that $y\circ f=x$. We write $x\utr y$\index{s}{triangle@$x\utr y$|ii} the statement that there exists such a morphism~$f$.

We identify, the usual way, preordered sets with categories in which there exists at most one morphism from any object to any other object.
For a poset~$P$, a \emph{$P$-indexed diagram}\index{i}{diagram!poset-indexed|ii} in a category~$\cC$ is a functor from~$P$ (viewed as a category) to~$\cC$. It can be viewed as a system $\overrightarrow{S}=\famm{S_p,\gs_p^q}{p\leq q\text{ in }P}$ where
$\gs_p^q\colon S_p\to S_q$ in~$\cC$, $\gs_p^p=\id_{S_p}$, and
$\gs_p^r=\gs_q^r\circ\gs_p^q$, for all $p\leq q\leq r$ in~$P$. Hence a cocone above~$\overrightarrow{S}$ consists of a family $\famm{T,\gt_p}{p\in P}$, for an object~$T$ of~$\cC$ and morphisms $\gt_p\colon S_p\to T$ such that $\gt_p=\gt_q\circ\gs_p^q$, for all $p\leq q$ in~$P$. In case~$\overrightarrow{S}$ has a colimit with underlying object~$S$, we shall often denote by~$\gs_p\colon S_p\to S$ the corresponding limiting morphism, for $p\in P$. A \emph{subdiagram}\index{i}{diagram!sub-|ii} of~$\overrightarrow{S}$ is a composite of the form $\overrightarrow{S}\circ\varphi$, where~$\varphi$ is an order-embedding from some poset into~$P$ (of course~$\varphi$ is viewed as a functor).

For a category~$\cA$ and a functor~$\Phi\colon\cA\to\cS$, we say that a diagram~$\overrightarrow{S}$ in~$\cS$ is \emph{liftable}\index{i}{liftable!diagram|seeonly{diagram, liftable}}\index{i}{diagram!liftable|ii} \emph{with respect to~$\Phi$} if there exists a diagram~$\overrightarrow{A}$ in~$\cA$ such that $\Phi\overrightarrow{A}\cong\overrightarrow{S}$. We say then that~$\overrightarrow{S}$ is \emph{lifted}\index{i}{diagram!lifted|ii} by~$\cA$, or that~$\cA$ is a \emph{lifting}\index{i}{diagram!lifting|ii} of~$\cS$.

For an infinite cardinal~$\gk$, we say that~$\overrightarrow{S}$ is

\begin{itemize}
\item \emph{$\gk$-directed}\index{i}{diagram!directedl@$\gl$-directed|ii} if the poset~$P$ is $\gk$-directed\index{i}{poset!directedl@$\gl$-directed};

\item \emph{continuous}\index{i}{diagram!continuous|ii} if~$\overrightarrow{S}$ preserves all small directed colimits, that is, $S_p=\varinjlim_{q\in X}S_q$ (with the obvious transition and limiting morphisms), for every nonempty directed\index{i}{poset!directed} subset~$X$ of~$P$ with least upper bound~$p=\bigvee X$.

\item \emph{$\gs$-continuous}\index{i}{diagram!$\gs$-continuous|ii} if $\overrightarrow{S}$ preserves all directed colimits of \emph{increasing sequences} in~$P$.
 \end{itemize}
 
A colimit is \emph{monomorphic}\index{i}{monomorphic colimit|ii} if all its limiting morphisms (thus also all its transition morphisms) are monic.

\begin{defn}\label{D:SmallDirColim}
Let~$\gl$ be an infinite cardinal and let~$\cC^\dagger$ be a full subcategory of a category~$\cC$. We say that
\begin{description}
\item[\tui] \emph{$\cC^\dagger$ has all $\gl$-small directed colimits within~$\cC$}\index{i}{having all $\gl$-small directed colimits within~$\cC$|ii} if every diagram in~$\cC^\dagger$ indexed by a nonempty \emph{directed}\index{i}{poset!directed} $\gl$-small poset~$P$ has a colimit in~$\cC$ whose underlying object belongs to~$\cC^\dagger$;

\item[\tuii] \emph{$\cC$ has all $\gl$-small directed colimits} if it has all $\gl$-small directed colimits\index{i}{having all $\gl$-small directed colimits within~$\cC$} within itself.
\end{description}

\end{defn}

According to \cite[Theorem~1.5]{AdRo}\index{c}{Ad\'amek, J.}\index{c}{Rosick\'y, J.} (see also \cite[Theorem~B.2.6.13]{John02}\index{c}{Johnstone, P.\,T.}), $\cC^\dagger$ has all $\gl$-small directed colimits within\index{i}{having all $\gl$-small directed colimits within~$\cC$}~$\cC$ if{f} every diagram in~$\cC^\dagger$, indexed by a nonempty $\gl$-small \emph{filtered category}, has a colimit whose underlying object belongs to~$\cC^\dagger$. (A category is \emph{filtered} if for all objets~$A$ and~$B$ there are an object~$C$ with morphisms $f\colon A\to C$ and $g\colon B\to C$, and for all objects~$A$ and~$B$ with morphisms $u,v\colon A\to B$ there are an object~$C$ and a morphism $f\colon B\to C$ such that $f\circ u=f\circ v$.)\index{i}{filtered category|ii}

For a category~$\cC$ and a poset~$P$, we denote by~$\cC^P$ the category of all $P$-indexed diagrams in~$\cC$, with natural transformations as morphisms. In particular, in case $P=\two=\set{0,1}$\index{s}{two@$\two$|ii}, the two-element chain, we obtain the category~$\cC^\two$ of arrows of~$\cC$.

\begin{defn}\label{D:projection}
Let~$\cC$ be a category. A \emph{projection}\index{i}{projection (in a category)|ii} in~$\cC$ is the canonical projection from a nonempty finite product in~$\cC$ to one of its factors.

An arrow $f\colon A\to B$ in~$\cC$ is an \emph{extended projection}\index{i}{extended projection (in a category)|ii} if there are diagrams $\overrightarrow{A}=\famm{A_i,\ga_i^j}{i\leq j\text{ in }I}$ and $\overrightarrow{B}=\famm{B_i,\gb_i^j}{i\leq j\text{ in }I}$ in~$\cC$, both indexed by a directed\index{i}{poset!directed} poset~$I$, together with colimits
 \begin{align*}
 \famm{A,\ga_i}{i\in I}&=\varinjlim\overrightarrow{A}\,,\\
 \famm{B,\gb_i}{i\in I}&=\varinjlim\overrightarrow{B}\,,
 \end{align*}
and a natural transformation $\famm{f_i}{i\in I}$ from~$\overrightarrow{A}$ to~$\overrightarrow{B}$ with limit morphism~$f$, such that each~$f_i$, for $i\in I$, is a projection.
\end{defn}

In particular, every isomorphism is a projection, and every extended projection is a directed colimit, in the category~$\cC^{\two}$\index{s}{two@$\two$} of all arrows of~$\cC$, of projections in~$\cC$.

\begin{defn}\label{D:ClosDirColim}
Let~$\gl$ be an infinite cardinal.
A subcategory~$\cC'$ of a category~$\cC$ is \emph{closed under $\gl$-small directed colimits}\index{i}{closed under $\gl$-small $\varinjlim$|ii}\index{i}{closed under $\gl$-small $\varinjlim$|ii}, if for every $\gl$-small set~$I$ and all diagrams $\overrightarrow{A}=\famm{A_i,\ga_i^j}{i\leq j\text{ in }I}$ and $\overrightarrow{B}=\famm{B_i,\gb_i^j}{i\leq j\text{ in }I}$ in~$\cC'$, together with colimits in~$\cC$
 \begin{align*}
 \famm{A,\ga_i}{i\in I}&=\varinjlim\overrightarrow{A}\,,\\
 \famm{B,\gb_i}{i\in I}&=\varinjlim\overrightarrow{B}\,,
 \end{align*}
and a natural transformation $\famm{f_i}{i\in I}$ from~$\overrightarrow{A}$ to~$\overrightarrow{B}$ with limit morphism~$f\colon A\to\nobreak B$, if each~$f_i$, for $i\in I$, is an arrow of~$\cC'$, then so is~$f$.

In case we formulate the above requirement with~$\overrightarrow{A}$ (resp., $\overrightarrow{B}$) \emph{constant} (i.e., all morphisms are the identity), we say that~$\cC'$ is \emph{left \pup{resp., right} closed under $\gl$-small directed colimits}\index{i}{closed (left, right-) under $\gl$-small $\varinjlim$|ii}. For example, $\cC'$ is right closed under $\gl$-small directed colimits if{f} for every directed colimit diagram
 \[
 \famm{A,\ga_i}{i\in I}=\varinjlim\famm{A_i,\ga_i^j}{i\leq j\text{ in }I}
 \]
in~$\cA$ with $\card I<\gl$, every object~$B$ in~$\cB$, and every $f\colon A\to B$, if $f\circ\ga_i$ is a morphism of~$\cC'$ for each $i\in I$, then so is~$f$.
\end{defn}

In case~$\cC'$ is a full subcategory of~$\cC$, this is equivalent to verifying that any colimit cocone in~$\cC$ of any directed poset-indexed diagram in~$\cC'$ is contained in~$\cC'$.

\subsubsection*{\textbf{Double arrows}}\index{i}{double arrow|ii}
We shall often encounter the situation of a subcategory~$\cS^\Rightarrow$\index{s}{RightarrowCat@$\cS^\Rightarrow$|ii} of a category~$\cS$, of which the arrows are denoted in the form $f\colon X\Rightarrow Y$\index{s}{AtorightarrowB@$f\colon A\Rightarrow B$|ii} and called \emph{the double arrows\index{i}{double arrow|ii} of $\cS$}. In such a context, for a poset~$P$ and objects of~$\cS^P$ (i.e., $P$-indexed diagrams) $\overrightarrow{X}$ and~$\overrightarrow{Y}$, a double arrow\index{i}{double arrow} from~$\overrightarrow{X}$ to~$\overrightarrow{Y}$ will be defined as an arrow $\overrightarrow{f}\colon\overrightarrow{X}\to\overrightarrow{Y}$ in~$\cS^P$ (i.e., a natural transformation from~$\overrightarrow{X}$ to~$\overrightarrow{Y}$), say $\overrightarrow{f}=\famm{f_p}{p\in P}$, such that~$f_p\colon X_p\Rightarrow Y_p$\index{s}{AtorightarrowB@$f\colon A\Rightarrow B$} for each $p\in P$; and then we shall write $\overrightarrow{f}\colon\overrightarrow{X}\Todot\overrightarrow{Y}$\index{s}{AtoRightarrowdotB@$f\colon\xA\Todot\xB$}.

\subsection{Directed colimits of first-order structures}\label{Su:DirColimFirstOrd}
We shall use standard definitions and facts about first-order structures as presented, for example, in Chang and Keisler~\cite[Section~1.3]{ChKe}\index{c}{Chang, C.\,C.}\index{c}{Keisler, H.\,J.}. We shall still denote by~$\Lg(\bA)$\index{s}{LgA@$\Lg(\bA)$|ii} the language of a first-order structure~$\bA$. This set breaks up as the set~$\Cst(\bA)$\index{s}{CstA@$\Cst(\bA)$|ii} of all constant symbols of~$\bA$, the set $\Op(\bA)$\index{s}{OpA@$\Op(\bA)$|ii} of all operation symbols of~$\bA$, and the set~$\Rel(\bA)$\index{s}{RelA@$\Rel(\bA)$|ii} of all relation symbols of~$\bA$. We denote by $\ari(s)$\index{s}{aris@$\ari(s)$|ii} the arity of a symbol $s\in\Lg(\bA)$; hence~$\ari(s)$ is a positive integer, unless $s\in\Cst(\bA)$ in which case $\ari(s)=0$.

We shall also denote the universe of a first-order structure by the corresponding lightface character, so, for instance, $A$ will be the universe of~$\bA$. We denote by~$s^\bA$ the interpretation of a symbol~$s$ of~$\Lg(\bA)$ in the structure~$\bA$, so that~$s^\bA$ is an element of~$A$ if~$s$ is a constant symbol and an $n$-ary operation (resp., relation) on~$A$ if~$s$ is an $n$-ary operation (resp., relation) symbol.

A map $f\colon\bA\to\bB$ between first-order structures of the same language~$\scL$ is an \emph{elementary embedding} if $\bA\models\vF(a_1,\dots,a_n))$ if{f} $\bB\models\vF(f(a_1),\dots,f(a_n))$ for every first-order formula~$\vF$ of~$\scL$, of arity~$n$, and all $a_1,\dots,a_n\in A$. Requiring this condition only for~$\vF$ \emph{atomic} means that~$f$ is an \emph{embedding}. Both conditions trivially imply that~$f$ is one-to-one.

We shall need the following description of directed colimits in the category~$\MOD_\scL$\index{s}{Mod@$\MOD_\scL$|ii} of all models for a given first-order language~$\scL$, see for example \cite[Section~1.2.5]{Gorb}\index{c}{Gorbunov, V.\,A.}. For a directed poset\index{i}{poset!directed}~$I$ and a poset-indexed diagram
$\famm{\bA_i,\gf_i^j}{i\leq j\text{ in }I}$ in $\MOD_\scL$\index{s}{Mod@$\MOD_\scL$}, we define an equivalence relation~$\equiv$ on~$\cA:=\bigcup\famm{A_i\times\set{i}}{i\in I}$ by
 \[
 (x,i)\equiv(y,j)\ \Longleftrightarrow\ (\exists k\in I\Upw\set{i,j})
 \bigl(\gf_i^k(x)=\gf_j^k(y)\bigr)\,,
 \]
then we set $A:=\cA/{\equiv}$ and we define $(\gf_i\colon A_i\to A$, $x\mapsto(x,i)/{\equiv})$, for each $i\in I$. Then the following relation holds in the category~$\SET$\index{s}{Set@\textbf{Set}|ii} of sets (with maps as homomorphisms):
 \begin{equation}\label{Eq:varinjliminSet}
 \famm{A,\gf_i}{i\in I}=\varinjlim\famm{A_i,\gf_i^j}{i\leq j\text{ in }I}
 \end{equation}
The directed colimit cocone $\famm{A,\gf_i}{i\in I}$ is also uniquely characterized by the properties
 \begin{align}
 A&=\bigcup\famm{\gf_i``(A_i)}{i\in I}\,,\label{Eq:ADirUngfiAi}\\
 \gf_i(x)=\gf_i(y)&\Leftrightarrow(\exists j\in I\upw i)
 \bigl(\gf_i^j(x)=\gf_i^j(y)\bigr)\,,\quad\text{for all }i\in I
 \text{ and all }x,y\in A_i\,.\label{Eq:gfix=gfiyimpl}
 \end{align}
The set~$A$ can be given a unique structure of model of~$\scL$ such that
\begin{description}
\item[\tui] For each constant symbol~$c$ in~$\scL$, $c^\bA=\gf_i(c^{\bA_i})$ for each~$i\in I$.

\item[\tuii] For each operation symbol~$f$ in~$\scL$, say of arity~$n$, and for each~$i\in I$, the equation
 \[
 f^\bA(\gf_i(x_1),\dots,\gf_i(x_n))=
 \gf_i\bigl(f^{\bA_i}(x_1,\dots,x_n)\bigr)
 \]
holds for all $x_1,\dots,x_n\in A_i$.

\item[\tuiii] For each relation symbol~$R$ in~$\scL$, say of arity~$n$, and for each $i\in I$, the following equivalence holds:
 \begin{multline}\label{Eq:RelatOnDirColim}
 (\gf_i(x_1),\dots,\gf_i(x_n))\in R^\bA\ \Longleftrightarrow\ 
 (\exists j\in I\upw i)
 \bigl((\gf_i^j(x_1),\dots,\gf_i^j(x_n))\in R^{\bA_j}\bigr)\,,\\
 \quad\text{for all }x_1,\dots,x_n\in A_i\,.
 \end{multline}
\end{description}
Furthermore, the following relation holds in the category~$\MOD_\scL$\index{s}{Mod@$\MOD_\scL$}:
 \begin{equation}\label{Eq:DirLimMODscL}
 \famm{\bA,\gf_i}{i\in I}=\varinjlim\famm{\bA_i,\gf_i^j}
 {i\leq j\text{ in }I}\,.
 \end{equation}
The conditions (i)--(iii) above, together with~\eqref{Eq:ADirUngfiAi} and~\eqref{Eq:gfix=gfiyimpl}, determine the colimit up to isomorphism. In categorical terms, the discussion above means that \emph{the forgetful functor from $\MOD_\scL$\index{s}{Mod@$\MOD_\scL$} to~$\SET$\index{s}{Set@\textbf{Set}} creates and preserves all small directed colimits}.

For a theory~$\scT$ in the language~$\scL$, we denote by $\MOD(\scT)$\index{s}{Modt@$\MOD(\scT)$|ii} the full subcategory of~$\MOD_\scL$\index{s}{Mod@$\MOD_\scL$} consisting of the models that satisfy all the axioms of~$\scT$. It is well-known that if~$\scT$ consists only of axioms of the form
 \begin{equation}\index{s}{AtorightarrowB@$f\colon A\Rightarrow B$}\label{Eq:SpecialForm}
 (\forall\overrightarrow{\vx})\bigl(\vE(\overrightarrow{\vx})\Rightarrow(\exists\overrightarrow{\vy})
 \vF(\overrightarrow{\vx},\overrightarrow{\vy})\bigr)\,,
 \end{equation}
with each of the formulas~$\vE$ and~$\vF$ either a tautology, or an antilogy, or a conjunction of atomic formulas, then $\MOD(\scT)$\index{s}{Modt@$\MOD(\scT)$} is closed under directed colimits (see, for example, the easy direction of \cite[Exercise~5.2.24]{ChKe}\index{c}{Chang, C.\,C.}\index{c}{Keisler, H.\,J.}). While there are more general first-order sentences preserving directed colimits, those of the form~\eqref{Eq:SpecialForm} have the additional advantage of preserving direct products, which will be of importance in the sequel.

\begin{exple}\label{Ex:QuasivarSpecial}
If~$\scT$ consists of \emph{universal Horn sentences} (cf. \cite[Section~1.2.2]{Gorb}\index{c}{Gorbunov, V.\,A.}), then all the axioms in~$\scT$ have the form~\eqref{Eq:SpecialForm}, thus $\MOD(\scT)$\index{s}{Modt@$\MOD(\scT)$} is closed under directed colimits. This is the case for the following examples:
\begin{description}
\item[\tui] $\scT$ is the theory of all groups, in the language $\set{\cdot,1,{}^{-1}}$, where~$\cdot$ is a binary operation symbol, $1$ is a constant symbol, and ${}^{-1}$ is a unary operation symbol.

\item[\tuii] $\scT$ is the theory of all partially ordered abelian groups, in the language $\set{-,0,\leq}$, where~$-$ is a binary operation symbol, $0$ is a constant symbol, and~$\leq$ is a binary relation symbol.

\item[\tuiii] $\scT$ is the theory of all \jzs s, in the language $\set{\vee,0}$, where~$\vee$ is a binary operation symbol and~$0$ is a constant symbol.
\end{description}
Many other examples are given in Gorbunov~\cite{Gorb}\index{c}{Gorbunov, V.\,A.}.
\end{exple}

\begin{exple}\label{Ex:DSLatSpecial}
Let $\scL:=\set{\vee,0}$, where~$\vee$ is a binary operation symbol and~$0$ is a constant symbol, and let~$\scT$ consist of the axioms for \jzs s (i.e., idempotent commutative monoids) together with the axiom
 \[
 (\forall\vx,\vy,\vz)
 \bigl(\vz\vee\vx\vee\vy=\vx\vee\vy\Rightarrow(\exists\vx',\vy')
 (\vx'\vee\vx=\vx\text{ and }\vy'\vee\vy=\vy\text{ and }
 \vz=\vx'\vee\vy'\bigr)\,.
 \]
Then all axioms of~$\scT$ have the form~\eqref{Eq:SpecialForm} and the models of~$\scT$ are exactly the \emph{distributive \jzs s}\index{i}{distributive!semilattice}.
\end{exple}

\begin{exple}\label{Ex:CRMSpecial}
Let $\scL:=\set{+,0}$, where~$+$ is a binary operation symbol, $0$ is a constant symbol, and let~$\scT$ consist of all the axioms of the theory of commutative monoids, together with the \emph{refinement axiom}\index{s}{conj@$\conj$|ii},
 \begin{multline*}
 (\forall\vx_0,\vx_1,\vy_0,\vy_1)\Bigl(\vx_0+\vx_1=\vy_0+\vy_1
 \Rightarrow\\
 (\exists_{i,j<2}\vz_{i,j})\conj_{i<2}
 (\vx_i=\vz_{i,0}+\vz_{i,1}\text{ and }\vy_i=\vz_{0,i}+\vz_{1,i})\Bigr)\,,
 \end{multline*}
where~$\conj$\index{s}{conj@$\conj$} denotes conjunction of a list of formulas. Then all axioms of~$\scT$ have the form~\eqref{Eq:SpecialForm} and the models of~$\scT$ are exactly the \emph{refinement monoids}\index{i}{monoid!refinement|ii}. If we add to~$\scT$ the \emph{conicality axiom}\index{i}{monoid!conical} (cf. Example~\ref{Ex:nsKth})
 \[
 (\forall\vx,\vy)(\vx+\vy=0\Rightarrow\vx=\vy=0)\,,
 \]
then we obtain the theory of \emph{conical refinement monoids}\index{i}{monoid!conical}\index{i}{monoid!refinement}.
\end{exple}

\begin{exple}\label{Ex:RegRngSpecial}
Let $\scL:=\set{-,\cdot,0}$, where~$-$ and~$\cdot$ are both binary operation symbols and~$0$ is a constant symbol, and~$\scT$ consists of all axioms of rings (they are identities) together with the axiom
 \[
 (\forall\vx)(\exists\vy)(\vx\cdot\vy\cdot\vx=\vx)\,.
 \]
Then all axioms of~$\scT$ have the form~\eqref{Eq:SpecialForm} and the models of~$\scT$ are exactly the (not necessarily unital) \emph{regular rings}\index{i}{regular ring}.
\end{exple}

The given description of directed colimits gives a particularly simple way to check whether a cocone $\famm{\bA,\gf_i}{i\in I}$, of models of a theory whose axioms all have the form~\eqref{Eq:SpecialForm}, above a directed poset-indexed diagram $\famm{\bA_i,\gf_i^j}{i\leq j\text{ in }I}$, is a colimit of that diagram: namely, we only need to check the statements~\eqref{Eq:ADirUngfiAi}, \eqref{Eq:gfix=gfiyimpl}, and~\eqref{Eq:RelatOnDirColim}. In particular, in all the examples described above except the one of partially ordered abelian groups, the directed colimit is characterized by~\eqref{Eq:ADirUngfiAi} and \eqref{Eq:gfix=gfiyimpl}; in the case of partially ordered abelian groups, one also needs to check that the statement
 \[
 \gf_i(x)\geq 0\Rightarrow(\exists j\in I\upw i)\bigl(\gf_i^j(x)\geq0\bigr)
 \]
holds for all $i\in I$ and all $x\in A_i$.

\section{Kappa-presented and weakly kappa-presented objects}
\label{S:WeaklKapPres}

In the present section we shall introduce (cf. Definition~\ref{D:wkapppres}) a weakening of the usual definition of a $\gk$-presented\index{i}{presented!$\gl$-} object in a category. We first recall that definition, as given in Gabriel and Ulmer \cite[Definition~6.1]{GaUl}\index{c}{Gabriel, P.}\index{c}{Ulmer, F.}, see also Ad\'amek and Rosick\'y \cite[Definitions~1.1 and~1.13]{AdRo}\index{c}{Ad\'amek, J.}\index{c}{Rosick\'y, J.}.

\begin{defn}\label{D:kapppres}
For an infinite regular cardinal~$\gk$, an object $A$ in a category
$\cC$ is \emph{$\gk$-presented}\index{i}{presented!$\gl$-|ii} if for every $\gk$-directed\index{i}{diagram!directedl@$\gl$-directed} colimit
 \[
 \famm{B,b_i}{i\in I}=\varinjlim\famm{B_i,b_i^j}{i\leq j\text{ in }I}
 \quad\text{in }\cC\,,
 \]
the following statements hold:

\begin{description}
\item[\tui] For every $f\colon A\to B$, there exists $i\in I$ such that
$f$ factors through $B_i$ (i.e., $f=b_i\circ f'$ for some morphism $f'\colon A\to B_i$).

\item[\tuii] For each $i\in I$ and morphisms $f,g\colon A\to B_i$ such that $b_i\circ f=b_i\circ g$, there exists~$j\geq i$ in~$I$ such that $b_i^j\circ f=b_i^j\circ g$.
\end{description}
In case~$\gk=\aleph_0$\index{s}{aleph0@$\aleph_{\ga}$}, we will say \emph{finitely presented}\index{i}{presented!finitely|ii} instead of $\gk$-presented\index{i}{presented!$\gl$-}.
\end{defn}

In the following definition we do not assume regularity of the cardinal~$\gk$.

\begin{defn}\label{D:wkapppres}
For an infinite cardinal~$\gk$, an object $A$ in a category
$\cC$ is \emph{weakly $\gk$-presented}\index{i}{presented!weakly $\gl$-|ii} if for every set~$\Omega$ and every \emph{continuous}\index{i}{diagram!continuous} directed colimit
$\famm{B,b_X}{X\in[\Omega]^{<\gk}}=
\varinjlim\famm{B_X,b_X^Y}{X\subseteq Y\text{ in }[\Omega]^{<\gk}}$
in~$\cC$, every morphism from~$A$ to~$B$ factors through $b_X\colon B_X\to B$ for some~$X\in[\Omega]^{<\gk}$.
\end{defn}

Trivially, in case~$\gk$ is regular, $\gk$-presented\index{i}{presented!$\gl$-} implies weakly $\gk$-presented\index{i}{presented!weakly $\gl$-}.
For \emph{cocomplete} categories (i.e., categories in which every small---not necessarily directed---diagram has a colimit), weakly $\gk$-presented structures\index{i}{presented!weakly $\gl$-} can be characterized as follows.

\begin{prop}\label{P:WgkPresCoCplte}
Let~$\gk$ be an infinite cardinal and let~$\cC$ be a cocomplete category. Then an object~$A$ of~$\cC$ is weakly $\gk$-presented\index{i}{presented!weakly $\gl$-} if{f} for every \pup{non necessarily directed}\index{i}{poset!directed} poset~$I$ and every colimit cocone
 \[
 \famm{B,b_i}{i\in I}=\varinjlim\famm{B_i,b_i^j}{i\leq j\text{ in }I}
 \quad\text{in }\cC\,,
 \]
every morphism from~$A$ to~$B$ factors through $\varinjlim_{i\in J}B_i$ for some $\gk$-small subset~$J$ of~$I$.
\end{prop}

\begin{proof}
Suppose that~$A$ satisfies the given condition and consider a morphism $f\colon A\to B$ with a continuous\index{i}{diagram!continuous} directed colimit
 \[
 \famm{B,b_X}{X\in[\Omega]^{<\gk}}=
 \varinjlim\famm{B_X,b_X^Y}{X\subseteq Y\text{ in }[\Omega]^{<\gk}}\quad\text{in }
 \cC\,.
 \]
As this colimit is continuous\index{i}{diagram!continuous}, $B_X=\varinjlim_{Y\in[X]^{<\go}}B_Y$ for each $X\in[\Omega]^{<\gk}$, thus
 \[
 \famm{B,b_X}{X\in[\Omega]^{<\go}}=
 \varinjlim\famm{B_X,b_X^Y}{X\subseteq Y\text{ in }[\Omega]^{<\go}}\quad\text{in }
 \cC\,.
 \]
By assumption, there exists a $\gk$-small subset~$I$ of $[\Omega]^{<\go}$ such that~$f$ factors through $\varinjlim_{X\in I}B_X$. It follows that the union~$X$ of all the elements of~$I$ is a $\gk$-small subset of~$\Omega$ (\emph{as every member of~$I$ is finite, this does not require any regularity assumption on~$\gk$}) and~$f$ factors through~$B_X$.

Conversely, suppose that~$A$ is weakly $\gk$-presented\index{i}{presented!weakly $\gl$-}. Let
 \[
 \famm{B,b_i}{i\in I}=\varinjlim\famm{B_i,b_i^j}{i\leq j\text{ in }I}\quad\text{in }\cC\,.
 \]
As~$\cC$ is cocomplete, we can define $C_X:=\varinjlim_{i\in X}B_i$ for each $X\in[I]^{<\gk}$, with the obvious transition and limiting morphisms. As $B=\varinjlim_{X\in[I]^{<\gk}}C_X$ is a continuous\index{i}{diagram!continuous} directed colimit and~$A$ is weakly $\gk$-presented\index{i}{presented!weakly $\gl$-}, every morphism from~$A$ to~$B$ factors through~$C_J=\varinjlim_{i\in J}B_i$, for some $J\in[I]^{<\gk}$.
\qed\end{proof}

\begin{cor}\label{C:WgkPresCoCplte}
Let $\gl$ and~$\gk$ be infinite cardinals with $\gl\leq\gk$ and let~$\cC$ be a cocomplete category. Then every weakly $\gl$-presented\index{i}{presented!weakly $\gl$-} object of~$\cC$ is also weakly $\gk$-presented\index{i}{presented!weakly $\gl$-}.
\end{cor}

\begin{exple}\label{Ex:wkapppres}
For an infinite cardinal~$\gk$, we find a complete lattice (thus a cocomplete category)~$L_\gk$ in which the unit is weakly $\gk$-presented\index{i}{presented!weakly $\gl$-} in~$L_\gk$ but not weakly $\ga$-presented\index{i}{presented!weakly $\gl$-} for any $\ga<\gk$.

Denote by~$A_\gk$ the set of all infinite cardinals smaller than~$\gk$ and endow the set
 \[
 L^*_\gk:=\setm{(\ga,x)}{\ga\in A_\gk\text{ and }x\in[\ga]^{<\ga}}
 \]
with the partial ordering defined by
 \[
 (\ga,x)\leq(\gb,y)\ \Longleftrightarrow\ (\ga=\gb\text{ and }x\subseteq y)\,.
 \]
We set $L_\gk:=L^*_\gk\cup\set{0,1}$, for a new smallest (resp., largest) element~$0$ (resp., $1$). Then~$L_\gk$ is a complete lattice, thus, viewed as a category, $L_\gk$ is cocomplete. Nontrivial joins in~$L_\gk$ are given by
 \begin{align}
 (\ga,x)\vee(\gb,y)&=1\quad\text{if }\ga\neq\gb\,,\label{Eq:JoinDistLayers}\\
 \bigvee\famm{(\ga,x_i)}{i\in I}&=\begin{cases}
 \bigl(\ga,\bigcup\famm{x_i}{i\in I}\bigr)\,,&\text{if }
 \card\bigcup\famm{x_i}{i\in I}<\ga\,,\\
 1\,,&\text{if }\card\bigcup\famm{x_i}{i\in I}=\ga\,.\label{Eq:JoinSameLayer}
 \end{cases} 
 \end{align}
Suppose that the unit of~$L_\gk$ is the colimit (i.e., join) of a family $\famm{(\ga_i,x_i)}{i\in I}$ of elements of~$L^*_\gk$, we must find a $\gk$-small subset~$J$ if~$I$ such that $1=\bigvee\famm{(\ga_i,x_i)}{i\in J}$. If $\ga_i\neq\ga_j$ for some~$i,j$ then, using~\eqref{Eq:JoinDistLayers}, $1=(\ga_i,x_i)\vee(\ga_j,x_j)$ and we are done. Now suppose that $\ga_i=\ga$ for all $i\in I$. It follows from~\eqref{Eq:JoinSameLayer} that the union~$x$ of all~$x_i$ has cardinality~$\ga$. Pick $i_\gx\in I$ such that $\gx\in x_{i_\gx}$, for each $\gx\in x$. Then the set $J:=\setm{i_\gx}{\gx\in x}$ has cardinality at most~$\ga$ (thus it is $\gk$-small), and $1=\bigvee\famm{(\ga_i,x_i)}{i\in J}$. Therefore, by Proposition~\ref{P:WgkPresCoCplte}, $1$ is weakly $\gk$-presented\index{i}{presented!weakly $\gl$-} in~$L_\gk$.

Now let $\ga\in A_\gk$. Then $1=\bigvee\famm{(\ga,\set{\gx})}{\gx<\ga}$ in~$L_\gk$, but (using~\eqref{Eq:JoinSameLayer}) there is no $\ga$-small subset~$u$ of~$\ga$ such that $1=\bigvee\famm{(\ga,\set{\gx})}{\gx\in u}$. Therefore, by Proposition~\ref{P:WgkPresCoCplte}, $1$ is not weakly $\ga$-presented\index{i}{presented!weakly $\gl$-} in~$L_\gk$.
\end{exple}

\section{Extension of a functor by directed colimits}\label{S:ExtFunct}

The main result of the present section, Proposition~\ref{P:ArrObj2Diag}, states a very intuitive and probably mostly well-known fact, parts of which are already present in the literature such as Pudl\'ak~\cite{Pudl85}\index{c}{Pudl\'ak, P.}, although we could not trace a reference where it is stated explicitly in the strong form that we shall require. We are given categories~$\cA$ and~$\cS$, together with a full subcategory~$\cA^\dagger$ of~$\cA$ and a functor $\Phi\colon\cA^\dagger\to\cS$. We wish to extend~$\Phi$ to a functor, which preserves all small directed colimits, from all small directed colimits of diagrams from~$\cA^\dagger$ to~$\cS$. Proposition~\ref{P:ArrObj2Diag} states sufficient conditions under which this can be done.

\begin{lem}\label{L:ArrObj2Diag}
Let~$\cA^\dagger$ be a full subcategory of a category~$\cA$, let~$\cS$ be a category with all small directed colimits, and let~$\Phi\colon\cA^\dagger\to\cS$ be a functor. Let $f\colon A\to B$ be a morphism in~$\cA$, together with colimits
 \begin{align}
 \famm{A,\ga_i}{i\in I}&=\varinjlim\overrightarrow{A}\,,\quad\text{with }
 \overrightarrow{A}=\famm{A_i,\ga_i^{i'}}{i\leq i'\text{ in }I}\,,
 \label{Eq:Aascolim}\\
 \famm{B,\gb_j}{j\in J}&=\varinjlim\overrightarrow{B}\,,\quad\text{with }
 \overrightarrow{B}=\famm{B_j,\gb_j^{j'}}{j\leq j'\text{ in }J},
 \label{Eq:Bascolim}
 \end{align}
with both posets~$I$ and~$J$ directed\index{i}{poset!directed}, all~$A_i$ finitely presented\index{i}{presented!finitely}, and all~$A_i$ and~$B_j$ belonging to~$\cA^\dagger$. Set
 \begin{align*}
 K:=\setm{(i,j,x)}{(i,j)\in I\times J,&\ x\colon A_i\to B_j,\text{ and }
 \gb_j\circ x=f\circ\ga_i}\,,\notag\\
 \famm{\oll{A},\ol{\ga}_i}{i\in I}&:=\varinjlim\Phi\overrightarrow{A}\,,\\
 \famm{\ol{B},\ol{\gb}_j}{j\in J}&:=\varinjlim\Phi\overrightarrow{B}\,.
 \end{align*}
Then there exists a unique morphism~$\ol{f}\colon\oll{A}\to\ol{B}$ such that $\ol{\gb}_j\circ\Phi(x)=\ol{f}\circ\ol{\ga}_i$ for each $(i,j,x)\in K$.
\end{lem}

\begin{snote}
While~$I$ and~$J$ are assumed to be sets, $K$ may be a proper class.
\end{snote}

\begin{proof}
We define a partial ordering~$\leq$ on the class~$K$ by the rule
 \[
 (i,j,x)\leq(i',j',x')\ \Longleftrightarrow\ (i\leq i',\ j\leq j',\text{ and }
 \gb_j^{j'}\circ x=x'\circ\ga_i^{i'})\,.
 \]
We claim that~$K$ is directed\index{i}{poset!directed}. Let $(i_0,j_0,x_0),(i_1,j_1,x_1)\in K$. We pick $i\geq i_0,i_1$ in~$I$. As~$A_i$ is finitely presented\index{i}{presented!finitely}, there exists $j\in J$ such that $j\geq j_0,j_1$ and the morphism $f\circ\ga_i\colon A_i\to B$ factors through~$B_j$; the latter means that there exists $x\colon A_i\to B_j$ such that $\gb_j\circ x=f\circ\ga_i$. It follows that
 \[
 \gb_j\circ x\circ\ga_{i_0}^i=f\circ\ga_i\circ\ga_{i_0}^i=f\circ\ga_{i_0}
 =\gb_{j_0}\circ x_0=\gb_j\circ\gb_{j_0}^j\circ x_0\,,
 \]
and thus, as~$A_{i_0}$ is finitely presented\index{i}{presented!finitely}, there exists $j'\in J\upw j$ such that
 \[
 \gb_j^{j'}\circ x\circ\ga_{i_0}^i=\gb_j^{j'}\circ\gb_{j_0}^j\circ x_0\,,
 \]
that is,
 \[
 \gb_j^{j'}\circ x\circ\ga_{i_0}^i=\gb_{j_0}^{j'}\circ x_0\,.
 \]
This implies that $(i,j',\gb_j^{j'}\circ x)$ belongs to $K\upw(i_0,j_0,x_0)$. A similar argument, using the finite presentability of~$A_{i_1}$, yields an element $j''\in J\upw j'$ such that $(i,j'',\gb_j^{j''}\circ x)$ belongs to $K\upw(i_1,j_1,x_1)$. Therefore, $(i,j'',\gb_j^{j''}\circ x)$ is an element of~$K$ above both elements $(i_0,j_0,x_0)$ and $(i_1,j_1,x_1)$, which completes the proof of our claim.

Both projections $\gl\colon K\to I$, $(i,j,x)\mapsto i$ and $\gm\colon K\to J$, $(i,j,x)\mapsto j$ are isotone. Let $i\in I$. As~$A_i$ is finitely presented\index{i}{presented!finitely}, there exists $j\in J$ such that the morphism $f\circ\ga_i\colon A_i\to B$ factors through~$B_j$; furthermore, $j$ can be chosen above any given element of~$J$. By definition, there exists~$x\colon A_i\to B_j$ such that $(i,j,x)\in K$. This proves that~$\gl$ is surjective (thus, \emph{a fortiori}, it has cofinal range) and~$\gm$ has cofinal range.

As both~$\gl$ and~$\gm$ are isotone with cofinal range and~$K$ is directed\index{i}{poset!directed},
it follows from \cite[Section~0.11]{AdRo}\index{c}{Ad\'amek, J.}\index{c}{Rosick\'y, J.} that
 \begin{align*}
 \famm{\oll{A},\ol{\ga}_i}{(i,j,x)\in K}&=\varinjlim\Phi\overrightarrow{A}\gl\,,\\
 \famm{\ol{B},\ol{\gb}_j}{(i,j,x)\in K}&=\varinjlim\Phi\overrightarrow{B}\gm\,.
 \end{align*}
By the definition of the ordering of~$K$, the family $\famm{\Phi(x)}{(i,j,x)\in K}$ defines a natural transformation from~$\Phi\overrightarrow{A}\gl$ to~$\Phi\overrightarrow{B}\gm$. The universal property of the colimit gives the desired conclusion.
\qed\end{proof}

\begin{prop}\label{P:ArrObj2Diag}
Let~$\cA^\dagger$ be a full subcategory of finitely presented\index{i}{presented!finitely} objects in a category~$\cA$, let~$\cS$ be a category with all small directed colimits. We assume that every object in~$\cA$ is a small directed colimit of objects in~$\cA^\dagger$. Then every functor $\Phi\colon\cA^\dagger\to\cS$ extends to a functor $\ol{\Phi}\colon\cA\to\cS$ which preserves all colimits of directed poset-indexed diagrams in~$\cA^\dagger$. Furthermore, if~$\cA^\dagger$ has small hom-sets, then~$\ol{\Phi}$ preserves all small directed colimits.
\end{prop}

\begin{proof}
The first part of the proof, up to the preservation of directed colimits from~$\cA^\dagger$, is established in a similar fashion as in the proof of the Corollary in Pudl\'ak~\cite[page~101]{Pudl85}\index{c}{Pudl\'ak, P.}.

For every object~$A$ of~$\cA$, we pick a representation of the form
 \begin{equation}\label{Eq:CanReprA}
  \famm{A,\ga_i}{i\in I}=\varinjlim\famm{A_i,\ga_i^{i'}}{i\leq i'\text{ in }I}
  \text{ in }\cA\,,
 \end{equation}
which we shall call the \emph{canonical representation of~$A$}, in such a way that~$I$ is directed\index{i}{poset!directed}, $A_i\in\cA^\dagger$ for all~$i\in I$, and $I=\set{\bot}$ with $A_\bot=A$ in case~$A\in\cA^\dagger$. We define~$\ol{\Phi}$ on objects by picking a cocone $\famm{\ol{\Phi}(A),\ol{\ga}_i}{i\in I}$ such that
 \[
 \famm{\ol{\Phi}(A),\ol{\ga}_i}{i\in I}=
 \varinjlim\famm{\Phi(A_i),\Phi(\ga_i^{i'})}{i\leq i'\text{ in }I}\text{ in }\cS\,.
 \]
By Lemma~\ref{L:ArrObj2Diag}, for each morphism~$f\colon A\to B$ in~$\cA$ with canonical representations~\eqref{Eq:Aascolim} and~\eqref{Eq:Bascolim}, there exists a unique morphism $\ol{\Phi}(f)\colon\ol{\Phi}(A)\to\ol{\Phi}(B)$ in~$\cS$ such that $\ol{\gb}_j\circ\Phi(x)=\ol{\Phi}(f)\circ\ol{\ga}_i$ for all $(i,j)\in I\times J$ and all $x\colon A_i\to B_j$ such that $\gb_j\circ x=f\circ\ga_i$.

We first prove that $\ol{\Phi}$ is a functor. It is clear that $\ol{\Phi}$ sends identities to identities. Now let $f\colon A\to B$ and $g\colon B\to C$ in~$\cA$, put $h:=g\circ f$, and let the canonical representations of~$A$, $B$, and~$C$ be respectively given by~\eqref{Eq:Aascolim}, \eqref{Eq:Bascolim}, and
 \[
 \famm{C,\gc_k}{k\in K}=
 \varinjlim\famm{C_k,\gc_k^{k'}}{k\leq k'\text{ in }K}\,.
 \]
Let $(i,k)\in I\times K$ and $z\colon A_i\to C_k$ such that $\gc_k\circ z=h\circ\ga_i$. As~$A_i$ is finitely presented\index{i}{presented!finitely}, there are~$j\in J$ and~$x\colon A_i\to B_j$ such that $\gb_j\circ x=f\circ\ga_i$. As~$B_j$ is finitely presented\index{i}{presented!finitely}, there are $k'\geq k$ in~$K$ and $y\colon B_j\to C_{k'}$ such that $\gc_{k'}\circ y=g\circ\gb_j$. As
 \[
 \gc_{k'}\circ y\circ x=g\circ\gb_j\circ x=g\circ f\circ\ga_i=
 h\circ\ga_i=\gc_k\circ z
 \]
and~$A_i$ is finitely presented\index{i}{presented!finitely}, there exists $k''\geq k'$ in~$K$ such that $\gc_{k'}^{k''}\circ y\circ x=\gc_k^{k''}\circ z$. Hence, by replacing~$k'$ by~$k''$ and~$y$ by~$\gc_{k'}^{k''}\circ y$, we may assume that~$k'=k''$, and hence $y\circ x=\gc_k^{k'}\circ z$ (see Figure~\ref{Fig:olPhiFunct}).

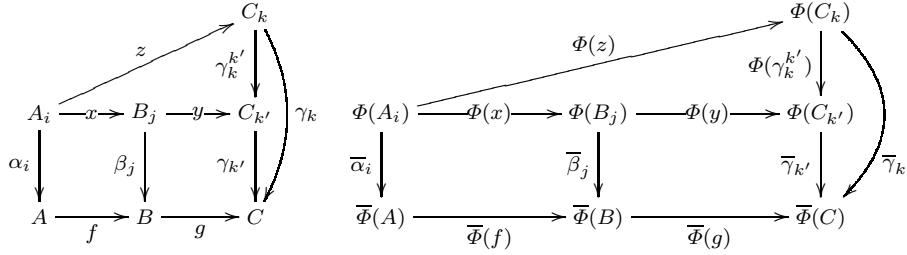
\begin{figure}[htb]
 \[
 \def\labelstyle{\displaystyle}
 \xymatrix{
 & & C_k\ar[d]_{\gc_k^{k'}}\ar@/^1pc/[dd]^{\gc_k}
 &&&&& \Phi(C_k)\ar[d]_{\Phi(\gc_k^{k'})}\ar@/^2pc/[dd]^(.7){\ol{\gc}_k}\\
 A_i\ar[r]|-x\ar[d]_{\ga_i}\ar[rru]^z & B_j\ar[r]|-y\ar[d]_{\gb_j} &
 C_{k'}\ar[d]_{\gc_{k'}} &
 \Phi(A_i)\ar[rr]|-{\Phi(x)}\ar[d]_{\ol{\ga}_i}\ar[rrrru]^{\Phi(z)}
 && \Phi(B_j)\ar[rr]|-{\Phi(y)}\ar[d]_{\ol{\gb}_j} &&
 \Phi(C_{k'})\ar[d]_{\ol{\gc}_{k'}}\\
 A\ar[r]_f & B\ar[r]_g & C &
 \ol{\Phi}(A)\ar[rr]_{\ol{\Phi}(f)} && \ol{\Phi}(B)\ar[rr]_{\ol{\Phi}(g)} && \ol{\Phi}(C)
 }
 \]
\caption{Proving that $\ol{\Phi}$ is a functor}
\label{Fig:olPhiFunct}
\end{figure}

It follows (see again Figure~\ref{Fig:olPhiFunct}) that
 \begin{align*}
 \ol{\Phi}(g)\circ\ol{\Phi}(f)\circ\ol{\ga}_i&=\ol{\Phi}(g)\circ\ol{\gb}_j\circ\Phi(x)\\
 &=\ol{\gc}_{k'}\circ\Phi(y)\circ\Phi(x)\\
 &=\ol{\gc}_{k'}\circ\Phi(\gc_k^{k'})\circ\Phi(z)\\
 &=\ol{\gc}_k\circ\Phi(z)\,. 
 \end{align*}
Therefore, by definition, $\ol{\Phi}(g)\circ\ol{\Phi}(f)=\ol{\Phi}(h)=\ol{\Phi}(g\circ f)$, and so~$\ol{\Phi}$ is a functor.

Now let~$A$ be an object of~$\cA$, with canonical representation~\eqref{Eq:CanReprA} and for which another representation
 \[
 \famm{A,\gb_j}{j\in J}=\varinjlim\famm{B_j,\gb_j^{j'}}{j\leq j'\text{ in }J}\,,
 \]
is given, where $J$ is directed\index{i}{poset!directed} and $B_j\in\cA^\dagger$ for all~$j\in J$. We shall prove that
 \begin{equation}\label{Eq:wtdcolA}
 \famm{\ol{\Phi}(A),\ol{\Phi}(\gb_j)}{j\in J}=
 \varinjlim\famm{\Phi(B_j),\Phi(\gb_j^{j'})}{j\leq j'\text{ in }J}\,.
 \end{equation}
We put
 \begin{align}
 \famm{\ol{B},\ol{\gb}_j}{j\in J}&:=
 \varinjlim\famm{\Phi(B_j),\Phi(\gb_j^{j'})}{j\leq j'\text{ in }J}\,,\label{Eq:DefolBA}\\
 U&:=\setm{(i,j,x)}{(i,j)\in I\times J,\ x\colon A_i\to B_j,\text{ and }
 \gb_j\circ x=\ga_i}\,,\notag\\
 V&:=\setm{(i,j,y)}{(i,j)\in I\times J,\ y\colon B_j\to A_i,\text{ and }
 \gb_j=\ga_i\circ y}\,.\notag
 \end{align}
By Lemma~\ref{L:ArrObj2Diag}, there are unique morphisms $u\colon\ol{\Phi}(A)\to\ol{B}$ and\linebreak $v\colon\ol{B}\to\ol{\Phi}(A)$ such that $\ol{\gb}_j\circ\Phi(x)=u\circ\ol{\ga}_i$ for all $(i,j,x)\in U$ and $\ol{\ga}_i\circ\Phi(y)=v\circ\ol{\gb}_j$ for all $(i,j,y)\in V$. As in the paragraph above, it follows that $v\circ u\circ\ol{\ga}_i=\ol{\ga}_i$ for each $i\in I$; whence $v\circ u=\id_{\ol{\Phi}(A)}$. Similarly, $u\circ v=\id_{\ol{B}}$, and thus~$u$ and~$v$ are mutually inverse isomorphisms. For all~$j\in J$, $\ol{\Phi}(\gb_j)$ is the unique morphism from~$\Phi(B_j)$ to~$\ol{\Phi}(A)$ such that $\ol{\Phi}(\gb_j)=\ol{\ga}_i\circ\Phi(y)$ for each $i\in I$ and all $y\colon B_j\to A_i$ with $\gb_j=\ga_i\circ y$. It follows that $\ol{\Phi}(\gb_j)=v\circ\ol{\gb}_j$, and hence, by~\eqref{Eq:DefolBA} and as~$v$ is an isomorphism, \eqref{Eq:wtdcolA} follows.

So far we have proved that~$\ol{\Phi}$ preserves all colimits of directed poset-indexed diagrams in~$\cA^\dagger$. Now we assume that~$\cA^\dagger$ has small hom-sets. We are given an arbitrary small directed colimit in~$\cA$ as in~\eqref{Eq:CanReprA}, where the~$A_i$ are no longer assumed to be in~$\cA^\dagger$, we must prove that the following statement holds:
 \begin{equation}\label{Eq:olPhi(colimA)}
 \famm{\ol{\Phi}(A),\ol{\Phi}(\ga_i)}{i\in I}=
 \varinjlim\famm{\ol{\Phi}(A_i),\ol{\Phi}(\ga_i^{i'})}{i\leq i'\text{ in }I}\text{ in }\cS\,.
 \end{equation}
For each~$i\in I$, we pick a representation
 \begin{equation}\label{Eq:CanReprAi}
  \famm{A_i,\ga_{i,j}^i}{j\in J_i}=
  \varinjlim\famm{A_{i,j},\ga_{i,j}^{i,j'}}{j\leq j'\text{ in }J_i}
  \text{ in }\cA\,,
 \end{equation}
where the poset~$J_i$ is directed\index{i}{poset!directed} and all~$A_{i,j}$ belong to~$\cA^\dagger$. We put
 \[
 P:=\bigcup\famm{\set{i}\times J_i}{i\in I}\,.
 \]
For $(i,j),(i',j')\in P$, we define a morphism from~$(i,j)$ to~$(i',j')$ as a morphism $x\colon A_{i,j}\to A_{i',j'}$ in~$\cA^\dagger$ such that $\ga_{i',j'}^{i'}\circ x=\ga_i^{i'}\circ\ga_{i,j}^i$ (this requires~$i\leq i'$). This defines a category~$\cP$ with underlying set~$P$. As~$\cA^\dagger$ has small hom-sets, $\cP$ is a small category. We put
 \[
 \ga_{i,j}^{i'}:=\ga_i^{i'}\circ\ga_{i,j}^i\,,\qquad\text{for all }i\leq i'\text{ in }I
 \text{ and all }j\in J_i\,.
 \]

\begin{claim}
The category~$\cP$ is filtered.
\end{claim}

\begin{proof}
First let $(i_0,j_0),(i_1,j_1)\in P$. There exists $i\geq i_0,i_1$ in~$I$. As~$A_{i_l,j_l}$ is finitely presented\index{i}{presented!finitely}, there exists~$j'_l\in J_i$ such that~$\ga_{i_l,j_l}^i$ factors through~$A_{i,j'_l}$, for all $l<2$. Hence, taking $j\geq j'_0,j'_1$ in~$J_i$, both morphisms~$\ga_{i_0,j_0}^i$ and~$\ga_{i_1,j_1}^i$ factor through~$A_{i,j}$, which yields $x_l\colon(i_l,j_l)\to(i,j)$, for all $l<2$. Part of the argument can be followed on Figure~\ref{Fig:CommSqAij}.

\begin{figure}[htb]\index{i}{square (shape of a diagram)}
 \[
 \def\labelstyle{\displaystyle}\xymatrixcolsep{1.5pc}
 \xymatrix{
 A_{i_l,j_l}\ar[rr]^{\ga_{i_l,j_l}^{i_l}}
 \ar[d]_{x_l}\ar[rrd]|-{\ga_{i_l,j_l}^i} &&
 A_{i_l}\ar[d]^{\ga_{i_l}^i}\\
 A_{i,j}\ar[rr]_{\ga_{i,j}^i} && A_i 
 }
 \]
\caption{A commutative square in~$\cA$}
\label{Fig:CommSqAij}
\end{figure}

Next, let $x,y\colon(i_0,j_0)\to(i_1,j_1)$ in~$\cP$, so
$\ga_{i_0,j_0}^{i_1}=\ga_{i_1,j_1}^{i_1}\circ x=\ga_{i_1,j_1}^{i_1}\circ y$. As~$A_{i_0,j_0}$ is finitely presented\index{i}{presented!finitely}, there exists~$j\geq j_1$ in~$J_{i_1}$ such that
$\ga_{i_1,j_1}^{i_1,j}\circ x=\ga_{i_1,j_1}^{i_1,j}\circ y$. Therefore, $\ga_{i_1,j_1}^{i_1,j}\colon(i_1,j_1)\to(i_1,j)$ coequalizes~$x$ and~$y$ in~$\cP$.
\qed\ Claim\end{proof}

Now we put
 \[
 \ga_{i,j}:=\ga_i\circ\ga_{i,j}^i\,,\qquad\text{for all }(i,j)\in P\,,
 \]
and we define a functor $\bA\colon\cP\to\cA$ by $\bA(i,j):=A_{i,j}$ for $(i,j)\in P$ and $\bA(x):=x$ for every morphism~$x$ in~$\cP$.
It is straightforward to verify that $\famm{A_{i,j},\ga_{i,j}}{(i,j)\in P}$ is a cocone above~$\bA$.

We shall now establish the statement
 \begin{equation}\label{Eq:A=colimbA}
 \famm{A,\ga_{i,j}}{(i,j)\in P}=\varinjlim\bA\,.
 \end{equation}
Let $\famm{B,\gb_{i,j}}{(i,j)\in P}$ be a cocone above~$\bA$. In particular, 
for all~$i\in I$, the family $\famm{B,\gb_{i,j}}{j\in J_i}$ is a cocone above
$\famm{A_{i,j},\ga_{i,j}^{i,j'}}{j\leq j'\text{ in }J_i}$, thus, by~\eqref{Eq:CanReprAi}, there exists a unique morphism~$\gb_i\colon A_i\to B$ such that $\gb_{i,j}=\gb_i\circ\ga_{i,j}^i$ for all $j\in J_i$. Let $i\leq i'$ in~$I$ and let $j\in J_i$. As $\ga_i^{i'}\circ\ga_{i,j}^i\colon A_{i,j}\to A_{i'}=\varinjlim_{j'\in J_{i'}}A_{i',j'}$ and~$A_{i,j}$ is finitely presented\index{i}{presented!finitely}, there are $j'\in J_{i'}$ and $x\colon(i,j)\to(i',j')$ in~$\cP$. Hence,
 \begin{align*}
 \gb_{i'}\circ\ga_i^{i'}\circ\ga_{i,j}^i&=\gb_{i'}\circ\ga_{i',j'}^{i'}\circ x\\
 &=\gb_{i',j'}\circ x\\
 &=\gb_{i,j}\\
 &=\gb_i\circ\ga_{i,j}^i\,. 
 \end{align*}
As this holds for all $j\in J_i$, it follows that $\gb_{i'}\circ\ga_i^{i'}=\gb_i$. Hence, as $A=\varinjlim_{i\in I}A_i$, there exists a unique morphism $\gf\colon A\to B$ such that $\gb_i=\gf\circ\ga_i$ for each $i\in I$. It follows that
 \[
 \gf\circ\ga_{i,j}=\gf\circ\ga_i\circ\ga_{i,j}^i
 =\gb_i\circ\ga_{i,j}^i=\gb_{i,j}\,,\qquad\text{for all }(i,j)\in P\,.
 \]
Now let $\gy\colon A\to B$ such that
 \[
 \gy\circ\ga_{i,j}=\gb_{i,j}\,,\qquad\text{for all }(i,j)\in P\,.
 \]
For $i\in I$ and $j\in J_i$, we compute
 \[
 \gy\circ\ga_i\circ\ga_{i,j}^i=\gy\circ\ga_{i,j}=\gb_{i,j}=\gb_i\circ\ga_{i,j}^i\,.
 \]
As this holds for all $j\in J_i$, it follows that $\gy\circ\ga_i=\gb_i$. As this holds for all~$i\in I$ and by the uniqueness statement defining~$\gf$, it follows that~$\gf=\gy$, thus completing the proof of~\eqref{Eq:A=colimbA}.

As, by the claim above, $\cP$ is a small filtered category\index{i}{filtered category}, it follows from~\cite[Theorem~1.5]{AdRo}\index{c}{Ad\'amek, J.}\index{c}{Rosick\'y, J.} that there are a small directed\index{i}{poset!directed} poset~$\ol{P}$ and a cofinal functor from~$\ol{P}$ to~$\cP$. As we have seen that $\ol{\Phi}$ preserves all colimits of directed poset-indexed diagrams in~$\cA^\dagger$, it follows from~\cite[Section~0.11]{AdRo}\index{c}{Ad\'amek, J.}\index{c}{Rosick\'y, J.} and~\eqref{Eq:A=colimbA} that
 \begin{equation}\label{Eq:PhiA=colimPhibA}
 \famm{\ol{\Phi}(A),\ol{\Phi}(\ga_{i,j})}{(i,j)\in P}=\varinjlim\Phi\bA\,.
 \end{equation}
Now we can conclude the proof of \eqref{Eq:olPhi(colimA)}. Let
$\famm{S,\gs_i}{i\in I}$ be a cocone in~$\cS$ above
$\famm{\ol{\Phi}(A_i),\ol{\Phi}(\ga_i^{i'})}{i\leq i'\text{ in }I}$.
Put $\gs_{i,j}:=\gs_i\circ\ol{\Phi}(\ga_{i,j}^i)$, for all $(i,j)\in P$. For every morphism $x\colon(i,j)\to(i',j')$ in~$\cP$,
 \begin{align*}
 \gs_{i',j'}\circ\Phi(x)&=\gs_{i'}\circ\ol{\Phi}(\ga_{i',j'}^{i'})\circ\Phi(x)\\
 &=\gs_{i'}\circ\ol{\Phi}(\ga_i^{i'})\circ\ol{\Phi}(\ga_{i,j}^i)\\
 &=\gs_i\circ\ol{\Phi}(\ga_{i,j}^i)\\
 &=\gs_{i,j}\,. 
 \end{align*}
This proves that $\famm{S,\gs_{i,j}}{(i,j)\in P}$ is a cocone above~$\Phi\bA$, thus, by~\eqref{Eq:PhiA=colimPhibA}, there exists a unique $\gf\colon\ol{\Phi}(A)\to S$ such that $\gs_{i,j}=\gf\circ\ol{\Phi}(\ga_{i,j})$ for all $(i,j)\in P$. For all $(i,j)\in P$,
 \[
 \gs_i\circ\ol{\Phi}(\ga_{i,j}^i)=\gs_{i,j}=\gf\circ\ol{\Phi}(\ga_{i,j})
 =\gf\circ\ol{\Phi}(\ga_i)\circ\ol{\Phi}(\ga_{i,j}^i)\,.
 \]
Fix $i\in I$. As the equation above is satisfied for all~$j\in J_i$, it follows that~$\gs_i=\gf\circ\ol{\Phi}(\ga_i)$.

Finally, let $\gy\colon\ol{\Phi}(A)\to S$ satisfy $\gs_i=\gy\circ\ol{\Phi}(\ga_i)$ for all~$i\in I$. Then $\gs_{i,j}=\gy\circ\ol{\Phi}(\ga_{i,j})$ for all $(i,j)\in P$, and thus, by the uniqueness statement defining~$\gf$, we get~$\gf=\gy$, which concludes the proof.
\qed\end{proof}

\begin{remk}\label{Rk:SmSetTh}
Under mild set-theoretical assumptions, it is easy to remove from the statement of the final sentence of Proposition~\ref{P:ArrObj2Diag} the hypothesis that~$\cA^\dagger$ has small hom-sets. In the proof of Proposition~\ref{P:ArrObj2Diag}, we need to replace~$\cP$ by a \emph{small} subcategory~$\cP^*$ satisfying the following conditions:
\begin{description}
\item[\tui] $\Ob\cP^*=P$.

\item[\tuii] The morphism $\ga_{i,j_0}^{i,j_1}$ belongs to $\Mor_{\cP^*}((i,j_0),(i,j_1))$, for all~$i\in I$ and all~$j_0\leq j_1$ in~$J_i$.

\item[\tuiii] $\Mor_{\cP}((i,j),(i',j'))\neq\es$ iff $\Mor_{\cP^*}((i,j),(i',j'))\neq\es$, for all elements~$(i,j)$ and~$(i',j')$ in~$P$.
\end{description}

For example, such a $\cP^*$ can be constructed in case the ambient set-theoretical universe satisfies the Bernays-G\"odel class theory with axiom of foundation. However, as we will need Proposition~\ref{P:ArrObj2Diag} only in case $\cA=\Bool_P$\index{s}{BoolP@$\Bool_P$} (cf. Section~\ref{S:Pnorm}), which has small hom-sets, we shall not expand on this further here.
\end{remk}

\begin{remk}\label{Rk:Restrkappa}
Let~$\gk$ be an infinite regular cardinal. If we assume only, in the statement of Proposition~\ref{P:ArrObj2Diag}, that~$\cS$ has all $\gk$-small directed colimits, that every object in~$\cA$ is a colimit of a $\gk$-small directed poset-indexed diagram in~$\cA^\dagger$, and that~$\cA^\dagger$ has $\gk$-small hom-sets, then the following analogue of the conclusion of Proposition~\ref{P:ArrObj2Diag} remains valid: \emph{The functor~$\Phi$ extends to a functor~$\ol{\Phi}\colon\cA\to\cS$ that preserves all~$\gk$-small directed colimits}. We shall use this result only in case $\cA=\Bool_P$\index{s}{BoolP@$\Bool_P$} (cf. Section~\ref{S:BoolP}), where these assumptions will be automatically satisfied.
\end{remk}

\begin{remk}\label{Rk:ExtPhi}
It is not hard to verify that any two extensions of~$\Phi$ to~$\cA$ preserving all colimits of directed poset-indexed diagrams in~$\cA^\dagger$ are isomorphic above~$\Phi$. Hence we shall call the functor~$\ol{\Phi}$ constructed in Proposition~\ref{P:ArrObj2Diag} the \emph{natural extension of~$\Phi$ to~$\cA$}.
\end{remk}

\section{Projectability witnesses}\label{S:ProjWit}

Given categories~$\cA$ and~$\cB$ together with a functor $\Psi\colon\cA\to\cB$, it will often be the case (though not always, a notable exception being developed in Chapter~\ref{Ch:CongPres}) that certain arrows of the form $\gf\colon\Psi(A)\to B$ can be turned into \emph{isomorphisms} of the form $\eps\colon\Psi(\oll{A})\to B$, for a ``canonical quotient''~$\oll{A}$ of~$A$. Lemma~\ref{L:LiftProj} will be our key idea to turn statements involving the ``double arrows''\index{i}{double arrow}, mentioned in Section~\ref{Su:Contents} and introduced formally in Section~\ref{S:CLL}, to isomorphisms. Roughly speaking, the existence of projectability witnesses is a categorical combination of the First and Second Isomorphism Theorems\index{i}{Isomorphism Theorem (First ${}_{-}$)}\index{i}{Isomorphism Theorem (Second ${}_{-}$)} for algebraic systems\index{i}{algebraic system} (cf. Lemmas~\ref{L:FirstIsomThm} and~\ref{L:SecIsomThm}).

We recall a definition from Wehrung~\cite{RetrLift}\index{c}{Wehrung, F.}.

\begin{defn}\label{D:ProjFunct}
Let $\Psi$ be a functor from a category~$\cA$ to a category~$\cB$,
let $A\in\Ob\cA$, $B\in\Ob\cB$, and $\gf\colon\Psi(A)\to B$.
A \emph{projectability witness}\index{i}{projectability witness} for $(\gf,A,B)$
(or, abusing notation, ``for $\gf\colon\Psi(A)\to B$'') with respect to~$\Psi$ is a pair $(a,\eps)$ satisfying the following conditions:
\begin{description}
\item[\tui] $a\colon A\onto\oll{A}$\index{s}{AtoonB@$f\colon A\onto B$} is an epimorphism in $\cA$.

\item[\tuii] $\eps\colon\Psi(\oll{A})\to B$ is an isomorphism in $\cB$.

\item[\tuiii] $\gf=\eps\circ\Psi(a)$.

\item[\tuiv] For every $f\colon A\to X$ in $\cA$ and every
$\gh\colon\Psi(\oll{A})\to\Psi(X)$ such that
$\Psi(f)=\gh\circ\Psi(a)$, there exists $g\colon\oll{A}\to X$ in
$\cA$ such that $f=g\circ a$ and $\gh=\Psi(g)$.
\end{description}
\end{defn}

We observe that the morphism $g$ in (iii) above is necessarily
\emph{unique} (because~$a$ is an epimorphism). Furthermore, the
projectability witness\index{i}{projectability witness} $(a,\eps)$ is unique up to isomorphism.

Definition~\ref{D:ProjFunct} is illustrated on Figure~\ref{Fig:ProjxF}.

\begin{figure}[htb]
 \[
 \def\labelstyle{\displaystyle}
 \xymatrix{
 \Psi(A)\ar[r]^{\gf}\ar[d]_{\Psi(a)} & B &
 \Psi(A)\ar[d]_{\Psi(a)}\ar[drr]^{\Psi(f)} & & &
 A\ar@{->>}[d]_a\ar[dr]^f & \\
 \Psi(\oll{A})\ar[ru]|-{\eps\cong} & &
 \Psi(\oll{A})\ar[rr]_{\gh=\Psi(g)} & & \Psi(X)
 & \oll{A}\ar[r]_g & X
 }
 \]
\caption{$(a,\eps)$ is a projectability witness for $\gf\colon\Psi(A)\to B$}
\label{Fig:ProjxF}
\end{figure}

Loosely speaking, the following property says that the existence of enough projectability witnesses\index{i}{projectability witness} entails the existence of liftings\index{i}{diagram!lifting}.

\begin{lem}\label{L:LiftProj}
Let $\cI$, $\cA$, and $\cS$ be categories, let $\xD\colon\cI\to\cA$,
$\Psi\colon\cA\to\cS$, and $\xS\colon\cI\to\cS$ be functors. Let
$\gt\colon\Psi\xD\todot\xS$\index{s}{AtorightarrowdotB@$f\colon\xA\todot\xB$} be a natural transformation such that
$\gt_i\colon\Psi\xD(i)\to\nobreak\xS(i)$ has a projectability witness\index{i}{projectability witness} $(a_i,\gh_i)$,
with $a_i\colon\xD(i)\onto\xE(i)$\index{s}{AtoonB@$f\colon A\onto B$}, for each $i\in\Ob\cI$. Then~$\xE$ can be uniquely extended to a functor from~$\cI$ to~$\cA$ such that $a\colon\xD\todot\xE$ is a natural transformation and $\gh\colon\Psi\xE\todot\xS$\index{s}{AtorightarrowdotB@$f\colon\xA\todot\xB$} is a natural equivalence with $\gt_i=\gh_i\circ\Psi(a_i)$ for each $i\in\Ob\cI$.
\end{lem}

\begin{proof}
Note that $\gh_i\colon\Psi\xE(i)\overset{\cong}{\to}\xS(i)$ and
$\gt_i=\gh_i\circ\Psi(a_i)$, for each $i\in\Ob\cI$. We need to extend~$\xE$ to a functor. Let $f\colon i\to j$ in $\cI$. As
 \begin{align*}
 (\gh_j^{-1}\circ\xS(f)\circ\gh_i)\circ\Psi(a_i)&=
 \gh_j^{-1}\circ\xS(f)\circ\gt_i\\
 &=\gh_j^{-1}\circ\gt_j\circ\Psi\xD(f)\\
 &=\Psi(a_j\circ\xD(f))\,, 
 \end{align*}
there exists a unique $\xE(f)\colon\xE(i)\to\xE(j)$ such that
$\Psi\xE(f)=\gh_j^{-1}\circ\xS(f)\circ\gh_i$ and
$\xE(f)\circ a_i=a_j\circ\xD(f)$, see Figure~\ref{Fig:FunctxE}.

\begin{figure}[htb]
 \[
 \def\labelstyle{\displaystyle}
 \xymatrix{
 \xD(i)\ar[r]^{\xD(f)}\ar@{->>}[d]_{a_i} & \xD(j)\ar@{->>}[d]^{a_j} & &
 \Psi\xE(i)\ar[rr]^{\Psi\xE(f)}\ar[d]_{\gh_i}^{\cong} & &
 \Psi\xE(j)\ar[d]^{\gh_j}_{\cong}\\
 \xE(i)\ar[r]^{\xE(f)} & \xE(j) & & \xS(i)\ar[rr]^{\xS(f)} & & \xS(j)
 }
 \]
\caption{The functor $\xE$ created from projectability witnesses}
\label{Fig:FunctxE}
\end{figure}
The uniqueness statement on~$\xE(f)$, together with routine calculations, show easily that~$\xE$ is a functor.
\qed\end{proof}

\chapter[Boolean algebras scaled with respect to a poset]{Boolean algebras that are scaled with respect to a poset}\label{Ch:PscaledBAs}

\textbf{Abstract.} Our main result, CLL\index{i}{Condensate Lifting Lemma (CLL)} (Lemma~\ref{L:CLL}), involves a construction that turns a \emph{diagram}~$\overrightarrow{A}$, indexed by a poset~$P$, from a category~$\cA$, to an \emph{object} of~$\cA$, called a \emph{condensate}\index{i}{condensate} of~$\overrightarrow{A}$ (cf. Definition~\ref{D:Condensate}). A condensate\index{i}{condensate} of~$\overrightarrow{A}$ will be written in the from $\bB\otimes\overrightarrow{A}$\index{s}{otimAS@$\bA\otimes\overrightarrow{S}$, $\gf\otimes\overrightarrow{S}$}, where~$\bB$ is a Boolean algebra\index{i}{algebra!Boolean} with additional structure---we shall say a \emph{$P$-scaled Boolean algebra}\index{i}{algebra!Pscaled Boolean@$P$-scaled Boolean} (Definition~\ref{D:BoolP}). It will turn out (cf. Proposition~\ref{P:StoneDualP}) that $P$-scaled Boolean algebras\index{i}{algebra!Pscaled Boolean@$P$-scaled Boolean} are the dual objects of topological objects called \emph{$P$-normed Boolean spaces}\index{i}{normed (Boolean) space} (cf. Definition~\ref{D:Pnorm}). By definition, a $P$-normed\index{i}{normed (Boolean) space} topological space is a topological space~$X$ endowed with a map (the ``norm function'') from~$X$ to~$\Id P$\index{s}{IdP@$\Id P$, $P$ poset} which is \emph{continuous} with respect to the given topology of~$X$ and the Scott topology on~$\Id P$\index{s}{IdP@$\Id P$, $P$ poset}. In case~$X$ is a one-point space with norm an ideal\index{i}{ideal!of a poset}~$H$ of~$P$ and~$\bB$ is the corresponding $P$-scaled Boolean algebra\index{i}{algebra!Pscaled Boolean@$P$-scaled Boolean}, $\bB\otimes\overrightarrow{A}=\varinjlim_{p\in H}A_p$\index{s}{otimAS@$\bA\otimes\overrightarrow{S}$, $\gf\otimes\overrightarrow{S}$}. In case~$X$ is finite and $\nu(x)=P\dnw f(x)$ (where $f(x)\in P$) for each $x\in X$, then $\bB\otimes\overrightarrow{A}=\prod\famm{A_{f(x)}}{x\in X}$\index{s}{otimAS@$\bA\otimes\overrightarrow{S}$, $\gf\otimes\overrightarrow{S}$}. The latter situation describes the case where~$\bB$ is a \emph{finitely presented}\index{i}{presented!finitely} $P$-scaled Boolean algebra\index{i}{algebra!Pscaled Boolean@$P$-scaled Boolean} (cf. Definition~\ref{D:CompactBP} and Corollary~\ref{C:FinPres=Comp}). In the general case, there is a directed colimit representation $\bB=\varinjlim_{i\in I}\bB_i$ where all the~$\bB_i$ are finitely presented\index{i}{presented!finitely} (cf. Proposition~\ref{P:FinRepresBP}) and then $\bB\otimes\overrightarrow{A}$\index{s}{otimAS@$\bA\otimes\overrightarrow{S}$, $\gf\otimes\overrightarrow{S}$} is defined as the corresponding directed colimit of the $\bB_i\otimes\overrightarrow{A}$\index{s}{otimAS@$\bA\otimes\overrightarrow{S}$, $\gf\otimes\overrightarrow{S}$}. That this can be done, and that the resulting functor $\bB\mapsto\bB\otimes\overrightarrow{A}$\index{s}{otimAS@$\bA\otimes\overrightarrow{S}$, $\gf\otimes\overrightarrow{S}$} preserves all small directed colimits, will follow from Proposition~\ref{P:ArrObj2Diag}.

\section[Pseudo join-semilattices]{Pseudo join-semilattices, supported posets, and almost join-semilattices}\label{S:TPS}

The statement of CLL\index{i}{Condensate Lifting Lemma (CLL)} (Lemma~\ref{L:CLL}) involves posets~$P$ for which there exists a ``$\gl$-lifter''\index{i}{lifter ($\gl$-)} $(X,\bX)$; here $X$ is a poset, endowed with an isotone map $\partial\colon X\to P$, and~$\bX$ is a certain set of ideals\index{i}{ideal!of a poset} of~$X$. Defining the condensate\index{i}{condensate} $\xF(X)\otimes\overrightarrow{A}$\index{s}{FxX@$\xF(X)$}\index{s}{otimAS@$\bA\otimes\overrightarrow{S}$, $\gf\otimes\overrightarrow{S}$}, involved in the statement of CLL\index{i}{Condensate Lifting Lemma (CLL)}, requires the construction of a certain $P$-scaled Boolean algebra\index{i}{algebra!Pscaled Boolean@$P$-scaled Boolean}, defined by generators and relations, $\xF(X)$\index{s}{FxX@$\xF(X)$}. And the definition of~$\xF(X)$\index{s}{FxX@$\xF(X)$} will require~$X$ be a \emph{\pjs}\index{i}{pseudo join-semilattice} (cf. Definition~\ref{D:PJS}). Together with \pjs s\index{i}{pseudo join-semilattice}, we will also need to introduce \emph{supported posets}\index{i}{poset!supported} and \emph{\ajs s}\index{i}{almost join-semilattice}.

\begin{notation}\label{Not:sor}
Let~$X$ be a subset in a poset~$P$. We denote by $\Sor X$, or~$\Sor_PX$ in case~$P$ needs to be specified, the set of all minimal elements of~$P\Upw X$. We shall write $a_0\sor\cdots\sor a_{n-1}$\index{s}{minmaj@$a_0\sor\cdots\sor a_{n-1}$, $X_0\sor\cdots\sor X_{n-1}$|ii}, or $\Sor_{i<n}a_i$\index{s}{Minmaj@$\Sor_{i<n}a_i$, $\Sor X$, $\Sor_PX$|ii}, instead of $\Sor\set{a_0,\dots,a_{n-1}}$.

For subsets $X_0$, \dots, $X_{n-1}$ of~$P$, we set
 \[
 X_0\sor\cdots\sor X_{n-1}:=\bigcup\famm{\Sor_{i<n}a_i}{a_i\in X_i\text{ for all }i<n}\,.
 \]
\end{notation}

\begin{defn}\label{D:PJS}
We say that a subset~$X$ in a poset~$P$ is \emph{$\sor$-closed}\index{i}{closssedsor@$\sor$-closed|ii} if $\Sor Y\subseteq\nobreak X$ for any finite $Y\subseteq X$. The \emph{$\sor$-closure}\index{i}{closssuresor@$\sor$-closure|ii} of a subset~$X$ of a poset~$P$ is the least~$\sor$-closed subset of~$P$ containing~$X$. We say that~$P$ is
\begin{itemize}
\item a \emph{\pjs}\index{i}{pseudo join-semilattice|ii} if the subset $P\Upw X$ is a finitely generated\index{i}{finitely generated!upper subset} upper subset of~$P$ (cf. Section~\ref{Su:Posets}), for every subset~$X$ of~$P$ which is either empty or a two-element set (\emph{and thus for every finite subset~$X$ of~$P$, cf. Lemma~\textup{\ref{L:PUpwX}}}).

\item \emph{supported}\index{i}{poset!supported|ii} if~$P$ is a \pjs\ and the $\sor$-closure of every finite subset of~$P$ is finite.

\item an \emph{\ajs}\index{i}{almost join-semilattice|ii} if~$P$ is a \pjs\ and~$P\dnw a$ is a \js\ for each $a\in P$.
\end{itemize}
\end{defn}

In particular, taking $X:=\es$ in the definition of a \pjs, observe that every \pjs\index{i}{pseudo join-semilattice}\ has only finitely many minimal elements and that every element lies above one of those minimal elements.

Some of these definitions are closely related to definitions used in domain theory\index{i}{continuous domain}. For example, a poset is supported\index{i}{poset!supported} if{f} it is \emph{mub-complete} as defined on~\cite[Definition~4.2.1]{AbJu}\index{c}{Abramsky, S.}\index{c}{Jung, A.}. Also, pointed continuous domains\index{i}{continuous domain} in which every principal ideal\index{i}{ideal!of a poset}\index{i}{ideal!principal ${}_{-}$, of a poset} is a \js\ are called \emph{L-domains} in~\cite[Definition~4.1.1]{AbJu}\index{c}{Abramsky, S.}\index{c}{Jung, A.}.

Our definition of a supported\index{i}{poset!supported} poset is equivalent to the one presented in Gillibert~\cite{Gill1}\index{c}{Gillibert, P.}. Every supported\index{i}{poset!supported} poset is a \pjs\index{i}{pseudo join-semilattice}, and we shall see in Corollary~\ref{C:TrSupp} that every \ajs\index{i}{almost join-semilattice}\ is supported\index{i}{poset!supported}.
The infinite poset represented on the left hand side of Figure~\ref{Fig:posets} is not a \pjs\index{i}{pseudo join-semilattice}, the one in the middle is a non-supported\index{i}{poset!supported} \pjs\index{i}{pseudo join-semilattice}, while the one on the right hand side is (as every finite poset) supported\index{i}{poset!supported}, while it is not an \ajs\index{i}{almost join-semilattice}.

\begin{figure}[htb]
\includegraphics{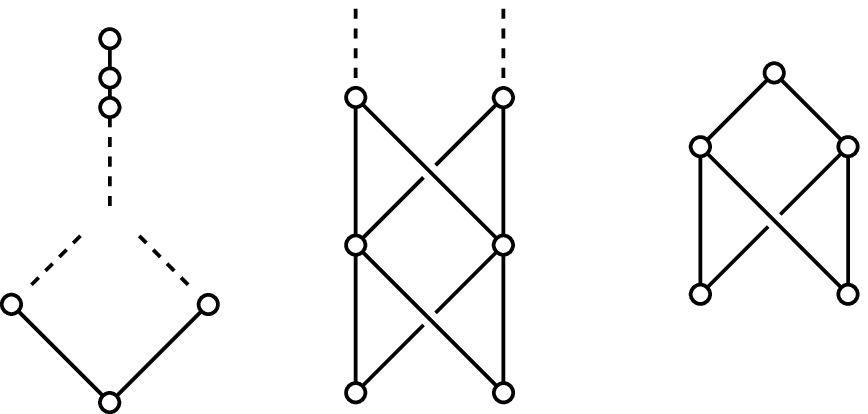}
\caption{A few posets that are not \ajs s}
\label{Fig:posets}
\end{figure}

Observe that for a finite subset~$X$ in a poset~$P$, $P\Upw X$ is a finitely generated\index{i}{finitely generated!upper subset} upper subset of~$P$ if{f} $\Sor X$ is finite and every element of $P\Upw X$ lies above some element of~$\Sor X$.

\begin{lem}\label{L:PUpwX}
Let $P$ be a \pjs\index{i}{pseudo join-semilattice}. Then $P\Upw X$ is a finitely generated\index{i}{finitely generated!upper subset} upper subset of~$P$, for every finite subset~$X$ of~$P$.
\end{lem}

\begin{proof}
The statement is trivial if $X=\es$. Suppose that~$X\neq\es$ and that the statement has been proved for all subsets of~$P$ of cardinality smaller than $\card X$. Pick $x\in X$, and set $Y:=X\setminus\set{x}$. By the induction hypothesis, there exists a finite subset~$V$ of~$P$ such that $P\Upw Y=P\upw V$. By assumption, for each $v\in V$ there exists a finite subset~$U_v$ of~$P$ such that $P\Upw\set{x,v}=P\upw U_v$. Then $U:=\bigcup\famm{U_v}{v\in V}$ is a finite subset of~$P$, and $P\Upw X=P\upw U$.
\qed\end{proof}

It follows that in the context of Lemma~\ref{L:PUpwX}, $\Sor X$ is finite in case~$X$ is finite. Every nonempty $\sor$-closed subset~$X$ in a \pjs\index{i}{pseudo join-semilattice}~$P$ is also a \pjs\index{i}{pseudo join-semilattice}, and $u\sor_Xv=u\sor_Pv$, for all $u,v\in X$.

\begin{lem}\label{L:AssocSor}
For a \pjs\index{i}{pseudo join-semilattice}~$P$, a positive integer~$n$, finite subsets $X_0$, \dots, $X_{n-1}$ of~$P$, and $X=\bigcup_{i<n}X_i$, the following statements hold:
\begin{description}
\item[\tui] $\Sor X\subseteq(\Sor X_0)\sor\cdots\sor(\Sor X_{n-1})$.

\item[\tuii] In case $P$ is an \ajs\index{i}{almost join-semilattice}, $\Sor X=(\Sor X_0)\sor\cdots\sor(\Sor X_{n-1})$.
\end{description}
\end{lem}

\begin{proof}
(i). Let $a\in\Sor X$. For all $i<n$, it follows from Lemma~\ref{L:PUpwX} that $P\Upw X_i$ is a finitely generated\index{i}{finitely generated!upper subset} upper subset of~$P$; as~$X_i\subseteq P\dnw a$, there exists $a_i\in\Sor X_i$ such that $a_i\leq a$. If~$x\in P$ such that $a_i\leq x\leq a$ for each $i<n$, then $X_i\subseteq P\dnw x$ for each $i<n$, thus $X\subseteq P\dnw x$, and thus, as $a\in\Sor X$ and $x\leq a$, we get that $x=a$; whence $a\in\Sor_{i<n}a_i$. So $a\in(\Sor X_0)\sor\cdots\sor(\Sor X_{n-1})$.

(ii). Let $a\in(\Sor X_0)\sor\cdots\sor(\Sor X_{n-1})$; by definition, $a\in\Sor_{i<n}a_i$, for some $a_i\in\Sor X_i$ for each $i<n$. As $a_i\in\Sor X_i$ and $\set{a_i}\cup X_i$ is contained in the \js\ $P\dnw a$, we obtain that
 \begin{equation}\label{Eq:aiXiPda}
 a_i\text{ is the join of }X_i\text{ in }P\dnw a\,.
 \end{equation}
Now let $x\in P\dnw a$ such that $P\dnw x$ contains~$X$. Hence all $X_i$s are below~$x$ in $P\dnw a$, thus, by~\eqref{Eq:aiXiPda}, $a_i\leq x$ for each $i<n$. As $x\leq a$ and $a\in\Sor_{i<n}a_i$, we obtain that $x=a$. This proves that $a\in\Sor X$.
\qed\end{proof}

\begin{cor}\label{C:TrSupp}
Every \ajs\index{i}{almost join-semilattice}\ $P$ is supported\index{i}{poset!supported}. Furthermore, the $\sor$-closure of any subset~$X$ of~$P$ is $\bigcup\famm{\Sor Y}{Y\in[X]^{<\go}}$.
\end{cor}

\begin{proof}
As~$P$ is a \pjs\index{i}{pseudo join-semilattice}, all the subsets $\Sor Y$, for $Y\in\nobreak[P]^{<\go}$, are finite (cf. Lemma~\ref{L:PUpwX}), and thus it is sufficient to establish the second statement. It suffices in turn to prove that the subset $\ol{X}:=\bigcup\famm{\Sor Y}{Y\in[X]^{<\go}}$ is $\sor$-closed. For $u,v\in\ol{X}$, there are finite subsets~$U$ and~$V$ of~$X$ such that $u\in\Sor U$ and $v\in\Sor V$. It follows from Lemma~\ref{L:AssocSor}(ii) that $u\sor v\subseteq\Sor(U\cup V)$, and thus $u\sor v\subseteq\ol{X}$. The conclusion follows.
\qed\end{proof}

\section[$P$-normed spaces, $P$-scaled Boolean algebras]{$P$-normed spaces and $P$-scaled Boolean algebras}\label{S:Pnorm}

A condensate\index{i}{condensate} of a diagram~$\overrightarrow{A}$, indexed by a poset~$P$, will be a structure of the form $\bB\otimes\overrightarrow{A}$\index{s}{otimAS@$\bA\otimes\overrightarrow{S}$, $\gf\otimes\overrightarrow{S}$}, where~$\bB$ is a Boolean algebra\index{i}{algebra!Boolean} with an additional structure---we shall say a \emph{$P$-scaled Boolean algebra}\index{i}{algebra!Pscaled Boolean@$P$-scaled Boolean} (cf. Definition~\ref{D:BoolP}). Many readers may find it more intuitive to describe $P$-scaled Boolean algebras\index{i}{algebra!Pscaled Boolean@$P$-scaled Boolean} by their dual (topological) spaces, the \emph{$P$-normed Boolean spaces}\index{i}{normed (Boolean) space} (cf. Definition~\ref{D:Pnorm}). However, as the statement of CLL\index{i}{Condensate Lifting Lemma (CLL)} will involve directed colimits, as opposed to direct limits, we chose to write all our proofs algebraically, that is, formulated in the language of $P$-scaled Boolean algebras\index{i}{algebra!Pscaled Boolean@$P$-scaled Boolean} rather than $P$-normed\index{i}{normed (Boolean) space} Boolean spaces. The present section is mainly devoted to introducing the duality between $P$-scaled Boolean algebras\index{i}{algebra!Pscaled Boolean@$P$-scaled Boolean} and $P$-normed\index{i}{normed (Boolean) space} Boolean spaces.

Throughout this section we shall fix a poset~$P$.

\begin{defn}\label{D:Pnorm}
A \emph{$P$-norm} on a topological space~$X$ is a map $\nu\colon X\to\Id P$\index{s}{IdP@$\Id P$, $P$ poset} such that $\setm{x\in X}{p\in\nu(x)}$ is open for each $p\in P$. A \emph{$P$-normed space}\index{i}{normed (Boolean) space|ii} is a pair $\bX=(X,\nu)$, where $\nu$ is a $P$-norm on a topological space~$X$. Then we shall call the \emph{norm} of an element~$x$ of~$X$ the ideal\index{i}{ideal!of a poset} $\nu(x)$, and often denote it by $\norm{x}$\index{s}{normx@$\Vert{x}\Vert$, $\Vert{x}\Vert_\bX$, $x$ point|ii}, or $\norm{x}_\bX$ in case~$\bX$ needs to be specified.
\end{defn}

\begin{defn}\label{D:morphPnorm}
For $P$-normed\index{i}{normed (Boolean) space} spaces~$\bX$ and~$\bY$, a \emph{morphism from~$\bX$ to~$\bY$} is a continuous map $f\colon X\to Y$ such that $\norm{f(x)}_{\bY}\subseteq\norm{x}_\bX$\index{s}{normx@$\Vert{x}\Vert$, $\Vert{x}\Vert_\bX$, $x$ point} for all~$x\in X$. We shall denote by $\BTop_P$\index{s}{BTop@$\BTop_P$|ii} the category of $P$-normed\index{i}{normed (Boolean) space} Boolean spaces with the morphisms defined above.
\end{defn}

The following definition gives a description of the dual objects to $P$-normed\index{i}{normed (Boolean) space} Boolean spaces. This duality will be achieved in Proposition~\ref{P:StoneDualP}.
\goodbreak

\begin{defn}\label{D:BoolP}
Denote by $\Bool_P$\index{s}{BoolP@$\Bool_P$|ii} the category described as follows.
\begin{itemize}
\item The objects of $\Bool_P$, called \emph{$P$-scaled Boolean algebras}\index{i}{algebra!Pscaled Boolean@$P$-scaled Boolean|ii}, are the families of the form $\bA=\bigl(A,\famm{A^{(p)}}{p\in P}\bigr)$\index{s}{Aidp@$A^{(p)}$|ii}, where~$A$ is a Boolean algebra\index{i}{algebra!Boolean} and $A^{(p)}$ is an ideal\index{i}{ideal!of a poset} of~$A$ for all $p\in P$, and the following conditions are satisfied:
\begin{description}
\item[\tui] $A=\bigvee\famm{A^{(p)}}{p\in P}$.

\item[\tuii] $A^{(p)}\cap A^{(q)}=\bigvee\famm{A^{(r)}}{r\geq p,q\text{ in }P}$, for all $p,q\in P$.
\end{description}

\item For objects~$\bA$ and $\bB$ in~$\Bool_P$\index{s}{BoolP@$\Bool_P$}, a \emph{morphism} from~$\bA$ to~$\bB$ is a morphism $f\colon A\to B$ of Boolean algebras\index{i}{algebra!Boolean} such that $f``(A^{(p)})\subseteq B^{(p)}$, for all $p\in P$.
\end{itemize}
\end{defn}

In (i) and (ii) above, joins are, of course, evaluated in the lattice~$\Id A$\index{s}{IdP@$\Id P$, $P$ poset} of ideals\index{i}{ideal!of a poset} of~$A$. For example, (i) means that there exists a decomposition in~$A$ of the form $1=\bigvee\famm{a_p}{p\in Q}$, where~$Q$ is a finite subset of~$P$ and $a_p\in A^{(p)}$, for all $p\in Q$. Observe also that $\famm{A^{(p)}}{p\in P}$ is necessarily \emph{antitone} (i.e., $p\leq q$ implies that $A^{(q)}\subseteq A^{(p)}$, for all $p,q\in P$).

For $P$-scaled Boolean algebra\index{i}{algebra!Pscaled Boolean@$P$-scaled Boolean} $\bA$, we put\index{s}{normfa@$\Vert{\fra}\Vert$, $\Vert{\fra}\Vert_\bA$, $\fra$ ultrafilter|ii}
 \begin{equation}\label{Eq:normfra}
 \norm{\fra}_\bA:=\setm{p\in P}{\fra\cap A^{(p)}\neq\es}\,,
 \qquad\text{for each }\fra\in\Ult A\,.
 \end{equation}
We will write~$\norm{\fra}$\index{s}{normfa@$\Vert{\fra}\Vert$, $\Vert{\fra}\Vert_\bA$, $\fra$ ultrafilter} instead of~$\norm{\fra}_\bA$ in case~$\bA$ is understood.

\begin{lem}\label{L:Ult2Pnorm}
The subset~$\norm{\fra}$ is an ideal\index{i}{ideal!of a poset} of~$P$, for every ultrafilter~$\fra$ of~$A$. Furthermore, the map $\fra\mapsto\norm{\fra}$\index{s}{normfa@$\Vert{\fra}\Vert$, $\Vert{\fra}\Vert_\bA$, $\fra$ ultrafilter} is a $P$-norm on~$\Ult A$\index{s}{UltB@$\Ult B$, $\Ult\bB$, $\Ult\gf$}.
\end{lem}

\begin{proof}
It is obvious that $\norm{\fra}$\index{s}{normfa@$\Vert{\fra}\Vert$, $\Vert{\fra}\Vert_\bA$, $\fra$ ultrafilter} is a lower subset\index{i}{lower subset} of~$P$.

As there exists a decomposition in~$A$ of the form $1=\bigvee\famm{a_p}{p\in Q}$, for a finite subset~$Q$ of~$P$ and elements~$a_p\in A^{(p)}$, for all $p\in Q$, and as~$\fra$ is an ultrafilter of~$A$, there exists~$p\in Q$ such that $a_p\in\fra$, so~$p\in\norm{\fra}$, and so~$\norm{\fra}$\index{s}{normfa@$\Vert{\fra}\Vert$, $\Vert{\fra}\Vert_\bA$, $\fra$ ultrafilter} is nonempty.

Let $p,q\in\norm{\fra}$. Pick $u\in\fra\cap A^{(p)}$ and $v\in\fra\cap A^{(q)}$. Then~$w:=u\wedge v$ belongs to $\fra\cap(A^{(p)}\cap A^{(q)})$, and so there exists a decomposition of the form $w=\bigvee\famm{w_r}{r\in Q}$ in~$A$, where~$R$ is a finite subset of $P\Upw\set{p,q}$ and~$w_r\in A^{(r)}$ for all~$r\in R$. As~$\fra$ is an ultrafilter, there exists~$r\in R$ such that $w_r\in\fra$, and so $r\in\norm{\fra}$\index{s}{normfa@$\Vert{\fra}\Vert$, $\Vert{\fra}\Vert_\bA$, $\fra$ ultrafilter}, with $r\geq p,q$. This proves that~$\norm{\fra}$ is directed. Therefore, $\norm{\fra}$\index{s}{normfa@$\Vert{\fra}\Vert$, $\Vert{\fra}\Vert_\bA$, $\fra$ ultrafilter} is an ideal\index{i}{ideal!of a poset} of~$P$.

For each $p\in P$, $\setm{\fra\in\Ult A}{p\in\norm{\fra}}=\setm{\fra\in\Ult A}{\fra\cap A^{(p)}\neq\es}$\index{s}{normfa@$\Vert{\fra}\Vert$, $\Vert{\fra}\Vert_\bA$, $\fra$ ultrafilter}\index{s}{UltB@$\Ult B$, $\Ult\bB$, $\Ult\gf$} is obviously an open subset of $\Ult A$\index{s}{UltB@$\Ult B$, $\Ult\bB$, $\Ult\gf$}.
\qed\end{proof}

In case $\fra=A\upw a$, for an atom~$a$ of~$A$, we obtain the following particular case of Lemma~\ref{L:Ult2Pnorm}:

\begin{lem}\label{L:normaId}
For every object~$\bA$ of~$\Bool_P$\index{s}{BoolP@$\Bool_P$} and every atom~$a$ of~$A$, the set
\index{s}{norma@$\Vert{a}\Vert$, $\Vert{a}\Vert_\bA$, $a$ atom|ii}
 \begin{equation}\label{Eq:normAtom}
 \norm{a}:=\setm{p\in P}{a\in A^{(p)}}
 \end{equation}
is an ideal\index{i}{ideal!of a poset} of~$P$.
\end{lem}

Observe that the equation $\norm{A\upw a}=\norm{a}$\index{s}{normfa@$\Vert{\fra}\Vert$, $\Vert{\fra}\Vert_\bA$, $\fra$ ultrafilter}\index{s}{norma@$\Vert{a}\Vert$, $\Vert{a}\Vert_\bA$, $a$ atom} is satisfied for every atom~$a$ of~$A$, where~$\norm{a}$ is evaluated using~\eqref{Eq:normAtom} and $\norm{A\upw a}$ is evaluated using~\eqref{Eq:normfra}.

It is easy to construct examples where $\norm{a}$ has no largest element, even in case~$A$ is finite. Denote by $|a|$\index{s}{nora@$\vert a\vert$, $a$ atom|ii} the largest element of $\norm{a}$\index{s}{norma@$\Vert{a}\Vert$, $\Vert{a}\Vert_\bA$, $a$ atom} if it exists.

We shall denote by $\Ult\bA$\index{s}{UltB@$\Ult B$, $\Ult\bB$, $\Ult\gf$|ii} the Boolean space $\Ult A$ endowed with the $P$-norm defined in~\eqref{Eq:normfra}, and we shall call it the \emph{dual $P$-normed space}\index{i}{normed (Boolean) space}\index{i}{dual $P$-normed space|ii} of~$\bA$.

The proof of the following lemma is straightforward (see \eqref{Eq:DefUltf} for the definition of $\Ult\gf$).

\begin{lem}\label{L:UltfMorph}
The map $\Ult\gf\colon\Ult\bB\to\Ult\bA$\index{s}{UltB@$\Ult B$, $\Ult\bB$, $\Ult\gf$} is a morphism in $\BTop_P$\index{s}{BTop@$\BTop_P$}, for every morphism $\gf\colon\bA\to\bB$ in $\Bool_P$\index{s}{BoolP@$\Bool_P$}.
\end{lem}

Hence $\Ult$ defines a contravariant functor from $\Bool_P$\index{s}{BoolP@$\Bool_P$} to $\BTop_P$\index{s}{BTop@$\BTop_P$}.

For a $P$-normed\index{i}{normed (Boolean) space} Boolean space~$\bX$ and $A:=\Clop X$\index{s}{ClopX@$\Clop X$, $\Clop\bX$, $\Clop f$}, we shall put
\index{s}{normx@$\Vert{x}\Vert$, $\Vert{x}\Vert_\bX$, $x$ point}
 \begin{equation}\label{Eq:DefClopbX}
 A^{(p)}:=\setm{U\in\Clop X}{(\forall x\in U)(p\in\norm{x})}\,,\qquad\text{for all }p\in P\,. 
 \end{equation}
 
\begin{lem}\label{L:ClopbXinBP}
The structure $\Clop\bX:=\left(A,\famm{A^{(p)}}{p\in P}\right)$\index{s}{ClopX@$\Clop X$, $\Clop\bX$, $\Clop f$|ii} is a $P$-scaled Boolean algebra\index{i}{algebra!Pscaled Boolean@$P$-scaled Boolean}. 
\end{lem}

\begin{proof}
Obviously, $p\leq q$ implies that $A^{(q)}\subseteq A^{(p)}$, for all $p,q\in P$.

As $\norm{x}$\index{s}{normx@$\Vert{x}\Vert$, $\Vert{x}\Vert_\bX$, $x$ point} is nonempty for all $x\in X$, we get $X=\bigcup\famm{X_p}{p\in P}$, where we set
 \[
 X_p:=\setm{x\in X}{p\in\norm{x}}\,,\qquad\text{for all }p\in P\,.
 \]
As all $X_p$s are open and~$X$ is Boolean, there are a finite $Q\subseteq P$ and clopen subsets $U_p\subseteq X_p$, for $p\in Q$, such that $X=\bigcup\famm{U_p}{p\in Q}$. Hence $A=\bigvee\famm{A^{(p)}}{p\in P}$ in $\Id A$\index{s}{IdP@$\Id P$, $P$ poset}.

Finally let $p,q\in P$ and let $U\in A^{(p)}\cap A^{(q)}$. As~$\norm{x}$\index{s}{normx@$\Vert{x}\Vert$, $\Vert{x}\Vert_\bX$, $x$ point} is directed\index{i}{poset!directed} for all~$x\in X$, we get $U\subseteq\bigcup\famm{X_r}{r\geq p,q}$. As~$U$ is clopen and~$X$ is Boolean, there are a finite $R\subseteq P\Upw\set{p,q}$ and clopen subsets $U_r\subseteq X_r$, for $r\in R$, such that $U=\bigcup\famm{U_r}{r\in R}$. It follows that $A^{(p)}\cap A^{(q)}$ is contained in $\bigvee\famm{A^{(r)}}{r\geq p,q}$.
\qed\end{proof}

The proof of the following lemma is straightforward (see \eqref{Eq:DefClopf} for the definition of $\Clop f$\index{s}{ClopX@$\Clop X$, $\Clop\bX$, $\Clop f$}).

\begin{lem}\label{L:ClopfMorph}
The map $\Clop f\colon\Clop\bY\to\Clop\bX$\index{s}{ClopX@$\Clop X$, $\Clop\bX$, $\Clop f$} is a morphism in $\Bool_P$\index{s}{BoolP@$\Bool_P$}, for every morphism $f\colon\bX\to\bY$ in $\BTop_P$\index{s}{BTop@$\BTop_P$}.
\end{lem}

Hence $\Clop$ defines a contravariant functor from~$\BTop_P$\index{s}{BTop@$\BTop_P$} to~$\Bool_P$\index{s}{BoolP@$\Bool_P$}. Our next result shows that the pair $(\Ult,\Clop)$\index{s}{UltB@$\Ult B$, $\Ult\bB$, $\Ult\gf$}\index{s}{ClopX@$\Clop X$, $\Clop\bX$, $\Clop f$} extends the classical Stone duality\index{i}{Stone duality} between Boolean algebras\index{i}{algebra!Boolean} and Boolean spaces.

\begin{prop}\label{P:StoneDualP}
The pair $(\Ult,\Clop)$\index{s}{UltB@$\Ult B$, $\Ult\bB$, $\Ult\gf$}\index{s}{ClopX@$\Clop X$, $\Clop\bX$, $\Clop f$} defines a duality between the categories~$\Bool_P$\index{s}{BoolP@$\Bool_P$} and~$\BTop_P$\index{s}{BTop@$\BTop_P$}.
\end{prop}

\begin{proof}
For a $P$-scaled Boolean algebra\index{i}{algebra!Pscaled Boolean@$P$-scaled Boolean}~$\bA$ and $\bX:=\Ult\bA$\index{s}{UltB@$\Ult B$, $\Ult\bB$, $\Ult\gf$}, the natural isomorphism
$\eps\colon A\to\Clop X$\index{s}{ClopX@$\Clop X$, $\Clop\bX$, $\Clop f$} is given by
 \[
 \eps(u):=\setm{\fra\in X}{u\in\fra}\,,\qquad\text{for all }u\in A\,.
 \]
We claim that $\eps``(A^{(p)})=(\Clop X)^{(p)}$, for all $p\in P$. Let~$u\in A$. If $u\in A^{(p)}$, then $\fra\cap A^{(p)}$ is nonempty (because $u$ belongs there) for all $\fra\in\eps(u)$, thus $p\in\norm{\fra}$\index{s}{normfa@$\Vert{\fra}\Vert$, $\Vert{\fra}\Vert_\bA$, $\fra$ ultrafilter}; whence $\eps(u)\in(\Clop X)^{(p)}$. Conversely, suppose that $\eps(u)\in(\Clop X)^{(p)}$, that is, $(\forall\fra\in X)(u\in\fra\Rightarrow\fra\cap A^{(p)}\neq\es)$. If $u\notin A^{(p)}$, then, as~$A^{(p)}$ is an ideal\index{i}{ideal!of a poset} of~$A$, there exists an ultrafilter~$\fra$ of~$A$ such that~$u\in\fra$ and $\fra\cap A^{(p)}=\es$; a contradiction. So, $u\in A^{(p)}$, which proves our claim. Therefore, $\eps$ defines an isomorphism from~$\bA$ onto~$\Clop\bX$\index{s}{ClopX@$\Clop X$, $\Clop\bX$, $\Clop f$}.

For a $P$-normed\index{i}{normed (Boolean) space} Boolean space~$\bX$ and $\bA:=\Clop\bX$\index{s}{ClopX@$\Clop X$, $\Clop\bX$, $\Clop f$}, the natural homeomorphism $\gh\colon X\to\Ult A$\index{s}{UltB@$\Ult B$, $\Ult\bB$, $\Ult\gf$} is given by
 \[
 \gh(x):=\setm{U\in A}{x\in U}\,,\qquad\text{for all }x\in X\,.
 \]
We claim that\index{s}{normx@$\Vert{x}\Vert$, $\Vert{x}\Vert_\bX$, $x$ point} $\norm{x}=\norm{\gh(x)}$, for all $x\in X$. For every $p\in P$\index{s}{normx@$\Vert{x}\Vert$, $\Vert{x}\Vert_\bX$, $x$ point},\index{s}{ClopX@$\Clop X$, $\Clop\bX$, $\Clop f$}
 \begin{align*}
 p\in\norm{\gh(x)}&\Leftrightarrow \gh(x)\cap A^{(p)}\neq\es\\
 &\Leftrightarrow (\exists U\in\Clop X)(x\in U\text{ and }U\in A^{(p)})\\
 &\Leftrightarrow (\exists U\in\Clop X)\bigl(x\in U\text{ and }(\forall y\in U)(p\in\norm{y})\bigr)\\
 &\Leftrightarrow p\in\norm{x}\\
 &\qquad\qquad(\text{because }
 \setm{y\in X}{p\in\norm{y}}\text{ is open and }X\text{ is Boolean})\,, 
 \end{align*}
which proves our claim. Therefore, $\gh$ defines an isomorphism from~$\bX$ onto~$\Ult\bA$\index{s}{UltB@$\Ult B$, $\Ult\bB$, $\Ult\gf$}.
\qed\end{proof}

\section[Directed colimits and finite products in $\Bool_P$]{Directed colimits and finite products of $P$-scaled Boolean algebras}\label{S:BoolP}

Throughout this section we shall fix a poset~$P$. In order to be able to manipulate condensates\index{i}{condensate} conveniently we shall need the following easy result.

\begin{prop}\label{P:DirColim}
The category~$\Bool_P$\index{s}{BoolP@$\Bool_P$} has arbitrary small directed colimits.
\end{prop}

\begin{proof}
We must prove that every poset-indexed diagram
 \[
 \overrightarrow{\bA}=\famm{\bA_i,\ga_i^j}{i\leq j\text{ in }I}
 \]
in~$\Bool_P$\index{s}{BoolP@$\Bool_P$}, with~$I$ a directed poset\index{i}{poset!directed}, has a colimit. The colimit
 \[
 \famm{A,\ga_i}{i\in I}:=\varinjlim\overrightarrow{A}\,,\qquad\text{where}\qquad
 \overrightarrow{A}=\famm{A_i,\ga_i^j}{i\leq j\text{ in }I}
 \]
in the category of Boolean algebras\index{i}{algebra!Boolean} is characterized, among cocones above~$\overrightarrow{A}$, by the statements
 \begin{align*}
 A&=\bigcup\famm{\ga_i``(A_i)}{i\in I}\,,\\
 \ga_i(x)\leq\ga_i(y)&\Leftrightarrow(\exists j\geq i)
 \bigl(\ga_i^j(x)\leq\ga_i^j(y)\bigr)\,,&&\text{for all }i\in I\text{ and all }x,y\in A_i
 \end{align*}
(cf. Section~\ref{Su:DirColimFirstOrd}).
Put $A^{(p)}:=\bigcup\famm{\ga_i\bigl(A_i^{(p)}\bigr)}{i\in I}$, for each $p\in P$. We start by verifying that~$A^{(p)}$ is a lower subset\index{i}{lower subset} of~$A$. Let $\ol{x}\in A$ and~$\ol{y}\in A^{(p)}$ such that $\ol{x}\leq\ol{y}$. There are $i\in I$, $x\in A_i$, and $y\in A_i^{(p)}$ such that $\ol{x}=\ga_i(x)$ and $\ol{y}=\ga_i(y)$. As $\ol{x}\leq\ol{y}$, we may augment~$i$ in such a way that $x\leq y$. Hence $x\in A_i^{(p)}$, and so $\ol{x}\in A^{(p)}$. As $A^{(p)}$ is a \jz-subsemilattice of~$A$ (for the union defining~$A^{(p)}$ is directed), it follows that~$A^{(p)}$ is an ideal\index{i}{ideal!of a poset} of~$A$.

As $I$ is nonempty, we may pick $i\in I$. {}From $A_i=\bigvee\famm{A_i^{(p)}}{p\in P}$ it follows that there exists a decomposition in~$A_i$ of the form
$1=\bigvee\famm{a_p}{p\in Q}$, where~$Q$ is a finite subset of~$P$ and $a_p\in A_i^{(p)}$ for all $p\in Q$. Therefore, $1=\bigvee\famm{\ga_i(a_p)}{p\in Q}$ in~$A$, with $\ga_i(a_p)\in A^{(p)}$ for all~$p\in Q$, and so $A=\bigvee\famm{A^{(p)}}{p\in P}$.

Finally, let $p,q\in P$ and let $\ol{x}\in A^{(p)}\cap A^{(q)}$. There are $i\in I$ and $(x',x'')\in A_i^{(p)}\times A_i^{(q)}$ such that $\ol{x}=\ga_i(x')=\ga_i(x'')$. Hence $\ol{x}=\ga_i(x)$, where $x:=x'\wedge x''$ belongs to $A_i^{(p)}\cap A_i^{(q)}=\bigvee\famm{A_i^{(r)}}{r\geq p,q}$, and so $x=\bigvee\famm{x_r}{r\in R}$ in $A_i$, for some finite set~$R$ of upper bounds of $\set{p,q}$ and elements $x_r\in A_i^{(r)}$, for all $r\in R$. Therefore, $\ol{x}=\bigvee\famm{\ga_i(x_r)}{r\in R}$ belongs to $\bigvee\famm{A^{(r)}}{r\geq p,q}$. This completes the proof that
$\bA=\left(A,\famm{A^{(p)}}{p\in P}\right)$ is a $P$-scaled Boolean algebra\index{i}{algebra!Pscaled Boolean@$P$-scaled Boolean}.

The verification of the statement
 \[
 \famm{\bA,\ga_i}{i\in I}=\varinjlim\famm{\bA_i,\ga_i^j}{i\leq j\text{ in }I}
 \]
is routine.
\qed\end{proof}

The statement of CLL\index{i}{Condensate Lifting Lemma (CLL)} assumes the category~$\cA$ of Lemma~\ref{L:CLL} being closed under nonempty finite products. Hence we record for further use the following easy fact.

\begin{prop}\label{P:ArbProd}
The category~$\Bool_P$\index{s}{BoolP@$\Bool_P$} has arbitrary nonempty finite products, and even arbitrary nonempty small products in case~$P$ is finite.
\end{prop}

\begin{proof}
Let $I$ be a nonempty set and let $\famm{\bA_i}{i\in I}$ be a family of $P$-scaled Boolean algebras\index{i}{algebra!Pscaled Boolean@$P$-scaled Boolean}. We set
 \begin{align*}
 A&:=\prod\famm{A_i}{i\in I}\,,\\
 A^{(p)}&:=\prod\famm{A_i^{(p)}}{i\in I}\,,\qquad\text{for each }p\in P\,,\\
 \bA&:=\Bigl(A,\famm{A^{(p)}}{p\in P}\Bigr)\,.
 \end{align*}
For each $i\in I$, as $A_i=\bigvee\famm{A_i^{(p)}}{p\in P}$, there are a finite subset~$Q_i$ of~$P$ and a family $\famm{x_{i,p}}{p\in Q_i}\in\prod\famm{A_i^{(p)}}{p\in Q_i}$ such that $1_{A_i}=\bigvee\famm{x_{i,p}}{p\in Q_i}$.

Now assume that either~$I$ or~$P$ is finite. Then $Q:=\bigcup\famm{Q_i}{i\in I}$ is a finite subset of~$P$. Set $x_{i,p}:=0_{A_i}$ for each $i\in I$ and each $p\in Q\setminus Q_i$, and then set $x_p:=\famm{x_{i,p}}{i\in I}$, for each $p\in Q$. Then $\famm{x_{i,p}}{p\in Q}\in\prod\famm{A_i^{(p)}}{p\in Q}$ for each $i\in I$, and so $x_p\in A^{(p)}$ for each $p\in Q$, with $\bigvee\famm{x_p}{p\in Q}=1_A$. Item~(i) of Definition~\ref{D:BoolP} follows. The proof for Item~(ii) is similar. Hence~$\bA$ is a $P$-scaled Boolean algebra\index{i}{algebra!Pscaled Boolean@$P$-scaled Boolean}. The verification that~$\bA$ is a product of $\famm{\bA_i}{i\in I}$ in~$\Bool_P$\index{s}{BoolP@$\Bool_P$} is straightforward.
\qed\end{proof}

\section{Finitely presented $P$-scaled Boolean algebras}\label{S:CompBP}

The present section is mostly devoted to describe the finitely presented\index{i}{presented!finitely} objects in~$\Bool_P$\index{s}{BoolP@$\Bool_P$} as the \emph{compact}\index{i}{compact!member of $\Bool_P$} ones (cf. Definition~\ref{D:CompactBP} and Corollary~\ref{C:FinPres=Comp}). Also, we shall express an arbitrary $P$-scaled Boolean algebra\index{i}{algebra!Pscaled Boolean@$P$-scaled Boolean} as a monomorphic\index{i}{monomorphic colimit} directed colimit of finitely presented\index{i}{presented!finitely} ones (cf. Proposition~\ref{P:FinRepresBP}).

Throughout this section we shall fix a poset~$P$.

\begin{defn}\label{D:CompactBP}
An object~$\bA$ in~$\Bool_P$\index{s}{BoolP@$\Bool_P$} is \emph{compact}\index{i}{compact!member of $\Bool_P$|ii} if $A$ is finite and $\norm{a}$\index{s}{norma@$\Vert{a}\Vert$, $\Vert{a}\Vert_\bA$, $a$ atom} has a largest element for each $a\in\At A$ (cf. Lemma~\ref{L:normaId}).
\end{defn}

\begin{remk}\label{Rk:DualCompact}
The dual spaces of the compact $P$-scaled Boolean algebras\index{i}{algebra!Pscaled Boolean@$P$-scaled Boolean} (cf. Definition~\ref{D:CompactBP}) are exactly the $P$-normed\index{i}{normed (Boolean) space} spaces~$\bX$ with finite underlying set~$X$ and~$\norm{x}$\index{s}{normx@$\Vert{x}\Vert$, $\Vert{x}\Vert_\bX$, $x$ point} a principal ideal\index{i}{ideal!of a poset}\index{i}{ideal!principal ${}_{-}$, of a poset} of~$P$ for each~$x\in X$.
\end{remk}

\begin{lem}\label{L:Comp2FP}
Any compact object~$\bA$ in~$\Bool_P$\index{s}{BoolP@$\Bool_P$} is finitely presented\index{i}{presented!finitely} in $\Bool_P$\index{s}{BoolP@$\Bool_P$}.
\end{lem}

\begin{proof}
Let $\famm{\bB,\gb_i}{i\in I}=\varinjlim\famm{\bB_i,\gb_i^j}{i\leq j\text{ in }I}$ in $\Bool_P$\index{s}{BoolP@$\Bool_P$}, where~$I$ is a directed poset\index{i}{poset!directed}, and let $f\colon\bA\to\bB$ in $\Bool_P$. Let $a\in\At A$. {}From $a\in A^{(|a|)}$\index{s}{nora@$\vert a\vert$, $a$ atom} it follows that
$f(a)\in B^{(|a|)}=\bigcup\famm{\gb_i``\bigl(B_i^{(|a|)}\bigr)}{i\in I}$. As~$\At A$ is finite, there exists $i_0\in I$ such that $f(a)\in\gb_{i_0}``\bigl(B_{i_0}^{(|a|)}\bigr)$ for all $a\in\At A$. As~$A$ is a finite Boolean algebra\index{i}{algebra!Boolean} and $B=\varinjlim_{i\in I}B_i$ for Boolean algebras\index{i}{algebra!Boolean}, there exists $i_1\geq i_0$ such that~$f$ factors through~$B_{i_1}$, that is, there exists a homomorphism $g\colon A\to B_{i_1}$ of Boolean algebras\index{i}{algebra!Boolean} such that $f=\gb_{i_1}\circ g$. As $i_1\geq i_0$, we also get that
$(\gb_{i_1}\circ g)(a)=f(a)$ belongs to $\gb_{i_1}``\bigl(B_{i_1}^{(|a|)}\bigr)$ for each $a\in\At A$. As $\At A$ is finite, there exists $i\geq i_1$ in~$I$ such that
$(\gb_{i_1}^{i}\circ g)(a)\in\gb_{i_1}^{i}``\bigl(B_{i_1}^{(|a|)}\bigr)$ for all $a\in\At A$. Hence we may replace~$i_1$ by~$i$ and~$g$ by $\gb_{i_1}^{i}\circ g$, and so $g(a)\in B_i^{(|a|)}$, for all $a\in\At A$.

We claim that $g``(A^{(p)})\subseteq B_i^{(p)}$, for each $p\in P$. We must prove that $g(a)\in B_i^{(p)}$, for each atom~$a$ of~$A^{(p)}$. {}From $p\leq|a|$ and  $g(a)\in B_i^{(|a|)}$\index{s}{nora@$\vert a\vert$, $a$ atom} it follows that $g(a)\in B_i^{(p)}$, which proves our claim. Therefore, $g\colon\bA\to\bB_i$.

If $f=\gb_i\circ g=\gb_i\circ h$ for morphisms $g,h\colon\bA\to\bB_i$, then, as~$A$ is finite, $\gb_i^j\circ g=\gb_i^j\circ h$ for some $j\geq i$.
\qed\end{proof}

For a $P$-scaled Boolean algebra\index{i}{algebra!Pscaled Boolean@$P$-scaled Boolean} $\bA$, we denote by $\Sigma_\bA$\index{s}{SigmaA@$\Sigma_\bA$|ii} the set of all functions $f\colon\At U\to P$, for a finite subalgebra~$U=A_f$ of~$A$, such that $u\in A^{(f(u))}$ for all $u\in\At U$. Denote by $x^U$ the least element~$u$ of~$U$ such that $x\leq u$, for every $x\in A$, and define a partial ordering~$\sqsubseteq $ on~$\Sigma_\bA$ by
 \[
 f\sqsubseteq g\ \Longleftrightarrow\ \Bigl(A_f\subseteq A_g\text{ and }
 (\forall v\in\At A_g)\bigl(f(v^{A_f})\leq g(v)\bigr)\Bigr)\,,\ \text{for all }
 f,g\in\Sigma_\bA\,.
 \]
 
\begin{lem}\label{L:leqSdir}
The partial ordering $\sqsubseteq$ is directed\index{i}{poset!directed} on $\Sigma_\bA$. 
\end{lem}

\begin{proof}
Let $f_0,f_1\in\Sigma_\bA$ and put $U_i:=A_{f_i}$ for each $i<2$. Denote by~$U$ the subalgebra of~$A$ generated by $U_0\cup U_1$, so it is a finite subalgebra of~$A$. For all $u\in\At U$ and all $i<2$, $u^{U_i}\in\At U_i$, thus $u^{U_i}\in A^{(f_i(u^{U_i}))}$, hence $u\in A^{(f_0(u^{U_0}))}\cap A^{(f_1(u^{U_1}))}$, so there are a finite partition~$X_u$ of~$u$ in~$A$ and $g_u\colon X_u\to P\Upw\set{f_0(u^{U_0}),f_1(u^{U_1})}$ such that $x\in A^{(g_u(x))}$ for all $x\in X_u$. Denote by~$V$ the subalgebra of~$A$ generated by $X:=\bigcup\famm{X_u}{u\in\At U}$. So~$V$ is a finite subalgebra of~$A$ containing~$U$, and thus containing $U_0\cup U_1$, and $\At V=X$. Putting $g:=\bigcup\famm{g_u}{u\in\At U}$, we obtain that $g\colon X\to P$, that $x\in A^{(g(x))}$ for each $x\in X$, and that $f_0(u^{U_0}),f_1(u^{U_1})\leq g(x)$ in case $x\in X_u$. In particular, $g\in\Sigma_\bA$. Finally, for all $x\in X$ and all $i<2$, we get, putting $u:=x^U$, that $f_i(x^{U_i})=f_i(u^{U_i})\leq g(x)$, and so $f_i\sqsubseteq g$.
\qed\end{proof}

For a $P$-scaled Boolean algebra\index{i}{algebra!Pscaled Boolean@$P$-scaled Boolean}~$\bA$, we shall put
 \begin{multline*}
 A_f^{(p)}:=\setm{x\in A_f}
 {(\forall u\in\At A_f)(u\leq x\Rightarrow p\leq f(u))}\,,\\
 \text{for all }f\in\Sigma_\bA\text{ and all }p\in P\,.
 \end{multline*}
 
\begin{lem}\label{L:AfinBP}
The structure~$\bA_f:=\left(A_f,\famm{A_f^{(p)}}{p\in P}\right)$ is a compact $P$-scaled Boolean algebra\index{i}{algebra!Pscaled Boolean@$P$-scaled Boolean} and the inclusion map defines a monomorphism from~$\bA_f$ into~$\bA$, for each~$f\in\Sigma_\bA$. 
\end{lem}

\begin{proof}
Put $U:=A_f$. Obviously, $U^{(p)}:=A_f^{(p)}$ is an ideal\index{i}{ideal!of a poset} of~$U$. As $1=\bigvee\At U$ and $u\in A^{(f(u))}$ for all $u\in\At U$, we get $U=\bigvee\famm{U^{(p)}}{p\in P}$ in~$\Id U$\index{s}{IdP@$\Id P$, $P$ poset}. Finally, for $p,q\in P$ and $u\in U^{(p)}\cap U^{(q)}$, we must prove that~$u$ belongs to $\bigvee\famm{U^{(r)}}{r\geq p,q}$. It suffices to consider the case where~$u$ is an atom of~$U$, in which case $p\leq f(u)$ and $q\leq f(u)$, thus $u\in U^{(r)}$ where $r=f(u)\geq p,q$.
This proves that $\bA_f$ is a $P$-scaled Boolean algebra\index{i}{algebra!Pscaled Boolean@$P$-scaled Boolean}. Obviously, the largest element of $\norm{u}_{\bA_f}$\index{s}{norma@$\Vert{a}\Vert$, $\Vert{a}\Vert_\bA$, $a$ atom} is $f(u)$, for every $u\in\At A_f$; whence~$\bA_f$ is compact. As $A_f\subseteq A$ and $A_f^{(p)}\subseteq A^{(p)}$ for all $p\in P$, the inclusion map from~$A_f$ into~$A$ is a morphism, obviously monic, in~$\Bool_P$\index{s}{BoolP@$\Bool_P$}.
\qed\end{proof}

\begin{prop}\label{P:FinRepresBP}
Every object~$\bA$ in~$\Bool_P$\index{s}{BoolP@$\Bool_P$} satisfies
 \begin{equation}\label{Eq:ComGenBP}
 A=\bigcup\famm{A_f}{f\in\Sigma_\bA}\text{ and }
 A^{(p)}=\bigcup\famm{A_f^{(p)}}{f\in\Sigma_\bA}\,,\ \text{for each }p\in P\,,
 \end{equation}
both unions in~\textup{\eqref{Eq:ComGenBP}} being directed. Consequently, $\bA$ is a $(\card A+\card P)^+$-small monomorphic\index{i}{monomorphic colimit} directed colimit of compact objects of~$\Bool_P$\index{s}{BoolP@$\Bool_P$}.
\end{prop}

\begin{proof}
It is obvious that both unions in~\eqref{Eq:ComGenBP} are directed.
As~$1$ belongs to $\bigvee\famm{A^{(p)}}{p\in P}$, there exists a decomposition of the form $1=\oplus_{i<n}a_i$ in~$A$, where~$n$ is a positive integer, $a_i\in A\setminus\set{0}$, $p_i\in P$, and $a_i\in A^{(p_i)}$, for each $i<n$. Let $a\in A$. By splitting~$a_i$ into $a_i\wedge a$ and $a_i\wedge\neg a$, for each $i<n$, we may assume that either $a_i\leq a$ or $a_i\leq\neg a$ for each $i<n$. Denoting by~$U$ the subalgebra of~$A$ generated by $\setm{a_i}{i<n}$ and setting $f(a_i):=p_i$ for each $i<n$, we obtain that $a\in U=A_f$. This proves the first half of \eqref{Eq:ComGenBP}.

Now suppose that $a\in A^{(p)}$, where $p\in P$. As $A^{(p)}$ is an ideal\index{i}{ideal!of a poset} of~$A$, $a_i\in A^{(p)}\cap A^{(p_i)}$ for each $i<n$ such that $a_i\leq a$, so $a_i\in\bigvee\famm{A^{(r)}}{r\geq p,p_i}$. Hence, by splitting further~$a_i$, we may assume that $a_i\leq a$ implies that $p\leq p_i=f(a_i)$, for each $i<n$; whence $a\in A_f^{(p)}$. This proves the second half of \eqref{Eq:ComGenBP}.

It follows from \eqref{Eq:ComGenBP} that $\bA=\varinjlim_{f\in\Sigma_\bA}\bA_f$, where all the transition maps and limiting maps are the respective inclusion maps, which are trivially monic.
\qed\end{proof}

\begin{cor}\label{C:FinPres=Comp}
The finitely presented\index{i}{presented!finitely} objects in~$\Bool_P$\index{s}{BoolP@$\Bool_P$} are exactly the compact objects.
\end{cor}

\begin{proof}
In view of Lemma~\ref{L:Comp2FP}, it suffices to prove that every finitely presented\index{i}{presented!finitely} object~$\bA$ of~$\Bool_P$\index{s}{BoolP@$\Bool_P$} is compact. By Proposition~\ref{P:FinRepresBP}, we can write
 \[
 \famm{\bA,\ga_i}{i\in I}=\varinjlim\famm{\bA_i,\ga_i^j}{i\leq j\text{ in }I}\,,
 \]
where~$I$ is a directed poset\index{i}{poset!directed} (in fact one can take $I=\Sigma_\bA$), all~$\bA_i$ are compact, and all morphisms $\ga_i$ are monic. As~$\bA$ is finitely presented\index{i}{presented!finitely}, there are $i\in I$ and $f\colon\bA\to\bA_i$ such that $\id_A=\ga_i\circ f$. As~$\ga_i$ is monic, it follows that $f\circ\ga_i=\id_{\bA_i}$, and so $\ga_i$ is an isomorphism. As~$\bA_i$ is compact, $\bA$ is compact.
\qed\end{proof}

\begin{remk}\label{Rk:MonomInj}
It can be proved that the monomorphisms in~$\Bool_P$\index{s}{BoolP@$\Bool_P$} are exactly the one-to-one morphisms. On the other hand, there are posets~$P$ for which some epimorphisms in~$\Bool_P$\index{s}{BoolP@$\Bool_P$} are not surjective.
\end{remk}

\section[Normal morphisms of $P$-scaled Boolean algebras]{Normal morphisms in $\Bool_P$ and in $\BTop_P$}\label{S:NormMorph}

In the CLL universe we shall need to pay a particular attention to those surjective morphisms $f\colon\bU\to\bV$ in~$\Bool_P$\index{s}{BoolP@$\Bool_P$} such that, under the assumptions of CLL\index{i}{Condensate Lifting Lemma (CLL)} (Lemma~\ref{L:CLL}), the morphism $\Phi(f\otimes\overrightarrow{A})\colon\Phi(\bU\otimes\overrightarrow{A})\to\Phi(\bV\otimes\overrightarrow{A})$\index{s}{otimAS@$\bA\otimes\overrightarrow{S}$, $\gf\otimes\overrightarrow{S}$} is a double arrow\index{i}{double arrow}. Such morphisms will be called \emph{normal morphisms}\index{i}{morphism!normal}. It can be proved that the normal morphisms in~$\Bool_P$\index{s}{BoolP@$\Bool_P$} are exactly the so-called \emph{regular epimorphisms}\index{i}{regular epimorphism} in~$\Bool_P$\index{s}{BoolP@$\Bool_P$}, that is, those morphisms that appear as coequalizers of two morphisms. However, we will not need this fact in the book. The present section aims at introducing the normal\index{i}{morphism!normal} morphisms in~$\Bool_P$\index{s}{BoolP@$\Bool_P$} together with their dual objects, and to prove that every normal\index{i}{morphism!normal} morphism is a directed colimit of normal\index{i}{morphism!normal} morphisms between compact $P$-scaled Boolean algebras\index{i}{algebra!Pscaled Boolean@$P$-scaled Boolean} (Proposition~\ref{P:NormColim}).

Throughout this section we shall fix a poset~$P$.

\begin{defn}\label{D:NormMorph}
A morphism $f\colon\bA\to\bB$ in~$\Bool_P$\index{s}{BoolP@$\Bool_P$} is \emph{normal}\index{i}{morphism!normal|ii} if it is surjective and $f``(A^{(p)})=B^{(p)}$ for each~$p\in P$.
A morphism $f\colon\bX\to\bY$ in~$\BTop_P$\index{s}{BTop@$\BTop_P$} is \emph{normal} if it is one-to-one and $\norm{f(x)}_\bY=\norm{x}_\bX$\index{s}{normx@$\Vert{x}\Vert$, $\Vert{x}\Vert_\bX$, $x$ point} for all~$x\in X$.
\end{defn}

\begin{prop}\label{P:DualityNormal}
A morphism~$\gf\colon\bA\to\bB$ in~$\Bool_P$\index{s}{BoolP@$\Bool_P$} is normal\index{i}{morphism!normal} if{f} the dual morphism~$\Ult\gf\colon\Ult\bB\to\Ult\bA$\index{s}{UltB@$\Ult B$, $\Ult\bB$, $\Ult\gf$} is normal\index{i}{morphism!normal} in~$\BTop_P$\index{s}{BTop@$\BTop_P$}.
\end{prop}

\begin{proof}
Put~$\bX:=\Ult\bA$, $\bY:=\Ult\bB$, and~$f:=\Ult\gf$\index{s}{UltB@$\Ult B$, $\Ult\bB$, $\Ult\gf$}. Suppose first that~$\gf$ is normal\index{i}{morphism!normal}. As $\gf$ is surjective, the dual map~$f$ is one-to-one. Let~$\frb\in\bY$ and let $p\in\norm{\frb}$\index{s}{normfa@$\Vert{\fra}\Vert$, $\Vert{\fra}\Vert_\bA$, $\fra$ ultrafilter}; pick $b\in\frb\cap B^{(p)}$. As~$\gf$ is normal\index{i}{morphism!normal}, there exists~$a\in A^{(p)}$ such that $b=\gf(a)$, so $\gf(a)\in\frb$; whence $a\in f(\frb)\cap A^{(p)}$, and so $p\in\norm{f(\frb)}$\index{s}{normfa@$\Vert{\fra}\Vert$, $\Vert{\fra}\Vert_\bA$, $\fra$ ultrafilter}. This proves the containment~$\norm{\frb}\subseteq\norm{f(\frb)}$, and so (cf. Lemma~\ref{L:UltfMorph}) $\norm{\frb}=\norm{f(\frb)}$\index{s}{normfa@$\Vert{\fra}\Vert$, $\Vert{\fra}\Vert_\bA$, $\fra$ ultrafilter}.

Conversely, assume that~$f$ is normal\index{i}{morphism!normal}. As~$f$ is one-to-one, $\gf$ is surjective.
Let $p\in P$ and $b\in B^{(p)}$. Suppose that $b\notin\gf``(A^{(p)})$. There exists an ultrafilter~$\frb$ of~$B$ such that~$b\in\frb$ while $\frb\cap\gf``(A^{(p)})=\es$. As $b\in\frb\cap B^{(p)}$, $p$ belongs to~$\norm{\frb}$\index{s}{normfa@$\Vert{\fra}\Vert$, $\Vert{\fra}\Vert_\bA$, $\fra$ ultrafilter}, thus to~$\norm{f(\frb)}$\index{s}{normfa@$\Vert{\fra}\Vert$, $\Vert{\fra}\Vert_\bA$, $\fra$ ultrafilter}, which means that $f(\frb)\cap A^{(p)}$ is nonempty. Pick $a\in f(\frb)\cap A^{(p)}$, then $\gf(a)\in\frb\cap\gf``(A^{(p)})$, a contradiction. We have proved that $B^{(p)}\subseteq\gf``(A^{(p)})$.
\qed\end{proof}

The proofs of the following two lemmas are straightforward.

\begin{lem}\label{L:bA/I}
For a $P$-scaled Boolean algebra\index{i}{algebra!Pscaled Boolean@$P$-scaled Boolean}~$\bA$ and an ideal\index{i}{ideal!of a poset}~$I$ of the underlying Boolean algebra\index{i}{algebra!Boolean}~$A$ of~$\bA$ with canonical projection~$\gp_I\colon A\onto A/I$\index{s}{piI@$\gp_I$, $I$ ideal|ii}\index{s}{piI@$\gp_I$, $I$ ideal}\index{s}{AtoonB@$f\colon A\onto B$}, put
 \[
 (A/I)^{(p)}:=\gp_I\bigl(A^{(p)}\bigr)\,,\quad\text{for all }p\in P\,.
 \]
Then $\bA/I=\left(A/I,\famm{(A/I)^{(p)}}{p\in P}\right)$\index{s}{AoverI@$\bA/I$ ($I$ ideal of $\bA\in\Bool_P$)|ii} is a $P$-scaled Boolean algebra\index{i}{algebra!Pscaled Boolean@$P$-scaled Boolean}, and~$\gp_I$\index{s}{piI@$\gp_I$, $I$ ideal} is a normal\index{i}{morphism!normal} morphism from~$\bA$ onto~$\bA/I$ in~$\Bool_P$\index{s}{BoolP@$\Bool_P$}.
\end{lem}

\begin{lem}\label{L:FactMorph}
Let $\gf\colon\bA\to\bB$ be a morphism in~$\Bool_P$\index{s}{BoolP@$\Bool_P$} and put~$I:=\gf^{-1}\set{0}$, an ideal\index{i}{ideal!of a poset} of~$A$. Then the factor map~$\gy\colon\bA/I\to\bB$, $\gp_I(x)\mapsto\gf(x)$\index{s}{piI@$\gp_I$, $I$ ideal} is a morphism in~$\Bool_P$\index{s}{BoolP@$\Bool_P$}. Furthermore, $\gy$ is an isomorphism if{f}~$\gf$ is normal\index{i}{morphism!normal}.
\end{lem}

Lemma~\ref{L:FactMorph} shows, in particular, that the normal\index{i}{morphism!normal} morphisms in~$\Bool_P$\index{s}{BoolP@$\Bool_P$} are, up to isomorphism, the projections~$\gp_I\colon\bA\onto\bA/I$\index{s}{piI@$\gp_I$, $I$ ideal}\index{s}{AtoonB@$f\colon A\onto B$}, for an ideal\index{i}{ideal!of a poset}~$I$ of a $P$-scaled Boolean algebra\index{i}{algebra!Pscaled Boolean@$P$-scaled Boolean}~$\bA$.

We say that an arrow $\gf\colon\bA\to\bB$ in~$\Bool_P$\index{s}{BoolP@$\Bool_P$} is \emph{compact}\index{i}{morphism!compact|ii} if both~$\bA$ and~$\bB$ are compact.

\begin{prop}\label{P:NormColim}
Every normal\index{i}{morphism!normal} morphism~$\gf\colon\bA\onto\bB$\index{s}{AtoonB@$f\colon A\onto B$} in~$\Bool_P$\index{s}{BoolP@$\Bool_P$} is a monomorphic\index{i}{monomorphic colimit} $(\card A+\card P)^+$-small directed colimit, in the category of arrows $(\Bool_P)^\two$\index{s}{two@$\two$}\index{s}{BoolP@$\Bool_P$}, of compact normal\index{i}{morphism!normal} morphisms.
\end{prop}

\begin{proof}
By Lemma~\ref{L:FactMorph}, we may assume that~$\bB=\bA/I$ and $\gf=\gp_I$\index{s}{piI@$\gp_I$, $I$ ideal}, for some ideal\index{i}{ideal!of a poset}~$I$ of~$A$. Observe that $\norm{u}_\bA=\norm{\gp_I(u)}_\bB$\index{s}{norma@$\Vert{a}\Vert$, $\Vert{a}\Vert_\bA$, $a$ atom}\index{s}{piI@$\gp_I$, $I$ ideal} holds for every~$u\in(\At A)\setminus I$. Hence, if~$\bA$ is compact, then so is~$\bB$. In the general case, $I_f:=I\cap A_f$ is an ideal\index{i}{ideal!of a poset} of~$A_f$, for every~$f\in\Sigma_\bA$. Denote by~$\gf_f\colon\bA_f\onto\bA_f/{I_f}$\index{s}{AtoonB@$f\colon A\onto B$} the canonical projection, for every~$f\in\Sigma_\bA$. It follows from Proposition~\ref{P:FinRepresBP} that $\bA=\varinjlim_{f\in\Sigma_\bA}\bA_f$, thus $\bA/I=\varinjlim_{f\in\Sigma_\bA}\bA_f/{I_f}$, with the obvious transition morphisms and limiting morphisms, hence $\gf=\varinjlim_{f\in\Sigma_\bA}\gf_f$. As each~$\bA_f$ and each~$\bA_f/{I_f}$ is compact, each~$\gf_f$ is compact. Furthermore, each morphism~$\bA_f\into\bA$\index{s}{AtoinB@$f\colon A\into B$} and~$\bA_f/{I_f}\into\bA/I$ is monic.
\qed\end{proof}

\section{Norm-coverings of a poset; the structures $\two[p]$ and $\xF(X)$}\label{S:twoxF}

Most of the present section will be devoted to introducing, for a poset~$P$ and a ``norm-covering''\index{i}{norm-covering} $X$ of~$P$, the basic properties of a construct denoted by~$\xF(X)$\index{s}{FxX@$\xF(X)$}. This structure is a $P$-scaled Boolean algebra\index{i}{algebra!Pscaled Boolean@$P$-scaled Boolean}, defined by a set of generators and relations. The construction $X\mapsto\xF(X)$\index{s}{FxX@$\xF(X)$} is functorial and preserves those small directed colimits that we need (cf. Lemma~\ref{L:fXYnormcov}). The condensates\index{i}{condensate} of a diagram~$\overrightarrow{A}$ that will be required in CLL\index{i}{Condensate Lifting Lemma (CLL)} will be those of the form $\xF(X)\otimes\overrightarrow{A}$\index{s}{FxX@$\xF(X)$}\index{s}{otimAS@$\bA\otimes\overrightarrow{S}$, $\gf\otimes\overrightarrow{S}$}, for suitable norm-coverings\index{i}{norm-covering}~$X$ of~$P$.

Our first definition gives a description of all compact $P$-scaled Boolean algebras\index{i}{algebra!Pscaled Boolean@$P$-scaled Boolean} with underlying algebra~$\two:=\set{0,1}$\index{s}{two@$\two$}.

\begin{defn}\label{D:two[p]}
Let $P$ be a poset. For each $p\in P$, we denote by~$\two_P[p]$, or~$\two[p]$\index{s}{twop@$\two[p]$|ii} if~$P$ is understood, the dual of the~$P$-normed\index{i}{normed (Boolean) space} space~$\set{\cdot}$ with~$\norm{\cdot}:=P\dnw p$\index{s}{norma@$\Vert{a}\Vert$, $\Vert{a}\Vert_\bA$, $a$ atom} (so $|\cdot|=p$). Hence~$\two[p]:=\left(\two,\famm{\two[p]^{(q)}}{q\in P}\right)$, where\index{s}{twopq@$\two[p]^{(q)}$|ii}
 \[
 \two[p]^{(q)}:=\begin{cases}
 \set{0,1}&(\text{if }q\leq p),\\
 \set{0}&(\text{otherwise}), 
 \end{cases}
 \quad\text{for all }q\in P\,.
 \]
It is obvious that for $p\leq q$ in~$P$, the containment $\two[p]^{(r)}\subseteq\two[q]^{(r)}$\index{s}{twop@$\two[p]$} holds for each~$r\in P$. Hence the identity map defines a monomorphism, which we shall denote by~$\eps_p^q$\index{s}{epspq@$\eps_p^q$|ii}, from~$\two[p]$ into~$\two[q]$\index{s}{twop@$\two[p]$}.
\end{defn}

\begin{notation}\label{Not:F(X)}
For a \pjs\index{i}{pseudo join-semilattice}\ $X$, we denote by $\rF(X)$\index{s}{FrX@$\rF(X)$|ii} the Boolean algebra\index{i}{algebra!Boolean} defined by generators $\tilde{u}$\index{s}{utilde@$\tilde{u}$, $\tilde{u}^X$|ii} (or $\tilde{u}^X$ in case~$X$ needs to be specified), for $u\in X$, and the relations
 \begin{align}
 \tilde{v}&\leq\tilde{u}\,,&&\text{for all }u\leq v\text{ in }X\,;\label{Eq:F(X)antitone}\\
 \tilde{u}\wedge\tilde{v}&=\bigvee\famm{\tilde{w}}{w\in u\sor v}\,,
 &&\text{for all }u,v\in X\,;\label{Eq:F(X)tuwedgetv}\\
 1&=\bigvee\famm{\tilde{u}}{u\in\Min X}\,.\label{Eq:F(X)repr1Sores}
 \end{align}
\end{notation}
(We refer to Section~\ref{S:TPS} for the definition of $u\sor v$.)
The assumption that~$X$ is a \pjs\index{i}{pseudo join-semilattice}\ ensures that the joins in~\eqref{Eq:F(X)tuwedgetv} and~\eqref{Eq:F(X)repr1Sores} are \emph{finite} joins.

\begin{lem}\label{L:fXY}
For every $\sor$-closed subset~$X$ in a \pjs\index{i}{pseudo join-semilattice}\ $Y$, there exists a unique homomorphism $f_X^Y\colon\rF(X)\to\rF(Y)$\index{s}{fXY@$f_X^Y$|ii}\index{s}{FrX@$\rF(X)$} of Boolean algebras\index{i}{algebra!Boolean} such that $f_X^Y(\tilde{u}^X)=\tilde{u}^Y$\index{s}{utilde@$\tilde{u}$, $\tilde{u}^X$}\index{s}{fXY@$f_X^Y$} for each $u\in X$.
\end{lem}

\begin{proof}
It suffices to prove that the elements $\tilde{u}^Y$\index{s}{utilde@$\tilde{u}$, $\tilde{u}^X$}, for $u\in X$, satisfy the relations~\eqref{Eq:F(X)antitone}--\eqref{Eq:F(X)repr1Sores} defining~$\rF(X)$\index{s}{FrX@$\rF(X)$}. As~$X$ is a $\sor$-closed subset of~$Y$, $\Min X=\Min Y$ and $u\sor_Xv=u\sor_Yv$ for all $u,v\in X$. The conclusion follows immediately.
\qed\end{proof}

The following definition is a slight weakening of the definition with the same name in~\cite{Gill1}\index{c}{Gillibert, P.}.

\begin{defn}\label{D:NormCov}
A \emph{norm-covering}\index{i}{norm-covering|ii} of a poset~$P$ is a pair $(X,\partial)$, where~$X$ is a \pjs\index{i}{pseudo join-semilattice}\ and $\partial\colon X\to P$ is an isotone map.

Then we shall say that an ideal\index{i}{ideal!of a poset}~$\bu$ of~$X$ is \emph{sharp}\index{i}{ideal!sharp|ii} if the set $\setm{\partial x}{x\in\bu}$ has a largest element, which we shall then denote by~$\partial\bu$. We shall denote by $\Ids X$\index{s}{Ids@$\Ids X$, $X$ norm-covering|ii} the set of all sharp ideals\index{i}{ideal!sharp} of~$X$, partially ordered by containment.
\end{defn}

Observe that in Definition~\ref{D:NormCov}, sharp\index{i}{ideal!sharp} ideals are defined \emph{relatively to a norm-covering}\index{i}{norm-covering} of a poset, while ordinary ideals\index{i}{ideal!of a poset} are only defined relatively to a poset.

We shall identify in notation the norm-covering\index{i}{norm-covering}~$(X,\partial)$ with its underlying poset~$X$, so that the second component of the norm-covering\index{i}{norm-covering} will always be denoted by~$\partial$. Likewise, for a $\sor$-closed subset~$Y$ of~$X$, the pair $(Y,\partial\res_Y)$ is a norm-covering\index{i}{norm-covering} of~$P$, which we shall identify in notation with~$Y$.

\begin{lem}\label{L:NC2F(X)}
For a norm-covering\index{i}{norm-covering}~$X$ of a poset~$P$, define $\rF(X)^{(p)}$\index{s}{FrX@$\rF(X)$} as the ideal\index{i}{ideal!of a poset} of~$\rF(X)$\index{s}{FrX@$\rF(X)$} generated by $\setm{\tilde{u}}{u\in X\text{ and }p\leq\partial u}$\index{s}{utilde@$\tilde{u}$, $\tilde{u}^X$}, for each $p\in P$. Then the structure\index{s}{FrX@$\rF(X)$}\index{s}{FxX@$\xF(X)$|ii}
 \[
 \xF(X):=\left(\rF(X),\famm{\rF(X)^{(p)}}{p\in P}\right)
 \]
is a $P$-scaled Boolean algebra\index{i}{algebra!Pscaled Boolean@$P$-scaled Boolean}.
\end{lem}

\begin{proof}
As $\tilde{u}\in\rF(X)^{(\partial u)}$\index{s}{utilde@$\tilde{u}$, $\tilde{u}^X$}\index{s}{FrX@$\rF(X)$} for each $u\in X$, Item~(i) of Definition~\ref{D:BoolP} follows from~\eqref{Eq:F(X)repr1Sores}. Furthermore, it is obvious that $p\leq q$ implies that $\rF(X)^{(q)}\subseteq\rF(X)^{(p)}$\index{s}{FrX@$\rF(X)$}, for all $p,q\in P$.

Finally let $p,q\in P$, we must prove that $\rF(X)^{(p)}\cap\rF(X)^{(q)}$\index{s}{FrX@$\rF(X)$} is contained in the join of all~$\rF(X)^{(r)}$\index{s}{FrX@$\rF(X)$} where $r\geq p,q$ in~$P$; as the converse containment holds (cf. paragraph above), Item~(ii) of Definition~\ref{D:BoolP} will follow. Denote by~$\seq{Y}$ the Boolean subalgebra of~$\rF(X)$\index{s}{FrX@$\rF(X)$} generated by a subset~$Y$ of~$\rF(X)$\index{s}{FrX@$\rF(X)$}. Then, using~\eqref{Eq:F(X)tuwedgetv}, we obtain\index{s}{utilde@$\tilde{u}$, $\tilde{u}^X$}\index{s}{FrX@$\rF(X)$}
 \begin{align*}
 \rF(X)^{(p)}\cap\rF(X)^{(q)}&=
 \seq{\setm{\tilde{u}}{u\in X\,,\ p\leq\partial u}}\cap
 \seq{\setm{\tilde{v}}{v\in X\,,\ q\leq\partial v}}\\
 &=\seq{\setm{\tilde{u}\wedge\tilde{v}}{u,v\in X\,,\ p\leq\partial u\,,
 \text{ and }q\leq\partial v}}\\
 &=\seq{\setm{\tilde{w}}{(\exists u,v\in X)
 (w\in u\sor v\,,\ p\leq\partial u\,,\text{ and }q\leq\partial v)}}\\
 &\subseteq\seq{\setm{\tilde{w}}{w\in X\,\text{ and }p,q\leq\partial w}}\\
 &=\bigvee\famm{\rF(X)^{(r)}}{r\in P\Upw\set{p,q}}\,.
 \end{align*}
This concludes the proof.\qed
\end{proof}

The verifications of the following lemma is straightforward, although slightly tedious in the case of Item~(ii), in which case one can use the construction of the colimit given in the proof of Proposition~\ref{P:DirColim}.

\begin{lem}\label{L:fXYnormcov}
The following statements hold, for every poset~$P$:
\begin{description}
\item[\tui] Let~$X$ be a $\sor$-closed subset in a norm-covering\index{i}{norm-covering}~$Y$ of~$P$. Then the homomorphism of Boolean algebras~$f_X^Y\colon\rF(X)\to\rF(Y)$\index{s}{fXY@$f_X^Y$}\index{s}{FrX@$\rF(X)$} defined in Lemma~\textup{\ref{L:fXY}} is a morphism from~$\xF(X)$\index{s}{FxX@$\xF(X)$} to~$\xF(Y)$\index{s}{FxX@$\xF(X)$} in~$\Bool_P$\index{s}{BoolP@$\Bool_P$}.

\item[\tuii] Let $I$ be a directed poset\index{i}{poset!directed} and let $\famm{X_i}{i\in I}$ be an isotone family of $\sor$-closed subsets in a norm-covering\index{i}{norm-covering}~$Y$ of~$P$. Put $X:=\bigcup_{i\in I}X_i$. Then the following statement holds in~$\Bool_P$:\index{s}{fXY@$f_X^Y$}\index{s}{BoolP@$\Bool_P$}
 \[
 \famm{\xF(X),f_{X_i}^X}{i\in I}=
 \varinjlim\famm{\xF(X_i),f_{X_i}^{X_j}}{i\leq j\text{ in }I}\,.
 \]
\end{description}
\end{lem}

The following lemma introduces a useful class of normal\index{i}{morphism!normal} morphisms in $\Bool_P$\index{s}{BoolP@$\Bool_P$}.

\begin{lem}\label{L:pixnormal}
Let~$X$ be a norm-covering\index{i}{norm-covering} of a poset~$P$. For each ideal\index{i}{ideal!of a poset}~$\bu$ of~$X$, there exists a unique homomorphism $\gp^X_\bu\colon\rF(X)\to\two$\index{s}{FrX@$\rF(X)$}\index{s}{piXu@$\gp^X_\bu$|ii} of Boolean algebras\index{i}{algebra!Boolean} such that\index{s}{utilde@$\tilde{u}$, $\tilde{u}^X$}\index{s}{piXu@$\gp^X_\bu$}
 \begin{equation}\label{Eq:DefpiKx}
 \gp^X_\bu(\tilde{v})=\begin{cases}
 1,&\text{if }v\in\bu\\
 0,&\text{otherwise}
 \end{cases}\,,
 \qquad\text{for every }v\in X\,.
 \end{equation}
Furthermore, if~$\bu$ is sharp\index{i}{ideal!sharp}, then $\gp^X_\bu$\index{s}{piXu@$\gp^X_\bu$} is a normal\index{i}{morphism!normal} morphism from $\xF(X)$\index{s}{FxX@$\xF(X)$} onto $\two[\partial\bu]$\index{s}{twop@$\two[p]$} in $\Bool_P$\index{s}{BoolP@$\Bool_P$}.
\end{lem}

\begin{proof}
Denote by $f(v)$ the right hand side of \eqref{Eq:DefpiKx}. It is obvious that~$f$ is antitone. Now we must prove that
 \[
 f(v_0)\wedge f(v_1)=\bigvee\famm{f(v)}{v\in v_0\sor v_1},
 \quad\text{for all }v_0,v_1\in X\,.
 \]
The only nontrivial case occurs when the left hand side of the equation above is equal to~$1$, in which case $v_0,v_1\in\bu$. As~$\bu$ is an ideal\index{i}{ideal!of a poset} of the \pjs\index{i}{pseudo join-semilattice}~$X$, there exists $v\in v_0\sor v_1$ such that $v\in\bu$, so we obtain, indeed, that $f(v)=1$.

Finally, we must prove that $1=\bigvee\famm{f(v)}{v\in\Min X}$, in other words, that $\bu\cap\Min X$ is nonempty. This follows from~$\bu$ being an ideal\index{i}{ideal!of a poset} in the \pjs\index{i}{pseudo join-semilattice}~$X$.

This completes the proof of the first statement of the lemma.

For each $p\in P$, $\gp^X_\bu``\bigl(\rF(X)^{(p)}\bigr)$\index{s}{FrX@$\rF(X)$}\index{s}{piXu@$\gp^X_\bu$} is nonzero if{f} there exists $v\in\bu$ such that $p\leq\partial v$, which holds if{f} $p\leq\partial\bu$, which is equivalent to $\two[\partial\bu]^{(p)}$\index{s}{twop@$\two[p]$} being nonzero. Hence $\gp^X_\bu``\bigl(\rF(X)^{(p)}\bigr)=\two[\partial\bu]^{(p)}$\index{s}{twop@$\two[p]$}\index{s}{FrX@$\rF(X)$}\index{s}{piXu@$\gp^X_\bu$} for each $p\in P$.
\qed\end{proof}

\chapter{The Condensate Lifting Lemma (CLL)}\label{Ch:CLL}

\textbf{Abstract.} In this chapter we shall finalize the approach to this work's main result, the Condensate Lifting Lemma CLL\index{i}{Condensate Lifting Lemma (CLL)}. The statement of CLL\index{i}{Condensate Lifting Lemma (CLL)} involves a ``condensate''\index{i}{condensate} $\xF(X)\otimes\overrightarrow{A}$\index{s}{FxX@$\xF(X)$}\index{s}{otimAS@$\bA\otimes\overrightarrow{S}$, $\gf\otimes\overrightarrow{S}$}. General condensates are defined in Section~\ref{S:AtensS}.

The statement of CLL\index{i}{Condensate Lifting Lemma (CLL)} also involves categories~$\cA$, $\cB$, $\cS$ with functors $\Phi\colon\cA\to\cS$ and $\Psi\colon\cB\to\cS$. For further applications of our work, such as Gillibert~\cite{Gill3}\index{c}{Gillibert, P.}, we need to divide CLL\index{i}{Condensate Lifting Lemma (CLL)} into two parts: the \emph{\underbar{A}rmature Lemma} (Lemma~\ref{L:Armature})\index{i}{Armature Lemma}, which deals with the functor $\Phi\colon\cA\to\cS$, and the \emph{\underbar{B}uttress Lemma} (Lemma~\ref{L:Buttress})\index{i}{Buttress Lemma}, which deals with the functor~$\Psi\colon\cB\to\cS$. The \emph{lifters}\index{i}{lifter ($\gl$-)}, which are objects of poset-theoretical nature with a set-theoretical slant, will be defined in Section~\ref{S:Armature}.

In Section~\ref{S:CLL} we shall put together the various assumptions surrounding $\cA$, $\cB$, $\cS$, $\Phi$, and~$\Psi$ in the definition of a \emph{larder}\index{i}{larder} (Definition~\ref{D:Larder}), and we shall state and prove CLL\index{i}{Condensate Lifting Lemma (CLL)}. In Section~\ref{S:Pos2Lift} we shall relate the poset-theoretical assumptions from CLL\index{i}{Condensate Lifting Lemma (CLL)} with infinite combinatorics, proving in particular that the shapes of the diagrams involved (the posets~$P$) are \ajs s\index{i}{almost join-semilattice} satisfying certain infinite combinatorial statements (cf. Corollary~\ref{C:CharLift}). In Section~\ref{S:MorePosets} we shall weaken both the assumptions and the conclusion from CLL\index{i}{Condensate Lifting Lemma (CLL)}, making it possible to consider diagrams indexed by \ajs s\index{i}{almost join-semilattice}~$P$ for which there is no lifter\index{i}{lifter ($\gl$-)}. In Section~\ref{S:LeftRightL} we shall split up the definition of a $\gl$-larder\index{i}{larder} between left larder\index{i}{larder!left} and right $\gl$-larder\index{i}{larder!right}, making it possible to write a large part of this work as a toolbox, in particular stating the right larderhood of many structures.

\section{The functor $\bA\mapsto\bA\otimes{\stackrel{\rightarrow}{S}}$; condensates}
\label{S:AtensS}

Throughout this section we shall fix a poset~$P$. We shall define and develop the basic properties of the functor $\bA\mapsto\bA\otimes\overrightarrow{S}$\index{s}{otimAS@$\bA\otimes\overrightarrow{S}$, $\gf\otimes\overrightarrow{S}$}, where~$\bA$ is a $P$-scaled Boolean algebra\index{i}{algebra!Pscaled Boolean@$P$-scaled Boolean} and~$\overrightarrow{S}$ is a $P$-indexed diagram in a category~$\cS$ with all nonempty finite products and all small directed colimits (cf. Definition~\ref{D:SmallDirColim}). In particular,  the objects and morphisms from the diagram~$\overrightarrow{S}$ can be recovered from that functor (Lemma~\ref{L:twopqotimesS}), and this functor sends normal\index{i}{morphism!normal} morphisms (cf. Definition~\ref{D:NormMorph}) to directed colimits of projections (Proposition~\ref{P:Norm2Proj}).

We recall that $|a|$\index{s}{nora@$\vert a\vert$, $a$ atom} denotes the largest element of~$\norm{a}$\index{s}{norma@$\Vert{a}\Vert$, $\Vert{a}\Vert_\bA$, $a$ atom}, for each atom~$a$ in a compact $P$-scaled Boolean algebra\index{i}{algebra!Pscaled Boolean@$P$-scaled Boolean}~$\bA$ (see \eqref{Eq:normAtom}).

\begin{defn}\label{D:AotSFin}
Let $\cS$ be a category where any two objects have a product.
For every compact $P$-scaled Boolean algebra\index{i}{algebra!Pscaled Boolean@$P$-scaled Boolean}~$\bA$ and every $P$-indexed diagram
$\overrightarrow{S}=\famm{S_p,\gs_p^q}{p\leq q\text{ in }P}$, we put\index{s}{nora@$\vert a\vert$, $a$ atom}\index{s}{otimAS@$\bA\otimes\overrightarrow{S}$, $\gf\otimes\overrightarrow{S}$|ii}
 \[
 \bA\otimes\overrightarrow{S}=\prod\famm{S_{|u|}}{u\in\At A}\,,
 \]
with canonical projections~$\gd_{\bA}^u\colon\bA\otimes\overrightarrow{S}\to S_{|u|}$\index{s}{deltaAu@$\gd_{\bA}^u$|ii}, for all $u\in\At A$.

For every compact morphism\index{i}{morphism!compact}~$\gf\colon\bA\to\bB$ in~$\Bool_P$\index{s}{BoolP@$\Bool_P$}, we define $\gf\otimes\overrightarrow{S}$ as the unique morphism from~$\bA\otimes\overrightarrow{S}$\index{s}{otimAS@$\bA\otimes\overrightarrow{S}$, $\gf\otimes\overrightarrow{S}$} to~$\bB\otimes\overrightarrow{S}$ such that $\gd_{\bB}^v\circ(\gf\otimes\overrightarrow{S})=\gs_{|v^\gf|}^{|v|}\circ\gd_{\bA}^{v^\gf}$\index{s}{nora@$\vert a\vert$, $a$ atom}, for all~$v\in\At B$, where~$v^\gf$ denotes the unique atom~$u$ of~$A$ such that~$v\leq\gf(u)$ (cf. Figure~\ref{Fig:defphiotimesS}).
\end{defn}

\begin{figure}[htb]
 \[
 \def\labelstyle{\displaystyle}\xymatrixcolsep{1.5pc}
 \xymatrix{
 \bA\otimes\overrightarrow{S}\ar[d]_{\gd_{\bA}^{v^\gf}}\ar[rr]^{\gf\otimes\overrightarrow{S}}&&
 \bB\otimes\overrightarrow{S}\ar[d]^{\gd_{\bB}^v}\\
 S_{|v^\gf|}\ar[rr]_{\gs_{|v^\gf|}^{|v|}}&& S_{|v|}
 }
 \]
\caption{The morphism $\gf\otimes{\stackrel{\rightarrow}{S}}$}
\label{Fig:defphiotimesS}
\end{figure}

Observe that if~$A=\two$\index{s}{two@$\two$}, then $\bA\otimes\overrightarrow{S}=S_{|1|}$\index{s}{nora@$\vert a\vert$, $a$ atom}\index{s}{otimAS@$\bA\otimes\overrightarrow{S}$, $\gf\otimes\overrightarrow{S}$}, with $\gd_{\bA}^1=\id_{S_{|1|}}$ (still for~$\bA$ compact). In general, the correctness of the definition of $\gf\otimes\overrightarrow{S}$ is ensured by the universal property of the product, together with the obvious containment $\norm{v^\gf}_\bA\subseteq\norm{v}_\bB$\index{s}{norma@$\Vert{a}\Vert$, $\Vert{a}\Vert_\bA$, $a$ atom}. In case~$\gf$ is normal\index{i}{morphism!normal}, we get $\norm{v^\gf}_\bA=\norm{v}_\bB$\index{s}{norma@$\Vert{a}\Vert$, $\Vert{a}\Vert_\bA$, $a$ atom}, thus $\bB\otimes\overrightarrow{S}\cong\prod\famm{S_{|u|}}{u\in(\At A)\setminus\gf^{-1}\set{0}}$\index{s}{nora@$\vert a\vert$, $a$ atom}\index{s}{otimAS@$\bA\otimes\overrightarrow{S}$, $\gf\otimes\overrightarrow{S}$} canonically and~$\gf\otimes\overrightarrow{S}$\index{s}{otimAS@$\bA\otimes\overrightarrow{S}$, $\gf\otimes\overrightarrow{S}$} is, up to isomorphism, the canonical projection from $\prod\famm{S_{|u|}}{u\in\At A}$ to $\prod\famm{S_{|u|}}{u\in(\At A)\setminus\gf^{-1}\set{0}}$\index{s}{nora@$\vert a\vert$, $a$ atom}. In particular, $\gf\otimes\overrightarrow{S}$\index{s}{otimAS@$\bA\otimes\overrightarrow{S}$, $\gf\otimes\overrightarrow{S}$} is a projection in~$\cS$ (cf. Definition~\ref{D:projection}). This proves the second statement of the following result. The proof of the first statement is straightforward.

\begin{prop}\label{P:TensFunctor}
Let $\overrightarrow{S}$ be a $P$-indexed diagram in a category~$\cS$ where any two objects have a product. The assignment $(\bA\mapsto\bA\otimes\overrightarrow{S}$, $\gf\mapsto\gf\otimes\overrightarrow{S})$\index{s}{otimAS@$\bA\otimes\overrightarrow{S}$, $\gf\otimes\overrightarrow{S}$} defines a functor from the full subcategory of compact $P$-scaled Boolean algebras\index{i}{algebra!Pscaled Boolean@$P$-scaled Boolean} to~$\cS$. This functor sends compact normal\index{i}{morphism!normal} morphisms in~$\Bool_P$\index{s}{BoolP@$\Bool_P$} to projections in~$\cS$.
\end{prop}

Now we apply this construction to the objects $\two[p]$\index{s}{twop@$\two[p]$} and morphisms~$\eps_p^q$\index{s}{epspq@$\eps_p^q$} introduced in Definition~\ref{D:two[p]}. We obtain immediately the following result.

\begin{lem}\label{L:twopqotimesS}
Let $\overrightarrow{S}=\famm{S_p,\gs_p^q}{p\leq q\text{ in }P}$ be a $P$-indexed diagram in a category~$\cS$ where any two objects have a product. Then the following statements hold:\index{s}{otimAS@$\bA\otimes\overrightarrow{S}$, $\gf\otimes\overrightarrow{S}$}
\begin{description}
\item[\tui] $\two[p]\otimes\overrightarrow{S}=S_p$\index{s}{twop@$\two[p]$}, for all $p\in P$.

\item[\tuii] $\eps_p^q\otimes\overrightarrow{S}=\gs_p^q$\index{s}{epspq@$\eps_p^q$}, for all~$p\leq q$ in~$P$.
\end{description}
\end{lem}

Now let~$\cS$ be a category where any two objects have a product, and with arbitrary small directed colimits. For a fixed $P$-indexed diagram~$\overrightarrow{S}$ in~$\cS$, we extend the functor ${}_{-}\otimes\overrightarrow{S}$ of Proposition~\ref{P:TensFunctor} to the whole category~$\Bool_P$\index{s}{BoolP@$\Bool_P$}, \emph{via} Proposition~\ref{P:ArrObj2Diag}, taking~$\cA:=\Bool_P$\index{s}{BoolP@$\Bool_P$} and defining~$\cA^\dagger$ as the full subcategory of compact $P$-scaled Boolean algebras\index{i}{algebra!Pscaled Boolean@$P$-scaled Boolean}. We obtain the following result.

\begin{prop}\label{P:TensFunctorLim}
Let $\overrightarrow{S}$ be a $P$-indexed diagram in a category~$\cS$ with small directed colimits where any two objects have a product. The assignment $(\bA\mapsto\bA\otimes\overrightarrow{S}$, $\gf\mapsto\gf\otimes\overrightarrow{S})$\index{s}{otimAS@$\bA\otimes\overrightarrow{S}$, $\gf\otimes\overrightarrow{S}$} defines a functor from~$\Bool_P$\index{s}{BoolP@$\Bool_P$} to~$\cS$. This functor preserves arbitrary small directed colimits.
\end{prop}

In case~$A=\two$\index{s}{two@$\two$}, we get $\bA\otimes\overrightarrow{S}=S_H=\varinjlim_{p\in H}S_p$\index{s}{otimAS@$\bA\otimes\overrightarrow{S}$, $\gf\otimes\overrightarrow{S}$} (with the obvious transition morphisms), where $H:=\norm{1}=\setm{p\in P}{1\in A^{(p)}}$\index{s}{norma@$\Vert{a}\Vert$, $\Vert{a}\Vert_\bA$, $a$ atom}.

\begin{defn}\label{D:Condensate}
In the context of Proposition~\ref{P:TensFunctorLim}, the object $\bA\otimes\overrightarrow{S}$\index{s}{otimAS@$\bA\otimes\overrightarrow{S}$, $\gf\otimes\overrightarrow{S}$} is a \emph{condensate}\index{i}{condensate|ii} of the poset-indexed diagram~$\overrightarrow{S}$.
\end{defn}

An immediate application of Propositions~\ref{P:NormColim} and~\ref{P:TensFunctor} also gives the following result.

\begin{prop}\label{P:Norm2Proj}
Let $\overrightarrow{S}$ be a $P$-indexed diagram in a category~$\cS$ with small directed colimits where any two objects have a product. The morphism~$\gf\otimes\overrightarrow{S}$\index{s}{otimAS@$\bA\otimes\overrightarrow{S}$, $\gf\otimes\overrightarrow{S}$} is an extended projection of~$\cS$, for every normal\index{i}{morphism!normal} morphism~$\gf$ in~$\Bool_P$\index{s}{BoolP@$\Bool_P$}.
\end{prop}

We refer the reader to Definition~\ref{D:projection} for extended projections.

\begin{remk}\label{Rk:kappasmallBool}
Let~$\gk$ be an infinite regular cardinal.
By using Remark~\ref{Rk:Restrkappa}, both Propositions~\ref{P:TensFunctorLim} and~\ref{P:Norm2Proj} can be easily relativized to the full subcategory~$(\Bool_P)\res\gk$\index{s}{BoolP@$\Bool_P$} of all~$\gk$-small directed colimits of compact $P$-scaled Boolean algebras\index{i}{algebra!Pscaled Boolean@$P$-scaled Boolean}, in case~$\cS$ has all~$\gk$-small directed colimits. If $\card P<\gk$, then, by Proposition~\ref{P:FinRepresBP}, the objects of $(\Bool_P)\res\gk$\index{s}{BoolP@$\Bool_P$} are exactly the $P$-scaled Boolean algebras\index{i}{algebra!Pscaled Boolean@$P$-scaled Boolean}~$\bA$ such that~$\card A<\gk$.
\end{remk}

\section{Lifters and the Armature Lemma}\label{S:Armature}

In this section we shall complete, with the \emph{lifters}\index{i}{lifter ($\gl$-)}, the introduction of all the concepts needed to formulate and prove the Armature Lemma\index{i}{Armature Lemma} (Lemma~\ref{L:Armature}). This lemma is a technical result with no obvious meaning, of which we can nevertheless try to give an intuitive idea. We are given categories~$\cA$ and~$\cS$ together with a functor $\Phi\colon\cA\to\cS$. We are also given a diagram~$\overrightarrow{A}$ indexed by a poset~$P$ satisfying a certain infinite combinatorial assumption called \emph{existence of a lifter}\index{i}{lifter ($\gl$-)}. The lifter (cf. Definition~\ref{D:Lifter}) consists of a pair~$(X,\bX)$, where~$X$ is a poset and~$\bX$ is a set of ideals\index{i}{ideal!of a poset} of~$X$. The set~$\bX^=$\index{s}{Xequal@$\bX^=$|ii} of all elements of~$\bX$ of non-maximal norm, partially ordered by containment, will be later intended as the~$U$ of the Buttress Lemma\index{i}{Buttress Lemma}, Lemma~\ref{L:Buttress}. Given an object~$S$ of~$\cS$ and a morphism $\chi\colon S\to\Phi\bigl(\xF(X)\otimes\overrightarrow{A}\bigr)$\index{s}{FxX@$\xF(X)$}\index{s}{otimAS@$\bA\otimes\overrightarrow{S}$, $\gf\otimes\overrightarrow{S}$}, together with extra categorical data, we get an ``image'' in~$\cS$ of the diagram~$\Phi\overrightarrow{A}$. Hence the Armature Lemma\index{i}{Armature Lemma} turns a morphism from an \emph{object} of~$\cS$ to the image under~$\Phi$ of a \emph{condensate}\index{i}{condensate} of~$\overrightarrow{A}$ to a natural transformation from a \emph{diagram} in~$\cS$ to the diagram~$\Phi\overrightarrow{A}$. Even more roughly speaking, the Armature Lemma\index{i}{Armature Lemma} states in which sense the condensate\index{i}{condensate} $\xF(X)\otimes\overrightarrow{A}$\index{s}{FxX@$\xF(X)$}\index{s}{otimAS@$\bA\otimes\overrightarrow{S}$, $\gf\otimes\overrightarrow{S}$} crystallizes the properties of the diagram~$\overrightarrow{A}$.

The following definition stems from the definition of $\overrightarrow{\gk}$-compatible norm-cov\-er\-ings introduced in~\cite{Gill1}\index{c}{Gillibert, P.}, with a tad more generality. It states the main combinatorial property that we shall require from a poset within the statement of CLL\index{i}{Condensate Lifting Lemma (CLL)}. We refer the reader to Definitions~\ref{D:PJS} and~\ref{D:NormCov} for norm-coverings\index{i}{norm-covering} and supported\index{i}{poset!supported} posets.

\begin{defn}\label{D:Lifter}
Let~$\gl$ be an infinite cardinal and let~$P$ be a poset. A \emph{$\gl$-lifter}\index{i}{lifter ($\gl$-)|ii} of~$P$ is a pair $(X,\bX)$, where~$X$ is a norm-covering\index{i}{norm-covering} of~$P$ and~$\bX$ is a subset of~$\Ids X$\index{s}{Ids@$\Ids X$, $X$ norm-covering}\index{i}{norm-covering} satisfying the following properties:
\begin{description}
\item[\tui] The set $\bX^=:=\setm{\bx\in\bX}{\partial\bx\text{ is not maximal in }P}$\index{s}{Xequal@$\bX^=$|ii} is lower $\cf(\gl)$-small\index{i}{poset!lower $\gl$-small}, that is, $\card(\bX\dnw\bx)<\cf(\gl)$ for each~$\bx\in\bX^=$.
 
\item[\tuii] For each map $S\colon\bX^=\to[X]^{<\gl}$, there exists an isotone map $\gs\colon P\to\nobreak\bX$ such that
\begin{description}
\item[\tua] the map $\gs$ is a \emph{section} of $\partial$, that is, $\partial\gs(p)=p$ holds for each $p\in P$;

\item[\tub] the containment $S(\gs(a))\cap\gs(b)\subseteq\gs(a)$ holds for all $a<b$ in~$P$.
(\emph{Observe that~$\gs(a)$ necessarily belongs to~$\bX^=$}.)
\end{description}

\item[\tuiii] If $\gl=\aleph_0$\index{s}{aleph0@$\aleph_{\ga}$}, then~$X$ is supported\index{i}{poset!supported}.
\end{description}
We say that~$P$ is \emph{$\gl$-liftable}\index{i}{liftable!$\gl$-${}_{-}$ poset|ii} if it has a $\gl$-lifter\index{i}{lifter ($\gl$-)}. Observe that in part (ii,b) of the definition above, it suffices to require~$S$ be \emph{isotone}: indeed, as~$\bX^=$ is lower $\cf(\gl)$-small\index{i}{poset!lower $\gl$-small} (this is part~(i)), every $S\colon\bX^=\to[X]^{<\gl}$ lies below some isotone $S'\colon\bX^=\to[X]^{<\gl}$ (e.g., $S'(\bx):=\bigcup\famm{S(\by)}{\by\in\bX^=\dnw\bx}$, for each $\bx\in\bX^=$).

We will often refer to the conclusion (ii,b) above by saying that the one-to-one map $\gs\colon P\mono\Pow(X)$\index{s}{AtomonoB@$f\colon A\mono B$} is \emph{free with respect to~$S$}\index{i}{freemapp@free map (wrt. set-mapping, poset)|ii}.
\end{defn}

Conditions for existence or nonexistence of lifters\index{i}{lifter ($\gl$-)} are given in Sections~\ref {S:Pos2Lift} and~\ref{S:LiftRetr}.

In the following statement of the Armature Lemma\index{i}{Armature Lemma}, we refer to Definition~\ref{D:wkapppres} for weakly $\gl$-presented\index{i}{presented!weakly $\gl$-} objects, Definition~\ref{D:NormCov} for norm-coverings\index{i}{norm-covering}, Lemma~\ref{L:pixnormal} for the morphisms $\gp^X_\bx$\index{s}{piXu@$\gp^X_\bu$}, and Definition~\ref{D:Lifter} for $\gl$-lifters\index{i}{lifter ($\gl$-)}\index{i}{Armature Lemma|ii}.

\begin{lem}[Armature Lemma]\label{L:Armature}
Let $\gl$ be an infinite cardinal, let~$P$ be a poset with a $\gl$-lifter~$(X,\bX)$\index{i}{lifter ($\gl$-)}, let $\cA$ and~$\cS$ be categories, let $\Phi\colon\cA\to\cS$ be a functor. We suppose that the following conditions are satisfied:
\begin{description}
\item[$(\CLOS(\cA))$] \index{s}{ClosA@$(\CLOS(\cA))$|ii}$\cA$ has all small directed colimits.

\item[$(\PROD(\cA))$] \index{s}{ProdA@$(\PROD(\cA))$|ii}Any two objects of~$\cA$ have a product in~$\cA$.

\item[$(\CONT(\Phi))$] \index{s}{Cont@$(\CONT(\Phi))$|ii}The functor~$\Phi$ preserves all small directed colimits.
\end{description}
Let $\overrightarrow{A}=\famm{A_p,\ga_p^q}{p\leq q\text{ in }P}$ be a $P$-indexed diagram in~$\cA$, let~$S$ be an object of~$\cS$, let $\chi\colon S\to\Phi\bigl(\xF(X)\otimes\overrightarrow{A}\bigr)$\index{s}{FxX@$\xF(X)$}\index{s}{otimAS@$\bA\otimes\overrightarrow{S}$, $\gf\otimes\overrightarrow{S}$}, and let $\famm{(S_\bx,\gf_\bx),\gf_{\bx}^{\by}}{\bx\subseteq\by\text{ in }\bX}$ be an $\bX$-indexed diagram in~$\cS\dnw S$ such that~$S_\bx$ is weakly $\gl$-presented\index{i}{presented!weakly $\gl$-} for each $\bx\in\bX^=$. We set\index{s}{piXu@$\gp^X_\bu$}\index{s}{otimAS@$\bA\otimes\overrightarrow{S}$, $\gf\otimes\overrightarrow{S}$}
 \[
 \gr_\bx:=\Phi\bigl(\gp^X_\bx\otimes\overrightarrow{A}\bigr)\circ\chi\,,
 \quad\text{for each }\bx\in\bX\,.
 \]
Then there exists an isotone section $\gs\colon P\into\bX$\index{s}{AtoinB@$f\colon A\into B$} of the lifter\index{i}{lifter ($\gl$-)} such that the family $\famm{\gr_{\gs(p)}\circ\gf_{\gs(p)}}{p\in P}$ is a natural transformation from the direct system $\famm{S_{\gs(p)},\gf_{\gs(p)}^{\gs(q)}}{p\leq q\text{ in }P}$ to $\Phi\overrightarrow{A}$.
\end{lem}

We remind the reader that condensates\index{i}{condensate} have been introduced in Definition~\ref{D:Condensate}. In particular, $\xF(X)\otimes\overrightarrow{A}$\index{s}{FxX@$\xF(X)$}\index{s}{otimAS@$\bA\otimes\overrightarrow{S}$, $\gf\otimes\overrightarrow{S}$} is a condensate of~$\overrightarrow{A}$.

\begin{proof}
The assumptions $(\CLOS(\cA))$\index{s}{ClosA@$(\CLOS(\cA))$} and $(\PROD(\cA))$\index{s}{ProdA@$(\PROD(\cA))$} are put there in order to ensure that the condensate\index{i}{condensate} $\xF(X)\otimes\overrightarrow{A}$\index{s}{FxX@$\xF(X)$}\index{s}{otimAS@$\bA\otimes\overrightarrow{S}$, $\gf\otimes\overrightarrow{S}$} is indeed well-defined (cf. Section~\ref{S:AtensS}).

As (the underlying poset of) $X$ is a \pjs\index{i}{pseudo join-semilattice}, if $\gl>\aleph_0$\index{s}{aleph0@$\aleph_{\ga}$}, then the $\sor$-closure $Z^{\sor}$\index{s}{minmajclos@$Z^{\sor}$|ii} of a $\gl$-small subset~$Z$ of~$X$ is also $\gl$-small; this also holds for $\gl=\aleph_0$\index{s}{aleph0@$\aleph_{\ga}$}, because in that case~$X$ is supported\index{i}{poset!supported}. In particular, in any case, $X$ is the directed union of the set~$[X]^{<\gl}_{\sor}$\index{s}{lammclos@$[X]^{<\gl}_{\sor}$|ii} of all its $\gl$-small $\sor$-closed subsets.

By using $(\CONT(\Phi))$\index{s}{Cont@$(\CONT(\Phi))$}, Lemma~\ref{L:fXYnormcov}, and Proposition~\ref{P:TensFunctorLim}, we obtain that for each $\bx\in\bX^=$,\index{s}{FxX@$\xF(X)$}\index{s}{otimAS@$\bA\otimes\overrightarrow{S}$, $\gf\otimes\overrightarrow{S}$}
 \[
 \chi\circ\gf_\bx\colon S_\bx\to
 \Phi\bigl(\xF(X)\otimes\overrightarrow{A}\bigr)=
 \varinjlim\famm{\Phi\bigl(\xF(Z^{\sor})\otimes\overrightarrow{A}\bigr)}
 {Z\in[X]^{<\gl}}
 \]
(the transition and limiting morphisms are all of the form~$\Phi\bigl(f_{Z_0}^{Z_1}\otimes\overrightarrow{A}\bigr)$\index{s}{fXY@$f_X^Y$}, where the~$f_{Z_0}^{Z_1}$\index{s}{fXY@$f_X^Y$} are given by Lemma~\ref{L:fXYnormcov}).
As~$S_\bx$ is, by assumption, weakly $\gl$-presented\index{i}{presented!weakly $\gl$-}, there exists~$V(\bx)\in[X]^{<\gl}$ such that $\chi\circ\gf_\bx$ factors through $\Phi\bigl(\xF(V(\bx)^{\sor})\otimes\overrightarrow{A}\bigr)$\index{s}{FxX@$\xF(X)$}\index{s}{otimAS@$\bA\otimes\overrightarrow{S}$, $\gf\otimes\overrightarrow{S}$}. Furthermore, for each $\bx\in\bX^=$, the set
 \[
 \ol{V}(\bx):=\bigcup\famm{V(\bx')}{\bx'\in\bX\dnw\bx}
 \]
is $\gl$-small (because $\bX\dnw\bx$ is $\cf(\gl)$-small and all the sets~$V(\bx')$ are $\gl$-small). Hence $\ol{V}(\bx)^{\sor}$ is also $\gl$-small. Therefore, replacing~$V(\bx)$ by~$\ol{V}(\bx)^{\sor}$, we may assume that the map~$V$ is isotone, with values in~$[X]^{<\gl}_{\sor}$.

We shall fix the map~$V$ until the end of the proof of Lemma~\ref{L:Armature}\index{i}{Armature Lemma}. Write\index{s}{fXY@$f_X^Y$}\index{s}{otimAS@$\bA\otimes\overrightarrow{S}$, $\gf\otimes\overrightarrow{S}$}
 \begin{equation}\label{Eq:chiotainfty}
 \chi\circ\gf_\bx=
 \Phi\bigl(f_{V(\bx)}^X\otimes\overrightarrow{A}\bigr)\circ\gy_\bx\,,
 \end{equation}
for some morphism $\gy_\bx\colon S_\bx\to\Phi\bigl(\xF(V(\bx))\otimes\overrightarrow{A}\bigr)$\index{s}{FxX@$\xF(X)$}\index{s}{otimAS@$\bA\otimes\overrightarrow{S}$, $\gf\otimes\overrightarrow{S}$}. The equation~\eqref{Eq:chiotainfty} can also be visualized on Figure~\ref{Fig:psibx}.

\begin{figure}[htb]
 \[
 \def\labelstyle{\displaystyle}
 \xymatrix{
 S_\bx\ar[rrr]^{\gf_\bx}\ar[d]_{\gy_{\bx}}&&&
 S\ar[d]^{\chi}\\
 \Phi\bigl(\xF(V(\bx))\otimes\overrightarrow{A}\bigr)
 \ar[rrr]_{\Phi\bigl(f_{V(\bx)}^X\otimes\overrightarrow{A}\bigr)}&&&
 \Phi\bigl(\xF(X)\otimes\overrightarrow{A}\bigr)
 }
 \]
\caption{The morphisms $\gy_\bx$}
\label{Fig:psibx}
\end{figure}

As~$(X,\bX)$ is a $\gl$-lifter\index{i}{lifter ($\gl$-)} of~$P$, its norm has a free isotone section $\gs\colon P\into\bX$\index{s}{AtoinB@$f\colon A\into B$} with respect to the set mapping~$V$, that is,
 \begin{equation}\label{Eq:gsisUfree}
 V(\gs(p))\cap\gs(q)\subseteq\gs(p)\quad\text{for all }p<q\text{ in }P\,.
 \end{equation}

\begin{claim}
The equation $\gp_{\gs(q)}^X\circ f_{V\gs(p)}^X=
\eps_p^q\circ\gp_{\gs(p)}^X\circ f_{V\gs(p)}^X$\index{s}{epspq@$\eps_p^q$}\index{s}{fXY@$f_X^Y$}\index{s}{piXu@$\gp^X_\bu$} is satisfied for all $p<q$ in~$P$.
\end{claim}

\begin{proof}
We need to verify that the given morphisms agree on the canonical generators of~$\xF(V\gs(p))$\index{s}{FxX@$\xF(X)$}, that is, the elements~$\tilde{u}^{V\gs(p)}$\index{s}{utilde@$\tilde{u}$, $\tilde{u}^X$}, for $u\in V\gs(p)$. We proceed:\index{s}{epspq@$\eps_p^q$}\index{s}{fXY@$f_X^Y$}\index{s}{piXu@$\gp^X_\bu$}
 \begin{align*}
 (\gp_{\gs(q)}^X\circ f_{V\gs(p)}^X)(\tilde{u}^{V\gs(p)})&=
 \gp_{\gs(q)}^X(\tilde{u}^X)=\begin{cases}
 1\,,&\text{if }u\in\gs(q)\\ 0\,,&\text{otherwise}\end{cases}\,,\\
 \intertext{while}
 (\eps_p^q\circ\gp_{\gs(p)}^X\circ f_{V\gs(p)}^X)(\tilde{u}^{V\gs(p)})&=
 (\eps_p^q\circ\gp_{\gs(p)}^X)(\tilde{u}^X)=\begin{cases}
 1\,,&\text{if }u\in\gs(p)\\ 0\,,&\text{otherwise}\end{cases}\,. 
 \end{align*}
By~\eqref{Eq:gsisUfree}, the two expressions agree.
\qed\ Claim\end{proof}

By applying the functor ${}_{-}\otimes\overrightarrow{A}$ (cf. Proposition~\ref{P:TensFunctor}) to the result of the Claim above and as $\eps_p^q\otimes\overrightarrow{A}=\ga_p^q$\index{s}{epspq@$\eps_p^q$}\index{s}{otimAS@$\bA\otimes\overrightarrow{S}$, $\gf\otimes\overrightarrow{S}$} (cf. Lemma~\ref{L:twopqotimesS}(ii)), we thus obtain the following equation, for all $p<q$ in~$P$:\index{s}{fXY@$f_X^Y$}\index{s}{piXu@$\gp^X_\bu$}\index{s}{otimAS@$\bA\otimes\overrightarrow{S}$, $\gf\otimes\overrightarrow{S}$}
 \begin{equation}\label{Eq:(Flag)otimesvecA}
 (\gp_{\gs(q)}^X\otimes\overrightarrow{A})\circ(f_{V\gs(p)}^X\otimes\overrightarrow{A})=
 \ga_p^q\circ(\gp_{\gs(p)}^X\otimes\overrightarrow{A})\circ(f_{V\gs(p)}^X\otimes\overrightarrow{A})\,.
 \end{equation}
Furthermore, from $p<q$ it follows that $\gs(p)\in\bX^=$, thus, substituting~$\gs(p)$ to~$\bx$ in the diagram of Figure~\ref{Fig:psibx}, we obtain that the diagram of Figure~\ref{Fig:psigs(p)} is commutative.\index{s}{otimAS@$\bA\otimes\overrightarrow{S}$, $\gf\otimes\overrightarrow{S}$}

\begin{figure}[htb]\index{s}{fXY@$f_X^Y$}
 \[
 \def\labelstyle{\displaystyle}
 \xymatrix{
 S_{\gs(p)}\ar[rrr]^{\gf_{\gs(p)}}\ar[d]_{\gy_{\gs(p)}}&&&
 S\ar[d]^{\chi}\\
 \Phi\bigl(\xF(V\gs(p))\otimes\overrightarrow{A}\bigr)
 \ar[rrr]_{\Phi\bigl(f_{V\gs(p)}^X\otimes\overrightarrow{A}\bigr)}&&&
 \Phi\bigl(\xF(X)\otimes\overrightarrow{A}\bigr)
 }
 \]
\caption{The morphisms $\gy_{\gs(p)}$}
\label{Fig:psigs(p)}
\end{figure}
Now we can proceed:\index{s}{fXY@$f_X^Y$}\index{s}{piXu@$\gp^X_\bu$}\index{s}{otimAS@$\bA\otimes\overrightarrow{S}$, $\gf\otimes\overrightarrow{S}$}
 \begin{align*}
 \Phi(\ga_p^q)\circ\gr_{\gs(p)}\circ\gf_{\gs(p)}&=
 \Phi(\ga_p^q)\circ\Phi\bigl(\gp_{\gs(p)}^X\otimes\overrightarrow{A}\bigr)
 \circ\chi\circ\gf_{\gs(p)}\\
 &\qquad\qquad\qquad\qquad\qquad\qquad
 (\text{by the definition of }\gr_{\gs(p)})\\
 &=\Phi(\ga_p^q)\circ\Phi\bigl(\gp_{\gs(p)}^X\otimes\overrightarrow{A}\bigr)
 \circ\Phi\bigl(f_{V\gs(p)}^X\otimes\overrightarrow{A}\bigr)\circ\gy_{\gs(p)}\\
 &\qquad\qquad\qquad\qquad\qquad\qquad\qquad\qquad
 (\text{cf. Figure~\ref{Fig:psigs(p)}})\\
 &=\Phi\bigl(\gp_{\gs(q)}^X\otimes\overrightarrow{A}\bigr)\circ
 \Phi\bigl(f_{V\gs(p)}^X\otimes\overrightarrow{A}\bigr)\circ\gy_{\gs(p)}
 \quad(\text{use~\eqref{Eq:(Flag)otimesvecA}})\\
 &=\Phi\bigl(\gp_{\gs(q)}^X\otimes\overrightarrow{A}\bigr)\circ\chi\circ\gf_{\gs(p)}
 \qquad\qquad(\text{cf. Figure~\ref{Fig:psigs(p)}})\\
 &=\gr_{\gs(q)}\circ\gf_{\gs(p)}
 \qquad\qquad(\text{by the definition of }\gr_{\gs(q)})\\
 &=\gr_{\gs(q)}\circ\gf_{\gs(q)}\circ\gf_{\gs(p)}^{\gs(q)}\,.
 \end{align*}

\begin{figure}[htb]
 \[
 \def\labelstyle{\displaystyle}
 \xymatrix{
 S_{\gs(p)}\ar[rr]^{\gf_{\gs(p)}^{\gs(q)}}\ar[d]_{\gr_{\gs(p)}\circ\gf_{\gs(p)}}&&
 S_{\gs(q)}\ar[d]^{\gr_{\gs(q)}\circ\gf_{\gs(q)}}\\
 \Phi(A_p)\ar[rr]_{\Phi(\ga_p^q)}&&\Phi(A_q)
  }
 \]
\caption{Getting the desired natural transformation}
\label{Fig:FinSqpq}
\end{figure}
Therefore, the diagram of Figure~\ref{Fig:FinSqpq} commutes, as desired.
\qed\end{proof}

\section{The L\"owenheim-Skolem Condition and the Buttress Lemma}\label{S:Buttress}

The main result of Section~\ref{S:Buttress} is a technical result with no obvious meaning, Lemma~\ref{L:Buttress}\index{i}{Buttress Lemma}. Roughly speaking, it states the following. We are given categories~$\cB$ and~$\cS$, together with a subcategory~$\cS^\Rightarrow$\index{s}{RightarrowCat@$\cS^\Rightarrow$} of~$\cS$ (the ``double arrows''\index{i}{double arrow}), a full subcategory~$\cB^\dagger$ of~$\cB$ (the ``small objects'' of~$\cB$), a functor $\Psi\colon\cB\to\cS$, and object~$B$ of~$\cB$, and a family of double arrows\index{i}{double arrow} $\gr_u\colon\Psi(B)\Rightarrow S_u$\index{s}{AtorightarrowB@$f\colon A\Rightarrow B$}, where~$u$ ranges over a poset~$U$. The latter says that all objects~$S_u$ are ``small''. A \emph{buttress} of the family $\famm{\gr_u}{u\in U}$ (cf. Definition~\ref{D:Buttress})\index{i}{buttress} is an $U$-indexed diagram in~$\cB^\dagger\dnw B$ that witnesses the \emph{collective} smallness of $\famm{S_u}{u\in U}$. The main assumption required in Lemma~\ref{L:Buttress} is a L\"owenheim-Skolem\index{i}{Lowenheim@L\"owenheim-Skolem Theorem} type property, denoted there by~$(\LSb_\gl(B))$\index{s}{LSb@$(\LSb_\gm(B))$}. In this sense, Lemma~\ref{L:Buttress} may be considered a \emph{diagram version} of the L\"owenheim-Skolem property.

All the uses that we have been able to find so far for Lemma~\ref{L:Buttress} have the poset~$U$ \emph{well-founded}\index{i}{poset!well-founded}. However, the future may also bring uses of that lemma in the non well-founded case, hence we record it also for that case. There are other possible variants of Lemma~\ref{L:Buttress}, that we shall not record here; we tried to include the one with the largest application range.

\begin{defn}\label{D:Buttress}
Let~$\cB$ and $\cS$ be categories together with a full subcategory~$\cB^\dagger$ of~$\cB$, a subcategory~$\cS^\Rightarrow$\index{s}{RightarrowCat@$\cS^\Rightarrow$} of~$\cS$, and a functor $\Psi\colon\cB\to\cS$, let~$U$ be a poset, let~$B$ be an object of~$\cB$. A \emph{buttress}\index{i}{buttress|ii} of a family $\famm{\gr_u\colon\Psi(B)\to S_u}{u\in U}$ of morphisms in~$\cS$ is an $U$-indexed diagram
 \[
 \famm{\gb_u\colon B_u\to B\,,\ \gb_u^v\colon\gb_u\to\gb_v}
 {u\leq v\text{ in }U}
 \]
in~$\cB^\dagger\dnw B$ such that $\gr_u\circ\Psi(\gb_u)$ is a morphism of~$\cS^\Rightarrow$\index{s}{RightarrowCat@$\cS^\Rightarrow$} for each $u\in U$. The buttress\index{i}{buttress} is \emph{monic} if all morphisms~$\gb_u$ are monic (thus all morphisms~$\gb_u^v$ are monic).
\end{defn}

Observe that the definition of a buttress\index{i}{buttress} is formulated relatively to the object~$B$ of~$\cB$, the full subcategory~$\cB^\dagger$ of~$\cB$, the subcategory~$\cS^\Rightarrow$\index{s}{RightarrowCat@$\cS^\Rightarrow$} of~$\cS$, and the functor~$\Psi$.

In the following result, we shall refer to $(\LSb_\gl(B))$\index{s}{LSb@$(\LSb_\gm(B))$} as the \emph{L\"owenheim-Skolem Condition with index~$\gl$ at~$B$}\index{i}{Lowenheim@L\"owenheim-Skolem Condition}. It is labeled after the classical L\"owenheim-Skolem Theorem\index{i}{Lowenheim@L\"owenheim-Skolem Theorem} in model theory.\index{i}{Buttress Lemma|ii}

\begin{lem}[The Buttress Lemma]\label{L:Buttress}
Let $\gl$ be an infinite cardinal, let~$\cB$ and~$\cS$ be categories together with a full subcategory~$\cB^\dagger$ of~$\cB$, a subcategory~$\cS^\Rightarrow$\index{s}{RightarrowCat@$\cS^\Rightarrow$} of~$\cS$, and a functor $\Psi\colon\cB\to\cS$, let~$U$ be a lower $\gl$-small\index{i}{poset!lower $\gl$-small} poset, let~$B$ be an object of~$\cB$, and let $\overrightarrow{\gr}=\famm{\gr_u\colon\Psi(B)\Rightarrow S_u}{u\in U}$\index{s}{AtorightarrowB@$f\colon A\Rightarrow B$} be an $U$-indexed family of morphisms in~$\cS^\Rightarrow$\index{s}{RightarrowCat@$\cS^\Rightarrow$}. We make the following assumption:
\begin{description}
\item[$(\LSb_\gl(B))$] \index{s}{LSb@$(\LSb_\gm(B))$|ii}For each $u\in U$, each double arrow\index{i}{double arrow} $\gy\colon\Psi(B)\Rightarrow S_u$\index{s}{AtorightarrowB@$f\colon A\Rightarrow B$}, each $\gl$-small set~$I$, and each family
$\famm{\gc_i\colon C_i\mono B}{i\in I}$\index{s}{AtomonoB@$f\colon A\mono B$} of monic objects in~$\cB^\dagger\dnw B$, there exists a monic object $\gc\colon C\mono B$\index{s}{AtomonoB@$f\colon A\mono B$} in~$\cB^\dagger\dnw B$ such that $\gc_i\utr\gc$ for each~$i\in I$ while $\gy\circ\Psi(\gc)$ is a morphism of~$\cS^\Rightarrow$\index{s}{RightarrowCat@$\cS^\Rightarrow$}.
\end{description}
Furthermore, we assume that either~$U$ is well-founded\index{i}{poset!well-founded} or the following additional assumptions are satisfied:
\begin{description}
\item[$(\CLOS_\gl(\cB^\dagger,\cB))$] \index{s}{ClosB@$(\CLOS_\gl(\cB^\dagger,\cB))$|ii}The full subcategory~$\cB^\dagger$ has all $\gl$-small directed colimits within\index{i}{having all $\gl$-small directed colimits within~$\cC$}~$\cB$ \pup{cf. Definition~\textup{\ref{D:SmallDirColim}}}.

\item[$(\CLOSr_\gl(\cS^\Rightarrow))$] \index{s}{ClosBr@$(\CLOSr_\gl(\cS^\Rightarrow))$|ii}The subcategory $\cS^\Rightarrow$\index{s}{RightarrowCat@$\cS^\Rightarrow$} is right closed under all $\gl$-small directed colimits \pup{cf. Definition~\textup{\ref{D:ClosDirColim}}}.

\item[$(\CONT_\gl(\Psi))$] \index{s}{Contl@$(\CONT_\gl(\Psi))$|ii}The functor~$\Psi$ preserves all $\gl$-small directed colimits.
\end{description}
Then~$\overrightarrow{\gr}$ has a buttress\index{i}{buttress}.
\end{lem}

\begin{proof}
We first assume that~$U$ is well-founded\index{i}{poset!well-founded}. Let~$w\in U$ and suppose having constructed an $(U\dnw w)$-indexed diagram in~$\cB^\dagger\dnw B$,\index{s}{AtomonoB@$f\colon A\mono B$}
 \[
 \famm{\gb_u\colon B_u\mono B\,,\ \gb_u^v\colon\gb_u\mono\gb_v}
 {u\leq v\text{ in }U\ddnw w}\,,
 \]
with all $\gb_u\colon B_u\mono B$\index{s}{AtomonoB@$f\colon A\mono B$} monic. By $(\LSb_\gl(B))$\index{s}{LSb@$(\LSb_\gm(B))$}, there exists a monic object $\gb_w\colon B_w\mono B$ of~$\cB^\dagger\dnw B$ such that $\gb_u\utr\gb_w$ for each $u\in U\ddnw w$ while $\gr_w\circ\Psi(\gb_w)$ is a double arrow\index{i}{double arrow} of~$\cS$. Denote by $\gb_u^w\colon B_u\mono B_w$\index{s}{AtomonoB@$f\colon A\mono B$} the unique morphism such that $\gb_u=\gb_w\circ\gb_u^w$, for each $u\in U\ddnw w$. For all $u\leq v<w$,
 \[
 \gb_w\circ\gb_u^w=\gb_u=\gb_v\circ\gb_u^v=\gb_w\circ\gb_v^w\circ\gb_u^v\,,
 \]
thus, as~$\gb_w$ is monic, $\gb_u^w=\gb_v^w\circ\gb_u^v$. This completes the proof in case~$U$ is well-founded\index{i}{poset!well-founded}.

Now remove the well-foundedness\index{i}{poset!well-founded} assumption on~$U$ but assume the conditions~$(\CLOS_\gl(\cB^\dagger,\cB))$\index{s}{ClosB@$(\CLOS_\gl(\cB^\dagger,\cB))$}, $(\CLOSr_\gl(\cS^\Rightarrow))$\index{s}{ClosBr@$(\CLOSr_\gl(\cS^\Rightarrow))$}, and $(\CONT_\gl(\Psi))$\index{s}{Contl@$(\CONT_\gl(\Psi))$}. We endow the set
 \[
 \widehat{U}:=\setm{(u,a)\in U\times[U]^{<\go}}{a\subseteq U\dnw u}
 \]
with the partial ordering defined by $(u,a)\leq(v,b)$ if{f} $u\leq v$ and $a\subseteq b$. As~$\widehat{U}$ is lower finite\index{i}{poset!lower finite}, it is well-founded\index{i}{poset!well-founded}, thus, by the paragraph above, the family $\famm{\gr_u}{(u,a)\in\widehat{U}}$ has a buttress\index{i}{buttress}, say
 \[
 \famm{\gb_{u,a}\colon B_{u,a}\to B\,,
 \ \gb_{u,a}^{v,b}\colon\gb_{u,a}\to\gb_{v,b}}
 {(u,a)\leq(v,b)\text{ in }\widehat{U}}\,.
 \]
By assumption, for each $u\in U$, there exists a directed colimit cocone in~$\cB$
 \begin{equation}\label{Eq:Colimit(Bu)}
 \famm{B_u,\gd_{u,a}}{a\in[U\dnw u]^{<\go}}=
 \varinjlim\famm{B_{u,a},\gb_{u,a}^{u,b}}{a\subseteq b\text{ in }[U\dnw u]^{<\go}}
 \end{equation}
with $B_u\in\cB^\dagger$. By the universal property of the colimit, there exists a unique morphism $\gb_u\colon B_u\to B$ such that $\gb_u\circ\gd_{u,a}=\gb_{u,a}$ for each $a\in[U\dnw u]^{<\go}$. Moreover, it follows from~$(\CONT_\gl(\Psi))$\index{s}{Contl@$(\CONT_\gl(\Psi))$} and~\eqref{Eq:Colimit(Bu)} that the following statement holds in~$\cS$:
 \begin{equation}\label{Eq:ColimitPsi(Bu)}
 \famm{\Psi(B_u),\Psi(\gd_{u,a})}{a\in[U\dnw u]^{<\go}}=
 \varinjlim\famm{\Psi(B_{u,a}),\Psi(\gb_{u,a}^{u,b})}
 {a\subseteq b\text{ in }[U\dnw u]^{<\go}}\,.
 \end{equation}
For each $a\in[U\dnw u]^{<\go}$, $(\gr_u\circ\Psi(\gb_u))\circ\Psi(\gd_{u,a})=\gr_u\circ\Psi(\gb_{u,a})$ is a double arrow\index{i}{double arrow}, thus, by the assumption~$(\CLOS_\gl(\cB^\dagger,\cB))$\index{s}{ClosB@$(\CLOS_\gl(\cB^\dagger,\cB))$} together with~\eqref{Eq:ColimitPsi(Bu)}, $\gr_u\circ\Psi(\gb_u)$ is a double arrow\index{i}{double arrow} of~$\cS$.

For all $u\leq v$ in~$U$ and all $a\subseteq b$ in~$[U\dnw u]^{<\go}$,
we obtain, by using~\eqref{Eq:Colimit(Bu)}, that
 \[
 (\gd_{v,b}\circ\gb_{u,b}^{v,b})\circ\gb_{u,a}^{u,b}=\gd_{v,b}\circ\gb_{u,a}^{v,b}
 =\gd_{v,b}\circ\gb_{v,a}^{v,b}\circ\gb_{u,a}^{v,a}=\gd_{v,a}\circ\gb_{u,a}^{v,a}\,,
 \] 
thus, by the universal property of the colimit, there exists a unique morphism $\gb_u^v\colon B_u\to B_v$ such that $\gb_u^v\circ\gd_{u,a}=\gd_{v,a}\circ\gb_{u,a}^{v,a}$ for each $a\in[U\dnw u]^{<\go}$. Hence
 \[
 \gb_v\circ\gb_u^v\circ\gd_{u,a}=\gb_v\circ\gd_{v,a}\circ\gb_{u,a}^{v,a}
 =\gb_{v,a}\circ\gb_{u,a}^{v,a}=\gb_{u,a}=\gb_u\circ\gd_{u,a}\,.
 \]
As this holds for each $a\in[U\dnw u]^{<\go}$ and by the universal property of the colimit, it follows that $\gb_v\circ\gb_u^v=\gb_u$.

A similar proof, with~$\gb_v^w$ in place of~$\gb_v$, for $u\leq v\leq w$, yields that $\gb_u^w=\gb_v^w\circ\gb_u^v$. This completes the proof.
\qed\end{proof}

\begin{remk}\label{Rk:Buttress}
Observe from the proof of Lemma~\ref{L:Buttress}\index{i}{Buttress Lemma} that if~$U$ is well-founded\index{i}{poset!well-founded}, then the buttress\index{i}{buttress} can be taken monic. In fact, in most applications of Lemma~\ref{L:Buttress} through this work, $U$ will be~$\bX^=$ for a $\gl$-lifter\index{i}{lifter ($\gl$-)} $(X,\bX)$ of a lower finite\index{i}{poset!lower finite} \ajs\index{i}{almost join-semilattice}~$P$, and in such a situation~$\bX$ may be taken lower finite\index{i}{poset!lower finite} (cf. Lemma~\ref{L:Part2Lift}). Therefore, in most of (but not all) our applications, $U$ can be taken lower finite\index{i}{poset!lower finite}, and in that case, in order to get the conclusion of Lemma~\ref{L:Buttress}\index{i}{Buttress Lemma}, it is sufficient to replace $(\LSb_\gl(B))$\index{s}{LSb@$(\LSb_\gm(B))$} by~$(\LSb_\go(B))$\index{s}{LSb@$(\LSb_\gm(B))$}.

Situations corresponding to the case where~$U$ is not lower finite\index{i}{poset!lower finite} are encountered in Gillibert~\cite{Gill1}\index{c}{Gillibert, P.}. It seems to us quite likely that future extensions of our work may require the emergence of new variants of the Buttress Lemma\index{i}{Buttress Lemma}, while, quite to the contrary, the Armature Lemma\index{i}{Armature Lemma} looks more stable.
\end{remk}

\section{Larders and the Condensate Lifting Lemma}\label{S:CLL}

In this section we shall complete, with the \emph{larders}\index{i}{larder}, the introduction of all the concepts needed to formulate and prove CLL\index{i}{Condensate Lifting Lemma (CLL)} (Lemma~\ref{L:CLL}).

Recall (cf. Section~\ref{Su:FunctSol}) that the basic categorical context of CLL\index{i}{Condensate Lifting Lemma (CLL)} consists of categories~$\cA$, $\cB$, $\cS$, functors~$\Phi$ and~$\Psi$, and a few add-ons. Definition~\ref{D:Larder} states what these add-ons should be and what they should be expected to satisfy.

\begin{defn}\label{D:Larder}
Let $\gl$ and~$\gm$ be infinite cardinals. We say that an octuple
$\Lambda=(\cA,\cB,\cS,\cA^\dagger,\cB^\dagger,\cS^\Rightarrow,\Phi,\Psi)$ is a \emph{$(\gl,\gm)$-larder\index{i}{larder|ii} at an object~$B$} if $\cA$, $\cB$, $\cS$ are categories, $B$ is an object of~$\cB$, $\Phi\colon\cA\to\cS$ and $\Psi\colon\cB\to\nobreak\cS$ are functors, $\cA^\dagger$ \pup{resp., $\cB^\dagger$} is a full subcategory of~$\cA$ \pup{resp., $\cB$}, and~$\cS^\Rightarrow$\index{s}{RightarrowCat@$\cS^\Rightarrow$} is a subcategory of~$\cS$ satisfying the following conditions:
\begin{description}
\item[$(\CLOS(\cA))$] \index{s}{ClosA@$(\CLOS(\cA))$}$\cA$ has all small directed colimits.

\item[$(\PROD(\cA))$] \index{s}{ProdA@$(\PROD(\cA))$}Any two objects in~$\cA$ have a product in~$\cA$.

\item[$(\CONT(\Phi))$] \index{s}{Cont@$(\CONT(\Phi))$}The functor~$\Phi$ preserves all small directed colimits.

\item[$(\PROJ(\Phi,\cS^\Rightarrow))$] \index{s}{Proj@$(\PROJ(\Phi,\cS^\Rightarrow))$|ii}$\Phi(f)$ is a morphism in $\cS^\Rightarrow$\index{s}{RightarrowCat@$\cS^\Rightarrow$}, for each extended projection~$f$ of~$\cA$ (cf. Definition~\ref{D:projection}).

\item[$(\PRES_\gl(\cB^\dagger,\Psi))$] \index{s}{Pres@$(\PRES_\gl(\cB^\dagger,\Psi))$|ii}The object $\Psi(B)$ is weakly $\gl$-presented\index{i}{presented!weakly $\gl$-} in~$\cS$, for each object~$B\in\cB^\dagger$.

\item[$(\LS_\gm(B))$] \index{s}{LS@$(\LS_\gm(B))$|ii}For each $S\in\Phi``(\cA^\dagger)$, each double arrow\index{i}{double arrow} $\gy\colon\Psi(B)\Rightarrow S$\index{s}{AtorightarrowB@$f\colon A\Rightarrow B$}, each $\gm$-small set~$I$, and each family
$\famm{\gc_i\colon C_i\mono B}{i\in I}$\index{s}{AtomonoB@$f\colon A\mono B$} of monic objects in~$\cB^\dagger\dnw B$, there exists a monic object $\gc\colon C\mono B$\index{s}{AtomonoB@$f\colon A\mono B$} in~$\cB^\dagger\dnw B$ such that $\gc_i\utr\gc$ for each~$i\in I$ while $\gy\circ\Psi(\gc)$ is a morphism in~$\cS^\Rightarrow$\index{s}{RightarrowCat@$\cS^\Rightarrow$}.

\end{description}
We say that~$\Lambda$ is a \emph{$(\gl,\gm)$-larder}\index{i}{larder} if it is a $(\gl,\gm)$-larder at every object of~$\cB$.

We say that~$\Lambda$ is \emph{strong}\index{i}{larder!strong|ii} if the following conditions are satisfied:
\begin{description}
\item[$(\CLOS_\gm(\cB^\dagger,\cB))$] \index{s}{ClosB@$(\CLOS_\gl(\cB^\dagger,\cB))$}The full subcategory~$\cB^\dagger$ has all $\gm$-small directed colimits within~$\cB$ \pup{cf. Definition~\textup{\ref{D:SmallDirColim}}}.

\item[$(\CLOSr_\gm(\cS^\Rightarrow))$] \index{s}{ClosBr@$(\CLOSr_\gl(\cS^\Rightarrow))$}The subcategory $\cS^\Rightarrow$\index{s}{RightarrowCat@$\cS^\Rightarrow$} is right closed under all $\gm$-small directed colimits \pup{cf. Definition~\textup{\ref{D:ClosDirColim}}}.

\item[$(\CONT_\gm(\Psi))$] \index{s}{Contl@$(\CONT_\gl(\Psi))$}The functor~$\Psi$ preserves all $\gm$-small directed colimits.
\end{description}

We say that~$\Lambda$ is \emph{projectable}\index{i}{larder!projectable|ii} if every double arrow\index{i}{double arrow} $\varphi\colon\Psi(C)\Rightarrow S$\index{s}{AtorightarrowB@$f\colon A\Rightarrow B$}, for objects~$C\in\cB$ and~$S\in\cS$, has a projectability witness\index{i}{projectability witness} (cf. Definition~\ref{D:ProjFunct}).

We shall usually say $\gl$-larder instead of $(\gl,\cf(\gl))$-larder\index{i}{larder}.
\end{defn}

The conditions $(\CLOS_\gm(\cB^\dagger,\cB))$, $(\CLOSr_\gm(\cS^\Rightarrow))$, and $(\CONT_\gm(\Psi))$ were formulated within the statement of the Buttress Lemma (Lemma~\ref{L:Buttress}). The condition $(\LS_\gm(B))$ is a modification of the condition $(\LSb_\gl(B))$ formulated within the statement of the Buttress Lemma.

Now at long last we have reached the statement of CLL\index{i}{Condensate Lifting Lemma (CLL)}. We refer to Definition~\ref{D:NormCov} for norm-coverings\index{i}{norm-covering}, Definition~\ref{D:Lifter} for $\gl$-lifters\index{i}{lifter ($\gl$-)}, and Definition~\ref{D:Larder} for $\gl$-larders\index{i}{larder}.\index{i}{Condensate Lifting Lemma (CLL)|ii}

\begin{lem}[Condensate Lifting Lemma]\label{L:CLL}
Let $\gl$ and~$\gm$ be infinite cardinals, let $\Lambda:=(\cA,\cB,\cS,\cA^\dagger,\cB^\dagger,\cS^\Rightarrow,\Phi,\Psi)$ be a $(\gl,\gm)$-larder\index{i}{larder} at an object~$B$ of~$\cB$, and let~$P$ be a poset with a $\gl$-lifter~$(X,\bX)$\index{i}{lifter ($\gl$-)}. Suppose that we are also given the following additional data:
\begin{itemize}
\item a $P$-indexed diagram $\overrightarrow{A}=\famm{A_p,\ga_p^q}{p\leq q\text{ in }P}$ in~$\cA$ such that $A_p$ belongs to~$\cA^\dagger$ for each non-maximal $p\in P$;

\item a double arrow\index{i}{double arrow} $\chi\colon\Psi(B)\Rightarrow\Phi\bigl(\xF(X)\otimes\overrightarrow{A}\bigr)$\index{s}{AtorightarrowB@$f\colon A\Rightarrow B$}\index{s}{FxX@$\xF(X)$}\index{s}{otimAS@$\bA\otimes\overrightarrow{S}$, $\gf\otimes\overrightarrow{S}$} in~$\cS$.
\end{itemize}
Then in each of the following cases,
\begin{description}
\item[\tui] $\gm=\aleph_0$\index{s}{aleph0@$\aleph_{\ga}$} and $\bX^=$ is lower finite\index{i}{poset!lower finite},

\item[\tuii] $\gm=\cf(\gl)$ and $\bX^=$ is well-founded\index{i}{poset!well-founded},

\item[\tuiii] $\gm=\cf(\gl)$ and $\Lambda$ is strong\index{i}{larder!strong},
\end{description}
there are an object~$\overrightarrow{B}$ of~$\cB^P$ and a double arrow\index{i}{double arrow} $\overrightarrow{\chi}\colon\Psi\overrightarrow{B}\Rightarrow\Phi\overrightarrow{A}$\index{s}{AtorightarrowB@$f\colon A\Rightarrow B$} in~$\cS^P$ such that $B_p\in\cB^\dagger$ for each non-maximal~$p\in P$ while $B_p=B$ for each maximal $p\in P$. Furthermore, if~$\Lambda$ is projectable\index{i}{larder!projectable}, then there are an object~$\overrightarrow{B}'$ of~$\cB^P$ and a natural equivalence $\overrightarrow{\chi}'\colon\Psi\overrightarrow{B}'\todot\Phi\overrightarrow{A}$\index{s}{AtorightarrowdotB@$f\colon\xA\todot\xB$}, while there exists a natural transformation from~$\overrightarrow{B}$ to~$\overrightarrow{B}'$ all whose components are epimorphisms.
\end{lem}

\begin{proof}
It follows from $(\CLOS(\cA))$\index{s}{ClosA@$(\CLOS(\cA))$} and $(\PROD(\cA))$\index{s}{ProdA@$(\PROD(\cA))$} that the results of Section~\ref{S:AtensS} apply to the category~$\cA$. As, by Lemma~\ref{L:pixnormal}, $\gp_\bx^X$\index{s}{piXu@$\gp^X_\bu$} is a normal\index{i}{morphism!normal} morphism from~$\xF(X)$\index{s}{FxX@$\xF(X)$} to~$\two$\index{s}{two@$\two$}, for each $\bx\in\bX$, it follows from Proposition~\ref{P:Norm2Proj} that $\gp_\bx^X\otimes\overrightarrow{A}$\index{s}{piXu@$\gp^X_\bu$}\index{s}{otimAS@$\bA\otimes\overrightarrow{S}$, $\gf\otimes\overrightarrow{S}$} is an extended projection of~$\cA$. {}From $(\PROJ(\Phi,\cS^\Rightarrow))$\index{s}{Proj@$(\PROJ(\Phi,\cS^\Rightarrow))$} it follows that $\Phi\bigl(\gp_\bx^X\otimes\overrightarrow{A}\bigr)$\index{s}{piXu@$\gp^X_\bu$}\index{s}{otimAS@$\bA\otimes\overrightarrow{S}$, $\gf\otimes\overrightarrow{S}$} is a double arrow\index{i}{double arrow} of~$\cS$, thus so is the morphism\index{s}{AtorightarrowB@$f\colon A\Rightarrow B$}\index{s}{piXu@$\gp^X_\bu$}
 \[
 \gr_\bx:=\Phi\bigl(\gp_\bx^X\otimes\overrightarrow{A}\bigr)\circ\chi\colon
 \Psi(B)\Rightarrow\Phi(A_{\partial\bx})\,.
 \]
Now we observe that each of the assumptions (i)--(iii) allows an application of the Buttress Lemma\index{i}{Buttress Lemma} (Lemma~\ref{L:Buttress}) to the family $\famm{\gr_\bx}{\bx\in\bX^=}$, with $S_\bx:=\Phi(A_{\partial\bx})$. We obtain an $\bX^=$-indexed diagram in $\cB^\dagger\dnw B$,
 \[
 \famm{\gc_\bx\colon C_\bx\to B,\ \gc_{\bx}^{\by}\colon\gc_\bx\to\gc_\by}
 {\bx\subseteq\by\text{ in }\bX^=}\,,
 \]
such that $\gr_\bx\circ\Psi(\gc_\bx)$ is a double arrow\index{i}{double arrow} for each $\bx\in\bX^=$. We extend this diagram to an $\bX$-indexed diagram in~$\cB\dnw B$ (no longer necessarily in~$\cB^\dagger\dnw B$), by setting
 \begin{align}
 C_\bx&:=B\text{ and }\gc_\bx:=\gc_{\bx}^{\by}=\id_B\,,
 \quad\text{for all }\bx\subseteq\by\text{ in }\bX\setminus\bX^=\,,
 \label{Eq:DefnCbxnonX=}\\
 \gc_{\bx}^{\by}&:=\gc_\bx\,,\quad\text{for all }
 (\bx,\by)\in\bX^=\times(\bX\setminus\bX^=)\text{ with }\bx\subseteq\by\notag
 \end{align}
(it is straightforward to verify that the diagram thus extended remains commutative). Observe that $\gr_\bx\circ\Psi(\gc_\bx)=\gr_\bx$ is also a double arrow\index{i}{double arrow}, for each $\bx\in\bX\setminus\bX^=$; thus $\gr_\bx\circ\Psi(\gc_\bx)$ is a double arrow\index{i}{double arrow} for each $\bx\in\bX$.

Now we apply the Armature Lemma\index{i}{Armature Lemma} (Lemma~\ref{L:Armature}) with $S_\bx:=\Psi(C_\bx)$ (which is not identical to the~$S_\bx$ used exclusively in the paragraph above), $\gf_\bx:=\Psi(\gc_\bx)$, and $\gf_{\bx}^{\by}:=\Psi(\gc_{\bx}^{\by})$. If $\bx\in\bX^=$, then $C_\bx\in\cB^\dagger$, thus, by~$(\PRES_\gl(\cB^\dagger,\Psi))$\index{s}{Pres@$(\PRES_\gl(\cB^\dagger,\Psi))$}, $S_\bx$ is weakly $\gl$-presented\index{i}{presented!weakly $\gl$-}. We thus obtain an isotone section~$\gs$ of~$\partial$ such that $\famm{\gr_{\gs(p)}\circ\gf_{\gs(p)}}{p\in P}$ is a natural transformation from $\famm{\Psi(C_{\gs(p)}),\Psi(\gc_{\gs(p)}^{\gs(q)})}{p\leq q\text{ in }P}$ to $\Phi\overrightarrow{A}$. We set $\chi_p:=\gr_{\gs(p)}\circ\gf_{\gs(p)}$, $B_p:=C_{\gs(p)}$, $\gb_p^q:=\gc_{\gs(p)}^{\gs(q)}$. All the~$\chi_p$ are double arrows\index{i}{double arrow} while $\overrightarrow{\chi}:=\famm{\chi_p}{p\in P}$ is a natural transformation from~$\Psi\overrightarrow{B}$ to~$\Phi\overrightarrow{A}$, and~$p$ maximal implies that $B_p=B$ (cf.~\eqref{Eq:DefnCbxnonX=}).

If~$\Lambda$ is projectable\index{i}{larder!projectable}, then the morphism $\chi_p\colon\Psi(B_p)\Rightarrow\Phi(A_p)$\index{s}{AtorightarrowB@$f\colon A\Rightarrow B$} has a projectability witness\index{i}{projectability witness} $(a_p,\gh_p)$, say $a_p\colon B_p\onto B'_p$\index{s}{AtoonB@$f\colon A\onto B$}, for each~$p\in P$. By Lemma~\ref{L:LiftProj}, we can find a system of morphisms ${\gb'}_p^q\colon B'_p\to B'_q$, for $p\leq q$ in~$P$, such that $\overrightarrow{B'}:=\famm{B'_p,{\gb'}_p^q}{p\leq q\text{ in }P}$ is a $P$-indexed diagram and $\Psi\overrightarrow{B'}\cong\Phi\overrightarrow{A}$.
\qed\end{proof}

\begin{remk}\label{Rk:CLL}
It may very well be the case that the part of CLL dealing with \emph{strong larders}\index{i}{larder!strong} is vacuous. This would mean that whenever a poset~$P$ has a $\gl$-lifter, then it has a $\gl$-lifter $(X,\bX)$ with~$\bX^=$ well-founded; this would imply, in particular, that~$P\setminus\Max P$ is well-founded\index{i}{poset!well-founded}. We do not know whether this holds, see Problems~\ref{Pb:gooplift} and~\ref{Pb:GenLifter} in Chapter~\ref{Ch:Discussion}.
\end{remk}

\section{Infinite combinatorics and lambda-lifters}\label{S:Pos2Lift}

For a given poset~$P$ and an infinite cardinal~$\gl$, the complexity of the definition of a $\gl$-lifter\index{i}{lifter ($\gl$-)} makes it quite unpractical to verify whether such an object exists. The main goal of the present section is to relate the existence of a lifter\index{i}{lifter ($\gl$-)} to the definition of an \ajs\index{i}{almost join-semilattice}\ (Definition~\ref{D:PJS}) on the one hand, and to known infinite combinatorial issues on the other hand. One of the main results from this section (Corollary~\ref{C:CharLift}) relates the existence of a $\gl$-lifter\index{i}{lifter ($\gl$-)} of~$P$ and an infinite combinatorial statement that we introduced in Gillibert and Wehrung~\cite{GiWe1}\index{c}{Gillibert, P.}\index{c}{Wehrung, F.}, denoted $(\gk,{<}\gl)\leadsto P$\index{s}{arr0x@$(\gk,{<}\gl)\leadsto P$}, that we shall recall here. Further results about the $(\gk,{<}\gl)\leadsto P$ statement, in particular for specific choices of the poset~$P$, can be found in~\cite{GiWe1}. A survey paper on the interaction between those matters and the present work can also be found in Wehrung~\cite{CombSurv}.

\begin{defn}\label{D:InfCombP}
For cardinals $\gk$, $\gl$ and a poset~$P$, let $(\gk,{<}\gl)\leadsto P$\index{s}{arr0x@$(\gk,{<}\gl)\leadsto P$|ii} hold if for every mapping $F\colon\Pow(\gk)\to[\gk]^{<\gl}$, there exists a one-to-one map $f\colon P\mono\nobreak\gk$\index{s}{AtomonoB@$f\colon A\mono B$} such that
 \begin{equation}\label{Eq:FreenuJPa}
 F(f``(P\dnw x))\cap f``(P\dnw y)\subseteq f``(P\dnw x)\,,
 \qquad\text{for all }x\leq y\text{ in }P\,.
 \end{equation}
\end{defn}

In many cases, $P$ is lower finite\index{i}{poset!lower finite}, and then it is of course sufficient to verify the conclusion above for $F\colon[\gk]^{<\go}\to[\gk]^{<\gl}$ \emph{isotone}.

\begin{defn}\label{D:KurInd}
The \emph{Kuratowski index}\index{i}{Kuratowski!index|ii} of a finite poset~$P$ is defined as~$0$ if~$P$ is an antichain, and the least positive integer~$n$ such that the relation\linebreak $(\gk^{+(n-1)},{<}\gk)\leadsto P$\index{s}{arr0x@$(\gk,{<}\gl)\leadsto P$} holds for each infinite cardinal~$\gk$, otherwise. We denote this integer by~$\kur(P)$\index{s}{kurP@$\kur(P)$|ii}.
\end{defn}

We prove in~\cite[Proposition~4.7]{GiWe1}\index{c}{Gillibert, P.}\index{c}{Wehrung, F.} that the order-dimension $\dim(P)$\index{s}{dimP@$\dim(P)$}\index{i}{order-dimension} of a finite poset~$P$ lies above the Kuratowski index\index{i}{Kuratowski!index} of~$P$; that is, $\kur(P)\leq\dim(P)$\index{s}{kurP@$\kur(P)$}.

\begin{defn}\label{D:KurInd0}
The \emph{restricted Kuratowski index}\index{i}{Kuratowski!restricted ${}_{-}$ index|ii} of a finite poset~$P$ is defined as zero if~$P$ is an antichain, and the least positive integer~$n$ such that the relation $(\aleph_{n-1},{<}\aleph_0)\leadsto P$\index{s}{aleph0@$\aleph_{\ga}$}\index{s}{arr0x@$(\gk,{<}\gl)\leadsto P$} holds, otherwise. We denote this integer by~$\kur_0(P)$\index{s}{kurP0@$\kur_0(P)$|ii}.
\end{defn}

In particular, by definition, $\kur_0(P)\leq\kur(P)$\index{s}{kurP@$\kur(P)$}\index{s}{kurP0@$\kur_0(P)$}. 

\begin{exple}\label{Ex:kur0<kur}
The following example shows that the inequality above may be strict, that is, $\kur_0(P)<\kur(P)$\index{s}{kurP@$\kur(P)$}\index{s}{kurP0@$\kur_0(P)$} may hold, for a certain finite lattice~$P$---at least in some generic extension of the universe.

Start with a universe~$\mathbf{V}$ of~$\mathsf{ZFC}$ satisfying the Generalized Continuum Hypothesis~$\mathsf{GCH}$. There are a finite lattice~$P$ and a generic extension of the universe in which $\kur_0(P)<\kur(P)$\index{s}{kurP@$\kur(P)$}\index{s}{kurP0@$\kur_0(P)$}. In order to see this, we set $t_0:=5$, $t_1:=7$, and for each positive integer~$n$, $t_{n+1}$ is the least positive integer such that $t_{n+1}\rightarrow(t_n,7)^5$\index{s}{arr0freepart2@$x\rightarrow(y,z)^5$|ii} (the latter notation meaning that for each mapping $f\colon[t_{n+1}]^5\to\set{0,1}$, either there exists a $t_n$-element subset~$X$ of~$t_{n+1}$ such that $f``[X]^5=\set{0}$ or there exists a $7$-element subset~$X$ of~$t_{n+1}$ such that $f``[X]^5=\set{1}$). The existence of the sequence $\famm{t_n}{n<\go}$ is ensured by Ramsey's Theorem\index{i}{Ramsey's Theorem}.

Now we order the set\index{s}{Bmr@$\rB_m({\les}r)$|ii}
 \[
 \rB_m({\les}r):=\setm{X\subseteq m}{\text{either }\card X\leq r\text{ or }X=m}
 \]
by containment, for all integers $m$, $r$ such that $1\leq r<m$.

The cardinal $\gt:=\aleph_{\go+1}$\index{s}{aleph0@$\aleph_{\ga}$} is regular, thus, by Komj\'ath and Shelah~\cite[Theorem~1]{KoSh00}\index{c}{Komj\'ath, P.}\index{c}{Shelah, S.}, there exists a generic extension~$\mathbf{V}[G]$ of~$\mathbf{V}$, with the same cardinals as~$\mathbf{V}$, in which~$\mathsf{GCH}$ holds below~$\gt$, $2^{\gt}=\gt^{+4}$, and $(\gt^{+4},4,\gt)\not\rightarrow t_4$\index{s}{arr0free@$(\gk,{<}\go,\gl)\rightarrow\gr$}. Hence, it follows from \cite[Proposition~4.11]{GiWe1}\index{c}{Gillibert, P.}\index{c}{Wehrung, F.} that $\kur\rB_{t_4}({\les}4)\geq 6$ in~$\mathbf{V}[G]$.

On the other hand, as~$\mathsf{GCH}$ holds below~$\gt$ in~$\mathbf{V}[G]$, it follows from \cite[Theorem~45.5]{EHMR}\index{c}{Erd\H{o}s, P.}\index{c}{Hajnal, A.}\index{c}{Mate@M\'at\'e, A.}\index{c}{Rado, R.} that the relation $(\aleph_n,n,\aleph_0)\rightarrow\aleph_1$\index{s}{arr0free@$(\gk,{<}\go,\gl)\rightarrow\gr$}\index{s}{aleph0@$\aleph_{\ga}$} holds in~$\mathbf{V}[G]$, for each positive integer~$n$. In particular, $(\aleph_4,4,\aleph_0)\rightarrow t_4$\index{s}{aleph0@$\aleph_{\ga}$}\index{s}{arr0free@$(\gk,{<}\go,\gl)\rightarrow\gr$}, and thus, by \cite[Proposition~4.11]{GiWe1}\index{c}{Gillibert, P.}\index{c}{Wehrung, F.}, $(\aleph_4,{<}\aleph_0)\leadsto\rB_{t_4}({\les}4)$\index{s}{arr0x@$(\gk,{<}\gl)\leadsto P$}\index{s}{aleph0@$\aleph_{\ga}$}. Therefore, $\kur_0\rB_{t_4}({\les}4)\leq 5$\index{s}{kurP0@$\kur_0(P)$}. In fact, as $\rB_{t_4}({\les}4)$ has breadth~$5$, it follows from \cite[Proposition~4.8]{GiWe1}\index{c}{Gillibert, P.}\index{c}{Wehrung, F.} that $\kur_0\rB_{t_4}({\les}4)=5$\index{s}{kurP0@$\kur_0(P)$}. In particular, $\kur_0\rB_{t_4}({\les}4)<\kur\rB_{t_4}({\les}4)$\index{s}{kurP0@$\kur_0(P)$} in~$\mathbf{V}[G]$.
\end{exple}

\begin{lem}\label{L:Part2Lift}
Let~$P$ be a lower finite\index{i}{poset!lower finite} \ajs\index{i}{almost join-semilattice}\ with zero, let~$\gl$ and~$\gk$ be infinite cardinals such that every element of~$P$ has less than~$\cf(\gl)$ upper covers and $(\gk,{<}\gl)\leadsto P$\index{s}{arr0x@$(\gk,{<}\gl)\leadsto P$}. Then there exists a norm-covering\index{i}{norm-covering}~$X$ of~$P$ such that
\begin{description}
\item[\tui] $X$ is a lower finite\index{i}{poset!lower finite} \ajs\index{i}{almost join-semilattice}\ with zero;
\item[\tuii] $\card X=\gk$;
\item[\tuiii] $X$, together with the collection of all its principal ideals\index{i}{ideal!of a poset}\index{i}{ideal!principal ${}_{-}$, of a poset}, is a $\gl$-lifter\index{i}{lifter ($\gl$-)} of~$P$.
\end{description}
\end{lem}

\begin{proof}
If $P$ is a singleton then the statement is trivial. Suppose that~$P$ is not a singleton. {}From $(\gk,{<}\gl)\leadsto P$\index{s}{arr0x@$(\gk,{<}\gl)\leadsto P$} it follows that $\gk\geq\gl$ (otherwise, consider the constant mapping on~$\Pow(\gk)$ with value~$\gk$, get a contradiction). Furthermore, from the assumption on~$P$ it follows, using an easy induction proof, that
 \begin{equation}\label{Eq:BdedHtleqn}
 \setm{x\in P}{\card(P\dnw x)\leq n}\text{ is }\cf(\gl)\text{-small, for each }n<\go\,.
 \end{equation}
In particular, as~$P$ is lower finite, $\card P\leq\cf(\gl)$ (and $\card P<\cf(\gl)$ if $\cf(\gl)>\go$).

For any set~$K$, denote by $P\seq{K}$ the set of all pairs $(a,x)$, where:
\begin{description}
\item[\tui] $a\in P$;

\item[\tuii] $x$ is a function from a subset~$X$ of~$P\dnw a$ to~$K$ such that~$a\in\Sor X$. (\emph{Of course, as~$P$ is lower finite\index{i}{poset!lower finite}, $X$ is necessarily finite}.)
\end{description}

We order $P\seq{K}$ componentwise, that is,
 \[
 (a,x)\leq(b,y)\ \Longleftrightarrow\ 
 (a\leq b\text{ and }y\text{ extends }x)\,,
 \quad\text{for all }(a,x),(b,y)\in P\seq{K}\,.
 \]
Furthermore, for each $(b,y)\in P\seq{K}$, the function~$y$ extends only finitely many functions~$x$; as $\Sor(\dom x)$ is finite for each of those~$x$, it follows that $P\seq{K}\dnw(b,y)$ is finite. Hence~$P\seq{K}$ is lower finite\index{i}{poset!lower finite}. Furthermore, $P\seq{K}$ has a smallest element, namely $(0_P,\es)$.

Let $(a_0,x_0),(a_1,x_1)\in P\seq{K}$ both below some element of~$P\seq{K}$. In particular, $x:=x_0\cup x_1$ is a function. Put $X_i:=\dom x_i$, for each $i<2$, and $X:=\dom x$. We claim that
 \begin{equation}\label{Eq:SorJJaixi}
 (a_0,x_0)\sor(a_1,x_1)=(a_0\sor a_1)\times\set{x}\,.
 \end{equation}
Let $a\in a_0\sor a_1$. As $a_i\in\Sor X_i$ for each $i<2$, it follows from Lemma~\ref{L:AssocSor}(ii) that $a\in\Sor X$, and so $(a,x)\in P\seq{K}$. Then it is obvious that $(a_i,x_i)\leq(a,x)$ for each $i<2$, and that $(a,x)$ is minimal such. This establishes the containment from the right hand side into the left hand side in~\eqref{Eq:SorJJaixi}. Conversely, let~$(b,y)$ belong to the left hand side of~\eqref{Eq:SorJJaixi}. As $a_i\leq b$ for each $i<2$, there exists $a\in a_0\sor a_1$ such that $a\leq b$. Moreover, $y$ obviously extends~$x$, so $(a_i,x_i)\leq(a,x)\leq(b,y)$ for each $i<2$, and so, by the minimality of~$(b,y)$, $(b,y)=(a,x)$ belongs to the right hand side of~\eqref{Eq:SorJJaixi}. This completes the proof of~\eqref{Eq:SorJJaixi}.

As~$P\seq{K}$ has a zero, it follows that $P\seq{K}$ is a \pjs\index{i}{pseudo join-semilattice}. Furthermore, it follows from~\eqref{Eq:SorJJaixi} that any pair $\set{(a_0,x_0),(a_1,x_1)}$ of elements of $P\seq{K}$ below some $(b,y)\in P\seq{K}$ has a join in $P\seq{K}\dnw(b,y)$, which is $(a,x)$ where~$a$ is the unique element of~$a_0\sor a_1$ below~$b$ (\emph{we use here the assumption that $P\dnw b$ is a \js}) and $x=x_0\cup x_1$. In particular, \emph{$P\seq{K}$ is an \ajs}\index{i}{almost join-semilattice}.

We put $\partial(a,x):=a$, for all $(a,x)\in P\seq{K}$. As we just verified that $P\seq{K}$ is a \pjs\index{i}{pseudo join-semilattice}, this defines a norm-covering\index{i}{norm-covering} of~$P$.

We shall consider the poset~$P\seq{\gk}$. {}From $\gk\geq\gl$ it follows that $\card P\seq{\gk}=\nobreak\gk$.

We must prove that $P\seq{\gk}$, together with the collection of all principal ideals\index{i}{ideal!of a poset}\index{i}{ideal!principal ${}_{-}$, of a poset} of~$P\seq{\gk}$, is a $\gl$-lifter\index{i}{lifter ($\gl$-)} of~$P$. Let $S\colon P\seq{\gk}\to[P\seq{\gk}]^{<\gl}$ be an isotone mapping. For each $U\in[\gk]^{<\go}$, denote by~$\Phi(U)$ the set of all $(c,z)\in P\seq{\gk}$ such that there are $(a,x)\in P\seq{U}$ and $b\in P$ such that
\begin{description}
\item[(F1)] $a\prec b$ and $c\leq b$;

\item[(F2)] $(c,z)\in S(a,x)$;

\item[(F3)] $\card(P\dnw a)=\card U$.
\end{description}
Furthermore, set $F(U):=\bigcup\famm{\rng z}{(c,z)\in\Phi(U)}$. It follows from~\eqref{Eq:BdedHtleqn} that there are less than~$\cf(\gl)$ elements $a\in P$ such that $\card(P\dnw a)=\card U$, and each of those~$a$ has less than~$\cf(\gl)$ upper covers. As~$P$ is lower finite, it follows that there are less than~$\cf(\gl)$ triples $(a,b,c)\in P^3$ satisfying both~(F1) and~(F3). As each $S(a,x)$ is $\gl$-small, it follows that~$\Phi(U)$ is $\gl$-small,  and thus~$F(U)$ is $\gl$-small.

As $(\gk,{<}\gl)\leadsto P$\index{s}{arr0x@$(\gk,{<}\gl)\leadsto P$}, there exists a one-to-one map $f\colon P\mono\gk$\index{s}{AtomonoB@$f\colon A\mono B$} such that
 \begin{equation}\label{Eq:Fnuanub}
 F(f``(P\dnw a))\cap f``(P\dnw b)
 \subseteq f``(P\dnw a)\,,\quad\text{for all }a<b\text{ in }P\,.
 \end{equation}
As~$P$ is lower finite\index{i}{poset!lower finite}, the element $\gs(a):=(a,f\res_{P\dnw a})$ belongs to~$P\seq{\gk}$, for each $a\in P$. Of course, $\gs$ is isotone.

It remains to prove the containment
 \begin{equation}\label{Eq:Ggsab}
 S(\gs(a))\dnw\gs(b)\subseteq P\seq{\gk}\dnw\gs(a)\,,
 \end{equation}
for all $a<b$ in~$P$. As~$S$ is isotone and any closed interval of~$P$ has a finite maximal chain, it suffices to establish~\eqref{Eq:Ggsab} in case $a\prec b$.

Let $(c,z)$ be an element of the left hand side of~\eqref{Eq:Ggsab}. The relation $(c,z)\in\gs(b)$ implies that $c\leq b$, $f$ extends~$z$, and $\rng z\subseteq f``(P\dnw b)$. As~$f$ is one-to-one, the set $U:=f``(P\dnw a)$ has the same cardinality as~$P\dnw a$, and $(c,z)\in\Phi(U)$; whence $\rng z\subseteq F(U)$. Using~\eqref{Eq:Fnuanub}, it follows that
 \[
 \rng z\subseteq F(f``(P\dnw a))\cap f``(P\dnw b)\subseteq f``(P\dnw a)\,,
 \]
thus, as~$f$ is one-to-one and extends~$z$, $\dom z$ is contained in~$P\dnw a$. As~$P\dnw b$ is a \js\ containing $\set{c}\cup\dom z$ with $c\in\Sor(\dom z)$, $c$ is the join of~$\dom z$ in~$P\dnw b$; whence $c\leq a$. It follows that $(c,z)\leq\gs(a)$, thus completing the proof of~\eqref{Eq:Ggsab}.
\qed\end{proof}

Recall that a poset~$T$ with zero is a \emph{tree} if~$T$ has a smallest element, that we shall denote by~$\bot$, and~$T\dnw a$ is a chain for each $a\in T$. The following result is proved in the first author's paper~\cite[Corollary~4.7]{Gill1}\index{c}{Gillibert, P.}.

\begin{prop}\label{P:ktighttrees}
Let~$\gl$ be an infinite cardinal and let~$T$ be a nontrivial lower $\cf(\gl)$-small\index{i}{poset!lower $\gl$-small} well-founded\index{i}{poset!well-founded} tree such that $\card T\leq\gl$. Then there exists a $\gl$-lifter~$(X,\bX)$\index{i}{lifter ($\gl$-)} of~$T$ such that~$X$ is a lower finite\index{i}{poset!lower finite} \ajs\index{i}{almost join-semilattice}\ with zero and $\card X=\gl$.
\end{prop}

The $\gl$-lifter\index{i}{lifter ($\gl$-)} in Proposition~\ref{P:ktighttrees} is constructed as follows. The poset~$X$ is the set of all functions from a finite subchain of~$T\setminus\set{\bot}$ to~$\gl$, endowed with the extension order. Although it is not explicitly stated in~\cite{Gill1}\index{c}{Gillibert, P.} that~$X$ is an \ajs\index{i}{almost join-semilattice}, the verification of that fact is straightforward. The norm on~$X$ is defined by $\partial x:=\bigvee\dom(x)$ for each $x\in X$. The set~$\bX$ consists of the so-called \emph{extreme ideals}\index{i}{ideal!extreme} of~$X$: by definition, a sharp\index{i}{ideal!sharp} ideal~$\bx$ of~$X$ is \emph{extreme}\index{i}{ideal!extreme|ii} if there is no sharp\index{i}{ideal!sharp} ideal~$\by$ of~$X$ such that $\bx<\by$ and $\partial\bx=\partial\by$. 

In the context of Lemma~\ref{L:Part2Lift}, the $\gl$-lifter\index{i}{lifter ($\gl$-)} $(X,\bX)$ is here chosen in such a way that~$\bX$ is the collection of all principal ideals\index{i}{ideal!of a poset}\index{i}{ideal!principal ${}_{-}$, of a poset} of~$X$. In addition, within the proof of Lemma~\ref{L:Part2Lift}, the ideals\index{i}{ideal!of a poset} in the range of~$\gs$ are extreme\index{i}{ideal!extreme} ideals, so we could have restricted~$\bX$ further to the collection of all extreme\index{i}{ideal!extreme} principal ideals\index{i}{ideal!principal ${}_{-}$, of a poset}. It is intriguing that in all known cases, including Proposition~\ref{P:ktighttrees}, every $\gl$-liftable\index{i}{liftable!$\gl$-${}_{-}$ poset} poset with zero admits a $\gl$-lifter\index{i}{lifter ($\gl$-)} of the form $(X,\bX)$ where~$X$ is a lower finite\index{i}{poset!lower finite} \ajs\index{i}{almost join-semilattice}\ with zero and~$\bX$ is the collection of all extreme\index{i}{ideal!extreme} ideals of~$X$.

The following result yields necessary conditions for $\gl$-liftability (cf. Definition~\ref{D:PJS} for \ajs s\index{i}{almost join-semilattice}). It is related to \cite[Lemme~3.3.1]{GillTh}\index{c}{Gillibert, P.}.

\begin{prop}\label{P:NoBowTie}
Let $\gl$ be an infinite cardinal and let~$(X,\bX)$ be a $\gl$-lifter\index{i}{lifter ($\gl$-)} of a poset~$P$. Put $\gk:=\card\bX$. Then the following statements hold:
\begin{description}
\item[\textup{(1)}] $P\setminus\Max P$ is lower $\cf(\gl)$-small\index{i}{poset!lower $\gl$-small}.

\item[\textup{(2)}] $P$ is an \ajs\index{i}{almost join-semilattice}.

\item[\textup{(3)}] $P$ is the disjoint union of finitely many posets with zero.

\item[\textup{(4)}] For every isotone map $f\colon[\bX^=]^{<\cf(\gl)}\to[\bX]^{<\gl}$, there exists an isotone section $\gs\colon P\into\bX$\index{s}{AtoinB@$f\colon A\into B$} of~$\partial$ such that
 \[
 (\forall a<b\text{ in }P)\bigl(f(\gs``(P\dnw a))\cap\gs``(P\dnw b)
 \subseteq\gs``(P\dnw a)\bigr)\,.
 \]

\item[\textup{(5)}] For every isotone map $f\colon[\gk]^{<\cf(\gl)}\to[\gk]^{<\gl}$, there exists a one-to-one map $\gs\colon P\mono\gk$\index{s}{AtomonoB@$f\colon A\mono B$} such that
 \[
 (\forall a<b\text{ in }P)\bigl(f(\gs``(P\dnw a))\cap\gs``(P\dnw b)
 \subseteq\gs``(P\dnw a)\bigr)\,.
 \]
\end{description}
\end{prop}

\begin{proof}
For each $\bx\in\bX$, it follows from the sharpness\index{i}{ideal!sharp} of~$\bx$ that there exists $\vgr(\bx)\in\bx$ such that $\partial\vgr(\bx)=\partial\bx$. Furthermore, if $\partial\bx$ is minimal in~$P$, then for each $x\in\bx$, it follows from the minimality of~$\partial\bx$ and the inequality $\partial x\leq\partial\bx$ that $\partial x=\partial\bx$. However, as~$X$ is a \pjs\index{i}{pseudo join-semilattice}\ and~$\bx$ is an ideal\index{i}{ideal!of a poset} of~$X$, the set $\bx\cap\Min X$ is nonempty. Therefore, \emph{we may assume that~$\vgr(\bx)$ is minimal in~$X$ whenever~$\partial\bx$ is minimal in~$P$}.
Now we set
 \begin{align*}
 S_0(\bx)&:=\setm{\vgr(\by)}{\by\in\bX\dnw\bx}\,,\\
 S(\bx)&:=(\Min X)\cup\bigcup\famm{u\sor v}{u,v\in S_0(\bx)}\,,
 \end{align*}
for each $\bx\in\bX^=$. It follows from the assumptions defining $\gl$-lifters\index{i}{lifter ($\gl$-)} (Definition~\ref{D:Lifter}) that~$S(\bx)$ is a $\cf(\gl)$-small subset of~$X$, for every~$\bx\in\bX^=$. As~$(X,\bX)$ is a $\gl$-lifter\index{i}{lifter ($\gl$-)} of~$P$, there exists an isotone section~$\gs\colon P\into\bX$\index{s}{AtoinB@$f\colon A\into B$} of~$\partial$ such that $S(\gs(a))\cap\gs(b)\subseteq\gs(a)$ for all $a<b$ in~$P$. As~$\gs$ restricts to an order-embedding from~$P\setminus\Max P$ into~$\bX^=$ and the latter is lower $\cf(\gl)$-small\index{i}{poset!lower $\gl$-small}, (1) follows.

(2). Let~$A$ be a finite subset of~$P$, and put $\da:=\vgr\gs(a)$, for each $a\in A$. In particular, $\partial\da=a$. Set $\dA:=\setm{\da}{a\in A}$. As~$X$ is a \pjs\index{i}{pseudo join-semilattice}, $\Sor\dA$ is a finite subset of~$X$. Hence,
$U:=\setm{\partial\du}{\du\in\Sor\dA}$ is a finite subset of~$P$.

We claim that $P\Upw A=P\upw U$. For the containment from the right into the left, we observe that each $\du\in\Sor\dA$ lies above all elements of~$\dA$, thus~$\partial\du$ lies above all elements of~$A$ (use the equation $\partial\da=a$ that holds for each $a\in A$), and thus $U\subseteq P\Upw A$. Conversely, let $x\in P\Upw A$. For each $a\in A$, $\da\in\gs(a)\subseteq\gs(x)$, thus, as~$\gs(x)$ is an ideal\index{i}{ideal!of a poset} of the \pjs\index{i}{pseudo join-semilattice}~$X$ and~$\dA$ is finite, there exists $\du\in\Sor\dA$ such that $\du\in\gs(x)$. Hence, $\partial\du\leq\partial\gs(x)=x$. This completes the proof of our claim. In particular, $P\Upw A$ is a finitely generated\index{i}{finitely generated!upper subset} upper subset of~$P$. This completes the proof of the statement that~$P$ is a \pjs\index{i}{pseudo join-semilattice}.

Now let $a,b,e\in P$ such that $a,b\leq e$, we must prove that $\set{a,b}$ has a join in~$P\dnw e$. Put again $\da:=\vgr\gs(a)$ and $\db:=\vgr\gs(b)$. As~$\gs(e)$ is an ideal\index{i}{ideal!of a poset} of~$X$, $\set{\da,\db}\subseteq\gs(e)$, and~$X$ is a \pjs\index{i}{pseudo join-semilattice}, there exists an element~$\dx$ in $(\da\sor\db)\cap\gs(e)$. {}From $\partial\da=a$, $\partial\db=b$, and $\partial\gs(e)=e$ it follows that $a,b\leq\partial\dx\leq e$.

Let $c\in P\dnw e$ such that $a,b\leq c$. {}From $\gs(a)\in\bX\dnw\gs(c)$ it follows that $\da=\vgr\gs(a)\in S_0(\gs(c))$. Similarly, $\db\in S_0(\gs(c))$. As $\dx$ belongs to~$\da\sor\db$, it belongs to~$S(\gs(c))$. On the other hand, $\dx\in\gs(e)$, so~$\dx$ belongs to $S(\gs(c))\cap\gs(e)$, thus to~$\gs(c)$, and so $\partial\dx\leq c$. We have proved that~$\partial\dx$ is the join of $\set{a,b}$ in~$P\dnw e$.

(3). For all $a\leq b$ in~$P$, $(\Min X)\cap\gs(b)\subseteq S(\gs(a))\cap\gs(b)\subseteq\gs(a)$, thus, as $\gs(a)\subseteq\gs(b)$, we obtain $(\Min X)\cap\gs(a)=(\Min X)\cap\gs(b)$. Now let~$a,b\in P$ with~$a$ minimal in~$P$ and $P\Upw\set{a,b}\neq\es$; pick $c\in P\Upw\set{a,b}$. By the previous remark, we obtain that $(\Min X)\cap\gs(a)=(\Min X)\cap\gs(c)=(\Min X)\cap\gs(b)$. As $\partial\gs(a)=a$ is minimal in~$P$ and by the choice of the map~$\vgr$, the element $\da:=\vgr\gs(a)$ is minimal in~$X$. Hence~$\da$ belongs to
 \[
 (\Min X)\cap\gs(a)=(\Min X)\cap\gs(b)\subseteq\gs(b)\,,
 \]
and thus $a=\partial\da\leq\partial\gs(b)=b$. As, by~(1) above, $\Min P$ is finite and~$P$ is the union of all its subsets~$P\upw a$, for $a\in\Min P$, it follows that this union is disjoint.

(4). As~$f$ is isotone, the assignment $F(\bx):=f(\bX\dnw\bx)$ defines an isotone map $F\colon\bX^=\to[\bX]^{<\gl}$. Now set $S(\bx):=\vgr``F(\bx)$, for each~$\bx\in\bX^=$. Hence $S\colon\bX^=\to\nobreak[X]^{<\gl}$ is isotone. As~$(X,\bX)$ is a $\gl$-lifter\index{i}{lifter ($\gl$-)} of~$P$, its norm has a free isotone section $\gs\colon P\into\bX$\index{s}{AtoinB@$f\colon A\into B$} with respect to~$S$. Let $a<b$ in~$P$, we shall prove that $f(\gs``(P\dnw a))\cap\gs``(P\dnw b)\subseteq\gs``(P\dnw a)$. Let $c\in P\dnw b$ such that $\gs(c)\in f(\gs``(P\dnw a))$, we must prove that $c\leq a$. As $\gs``(P\dnw a)$ is contained in $\bX\dnw\gs(a)$ and~$f$ is isotone, we obtain
 \[
 f(\gs``(P\dnw a))\subseteq f(\bX\dnw\gs(a))=F(\gs(a))\,,
 \]
hence $\gs(c)\in F(\gs(a))$, and hence $\vgr\gs(c)\in S(\gs(a))$. As $\vgr\gs(c)\in\gs(c)\subseteq\gs(b)$, it follows that~$\vgr\gs(c)$ belongs to $S(\gs(a))\cap\gs(b)$, thus to~$\gs(a)$, and so
$c=\partial\gs(c)=\partial\vgr\gs(c)\leq\partial\gs(a)=a$.

In order to obtain~(5) from~(4), it suffices to prove that for every isotone map $f\colon[\bX]^{<\cf(\gl)}\to[\bX]^{<\gl}$, there exists a one-to-one map $\gs\colon P\mono\bX$\index{s}{AtomonoB@$f\colon A\mono B$} such that
 \begin{equation}\label{Eq:sfreewrtf}
 (\forall a<b\text{ in }P)\bigl(f(\gs``(P\dnw a))\cap\gs``(P\dnw b)
 \subseteq\gs``(P\dnw a)\bigr)\,.
 \end{equation}
By applying~(4) to the restriction~$g$ of~$f$ to $[\bX^=]^{<\cf(\gl)}$, we obtain an isotone section~$\gs\colon P\into\bX$\index{s}{AtoinB@$f\colon A\into B$} of~$\partial$ such that
 \[
 (\forall a<b\text{ in }P)\bigl(g(\gs``(P\dnw a))\cap\gs``(P\dnw b)
 \subseteq\gs``(P\dnw a)\bigr)\,.
 \]
However, for $a<b$ in~$P$, $\gs(a)\in\bX^=$, thus $\gs``(P\dnw a)\subseteq\bX^=$, and thus we get $f(\gs``(P\dnw a))=g(\gs``(P\dnw a))$, so~\eqref{Eq:sfreewrtf} is satisfied as well.
\qed\end{proof}

In particular, we obtain the following characterization of $\gl$-liftability for lower finite\index{i}{poset!lower finite} and sufficiently small posets. This characterization is related to \cite[Th\'e\-o\-r\`e\-me~3.3.2]{GillTh}\index{c}{Gillibert, P.}.

\begin{cor}\label{C:CharLift}
Let~$\gl$ and~$\gk$ be infinite cardinals and let~$P$ be a lower finite\index{i}{poset!lower finite} poset in which every element has less than~$\cf(\gl)$ upper covers. Then the following are equivalent:
\begin{description}
\item[\tui] $P$ has a $\gl$-lifter\index{i}{lifter ($\gl$-)} $(X,\bX)$ such that~$\bX$ consists of all principal ideals\index{i}{ideal!of a poset}\index{i}{ideal!principal ${}_{-}$, of a poset} of~$X$ and $\card X=\gk$, while~$X$ is a lower finite\index{i}{poset!lower finite} \ajs\index{i}{almost join-semilattice}.

\item[\tuii] $P$ has a $\gl$-lifter\index{i}{lifter ($\gl$-)} $(X,\bX)$ such that $\card\bX=\gk$.

\item[\tuiii] $P$ is the disjoint union of finitely many \ajs s\index{i}{almost join-semilattice} with zero, and $(\gk,{<}\gl)\leadsto\nobreak P$\index{s}{arr0x@$(\gk,{<}\gl)\leadsto P$}.
\end{description}
\end{cor}

\begin{proof}
(i)$\Rightarrow$(ii) is trivial.

(ii)$\Rightarrow$(iii). The statement that $P$ is the disjoint union of finitely many \ajs s\index{i}{almost join-semilattice} with zero follows from Proposition~\ref{P:NoBowTie}(2,3).

Now let $F\colon[\gk]^{<\go}\to[\gk]^{<\gl}$. We can associate to~$F$ an isotone mapping\linebreak $G\colon[\gk]^{<\cf(\gl)}\to[\gk]^{<\gl}$ by setting
 \begin{equation}\label{Eq:FfromS}
 G(X):=\bigcup\famm{F(Y)}{Y\in[X]^{<\go}}
 \qquad\text{for each }X\in[\gk]^{<\cf(\gl)}\,.
 \end{equation}
Observe that $F(X)\subseteq G(X)$ for every $X\in[\gk]^{<\go}$.
Now apply Proposition~\ref{P:NoBowTie}(5) to~$G$.

(iii)$\Rightarrow$(i). By assumption, we can write~$P$ as a disjoint union $P=\bigcup\famm{P_i}{i<n}$, where~$n$ is a positive integer and each~$P_i$ is an \ajs\index{i}{almost join-semilattice}\ with zero. For each $i<n$, as $P_i$ is a subset of~$P$ and $(\gk,{<}\gl)\leadsto P$\index{s}{arr0x@$(\gk,{<}\gl)\leadsto P$}, we obtain that $(\gk,{<}\gl)\leadsto P_i$\index{s}{arr0x@$(\gk,{<}\gl)\leadsto P$}, and thus, by Lemma~\ref{L:Part2Lift}, $P_i$ has a $\gl$-lifter\index{i}{lifter ($\gl$-)} $(X_i,\bX_i)$ such that~$\bX_i$ is the set of all principal ideals\index{i}{ideal!of a poset}\index{i}{ideal!principal ${}_{-}$, of a poset} of~$X_i$ and $\card X_i=\gk$, while each~$X_i$ is a lower finite\index{i}{poset!lower finite} \ajs\index{i}{almost join-semilattice}. We may assume that the~$X_i$s are pairwise disjoint. Then form the disjoint union (as posets) $X:=\bigcup\famm{X_i}{i<n}$ and set $\bX:=\bigcup\famm{\bX_i}{i<n}$, which is also the set of all principal ideals\index{i}{ideal!of a poset}\index{i}{ideal!principal ${}_{-}$, of a poset} of~$X$. Furthermore, let~$\partial_X$ extend each~$\partial_{X_i}$. Let $S\colon\bX^=\to[X]^{<\gl}$ and define
 \[
 S_i(\bx):=S(\bx)\cap X_i\,,\quad\text{for all }i<n\text{ and all }\bx\in\bX_i^=\,.
 \]
As $(X_i,\bX_i)$ is a $\gl$-lifter\index{i}{lifter ($\gl$-)} of~$P_i$, it has an isotone section~$\gs_i$ which is free with respect to~$S_i$\index{i}{freemapp@free map (wrt. set-mapping, poset)}. It follows easily that the union of the~$\gs_i$ is an isotone section of the lifter\index{i}{lifter ($\gl$-)} $(X,\bX)$ which is free with respect to~$S$\index{i}{freemapp@free map (wrt. set-mapping, poset)}. Therefore, $(X,\bX)$ is a $\gl$-lifter\index{i}{lifter ($\gl$-)} of~$P$.
\qed\end{proof}

\begin{cor}\label{C:SubsLift}
Let~$P$ and~$Q$ be lower finite\index{i}{poset!lower finite}, disjoint unions of finitely many \ajs s\index{i}{almost join-semilattice} with zero and let~$\gl$ and~$\gk$ be infinite cardinals such that every element of~$P$ has less than $\cf(\gl)$ upper covers. If~$P$ embeds into~$Q$ as a poset and~$Q$ has a $\gl$-lifter\index{i}{lifter ($\gl$-)} $(X,\bX)$ such that $\card\bX=\gk$, then so does~$P$.
\end{cor}

\begin{proof}
Observe that the proof of (ii)$\Rightarrow$(iii) in Corollary~\ref{C:CharLift} does not use the assumption that $\card P<\cf(\gl)$. It follows that the relation $(\gk,{<}\gl)\leadsto Q$\index{s}{arr0x@$(\gk,{<}\gl)\leadsto P$} holds. As~$P$ embeds into~$Q$, the relation $(\gk,{<}\gl)\leadsto P$\index{s}{arr0x@$(\gk,{<}\gl)\leadsto P$} holds as well. Therefore, by applying again Corollary~\ref{C:CharLift}, the desired conclusion follows.
\qed\end{proof}

Observe also the following characterization of $\gl$-liftable\index{i}{liftable!$\gl$-${}_{-}$ poset} finite posets.

\begin{cor}\label{C:LiftFinPos}
Let $P$ be a finite poset and let $\gl$ be an infinite cardinal. Then~$P$ has a $\gl$-lifter\index{i}{lifter ($\gl$-)} if{f} $P$ is a finite disjoint union of \ajs s\index{i}{almost join-semilattice} with zero. Furthermore, if~$P$ is not an antichain, then the cardinality of such a lifter\index{i}{lifter ($\gl$-)} can be taken equal to $\gl^{+(n-1)}$, where $n:=\kur(P)$\index{s}{kurP@$\kur(P)$}.
\end{cor}

\begin{proof}
If $P$ has a $\gl$-lifter\index{i}{lifter ($\gl$-)}, then it follows from Corollary~\ref{C:CharLift} that $P$ is a disjoint union of \ajs s\index{i}{almost join-semilattice} with zero.

Conversely, suppose that~$P$ is not an antichain. Set $n:=\kur(P)$\index{s}{kurP@$\kur(P)$} and $\gk:=\gl^{+(n-1)}$. The relation $(\gk,{<}\gl)\leadsto P$\index{s}{arr0x@$(\gk,{<}\gl)\leadsto P$} holds by definition. The conclusion follows from Corollary~\ref{C:CharLift}.
\qed\end{proof}

\section{Lifters, retracts, and pseudo-retracts}\label{S:LiftRetr}

It is obvious that if~$P$ is a subposet of a lower finite poset~$Q$, then $(\gk,{<}\gl)\leadsto Q$\index{s}{arr0x@$(\gk,{<}\gl)\leadsto P$} implies that $(\gk,{<}\gl)\leadsto P$\index{s}{arr0x@$(\gk,{<}\gl)\leadsto P$} (cf. \cite[Lemma~3.2]{GiWe1}). We do not know whether a similar statement can be proved for lifters\index{i}{lifter ($\gl$-)}, the infinite case included: that is, whether if~$P$ is an \ajs\index{i}{almost join-semilattice}\ and~$Q$ has a $\gl$-lifter\index{i}{lifter ($\gl$-)}, then so does~$P$. Corollary~\ref{C:SubsLift} shows additional assumptions on~$P$ and~$Q$ under which this holds. The following result shows another type of assumption under which this can be done.

\begin{lem}\label{L:RetrLift}
Let $\gl$ be an infinite cardinal and let~$P$ be a retract of a poset~$Q$. If~$Q$ has a $\gl$-lifter\index{i}{lifter ($\gl$-)} $(Y,\bY)$, then~$P$ has a $\gl$-lifter\index{i}{lifter ($\gl$-)} $(X,\bX)$ such that~$X$ and~$Y$ have the same underlying set and~$\bX^=$ is a subset of~$\bY^=$.
\end{lem}

\begin{proof}
Denote by $\gr$ a retraction of~$Q$ onto~$P$ and set
 \[
 R:=\setm{q\in\Max Q}{\gr(q)\notin\Max P}\,.
 \]
We endow the underlying set of~$Y$ with the norm~$\gr\partial$; we obtain a norm-covering\index{i}{norm-covering}~$X$ of~$P$. We set
 \[
 \bX:=\setm{\bx\in\bY}{\partial\bx\notin R}\,.
 \]
Every element of~$\bX$ is a sharp\index{i}{ideal!sharp} ideal of~$Y$ with respect to~$\partial$, thus \emph{a fortiori} with respect to~$\gr\partial$; that is, $\bX\subseteq\Ids X$\index{s}{Ids@$\Ids X$, $X$ norm-covering}.

We claim that $\bX^=$ (defined with respect to the norm~$\gr\partial$) is a subset of~$\bY^=$ (defined with respect to the norm~$\partial$). Let $\bx\in\bX^=$ and suppose that $\bx\notin\bY^=$, that is, $\partial\bx\in\Max Q$. {}From $\bx\in\bX$ it follows that $\partial\bx\notin R$, thus $\gr\partial\bx\in\Max P$, which contradicts the assumption $\bx\in\bX^=$ and thus proves our claim. In particular, as~$\bY^=$ is lower $\cf(\gl)$-small\index{i}{poset!lower $\gl$-small}, so is~$\bX^=$.

We can extend any mapping $S\colon\bX^=\to[X]^{<\gl}$ to a mapping $\ol{S}\colon\bY^=\to\nobreak[Y]^{<\gl}$ by setting $\ol{S}(\by):=\es$ for each $\by\in\bY^=\setminus\bX^=$. As $(Y,\bY)$ is a $\gl$-lifter\index{i}{lifter ($\gl$-)} of~$Q$, it has an isotone section~$\gs$ which is free with respect to~$\ol{S}$\index{i}{freemapp@free map (wrt. set-mapping, poset)}. It follows easily that~$\gs\res_P$ is an isotone section of~$\gr\partial$ which is free with respect to~$S$\index{i}{freemapp@free map (wrt. set-mapping, poset)}.
\qed\end{proof}

We do not know whether every liftable\index{i}{liftable!$\gl$-${}_{-}$ poset} poset is well-founded\index{i}{poset!well-founded} (cf. Problem~\ref{Pb:gooplift} in Chapter~\ref{Ch:Discussion}). However, the following sequence of results sheds some light on the gray zone where this could occur. We use standard notation for the addition of ordinals, so, in particular, $\go+1=\set{0,1,2,\dots}\cup\set{\go}$.

\begin{lem}\label{L:NonWFgo+1}
Let~$\gl$ be an infinite cardinal and let~$P$ a non well-founded\index{i}{poset!well-founded} poset. If~$P$ is $\gl$-liftable\index{i}{liftable!$\gl$-${}_{-}$ poset}, then so is $(\go+1)^{\op}$\index{s}{omega1op@$(\omega+1)^{\op}$}.
\end{lem}

\begin{proof}
It follows from Proposition~\ref{P:NoBowTie}(3) that we can write~$P$ as a disjoint union $P=\bigcup\famm{P_i}{i<n}$, where~$n$ is a positive integer and each~$P_i$ is an \ajs\index{i}{almost join-semilattice}\ with zero. As each~$P_i$ is a retract of~$P$ (send all the elements of~$P\setminus P_i$ to the zero element of~$P_i$) and one of the~$P_i$ is not well-founded\index{i}{poset!well-founded}, it follows from Lemma~\ref{L:RetrLift} that we may replace~$P$ by that~$P_i$ and thus assume that~$P$ is an \ajs\index{i}{almost join-semilattice}\ with zero.

As~$P$ is not well-founded\index{i}{poset!well-founded}, there exists a strictly decreasing $(\go+1)$-sequence $\famm{p_n}{0\leq n\leq\go}$ in~$P$ such that~$p_\go$ is the zero element of~$P$. Denote by~$\gr(x)$ the least $n\leq\go$ such that $p_n\leq x$, for each $x\in P$. As the assignment $(n\mapsto p_n)$ defines an order-embedding from $(\go+1)^{\op}$\index{s}{omega1op@$(\omega+1)^{\op}$} into~$P$, with retraction~$\gr$, the conclusion follows from Lemma~\ref{L:RetrLift}.
\qed\end{proof}

\begin{prop}\label{P:LiftnotWFrestr}
Let~$\gl$ be an infinite cardinal such that\index{s}{aleph0@$\aleph_{\ga}$}
 \begin{equation}\label{Eq:CardHypgl}
 \gm^{\aleph_0}<\gl\quad\text{for each }\gm<\cf(\gl)\,.
 \end{equation}
If a poset~$P$ is $\gl$-liftable\index{i}{liftable!$\gl$-${}_{-}$ poset}, then~$P$ is well-founded\index{i}{poset!well-founded}.
\end{prop}

\begin{proof}
Suppose that~$P$ is not well-founded\index{i}{poset!well-founded}.
By Lemma~\ref{L:NonWFgo+1}, we may assume that $P=(\go+1)^{\op}$\index{s}{omega1op@$(\omega+1)^{\op}$}. In particular, $P\setminus\Max P$ is not lower finite\index{i}{poset!lower finite}, thus, by Proposition~\ref{P:NoBowTie}(1), $\aleph_1\leq\cf(\gl)$\index{s}{aleph0@$\aleph_{\ga}$}.

\begin{claim}
For every \pup{\emph{not necessarily isotone}} map $f\colon[\bX^=]^{\les\aleph_0}\to[\bX]^{\les\aleph_0}$\index{s}{aleph0@$\aleph_{\ga}$}, there exists an isotone section $\gs\colon P\into\bX$\index{s}{AtoinB@$f\colon A\into B$} of~$\partial$ such that
 \[
 (\forall a<b\text{ in }P)\bigl(f(\gs``(P\dnw a))\cap\gs``(P\dnw b)
 \subseteq\gs``(P\dnw a)\bigr)\,.
 \]
\end{claim}

\begin{proof}
A slight modification of the proof of Proposition~\ref{P:NoBowTie}(4). We set\index{s}{aleph0@$\aleph_{\ga}$}
 \begin{equation}\label{Eq:DefnF(bx)}
 F(\bx):=\bigcup\famm{f(\bZ)}{\bZ\in[\bX\dnw\bx]^{\les\aleph_0}}\,,\qquad
 \text{for each }\bx\in\bX^=\,.
 \end{equation}
Let $\bx\in\bX^=$. The cardinal $\gm:=\card(\bX\dnw\bx)$ is smaller than~$\cf(\gl)$. As~$f(\bZ)$ is at most countable for each $\bZ\in[\bX\dnw\bx]^{\les\aleph_0}$\index{s}{aleph0@$\aleph_{\ga}$}, it follows that $\card F(\bx)\leq\gm^{\aleph_0}$\index{s}{aleph0@$\aleph_{\ga}$}, thus, by the assumption~\eqref{Eq:CardHypgl}, $F(\bx)$ is $\gl$-small. Hence $F\colon\bX^=\to[\bX]^{<\gl}$.

As in the proof of Proposition~\ref{P:NoBowTie}, for each $\bx\in\bX$, there exists $\vgr(\bx)\in\nobreak\bx$ such that $\partial\vgr(\bx)=\partial\bx$. We set $S(\bx):=\vgr``F(\bx)$, for each~$\bx\in\bX^=$; so $S\colon\bX^=\to[X]^{<\gl}$. As~$(X,\bX)$ is a $\gl$-lifter\index{i}{lifter ($\gl$-)} of~$P$, its norm has a free isotone section $\gs\colon P\into\bX$\index{s}{AtoinB@$f\colon A\into B$} with respect to~$S$. Let $a<b$ in~$P$, we shall prove the containment $f(\gs``(P\dnw a))\cap\gs``(P\dnw b)\subseteq\gs``(P\dnw a)$. Let $c\in P\dnw b$ such that $\gs(c)\in f(\gs``(P\dnw a))$, we must prove that $c\leq a$. As $\gs``(P\dnw a)$ is an at most countable subset of $\bX\dnw\gs(a)$, it follows from~\eqref{Eq:DefnF(bx)} that
 \[
 f(\gs``(P\dnw a))\subseteq F(\gs(a))\,,
 \]
hence $\gs(c)\in F(\gs(a))$, and hence $\vgr\gs(c)\in S(\gs(a))$. As $\vgr\gs(c)\in\gs(c)\subseteq\gs(b)$, it follows that~$\vgr\gs(c)$ belongs to $S(\gs(a))\cap\gs(b)$, thus to~$\gs(a)$, and so
$c=\partial\gs(c)=\partial\vgr\gs(c)\leq\partial\gs(a)=a$.
\qed\ Claim\end{proof}

We shall apply the Claim above to a specific~$f$, as in the proof of \cite[Proposition~3.5]{GiWe1}\index{c}{Gillibert, P.}\index{c}{Wehrung, F.}.

Let $U\equiv_{\mathrm{fin}}V$ hold if the symmetric difference~$U\mathbin{\triangle}V$ is finite, for all subsets~$U$ and~$V$ of~$\bX^=$. Let~$\Delta$ be a set that meets the $\equiv_{\mathrm{fin}}$-equivalence class~$[U]_{\equiv_{\mathrm{fin}}}$ of~$U$ in exactly one element, for each at most countable subset~$U$ of~$\bX^=$, and denote by~$f(U)$ the unique element of~$[U]_{\equiv_{\mathrm{fin}}}\cap\Delta$. By our Claim, there is a one-to-one map $\gs\colon P\mono\bX^=$\index{s}{AtomonoB@$f\colon A\mono B$} such that
 \begin{equation}\label{Eq:gsfreewrtf}
 f(\gs``(P\dnw p))\cap\gs``(P\dnw q)\subseteq\gs``(P\dnw p)\quad
 \text{for all }p<q\text{ in }P\,.
 \end{equation}
We set
 \[
 U_n:=\gs``(P\dnw n)=\set{\gs(n),\gs(n+1),\dots}\cup\set{\gs(\go)}\,,
 \quad\text{for each }n<\go\,.
 \]
By the definition of~$f$, the difference $U_0\setminus f(U_0)$ is finite, thus there exists a natural number~$m$ such that
 \[
 U_0\setminus f(U_0)\subseteq\set{\gs(0),\gs(1),\dots,\gs(m-1)}
 \cup\set{\gs(\go)}\,.
 \]
In particular, $\gs(m)\notin U_0\setminus f(U_0)$, but $\gs(m)\in U_0$, thus $\gs(m)\in f(U_0)$. Now from $U_0\equiv_{\mathrm{fin}}U_{m+1}$ it follows that $f(U_0)=f(U_{m+1})$, thus $\gs(m)\in f(U_{m+1})$. As $\gs(m)\in U_0$, it follows from~\eqref{Eq:gsfreewrtf} that $\gs(m)\in U_{m+1}$, a contradiction.
\qed\end{proof}

\begin{cor}\label{C:LiftnotWFrestr}
If a poset~$P$ is $(2^{\aleph_0})^+$-liftable\index{i}{liftable!$\gl$-${}_{-}$ poset}\index{s}{aleph0@$\aleph_{\ga}$}, then it is well-founded\index{i}{poset!well-founded}.
\end{cor}

Nevertheless, we shall see in Section~\ref{S:MorePosets} that a variant of CLL\index{i}{Condensate Lifting Lemma (CLL)} can be formulated for posets, such as $(\go+1)^{\op}$\index{s}{omega1op@$(\omega+1)^{\op}$}, that are not $\gl$-liftable\index{i}{liftable!$\gl$-${}_{-}$ poset} for certain infinite cardinals~$\gl$. The key concept is the one of a \emph{pseudo-retract}.

\begin{defn}\label{D:PseudoRetr}
A \emph{pseudo-retraction pair}\index{i}{pseudo-retraction|ii} for a pair $(P,Q)$ of posets is a pair $(e,f)$ of isotone maps $e\colon P\to\Id Q$\index{s}{IdP@$\Id P$, $P$ poset} and $f\colon Q\to P$ such that\linebreak $P\dnw f``(e(p))=P\dnw p$ for each~$p\in P$. We say that~$P$ is a \emph{pseudo-retract} of~$Q$.
\end{defn}

\begin{lem}\label{L:PseudoRetr}
Every \ajs\index{i}{almost join-semilattice}\ $P$ is a pseudo-retract of some lower finite\index{i}{poset!lower finite} \ajs\index{i}{almost join-semilattice}~$Q$ of cardinality $\card P$.
\end{lem}

\begin{proof}
If~$P$ is finite take $Q:=P$. Now suppose that~$P$ is infinite, and define
 \[
 \cI:=\setm{X\subseteq P}{X\text{ is a finite $\sor$-closed subset of }P}\,,
 \]
ordered by containment, and then $Q:=\bigcup\famm{\set{X}\times X}{X\in\cI}$, ordered componentwise. It is obvious that~$Q$ is lower finite\index{i}{poset!lower finite} with the same cardinality as~$P$. Now let~$U:=\set{(X_0,a_0),\dots,(X_{n-1},a_{n-1})}$ (with $n<\go$) be a finite subset of~$Q$. As~$P$ is a \pjs\index{i}{pseudo join-semilattice}, $A:=\sor_{i<n}a_i$ is a finite subset of~$P$. Furthermore, as~$P$ is an \ajs\index{i}{almost join-semilattice}, it follows from Corollary~\ref{C:TrSupp} that $X:=\bigl(\bigcup\famm{X_i}{i<n}\bigr)^{\sor}$ is a finite subset of~$P$. As all~$a_i$ belong to~$X$ and~$X$ is $\sor$-closed, $A$ is contained in~$X$. Thus $V:=\set{X}\times A$ is a finite subset of~$Q\Upw U$. Conversely, let $(Y,b)\in Q\Upw U$, so $(X_i,a_i)\leq(Y,b)$ for each $i<n$. As $X_i\subseteq Y$ for each $i<n$ and~$Y$ is $\sor$-closed, $Y$ contains~$X$. As $a_i\leq b$ for each $i<n$ and~$P$ is a \pjs\index{i}{pseudo join-semilattice}, there exists $a\in A$ such that $a\leq b$. So $(X,a)\leq(Y,b)$ with $(X,a)\in V$. We have proved that $Q\Upw U=Q\upw V$. Therefore, $Q$ is a \pjs\index{i}{pseudo join-semilattice}. Now let $(X_0,a_0),(X_1,a_1),(Y,b)\in Q$ such that $(X_i,a_i)\leq(Y,b)$ for each~$i<2$. A similar argument to the one used above shows that if $X:=(X_0\cup X_1)^{\sor}$ and~$a$ denotes the join of $\set{a_0,a_1}$ in~$P\dnw b$, then~$(X,a)$ is the join of $\set{(X_0,a_0),(X_1,a_1)}$ in~$Q\dnw(Y,b)$. Therefore, $Q$ is an \ajs\index{i}{almost join-semilattice}.

The maps $e\colon P\to\Id Q$\index{s}{IdP@$\Id P$, $P$ poset} and $f\colon Q\to P$ defined by
 \begin{align*}
 e(p)&:=\setm{(X,a)\in Q}{a\leq p}&&(\text{for each }p\in P)\\
 f(X,a)&:=a&&(\text{for each }(X,a)\in Q) 
 \end{align*}
form a pseudo-retraction\index{i}{pseudo-retraction} for $(P,Q)$ with the required cardinality property.
\qed\end{proof}

Recall \cite[Section~45]{EHMR}\index{c}{Erd\H{o}s, P.}\index{c}{Hajnal, A.}\index{c}{Mate@M\'at\'e, A.}\index{c}{Rado, R.} that for cardinals $\gk$, $\gl$, $\gr$, the statement $(\gk,{<}\go,\gl)\rightarrow\nobreak\gr$\index{s}{arr0free@$(\gk,{<}\go,\gl)\rightarrow\gr$|ii} holds if every map $F\colon[\gk]^{<\go}\to[\gk]^{<\gl}$ has a $\gr$-element free set\index{i}{freeset@free set (wrt. set-mapping)|ii}, that is, $H\in[\gk]^{\gr}$ such that $F(X)\cap H\subseteq X$ for each $X\in[H]^{<\go}$. In case $\gl\geq\aleph_1$\index{s}{aleph0@$\aleph_{\ga}$} and~$\gr\geq\aleph_0$\index{s}{aleph0@$\aleph_{\ga}$}, the existence of~$\gk$ such that $(\gk,{<}\go,\gl)\rightarrow\gr$\index{s}{arr0free@$(\gk,{<}\go,\gl)\rightarrow\gr$} is a large cardinal axiom, that entails the existence of~$0^{\#}$ in case~$\gr\geq\aleph_1$\index{s}{aleph0@$\aleph_{\ga}$} (cf. Devlin and Paris~\cite{DePa}\index{c}{Devlin, K.\,J.}\index{c}{Paris, J.\,B.}, and also Koepke \cite{Koep84,Koep89}\index{c}{Koepke, P.} for further related consistency strength results). The relation $(\gk,{<}\go,\gl)\rightarrow\gr$\index{s}{arr0free@$(\gk,{<}\go,\gl)\rightarrow\gr$} follows from the infinite partition relation $\gk\rightarrow(\gq)^{<\go}_2$\index{s}{arr0freepart@$\gk\rightarrow(\gq)^{<\go}_2$|ii} (\emph{existence of the $\gq^{\mathrm{th}}$ Erd\H{o}s cardinal})\index{i}{Erd\H{o}s cardinal} where $\gq:=\max\set{\gr,\gl^+}$, see \cite[Theorem~45.2]{EHMR}\index{c}{Erd\H{o}s, P.}\index{c}{Hajnal, A.}\index{c}{Mate@M\'at\'e, A.}\index{c}{Rado, R.} and the discussion preceding it.

We prove in \cite[Proposition~3.4]{GiWe1}\index{c}{Gillibert, P.}\index{c}{Wehrung, F.} that the statements $(\gk,{<}\gl)\leadsto([\gr]^{<\go},\subseteq)$\index{s}{arr0x@$(\gk,{<}\gl)\leadsto P$} and $(\gk,{<}\go,\gl)\rightarrow\gr$\index{s}{arr0free@$(\gk,{<}\go,\gl)\rightarrow\gr$} are equivalent, for all cardinals $\gk$, $\gl$, and $\gr$. In particular, suppose that this holds for $\gl=(2^{\aleph_0})^+$\index{s}{aleph0@$\aleph_{\ga}$} and $\gr=\aleph_0$\index{s}{aleph0@$\aleph_{\ga}$}, that is,\index{s}{aleph0@$\aleph_{\ga}$}
 \[
 \bigl(\gk,{<}\go,(2^{\aleph_0})^+\bigr)\to\aleph_0\,.
 \]
(This is a large cardinal axiom, that follows from the existence of the Erd\H os cardinal\index{i}{Erd\H{o}s cardinal}\index{c}{Erdos@Erd\H os, P.} of index $(2^{\aleph_0})^{++}$\index{s}{aleph0@$\aleph_{\ga}$}, and thus in particular from the existence of a Ramsey cardinal\index{i}{Ramsey cardinal}.) By the above comment, the relation\index{s}{arr0x@$(\gk,{<}\gl)\leadsto P$}\index{s}{aleph0@$\aleph_{\ga}$}
 \begin{equation}\label{Eq:LargeKurat}
 \bigl(\gk,{<}(2^{\aleph_0})^+\bigr)\leadsto([\go]^{<\go},\subseteq)
 \end{equation}
holds. Now set $P:=(\go+1)^{\op}$\index{s}{omega1op@$(\omega+1)^{\op}$}. By the construction of Lemma~\ref{L:PseudoRetr}, $P$ is a pseudo-retract of a countable lower finite\index{i}{poset!lower finite} \ajs\index{i}{almost join-semilattice}~$Q$ with zero. As~$Q$ is lower finite\index{i}{poset!lower finite} and countable, it embeds (\emph{via} $(x\mapsto Q\dnw x)$) into~$([\go]^{<\go},\subseteq)$, thus from~\eqref{Eq:LargeKurat} it follows that\index{s}{arr0x@$(\gk,{<}\gl)\leadsto P$}\index{s}{aleph0@$\aleph_{\ga}$}
 \[
 \bigl(\gk,{<}(2^{\aleph_0})^+\bigr)\leadsto Q\,.
 \]
By Lemma~\ref{L:Part2Lift}, it follows that~$Q$ is $(2^{\aleph_0})^+$-liftable\index{s}{aleph0@$\aleph_{\ga}$}. However, by Corollary~\ref{C:LiftnotWFrestr}, $P$ is not $(2^{\aleph_0})^+$-liftable\index{i}{liftable!$\gl$-${}_{-}$ poset}. In conclusion,
\begin{quote}\normalsize\em
In the presence of large cardinals, a countable pseudo-retract of a $(2^{\aleph_0})^+$-liftable\index{i}{liftable!$\gl$-${}_{-}$ poset} countable lower finite\index{i}{poset!lower finite}\index{s}{aleph0@$\aleph_{\ga}$} \ajs\index{i}{almost join-semilattice}\ may not be $(2^{\aleph_0})^+$-liftable\index{s}{aleph0@$\aleph_{\ga}$}\index{i}{liftable!$\gl$-${}_{-}$ poset}.
\end{quote}
This is to be put in contrast with Lemma~\ref{L:RetrLift}.

\section{Lifting diagrams without assuming lifters}
\label{S:MorePosets}

In some applications such as the proof of Theorem~\ref{T:MindConcLift}, we will need a version of CLL\index{i}{Condensate Lifting Lemma (CLL)} involving more general posets, meaning that those will not necessarily be assumed with a lifter\index{i}{lifter ($\gl$-)}, and a weaker conclusion, namely where the nature and properties of the condensate\index{i}{condensate}~$\xF(X)\otimes\overrightarrow{A}$\index{s}{FxX@$\xF(X)$}\index{s}{otimAS@$\bA\otimes\overrightarrow{S}$, $\gf\otimes\overrightarrow{S}$} will no longer matter. Furthermore, this version of CLL\index{i}{Condensate Lifting Lemma (CLL)} (Corollary~\ref{C:CLLnoLF}) will involve the large cardinal axiom $(\gk,{<}\go,\gl)\rightarrow\gl$\index{s}{arr0free@$(\gk,{<}\go,\gl)\rightarrow\gr$}. An example of application of this ``CLL without lifters''\index{i}{lifter ($\gl$-)} is given in Theorem~\ref{T:MindConcLift} (diagram extension of the Gr\"atzer-Schmidt\index{c}{Gr\"atzer, G.}\index{c}{Schmidt, E.\,T.} Theorem). Examples where we really need the original form of CLL\index{i}{Condensate Lifting Lemma (CLL)} (not only for the cardinality estimates of the condensates\index{i}{condensate} but even for their structure) are given in \cite{Gill3,Banasch2}\index{c}{Gillibert, P.}\index{c}{Wehrung, F.}.

Our next lemma shows that under mild assumptions, the liftability, with respect to a given functor and a given class of double arrows\index{i}{double arrow}, of a $P$-indexed diagram can be reduced to the liftability of a certain $Q$-indexed diagram provided~$P$ is a pseudo-retract of~$Q$.

\begin{lem}\label{L:LowFin2Arbv0}
Let $\gl$ be an infinite cardinal, let $\cB$ and $\cS$ be categories, let~$\cS^\Rightarrow$\index{s}{RightarrowCat@$\cS^\Rightarrow$} be a subcategory of~$\cS$, and let $\Psi\colon\cB\to\cS$ be a functor. We assume that $\cB$ has all $\gl$-small directed colimits, that~$\Psi$ preserves all $\gl$-small directed colimits, and that~$\cS^\Rightarrow$\index{s}{RightarrowCat@$\cS^\Rightarrow$} is right closed under $\gl$-small directed colimits\index{i}{closed (left, right-) under $\gl$-small $\varinjlim$}. For a poset~$P$ and an object~$\overrightarrow{S}$ of~$\cS^P$, let $\Lift(\overrightarrow{S})$\index{s}{Lift@$\Lift(\overrightarrow{S})$|ii} hold if there are an object~$\overrightarrow{B}$ of~$\cB^P$ and a morphism~$\overrightarrow{\chi}\colon\Psi\overrightarrow{B}\Todot\overrightarrow{S}$\index{s}{AtoRightarrowdotB@$f\colon\xA\Todot\xB$} in~$(\cS^\Rightarrow)^P$.

Let~$P$ and~$Q$ be posets, let $(e,f)$ be a pseudo-retraction\index{i}{pseudo-retraction} for $(P,Q)$ such that~$e(p)$ has a $\gl$-small cofinal subset, for each $p\in P$. Then $\Lift(\overrightarrow{S}f)$ implies $\Lift(\overrightarrow{S})$\index{s}{Lift@$\Lift(\overrightarrow{S})$}, for every object~$\overrightarrow{S}$ of~$\cS^P$.
\end{lem}

\begin{proof}
By assumption, the set $e^*(p):=\setm{q\in e(p)}{f(q)=p}$ is a cofinal upper subset\index{i}{upper subset} of~$e(p)$, for each~$p\in P$.
For an object $\overrightarrow{S}=\famm{S_{p_0},\gs_{p_0}^{p_1}}{p_0\leq p_1\text{ in }P}$ of~$\cS^P$, the composite diagram~$\overrightarrow{S} f$ is an object of~$\cS^Q$, thus, as $\Lift(\overrightarrow{S}f)$\index{s}{Lift@$\Lift(\overrightarrow{S})$} holds, there are a diagram $\overrightarrow{C}=\famm{C_{q_0},\gc_{q_0}^{q_1}}{q_0\leq q_1\text{ in }Q}$ and a morphism\index{s}{AtorightarrowB@$f\colon A\Rightarrow B$}\index{s}{RightarrowCat@$\cS^\Rightarrow$}
 \[
 \overrightarrow{\gy}=\famm{\gy_q}{q\in Q}\colon\Psi\overrightarrow{C}\Rightarrow
 \overrightarrow{S} f\quad\text{in }(\cS^\Rightarrow)^Q\,.
 \]
It follows from our assumptions that we can define
 \begin{equation}\label{Eq:BpColimCq}
 \famm{B_p,\gc_{q,p}}{q\in e(p)}:=
 \varinjlim\famm{C_{q_0},\gc_{q_0}^{q_1}}{q_0\leq q_1\text{ in }e(p)}
 \ \text{in }\cB\,,\quad
 \text{for each }p\in P\,.
 \end{equation}
(The colimit can be defined on a $\gl$-small cofinal subset of~$e(p)$, and then it is also a colimit on the whole~$e(p)$.) Again by our assumptions, it follows that
 \begin{equation}\label{Eq:PsiBpColimCq}
 \famm{\Psi(B_p),\Psi(\gc_{q,p})}{q\in e(p)}=
 \varinjlim\famm{\Psi(C_{q_0}),\Psi(\gc_{q_0}^{q_1})}{q_0\leq q_1\text{ in }e(p)}
 \quad\text{in }\cS\,.
 \end{equation}
We set
 \[
 R_{p_0,p_1}:=\setm{(q_0,q_1)\in e(p_0)\times e(p_1)}{q_0\leq q_1}\,,\quad
 \text{for all }p_0\leq p_1\text{ in }P\,.
 \]
We claim that~$R(p_0,p_1)$ is cofinal in~$e(p_0)\times e(p_1)$, for all $p_0\leq p_1$ in~$P$. Indeed, let $(q_0,q_1)\in e(p_0)\times e(p_1)$. As $q_0\in e(p_0)\subseteq e(p_1)$, $q_1\in e(p_1)$, and~$e(p_1)$ is an ideal\index{i}{ideal!of a poset} of~$Q$, there exists~$q\in e(p_1)$ such that $q\geq q_0,q_1$. It follows that $(q_0,q)\in R_{p_0,p_1}$, thus proving our claim. {}From this claim, \eqref {Eq:BpColimCq}, and the universal property of the colimit, it follows that for all $p_0\leq p_1$ in~$P$, there exists a unique morphism~$\gb_{p_0}^{p_1}\colon B_{p_0}\to\nobreak B_{p_1}$ such that $\gb_{p_0}^{p_1}\circ\gc_{q_0,p_0}=\gc_{q_1,p_1}\circ\gc_{q_0}^{q_1}$ holds for all $(q_0,q_1)\in R_{p_0,p_1}$. It is straightforward to verify that $\overrightarrow{B}:=\famm{B_{p_0},\gb_{p_0}^{p_1}}{p_0\leq p_1\text{ in }P}$ is an object of~$\cB^P$. Furthermore, for a given~$p\in P$, as $\famm{\gy_q}{q\in e^*(p)}$ defines a natural transformation from~$\Psi\overrightarrow{C}\res_{e^*(p)}$ to the one-vertex diagram~$S_p$ and by~\eqref{Eq:PsiBpColimCq}, there exists a unique morphism~$\chi_p\colon\Psi(B_p)\to S_p$ such that $\chi_p\circ\Psi(\gc_{q,p})=\gy_q$ holds for each $q\in e^*(p)$. Using the assumption that~$\cS^\Rightarrow$\index{s}{RightarrowCat@$\cS^\Rightarrow$} is right closed under $\gl$-small directed colimits\index{i}{closed (left, right-) under $\gl$-small $\varinjlim$}, it follows from~\eqref{Eq:PsiBpColimCq} and the fact that all~$\gy_q$, for~$q\in e^*(p)$, belong to~$\cS^\Rightarrow$\index{s}{RightarrowCat@$\cS^\Rightarrow$} that~$\chi_p$ is a double arrow\index{i}{double arrow}, that is, $\chi_p\colon\Psi(B_p)\Rightarrow S_p$\index{s}{AtorightarrowB@$f\colon A\Rightarrow B$}.

Let $p_0\leq p_1$ in~$P$ and let $q_0\in e^*(p_0)$. As seen above, there exists~$q_1\in e^*(p_1)$ such that $q_0\leq q_1$.

 \begin{figure}[htb]
 \[
 \def\labelstyle{\displaystyle}
 \xymatrix{
 \Psi(C_{q_0})\ar[rr]_{\Psi(\gc_{q_0,p_0})}^(.7){\scriptstyle{\varinjlim}}
 \ar[d]_{\Psi(\gc_{q_0}^{q_1})}
 \ar@{=>}@/^1.7pc/[rrrr]^(.7){\gy_{q_0}} && \Psi(B_{p_0})
 \ar@{=>}[rr]^{\chi_{p_0}}\ar[d]^{\Psi(\gb_{p_0}^{p_1})} && S_{q_0}
 \ar[d]^{\gs_{p_0}^{p_1}}\\
 \Psi(C_{q_1})\ar[rr]_(.7){\scriptstyle{\varinjlim}}^{\Psi(\gc_{q_1,p_1})}
 \ar@{=>}@/_1.7pc/[rrrr]_(.7){\gy_{q_1}} &&
 \Psi(B_{p_1})\ar@{=>}[rr]^{\chi_{p_1}}&&
 S_{q_1}
 }
 \]
\caption{Proving the naturality of $\stackrel{\rightarrow}{\chi}$}
\label{Fig:Natchi2}
\end{figure}

We obtain, following the calculations on Figure~\ref{Fig:Natchi2},
 \begin{align*}
 \gs_{p_0}^{p_1}\circ\chi_{p_0}\circ\Psi(\gc_{q_0,p_0})&=
 \gs_{p_0}^{p_1}\circ\gy_{q_0}\\
 &=\gy_{q_1}\circ\Psi(\gc_{q_0}^{q_1})\\
 &=\chi_{p_1}\circ\Psi(\gc_{q_1,p_1})\circ\Psi(\gc_{q_0}^{q_1})\\
 &=\chi_{p_1}\circ\Psi(\gb_{p_0}^{p_1})\circ\Psi(\gc_{q_0,p_0})\,. 
 \end{align*}
As this holds for all $q_0\in e^*(p_0)$ and by \eqref{Eq:PsiBpColimCq}, we get
$\gs_{p_0}^{p_1}\circ\chi_{p_0}=\chi_{p_1}\circ\Psi(\gb_{p_0}^{p_1})$. We have proved that $\famm{\chi_p}{p\in P}$ is a double arrow\index{i}{double arrow} from~$\Psi\overrightarrow{B}$ to~$\overrightarrow{S}$ in~$\cS$.
\qed\end{proof}

We obtain a version of CLL\index{i}{Condensate Lifting Lemma (CLL)} that does not require the existence of a lifter\index{i}{lifter ($\gl$-)} for the indexing poset~$P$, at the expense of a large cardinal assumption.

\begin{cor}\label{C:CLLnoLF}
Let $\gk$, $\gl$ be infinite cardinals with $(\gk,{<}\go,\gl)\rightarrow\gl$\index{s}{arr0free@$(\gk,{<}\go,\gl)\rightarrow\gr$}, let $\Lambda=(\cA,\cB,\cS,\cA^\dagger,\cB^\dagger,\cS^\Rightarrow,\Phi,\Psi)$ be a $(\gl,\aleph_0)$-larder\index{s}{aleph0@$\aleph_{\ga}$}\index{i}{larder}, and let~$P$ be an \ajs\index{i}{almost join-semilattice}\ such that $\card P<\cf(\gl)$. We further assume that for every $A\in\Ob\cA$, there are $B\in\Ob\cB$ and $\chi\colon\Psi(B)\Rightarrow\Phi(A)$\index{s}{AtorightarrowB@$f\colon A\Rightarrow B$}.

Then for every object $\overrightarrow{A}$ of $(\cA^\dagger)^P$, there are an object~$\overrightarrow{B}$ of~$\cB^P$ and a double arrow\index{i}{double arrow} $\overrightarrow{\chi}\colon\Psi\overrightarrow{B}\Rightarrow\Phi\overrightarrow{A}$\index{s}{AtorightarrowB@$f\colon A\Rightarrow B$}. Furthermore, if~$\Lambda$ is projectable\index{i}{larder!projectable}, then for every object~$\overrightarrow{A}$ of~$(\cA^\dagger)^P$ there exists an object~$\overrightarrow{B}$ of~$\cB^P$ such that $\Psi\overrightarrow{B}\cong\Phi\overrightarrow{A}$.
\end{cor}

\begin{proof}
Obviously $\gk\geq\gl$---just use for~$F\colon[\gk]^{<\go}\to[\gk]^{<\gl}$ the constant mapping with value~$\gk$, otherwise. In fact, it follows from~\cite[Theorem~45.7]{EHMR}\index{c}{Erd\H{o}s, P.}\index{c}{Hajnal, A.}\index{c}{Mate@M\'at\'e, A.}\index{c}{Rado, R.} that $\gk\geq\gl^{+\go}>\gl$.

Now suppose that~$P$ is an \ajs\index{i}{almost join-semilattice}. We first deal with the case where~$P$ is lower finite\index{i}{poset!lower finite}. It follows from \cite[Proposition~3.4]{GiWe1}\index{c}{Gillibert, P.}\index{c}{Wehrung, F.} that the relation $(\gk,{<}\gl)\leadsto([\gl]^{<\go},\subseteq)$\index{s}{arr0x@$(\gk,{<}\gl)\leadsto P$} holds. As the assignment $(x\mapsto P\dnw x)$ defines an order-embedding from~$(P,\leq)$ into $([P]^{<\go},\subseteq)$ and $\card P\leq\gl$, the relation $(\gk,{<}\gl)\leadsto P$\index{s}{arr0x@$(\gk,{<}\gl)\leadsto P$} holds as well. By Corollary~\ref{C:CharLift}, $P$ has a $\gl$-lifter\index{i}{lifter ($\gl$-)} $(X,\bX)$ with~$\bX$ lower finite\index{i}{poset!lower finite}. The desired conclusion follows from CLL\index{i}{Condensate Lifting Lemma (CLL)} (Lemma~\ref{L:CLL}).

In the general case, it follows from Lemma~\ref{L:PseudoRetr} that~$P$ is a pseudo-retract of a lower finite\index{i}{poset!lower finite} \ajs\index{i}{almost join-semilattice}~$Q$ with the same cardinality as~$P$. By the lower finite\index{i}{poset!lower finite} case, $\Lift(\Phi\overrightarrow{A}f)$ holds. By Lemma~\ref{L:LowFin2Arbv0}, $\Lift(\Phi\overrightarrow{A})$\index{s}{Lift@$\Lift(\overrightarrow{S})$} holds as well.
\qed\end{proof}

\begin{remk}\label{Rk:CLLnoLF}
Corollary~\ref{C:CLLnoLF} makes it possible to lift diagrams indexed by \ajs s\index{i}{almost join-semilattice} much more general than those taken care of by CLL\index{i}{Condensate Lifting Lemma (CLL)} (Lemma~\ref{L:CLL}). The easiest example is where $P:=(\go+1)^{\op}$\index{s}{omega1op@$(\omega+1)^{\op}$} is the dual of the chain $\go+1=\set{0,1,2,\dots}\cup\set{\go}$. Corollary~\ref{C:CLLnoLF} says that under a suitable large cardinal assumption, for every object $\overrightarrow{A}$ of $(\cA^\dagger)^P$, there are an object~$\overrightarrow{B}$ of~$\cB^P$ and a double arrow\index{i}{double arrow} $\overrightarrow{\chi}\colon\Psi\overrightarrow{B}\Rightarrow\Phi\overrightarrow{A}$\index{s}{AtorightarrowB@$f\colon A\Rightarrow B$}. This is not a trivial consequence of Lemma~\ref{L:CLL}, because~$P$ is not lower finite\index{i}{poset!lower finite} so Lemma~\ref{L:Part2Lift} does not apply. In fact, by Corollary~\ref{C:LiftnotWFrestr}, $P$ is not $(2^{\aleph_0})^+$-liftable\index{s}{aleph0@$\aleph_{\ga}$}\index{i}{liftable!$\gl$-${}_{-}$ poset} (cf. Problem~\ref{Pb:gooplift} in Chapter~\ref{Ch:Discussion}).

For an application of Corollary~\ref{C:CLLnoLF}, see the proof of Theorem~\ref{T:MindConcLift}.
\end{remk}

\section{Left and right larders}\label{S:LeftRightL}

The definition of a larder (Definition~\ref{D:Larder}) involves categories~$\cA$, $\cB$, and~$\cS$, with a few subcategories, together with functors $\Phi\colon\cA\to\cS$ and $\Psi\colon\cB\to\cS$. One half of the definition describes the interaction between~$\cA$ and~$\cS$, while the other half involves the interaction between~$\cB$ and~$\cS$; furthermore, it is only the second part that involves the cardinal parameter~$\gl$.

In most applications, the desired interaction between~$\cB$ and~$\cS$ is noticeably harder to establish than the one between~$\cA$ and~$\cS$. Therefore, in order to make our work more user-friendly in view of further applications, we shall split the definition of a larder in two parts.

\begin{defn}\label{D:LeftLard}
A \emph{left larder}\index{i}{larder!left|ii} is a quadruple $\Lambda=(\cA,\cS,\cS^\Rightarrow,\Phi)$, where~$\cA$ and~$\cS$ are categories, $\cS^\Rightarrow$\index{s}{RightarrowCat@$\cS^\Rightarrow$} is a subcategory of~$\cS$ (whose arrows we shall call the \emph{double arrows}\index{i}{double arrow} of~$\cS$), and $\Phi\colon\cA\to\cS$ is a functor satisfying the following properties:
\begin{description}
\item[$(\CLOS(\cA))$] \index{s}{ClosA@$(\CLOS(\cA))$}$\cA$ has all small directed colimits.

\item[$(\PROD(\cA))$] \index{s}{ProdA@$(\PROD(\cA))$}Any two objects in~$\cA$ have a product in~$\cA$.

\item[$(\CONT(\Phi))$] \index{s}{Cont@$(\CONT(\Phi))$}The functor~$\Phi$ preserves all small directed colimits.

\item[$(\PROJ(\Phi,\cS^\Rightarrow))$] \index{s}{Proj@$(\PROJ(\Phi,\cS^\Rightarrow))$}$\Phi(f)$ is a morphism in $\cS^\Rightarrow$\index{s}{RightarrowCat@$\cS^\Rightarrow$}, for each extended projection~$f$ of~$\cA$ (cf. Definition~\ref{D:projection}).
\end{description}
\end{defn}

We already formulated the conditions $(\CLOS(\cA))$, $(\PROD(\cA))$, $(\CONT(\Phi))$, and $(\PROJ(\Phi,\cS^\Rightarrow))$ within the statement of the Armature Lemma (Lemma~\ref{L:Armature}).

\begin{defn}\label{D:RightLard}
Let~$\gl$ and~$\gm$ be infinite cardinals. A $6$-uple\linebreak
$\Lambda=(\cB,\cB^\dagger,\cS,\cS^\dagger,\cS^\Rightarrow,\Psi)$ is a \emph{right $(\gl,\gm)$-larder\index{i}{larder!right|ii} at an object~$B$} if~$\cB$ and~$\cS$ are categories, $B$ is an object of~$\cB$, $\cB^\dagger$ (resp., $\cS^\dagger$) is a full subcategory of~$\cB$ (resp., $\cS$), $\cS^\Rightarrow$\index{s}{RightarrowCat@$\cS^\Rightarrow$} is a subcategory of~$\cS$ (whose arrows we shall call the \emph{double arrows}\index{i}{double arrow} of~$\cS$), and $\Psi\colon\cB\to\cS$ is a functor satisfying the following properties:
\begin{description}
\item[$(\PRES_\gl(\cB^\dagger,\Psi))$]  \index{s}{Pres@$(\PRES_\gl(\cB^\dagger,\Psi))$}The object $\Psi(B)$ is weakly $\gl$-presented\index{i}{presented!weakly $\gl$-} in~$\cS$, for each object $B\in\cB^\dagger$.

\item[$(\LSr_\gm(B))$] \index{s}{LSr@$(\LSr_\gm(B))$|ii}For each $S\in\cS^\dagger$, each double arrow\index{i}{double arrow} $\gy\colon\Psi(B)\Rightarrow S$\index{s}{AtorightarrowB@$f\colon A\Rightarrow B$}, each $\gm$-small set~$I$, and each family
$\famm{\gc_i\colon C_i\mono B}{i\in I}$\index{s}{AtomonoB@$f\colon A\mono B$} of monic objects in~$\cB^\dagger\dnw B$, there exists a monic object $\gc\colon C\mono B$\index{s}{AtomonoB@$f\colon A\mono B$} in~$\cB^\dagger\dnw B$ such that $\gc_i\utr\gc$ for each~$i\in I$ while $\gy\circ\Psi(\gc)$ is a morphism in~$\cS^\Rightarrow$\index{s}{RightarrowCat@$\cS^\Rightarrow$}.
\end{description}

We say that~$\Lambda$ is a \emph{right $(\gl,\gm)$-larder}\index{i}{larder!right|ii} if it is a right $(\gl,\gm)$-larder at every object of~$\cB$.

We say that~$\Lambda$ is \emph{strong}\index{i}{larder!right!strong|ii} if the following conditions are satisfied:
\begin{description}
\item[$(\CLOS_\gm(\cB^\dagger,\cB))$] \index{s}{ClosB@$(\CLOS_\gl(\cB^\dagger,\cB))$}The full subcategory~$\cB^\dagger$ has all $\gm$-small directed colimits within~$\cB$ \pup{cf. Definition~\textup{\ref{D:SmallDirColim}}}.

\item[$(\CLOSr_\gm(\cS^\Rightarrow))$] \index{s}{ClosBr@$(\CLOSr_\gl(\cS^\Rightarrow))$}The subcategory $\cS^\Rightarrow$\index{s}{RightarrowCat@$\cS^\Rightarrow$} is right closed under all $\gm$-small directed colimits \pup{cf. Definition~\textup{\ref{D:ClosDirColim}}}.

\item[$(\CONT_\gm(\Psi))$] \index{s}{Contl@$(\CONT_\gl(\Psi))$}The functor~$\Psi$ preserves all $\gm$-small directed colimits.
\end{description}

We say that~$\Lambda$ is \emph{projectable}\index{i}{larder!right!projectable|ii} if every double arrow\index{i}{double arrow} $\varphi\colon\Psi(C)\Rightarrow S$\index{s}{AtorightarrowB@$f\colon A\Rightarrow B$}, for~$C\in\Ob\cB$ and~$S\in\Ob\cS$, has a projectability witness\index{i}{projectability witness} (cf. Definition~\ref{D:ProjFunct}).
\end{defn}

We already formulated the condition $(\PRES_\gl(\cB^\dagger,\Psi))$ in Definition~\ref{D:Larder}. The conditions $(\CLOS_\gm(\cB^\dagger,\cB))$, $(\CLOSr_\gm(\cS^\Rightarrow))$, and $(\CONT_\gm(\Psi))$ were formulated within the statement of the Buttress Lemma (Lemma~\ref{L:Buttress}). The condition $(\LSr_\gm(B))$ is a modification of the condition $(\LSb_\gl(B))$ formulated within the statement of the Buttress Lemma.

Observe that in Definitions~\ref{D:LeftLard} and~\ref{D:RightLard} above, we just split the items $(\CLOS(\cA))$--$(\PROJ(\Phi,\cS^\Rightarrow))$\index{s}{ClosA@$(\CLOS(\cA))$}\index{s}{Proj@$(\PROJ(\Phi,\cS^\Rightarrow))$} from Definition~\ref{D:Larder} into the parts devoted to~$\cA$ and~$\cB$, respectively, while~$(\LSr_\gm(B))$\index{s}{LSr@$(\LSr_\gm(B))$} is obtained from~$(\LS_\gm(B))$\index{s}{LS@$(\LS_\gm(B))$} by replacing~$\Phi``(\cA^\dagger)$ by~$\cS^\dagger$. The proof of the following observation is trivial.

\begin{prop}\label{P:LR2Larder}
Let~$\gl$ and $\gm$ be infinite cardinals with~$\gl$ regular, let $\cA$, $\cB$, $\cS$, $\cA^\dagger$, $\cB^\dagger$, $\cS^\dagger$, $\cS^\Rightarrow$\index{s}{RightarrowCat@$\cS^\Rightarrow$} be categories, and let~$\Phi$, $\Psi$ be functors. If $(\cA,\cS,\cS^\Rightarrow,\Phi)$ is a left larder\index{i}{larder!left}, $(\cB,\cB^\dagger,\cS,\cS^\dagger,\cS^\Rightarrow,\Psi)$ is a right $(\gl,\gm)$-larder\index{i}{larder!right} at an object~$B$ of~$\cB$, $\cA^\dagger$ is a full subcategory of~$\cA$, and $\Phi``(\cA^\dagger)$ is contained in~$\cS^\dagger$, then $(\cA,\cB,\cS,\cA^\dagger,\cB^\dagger,\cS^\Rightarrow,\Phi,\Psi)$ is a $(\gl,\gm)$-larder\index{i}{larder} at~$B$.
\end{prop}

We shall usually say right $\gl$-larder instead of right $(\gl,\cf(\gl))$-larder\index{i}{larder}.

\chapter[Larders from first-order structures]{Getting larders from congruence lattices of first-order structures}\label{Ch:FirstOrd2Lard}

\textbf{Abstract.} One of the main origins of our work is the first author's paper Gillibert~\cite{Gill1}\index{c}{Gillibert, P.}, where it is proved, in particular, that the critical point\index{i}{critical point} $\crit(\cA;\cB)$\index{s}{critAB@$\crit(\cA;\cB)$} between a locally finite\index{i}{variety!locally finite} variety~$\cA$ and a finitely generated\index{i}{variety!finitely generated} congruence-distributive\index{i}{variety!congruence-distributive} variety~$\cB$ such that $\Conc\cA\not\subseteq\Conc\cB$\index{s}{compcon2@$\Conc\cA$} is always less than~$\aleph_\go$\index{s}{aleph0@$\aleph_{\ga}$}. One of the goals of the present chapter is to show how routine categorical verifications about algebraic systems\index{i}{algebraic system} make it possible, using CLL\index{i}{Condensate Lifting Lemma (CLL)}, to extend this result to relative compact congruence semilattices of quasivarieties of algebraic systems\index{i}{algebraic system} (i.e., the languages now have relations as well as operations, and we are dealing with quasivarieties rather than varieties). That particular extension is stated and proved in Theorem~\ref{T:DichotCritPt}. We also obtain a version of Gr\"atzer-Schmidt's\index{c}{Gr\"atzer, G.}\index{c}{Schmidt, E.\,T.} Theorem for poset-indexed diagrams of \jzs s and \jzh s in Theorem~\ref{T:MindConcLift}. With further potential applications in view, most of Chapter~\ref{Ch:FirstOrd2Lard} is designed to build up a framework for being able to easily verify larderhood of many structures arising from (generalized) quasivarieties of algebraic systems\index{i}{algebraic system}. Although we included in this chapter, for convenience sake, a number of already known or folklore results, it also contains results which, although they could be in principle obtained from already published results, could not be so in a straightforward fashion. Such results are Proposition~\ref{P:glPresMIND} (description of some weakly $\gk$-presented\index{i}{presented!weakly $\gl$-} structures in~$\MIND$\index{s}{Mind@$\MIND$}) or Theorem~\ref{T:ConcVPresDirColim} (preservation of all small directed colimits by the relative compact congruence semilattice functor\index{i}{relative!compact congruence semilattice functor} within a given \gqv\index{i}{generalized quasivariety}).

The structures studied in Chapter~\ref{Ch:FirstOrd2Lard} will be called \emph{monotone-indexed structures}. They form a category, that we shall denote by~$\MIND$\index{s}{Mind@$\MIND$}. The objects of~$\MIND$\index{s}{Mind@$\MIND$} are just the first-order structures. For first-order structures~$\bA$ and~$\bB$, a morphism from~$\bA$ to~$\bB$ in~$\MIND$\index{s}{Mind@$\MIND$} can exist only if the language of~$\bA$ is contained in the language of~$\bB$, and then it is defined as a homomorphism (in the usual sense) from~$\bA$ to the reduct of~$\bB$ to the language of~$\bA$.

\section{The category of all monotone-indexed structures}\label{S:NonInd}

In the present section we shall introduce basic definitions and facts about first-order structures, congruences, homomorphisms, with the following twist: we will allow homomorphisms between first-order structures with different languages.

Our definition of a morphism will extend the standard definition of a homomorphism between first-order structures. A \emph{morphism} from a first-order structure~$\bA$ to a first-order structure $\bB$, defined only in case $\Lg(\bA)\subseteq\Lg(\bB)$ (\emph{and not only, as usual, $\Lg(\bA)=\Lg(\bB)$}), is a map $\gf\colon A\to B$ whose restriction from~$\bA$ to the $\Lg(\bA)$-reduct of~$\bB$ is a homomorphism of~$\Lg(\bA)$-structures: that is, $\gf(c^\bA)=c^\bB$ for each constant symbol~$c\in\Lg(\bA)$, and
 \begin{align}
 \gf\bigl(f^\bA(x_1,\dots,x_n)\bigr)&=f^\bB(\gf(x_1),\dots,\gf(x_n))\,,
 \label{Eq:FuncHom}\\
 (x_1,\dots,x_n)\in R^\bA&\Rightarrow(\gf(x_1),\dots,\gf(x_n))\in R^\bB\,,
 \label{Eq:RelHom}
 \end{align}
for all $n\in\gos$, all $x_1,\dots,x_n\in A$, and each $n$-ary $f\in\Op(\bA)$ (resp., each $n$-ary $R\in\Rel(\bA)$). If, in addition, $\gf$ is one-to-one and the implication in~\eqref{Eq:RelHom} is a logical equivalence, we say that~$\gf$ is an \emph{embedding}. In the stronger case where~$\gf$ is an inclusion map and the implication in~\eqref{Eq:RelHom} is still an equivalence, we say that~$\bA$ is a \emph{substructure} of~$\bB$. We shall denote by~$\MIND$\index{s}{Mind@$\MIND$|ii} the category of all first-order structures with this extended definition of a morphism, and we shall call it the category of all \emph{monotone-indexed structures}.

The present definition is reminiscent of the few proposed for \emph{nonindexed algebras}\index{i}{algebra!nonindexed}. A major difference is that the category~$\MIND$\index{s}{Mind@$\MIND$} has all small directed colimits (cf. Definition~\ref{D:SmallDirColim}).

The following definition of a congruence is equivalent to the one given in~\cite[Section~1.4]{Gorb}\index{c}{Gorbunov, V.\,A.}.

\begin{defn}\label{D:Cong1stOrd}
A \emph{congruence} of a first-order structure~$\bA$ is a pair\linebreak
$\bgq=\bigl(\gq,\famm{R_\bgq}{R\in\Rel(\bA)}\bigr)$, where
\begin{description}
\item[\tui] $\gq$ is an equivalence relation on~$A$;

\item[\tuii] $\gq$ is compatible with each function in~$\bA$, that is, for each $f\in\Op(\bA)$, say of arity~$n$, and all $x_1,\dots,x_n,y_1,\dots,y_n\in A$, if $(x_s,y_s)\in\gq$ for each $s\in\set{1,\dots,n}$, then $(f(\overrightarrow{x}),f(\overrightarrow{y}))\in\gq$ (\emph{here and elsewhere, we use the abbreviation~$\overrightarrow{x}$ for the $n$-uple $(x_1,\dots,x_n)$}\index{s}{vecx@$\overrightarrow{x}$ (notation for tuples)|ii}).

\item[\tuiii] $R^\bA\subseteq R_\bgq\subseteq A^{\ari(R)}$\index{s}{aris@$\ari(s)$}, for each $R\in\Rel(\bA)$;

\item[\tuiv] For each $R\in\Rel(\bA)$, say of arity~$n$, and all $x_1,\dots,x_n,y_1,\dots,y_n\in A$, if $\overrightarrow{x}\in R_\bgq$ and $(x_s,y_s)\in\gq$ for each $s\in\set{1,\dots,n}$, then $\overrightarrow{y}\in R_\bgq$.
\end{description}
\end{defn}

\begin{notation}\label{Not:modNotation}
Let~$\bga$ be a congruence of a first-order structure~$\bA$. For elements $x,y\in A$, let $x\equiv y\pmod{\bga}$ hold if $(x,y)\in\ga$. Likewise, for $R\in\Rel(\bA)$, say of arity~$n$, and $x_1,\dots,x_n\in A$, let $R(x_1,\dots,x_n)\pmod{\bga}$ (often abbreviated $R\overrightarrow{x}\pmod{\bga}$) hold if $(x_1,\dots,x_n)\in R_\bga$.
\end{notation}

The set~$\Con\bA$\index{s}{conA@$\Con\bA$|ii} of all congruences of a first-order structure~$\bA$ is partially ordered componentwise, that is, $\bga\leq\bgb$ if{f} $\ga\subseteq\gb$ and $R_\bga\subseteq R_\bgb$ for each $R\in\Rel(\bA)$. It follows easily that~$\Con\bA$\index{s}{conA@$\Con\bA$} is an algebraic subset\index{i}{algebraic subset} of the lattice\index{s}{aris@$\ari(s)$}
 \[
 \Pow(A\times A)\times\prod\famm{\Pow(A^{\ari(R)})}{R\in\Rel(\bA)}\,;
 \]
in particular, it is an algebraic lattice\index{i}{lattice!algebraic}, see the comments at the beginning of \cite[Section~1.4.2]{Gorb}\index{c}{Gorbunov, V.\,A.}. The smallest congruence of~$\bA$ is\index{s}{congspzero@$\zero_\bA$|ii}
 \[
 \zero_\bA:=\bigl(\id_A,\famm{R^\bA}{R\in\Rel(\bA)}\bigr)\,,
 \]
while the largest congruence of~$\bA$ is\index{s}{congspone@$\one_\bA$|ii}
 \[
 \one_\bA:=\Bigl(A\times A,\famm{A^{\ari(R)}}{R\in\Rel(\bA)}\Bigr)\,.
 \]
For an equivalence relation~$\gq$ on a set~$A$ and an element $x\in A$, we shall denote by $x/\gq$ the block (=equivalence class) of~$x$ modulo~$\gq$. For a $n$-uple $(x_1,\dots,x_n)$, we abbreviate the $n$-uple $(x_1/\gq,\dots,x_n/\gq)$ as $\overrightarrow{x}/\gq$.

\begin{defn}\label{D:Quot1ordStruct}
For a congruence~$\bgq$ of a first-order structure~$\bA$, we shall denote by $\bA/\bgq$\index{s}{Aovertheta@$\bA/\bgq$, where $\bgq\in\Con\bA$|ii} the first-order structure with universe~$A/\gq$, the same language as~$\bA$, $c^{\bA/\bgq}=c^\bA/\gq$ for each $c\in\Cst(\bA)$, and
 \begin{align*}
 f^{\bA/\bgq}(\overrightarrow{x}/\gq)&=f(\overrightarrow{x})/\gq\,,\\
 \overrightarrow{x}/\gq\in R^{\bA/\bgq}&\Leftrightarrow\overrightarrow{x}\in R_\bgq\,, 
 \end{align*}
for all $n\in\gos$, all $x_1,\dots,x_n\in A$, and each $n$-ary $f\in\Op(\bA)$ (resp., each $n$-ary $R\in\Rel(\bA)$). The \emph{canonical homomorphism} (or \emph{canonical projection}) from~$\bA$ onto~$\bA/\bgq$ is the map ($A\onto A/\gq$, $x\mapsto x/\gq$)\index{s}{AtoonB@$f\colon A\onto B$}. More generally, for congruences~$\bga$ and~$\bgb$ of a first-order structure~$\bA$, if~$\bga\leq\bgb$, then the canonical map ($A/\ga\onto A/\gb$, $x/\ga\mapsto x/\gb$)\index{s}{AtoonB@$f\colon A\onto B$} is a morphism in~$\MIND$\index{s}{Mind@$\MIND$}, that we shall also call the canonical projection from~$\bA/\bga$ onto~$\bA/\bgb$.
\end{defn}

\begin{defn}\label{D:KerFirsrtOrdHom}
The \emph{kernel} of a morphism~$\gf\colon\bA\to\bB$ in~$\MIND$\index{s}{Mind@$\MIND$} is the pair $\Ker\gf:=\bigl(\gq,\famm{R_\bgq}{R\in\Rel(\bA)}\bigr)$\index{s}{Kerf@$\Ker\gf$|ii}, where
 \begin{align*}
 \gq&:=\setm{(x,y)\in A\times A}{\gf(x)=\gf(y)}\,,\\
 R_\bgq&:=\setm{\overrightarrow{x}\in A^{\ari(R)}}{\gf(\overrightarrow{x})\in R^\bB}
 \end{align*}
(where we set $\gf(\overrightarrow{x}):=(\gf(x_1),\dots,\gf(x_{\ari(R)}))$)\index{s}{aris@$\ari(s)$}, for each $R\in\Rel(\bA)$.
\end{defn}

It is straightforward to verify that the kernel of a morphism from~$\bA$ to $\bB$ is a congruence of~$\bA$. Furthermore, the kernel of the canonical projection from~$\bA$ onto~$\bA/\bgq$ is~$\bgq$, for any congruence~$\bgq$ of~$\bA$. The following result extends \cite[Proposition~1.4.1]{Gorb}\index{c}{Gorbunov, V.\,A.}. As our notation differs from the one used in that reference, we include a proof for convenience.

\begin{lem}[First Isomorphism Theorem]\label{L:FirstIsomThm}
\index{i}{Isomorphism Theorem (First ${}_{-}$)|ii}
Let $\gf\colon\bA\to\bB$ be a morphism in~$\MIND$\index{s}{Mind@$\MIND$} and let $\bga$ be a congruence of~$\bA$. Denote by $\gp\colon\bA\onto\bA/\bga$\index{s}{AtoonB@$f\colon A\onto B$} the canonical projection. Then $\bga\leq\Ker\gf$\index{s}{Kerf@$\Ker\gf$} if{f} there exists a morphism $\gy\colon\bA/\bga\to\bB$ such that $\gf=\gy\circ\gp$, and if this occurs then~$\gy$ is unique. Furthermore, $\gy$ is an embedding if{f} $\bga=\Ker\gf$\index{s}{Kerf@$\Ker\gf$}.
\end{lem}

\begin{proof}
Set $\bgq:=\Ker\gf$\index{s}{Kerf@$\Ker\gf$}. Suppose first that there exists a morphism\linebreak $\gy\colon\bA/\bga\to\bB$ such that $\gf=\gy\circ\gp$. For all $x,y\in A$, $(x,y)\in\ga$ means that $\gp(x)=\gp(y)$, thus $\gf(x)=\gf(y)$, that is, $(x,y)\in\gq$; whence $\ga\subseteq\gq$. Now let $R\in\Rel(\bA)$ and let $\overrightarrow{x}\in R_\bga$. This means that $\overrightarrow{x}/\ga\in R^{\bA/\bga}$, thus, as~$\gy$ is a morphism in~$\MIND$\index{s}{Mind@$\MIND$}, $\gy(\overrightarrow{x}/\ga)\in R^\bB$, that is, $\gf(\overrightarrow{x})\in R^\bB$, so $\overrightarrow{x}\in R_\bgq$. We have proved the inequality $\bga\leq\bgq$.

Conversely, suppose that $\bga\leq\bgq$. As $\ga\subseteq\gq$, there exists a unique map $\gy\colon A/\ga\to\nobreak B$ such that $\gf=\gy\circ\gp$. We must prove that~$\gy$ is a morphism in~$\MIND$\index{s}{Mind@$\MIND$}. For each $f\in\Op(\bA)$, say with arity~$n$, and all $x_1,\dots,x_n\in A$,
 \[
 \gy\bigl(f^{\bA/\bga}(\overrightarrow{x}/\ga)\bigr)=\gy\bigl(f^\bA(\overrightarrow{x})/\ga\bigr)
 =\gf\bigl(f^\bA(\overrightarrow{x})\bigr)=f^\bB(\gf(\overrightarrow{x}))
 =f^\bB\bigl(\gy(\overrightarrow{x}/\ga)\bigr)\,.
 \]
Furthermore, for any $R\in\Rel(\bA)$ and any $\overrightarrow{x}\in A^{\ari(R)}$,\index{s}{Kerf@$\Ker\gf$}\index{s}{aris@$\ari(s)$}
 \begin{align}
 \overrightarrow{x}/\ga\in R^{\bA/\bga}&\Leftrightarrow\overrightarrow{x}\in R_\bga
 &&(\text{by the definition of }R^{\bA/\bga})\notag\\
 &\Rightarrow\overrightarrow{x}\in R_\bgq
 &&(\text{by the assumption that }\bga\leq\bgq)
 \label{Eq:ImplnotEquivRar}\\
 &\Leftrightarrow\gf(\overrightarrow{x})\in R^\bB
 &&(\text{by the definition of }\bgq:=\Ker\gf)\notag\\
 &\Leftrightarrow\gy\bigl(\overrightarrow{x}/\ga\bigr)\in R^\bB\
 &&(\text{by the definition of }\gy).\notag
 \end{align}
This completes the proof that~$\gy$ is a morphism in~$\MIND$\index{s}{Mind@$\MIND$}. The uniqueness statement on~$\gy$ follows trivially from the surjectivity of the map~$\gp$.

If $\bga=\bgq$, then $\ga=\gq$ so~$\gy$ is one-to-one, and further, the implication~\eqref{Eq:ImplnotEquivRar} above is an equivalence, so~$\gy$ is an embedding. Conversely, if~$\gy$ is an embedding, then similar arguments to those above show easily that $\bga=\Ker\gf$\index{s}{Kerf@$\Ker\gf$}.
\qed\end{proof}

For congruences $\bga$, $\bgb$ of a first-order structure~$\bA$ such that $\bga\leq\bgb$, we denote by~$\bgb/\bga$\index{s}{betaoveralpha@$\bgb/\bga$|ii} the kernel of the canonical projection $\bA/\bga\onto\bA/\bgb$\index{s}{AtoonB@$f\colon A\onto B$}. Observe that this congruence of~$\bA/\bga$ can be described by
 \begin{align*}
 \gb/\ga&:=\setm{(x/\ga,y/\ga)}{(x,y)\in\gb}\,,\\
 R_{\bgb/\bga}&:=
 \setm{(x_1/\ga,\dots,x_n/\ga)\in(A/\ga)^n}{(x_1,\dots,x_n)\in R_\bgb}\,, 
 \end{align*}
for each $R\in\Rel(\bA)$ of arity, say, $n$.

The following result is established in \cite[Proposition~1.4.3]{Gorb}\index{c}{Gorbunov, V.\,A.}. As our notation differs from the one used in that reference, we include a proof for convenience.

\begin{lem}[Second Isomorphism Theorem]\label{L:SecIsomThm}
\index{i}{Isomorphism Theorem (Second ${}_{-}$)|ii}
Let $\bga$ be a congruence of a first-order structure~$\bA$. Then the assignment $(\bgb\mapsto\bgb/\bga)$ defines a lattice isomorphism from $(\Con\bA)\upw\bga$\index{s}{conA@$\Con\bA$} onto $\Con(\bA/\bga)$, and the assignment $\bigl(x/\gb\mapsto(x/\ga)/(\gb/\ga)\bigr)$ defines an isomorphism from $\bA/\bgb$ onto $(\bA/\bga)/(\bgb/\bga)$, for each $\bgb\in(\Con\bA)\upw\bga$\index{s}{conA@$\Con\bA$}.
\end{lem}

\begin{proof}
Again, similar to the case without relations (cf. \cite[Section~1.11]{GrUA})\index{c}{Gr\"atzer, G.}. We set $\gs(\bgb):=\bgb/\bga$, for each $\bgb\in(\Con\bA)\upw\bga$\index{s}{conA@$\Con\bA$}; so~$\gs$ defines an isotone map from $(\Con\bA)\upw\bga$\index{s}{conA@$\Con\bA$} to $\Con(\bA/\bga)$\index{s}{conA@$\Con\bA$}. We need to define the converse map. To each congruence~$\tbgb\in\Con(\bA/\bga)$, we associate the kernel of the composite of the canonical projections $\bA\onto\bA/\bga\onto(\bA/\bga)/\tbgb$\index{s}{AtoonB@$f\colon A\onto B$}. Hence $\gt(\tbgb)$ is a congruence of~$\bA$ containing~$\bga$, and $\gt(\tbgb):=\bigl(\gb,\famm{R_\bgb}{R\in\Rel(\bA)}\bigr)$ with
 \begin{align*}
 \gb&:=\setm{(x,y)\in A\times A}{(x/\ga,y/\ga)\in\tgb}\,,\\
 R_\bgb&:=\setm{\overrightarrow{x}\in A^n}{\overrightarrow{x}/\ga\in R_\tbgb}\,. 
 \end{align*}
Denote this congruence by~$\bgb$. Then $\gb/\ga=\tgb$ by the definition of~$\gb$, while, for each $R\in\Rel(\bA)$,
 \[
 R_{\bgb/\bga}=\setm{\overrightarrow{x}/\ga}{\overrightarrow{x}\in R_\bgb}=
 \setm{\overrightarrow{x}/\ga}{\overrightarrow{x}/\ga\in R_\tbgb}=R_\tbgb\,,
 \]
so $\tbgb=\bgb/\bga$. This proves that $\gs\circ\gt$ is the identity on $\Con(\bA/\bga)$.

Conversely, let $\bgb\in(\Con\bA)\upw\bga$\index{s}{conA@$\Con\bA$}, set $\tbgb:=\bgb/\bga$ and $\bgb':=\gt(\tbgb)$. We must prove that $\bgb=\bgb'$. First,
 \begin{align*}
 \gb'&=\setm{(x,y)\in A\times A}{(x/\ga,y/\ga)\in\tgb}\\
 &=\setm{(x,y)\in A\times A}{(x/\ga,y/\ga)\in\gb/\ga}\\
 &=\gb\,,
 \end{align*}
and for all $R\in\Rel(\bA)$ and all $\overrightarrow{x}\in A^{\ari(R)}$\index{s}{aris@$\ari(s)$}, $\overrightarrow{x}\in R_{\bgb'}$ if{f} $\overrightarrow{x}/\ga\in R_\tbgb$ if{f} $\overrightarrow{x}/\ga\in R_{\bgb/\bga}$ if{f} $\overrightarrow{x}\in R_\bgb$. Hence $\bgb=\bgb'$, which completes the proof that $\gt\circ\gs$ is the identity on $(\Con\bA)\upw\bga$\index{s}{conA@$\Con\bA$}. As both~$\gs$ and~$\gt$ are isotone, it follows that they are mutually inverse isomorphisms.

Finally, as $\bgb/\bga$ is the kernel of the canonical projection from $\bA/\bga$ onto $\bA/\bgb$, it follows from Lemma~\ref{L:FirstIsomThm} that it induces an isomorphism from $(\bA/\bga)/(\bgb/\bga)$ onto $\bA/\bgb$.
\qed\end{proof}

\section{Directed colimits of monotone-indexed structures}
\label{S:ColimMIND}

The following result shows that the category~$\MIND$\index{s}{Mind@$\MIND$} has all small directed colimits (cf. Definition~\ref{D:SmallDirColim}), and gives a description of those colimits. Of course, this result extends the classical description of directed colimits for models of a given language, given for example in \cite[Section~1.2.5]{Gorb}\index{c}{Gorbunov, V.\,A.} and recalled in Section~\ref{Su:DirColimFirstOrd}.

\begin{prop}\label{P:DirColimMIND}
Let $\famm{\bA_i,\gf_i^j}{i\leq j\text{ in }I}$ be a directed poset-indexed diagram in~$\MIND$\index{s}{Mind@$\MIND$}. We form the colimit
 \begin{equation}\label{Eq:DirLimSet}
 \famm{A,\gf_i}{i\in I}=\varinjlim\famm{A_i,\gf_i^j}{i\leq j\text{ in }I}
 \end{equation}
in the category~$\SET$\index{s}{Set@\textbf{Set}} of all sets, and we set~$\scL:=\bigcup\famm{\Lg(\bA_i)}{i\in I}$. Then~$A$ can be extended to a unique first-order structure~$\bA$ such that $\Lg(\bA)=\scL$ and the following statements hold:
\begin{description}
\item[\tui] For each constant symbol~$c$ in~$\scL$, $c^\bA=\gf_i(c^{\bA_i})$ for all large enough $i\in I$.

\item[\tuii] For each operation symbol~$f$ in~$\scL$, say of arity~$n$, and for each~$i\in I$ such that $f\in\Op(\bA_i)$, the equation
 \[
 f^\bA(\gf_i(x_1),\dots,\gf_i(x_n))=
 \gf_i\bigl(f^{\bA_i}(x_1,\dots,x_n)\bigr)
 \]
holds for all $x_1,\dots,x_n\in A_i$.

\item[\tuiii] For each relation symbol~$R$ in~$\scL$, say of arity~$n$, and for each $i\in I$ such that $R\in\Rel(\bA_i)$, the equivalence
 \[
 (\gf_i(x_1),\dots,\gf_i(x_n))\in R^\bA\ \Longleftrightarrow\ (\exists j\in I\upw i)
 \bigl((\gf_i^j(x_1),\dots,\gf_i^j(x_n))\in R^{\bA_j}\bigr)\,,
 \]
holds for all $x_1,\dots,x_n\in A_i$.
\end{description}
Furthermore, the following relation holds in~$\MIND$\index{s}{Mind@$\MIND$}.
 \begin{equation}\label{Eq:DirLimMIND}
 \famm{\bA,\gf_i}{i\in I}=\varinjlim\famm{\bA_i,\gf_i^j}
 {i\leq j\text{ in }I}\,.
 \end{equation}
\end{prop}

In categorical terms, Proposition~\ref{P:DirColimMIND} means that \emph{the forgetful functor from~$\MIND$\index{s}{Mind@$\MIND$} to~$\SET$\index{s}{Set@\textbf{Set}} creates and preserves all small directed colimits}.

\begin{proof}
The construction of $\famm{A,\gf_i}{i\in I}$ is given just before~\eqref{Eq:varinjliminSet}.
The existence of the structure~$\bA$ is then a straightforward exercise. For a cocone $\famm{\bB,\gy_i}{i\in I}$ in~$\MIND$\index{s}{Mind@$\MIND$} above~$\famm{\bA_i,\gf_i^j}{i\leq j\text{ in }I}$, it follows from~\eqref{Eq:DirLimSet} that there exists a unique map $\gy\colon A\to B$ such that $\gy_i=\gy\circ\gf_i$ for each $i\in I$. Then, using the definition of~$\bA$, it is straightforward to verify that~$\gy$ is a morphism from~$\bA$ to~$\bB$ in~$\MIND$\index{s}{Mind@$\MIND$}. This completes the proof of~\eqref{Eq:DirLimMIND}.
\qed\end{proof}

\begin{defn}\label{D:kappaSmallStruct}
Let~$\gk$ be an infinite cardinal. A first-order structure~$\bA$ is 
\begin{itemize}
\item \emph{$\gk$-small}\index{i}{kapsmall@$\gk$-small!first-order structure|ii} if $\card A<\gk$;
\item \emph{completely $\gk$-small}\index{i}{kapsmall@$\gk$-small!completely ${}_{-}$ first-order structure|ii} if $\card A+\card\Lg\bA<\gk$.
\end{itemize}
\end{defn}

Now that we know what directed colimits look like in~$\MIND$\index{s}{Mind@$\MIND$} we can prove the following result.

\begin{prop}\label{P:glPresMIND}
Let~$\gk$ be an infinite cardinal and let~$\bA$ be a first-order structure. If~$\bA$ is completely~$\gk$-small, then~$\bA$ is weakly $\gk$-presented\index{i}{presented!weakly $\gl$-} in~$\MIND$\index{s}{Mind@$\MIND$}. Conversely, if~$\gk$ is uncountable, then, defining~$\Omega$ as the disjoint union of~$A$ and~$\Lg(\bA)$, $\bA$ is a continuous directed union, indexed by $[\Omega]^{<\gk}$, of completely $\gk$-small structures in~$\MIND$\index{s}{Mind@$\MIND$}, and~$\bA$ is weakly $\gk$-presented\index{i}{presented!weakly $\gl$-} in~$\MIND$\index{s}{Mind@$\MIND$} if{f} it is completely $\gk$-small.
\end{prop}

\begin{proof}
A standard argument, obtained by adapting the proof of the corresponding statements for $\gk$-presented\index{i}{presented!$\gl$-} structures with~$\gk$ regular (cf., for example, \cite[Corollary~3.13]{AdRo}\index{c}{Ad\'amek, J.}\index{c}{Rosick\'y, J.}) to our definition of a weakly $\gk$-presented\index{i}{presented!weakly $\gl$-} structure.

Assume first that~$\bA$ is completely $\gk$-small, let~$\Omega$ be a set, and consider a continuous\index{i}{diagram!continuous} directed colimit cocone in~$\MIND$\index{s}{Mind@$\MIND$} of the form
 \begin{equation}\label{Eq:Colimkappa}
 \famm{\bB,\gb_X}{X\in[\Omega]^{<\gk}}=
 \varinjlim\famm{\bB_X,\gb_X^Y}{X \subseteq Y\text{ in }[\Omega]^{<\gk}}\,,
 \end{equation}
and let $\gf\colon\bA\to\bB$ be a morphism in~$\MIND$\index{s}{Mind@$\MIND$}. It follows from the continuity of the colimit~\eqref{Eq:Colimkappa} that the following relation holds for each $X\in[\Omega]^{<\gk}$:
 \begin{equation}\label{Eq:ColimTransl}
 \famm{\bB,\gb_{X\cup Y}}{Y\in[\Omega]^{<\go}}=
 \varinjlim\famm{\bB_{X\cup Y},\gb_{X\cup Y}^{X\cup Z}}
 {Y\subseteq Z\text{ in }[\Omega]^{<\go}}\,.
 \end{equation}
In particular, it follows from Proposition~\ref{P:DirColimMIND} that both relations
 \begin{align}
 B&=\bigcup\famm{\gb_{X\cup Y}``(B_Y)}{Y\in[\Omega]^{<\go}}
 \label{Eq:BasUnion}\\
 \Lg(\bB)&=\bigcup\famm{\Lg(\bB_{X\cup Y})}{Y\in[\Omega]^{<\go}}
 \label{Eq:LgBasUnion}
 \end{align}
hold for each $X\in[\Omega]^{<\gk}$. By applying~\eqref{Eq:BasUnion} for $X:=\es$, we obtain, for each $a\in A$, a finite subset~$U_a$ of~$\Omega$ such that $\gf(a)\in\gb_{U_a}``(B_{U_a})$. Likewise, by using~\eqref{Eq:LgBasUnion}, we obtain, for each $s\in\Lg(\bA)$, a finite subset~$V_s$ of~$\Omega$ such that $s\in\Lg(B_{V_s})$. The set
 \[
 X:=\bigcup\famm{U_a}{a\in A}\cup\bigcup\famm{V_s}{s\in\Lg(\bA)}
 \]
belongs to~$[\Omega]^{<\gk}$, $\gf``(A)$ is contained in~$\gb_X``(B_X)$, and~$\Lg(\bA)$ is contained in~$\Lg(\bB_X)$. The first containment implies that there exists a map (not necessarily a morphism) $\gy\colon A\to B_X$ such that $\gf=\gb_X\circ\gy$.

For each $f\in\Op(\bA)$, say of arity~$n$, and all $x_1,\dots,x_n\in A$, setting $x:=f^\bA(x_1,\dots,x_n)$, we obtain, as~$\gf$ is a morphism in~$\MIND$\index{s}{Mind@$\MIND$}, that the equation $\gf(x)=f^\bB\bigl(\gf(x_1),\dots,\gf(x_n)\bigr)$ is satisfied, that is, $(\gb_X\circ\gy)(x)=\gb_X\bigl(f^{\bB_i}\bigl(\gy(x_1),\dots,\gy(x_n)\bigr)\bigr)$. By using~\eqref{Eq:ColimTransl} together with Proposition~\ref{P:DirColimMIND}, we obtain a finite subset~$V$ of~$\Omega$ such that
 \begin{align}
 (\gb_X^{X\cup V}\circ\gy)(x)&=\gb_X^{X\cup V}
 \bigl(f^{\bB_X}\bigl(\gy(x_1),\dots,\gy(x_n)\bigr)\bigr)\,,\notag\\
 \intertext{that is,}
 (\gb_X^{X\cup V}\circ\gy)(x)&=
 f^{\bB_{X\cup V}}
 \bigl((\gb_X^{X\cup V}\circ\gy)(x_1),\dots,(\gb_X^{X\cup V}\circ\gy)(x_n)\bigr)\,.
 \label{Eq:gyFcHomBr2Xj}
 \end{align}
Similarly, for each $R\in\Rel(\bA)$, say of arity~$n$, and all $(x_1,\dots,x_n)\in R^\bA$, there exists a finite subset~$V$ of~$\Omega$ such that
 \begin{equation}\label{Eq:gyRelHomBr2Xj}
 \bigl((\gb_X^{X\cup V}\circ\gy)(x_1),\dots,(\gb_X^{X\cup V}\circ\gy)(x_n)\bigr)
 \in R^{\bB_{X\cup V}}\,.
 \end{equation}
As $\card A+\card\Lg(\bA)<\gk$, the union~$Y$ of all~$V$s in~\eqref{Eq:gyFcHomBr2Xj} and~\eqref{Eq:gyRelHomBr2Xj} belongs to~$[\Omega]^{<\gk}$. As the relations~\eqref{Eq:gyFcHomBr2Xj} and~\eqref{Eq:gyRelHomBr2Xj}, with~$V$ replaced by~$Y$, are satisfied for all possible choices of~$f$, $R$, and $x_1,\dots,x_n$, the map $\gf':=\gb_X^{X\cup Y}\circ\gy$ is a morphism from~$\bA$ to~$\bB_{X\cup Y}$ in~$\MIND$\index{s}{Mind@$\MIND$}. Observe that $\gf=\gb_{X\cup Y}\circ\gf'$. This means that~$\bA$ is weakly $\gk$-presented\index{i}{presented!weakly $\gl$-} in~$\MIND$\index{s}{Mind@$\MIND$}.

Now assume that~$\gk$ is uncountable, let~$\bA$ be an arbitrary first-order structure, and set $\Omega:=A\cup\Lg(\bA)$ (disjoint union). Pick $a\in A$. For each $X\subseteq\Omega$, there exists a least (with respect to containment) subset~$\ol{X}$ of~$\Omega$ containing~$X\cup\set{a}$ such that $\ol{X}\cap A$ is closed under all operations (and constants) from $\ol{X}\cap\Lg(\bA)$. Furthermore, as~$\gk$ is uncountable, if~$X$ is $\gk$-small, then so is~$\ol{X}$. For each $X\in[\Omega]^{<\gk}$, denote by~$\bA_X$ the first-order structure with universe~$\ol{X}\cap A$ and language $\ol{X}\cap\Lg(A)$. By using Proposition~\ref{P:DirColimMIND}, we obtain the continuous\index{i}{diagram!continuous} directed colimit\index{s}{Mind@$\MIND$}
 \begin{equation}\label{Eq:bAasContDirColim}
 \bA=\varinjlim_{X\in[\Omega]^{<\gk}}{\bA_X}\quad\text{in }\MIND\,,
 \end{equation}
with all transition maps and limiting maps being the corresponding inclusion maps. In particular, if~$\bA$ is weakly $\gk$-presented\index{i}{presented!weakly $\gl$-} in~$\MIND$\index{s}{Mind@$\MIND$}, it follows from the continuity of the directed colimit~\eqref{Eq:bAasContDirColim} that~$\bA=\bA_X$ for some $X\in[\Omega]^{<\gk}$, so $\card A+\card\Lg(\bA)<\gk$.
\qed\end{proof}

There are many examples of infinite structures that are finitely presented\index{i}{presented!finitely} in~$\MIND$\index{s}{Mind@$\MIND$}. For example, let $A:=\go$, $\Lg(\bA):=\set{f}$ where~$f$ is a unary operation symbol, and define~$f^\bA$ as the successor map on~$\go$. This shows that the converse part of Proposition~\ref{P:glPresMIND} does not extend to the case where $\gk=\aleph_0$\index{s}{aleph0@$\aleph_{\ga}$}.

\section[Congruence lattices in generalized quasivarieties]{The relative congruence lattice functor with respect to a \gqv}\label{S:ConFunct}

For a member~$\bA$ of a class~$\cV$ of structures, the only congruences of~$\bA$ we will be interested in will mostly be \emph{$\cV$-congruences}\index{i}{congruencesV@$\cV$-congruences|ii}, that is, congruences~$\bgq$ of~$\bA$ such that $\bA/{\bgq}$ belongs to~$\cV$. This approach is also widely used in Gorbunov~\cite{Gorb}\index{c}{Gorbunov, V.\,A.}. However, here all the structures in~$\cV$ need not have the same language. This will be useful for our extension of the Gr\"atzer-Schmidt\index{c}{Gr\"atzer, G.}\index{c}{Schmidt, E.\,T.} Theorem stated in Theorem~\ref{T:MindConcLift}. The relevant classes~$\cV$ will be called \emph{\gqvs}\index{i}{generalized quasivariety} (Definition~\ref{D:GQV}).

\begin{defn}\label{D:ResAssignm}
For a morphism $\gf\colon\bA\to\bB$ in~$\MIND$\index{s}{Mind@$\MIND$} and a congruence~$\bgb$ of~$\bB$, we define $(\Res\gf)(\bgb):=\bigl(\ga,\famm{R_\bga}{R\in\Rel(\bA)}\bigr)$\index{s}{Resf@$\Res\gf$|ii}, where\index{s}{aris@$\ari(s)$}
 \begin{align*}
 \ga&:=\setm{(x,y)\in A\times A}{(\gf(x),\gf(y))\in\gb}\,,\\
 R_\bga&:=\setm{\overrightarrow{x}\in A^{\ari(R)}}{\gf(\overrightarrow{x})\in R_\bgb}\,,
 &&\text{for each }R\in\Rel(\bA)\,.
 \end{align*}
\end{defn}

The verification of the following result is routine.

\begin{prop}\label{P:ResContrFunct}\hfill
\begin{description}
\item[\tui] The assignment~$\Res\gf$\index{s}{Resf@$\Res\gf$} maps $\Con\bB$\index{s}{conA@$\Con\bA$} to $\Con\bA$, for any\linebreak morphism $\gf\colon\bA\to\nobreak\bB$ in~$\MIND$\index{s}{Mind@$\MIND$}. Furthermore, this assignment preserves all meets and all nonempty directed joins in~$\Con\bB$\index{s}{conA@$\Con\bA$}.

\item[\tuii] The assignment $(\bA\mapsto\Con\bA$, $\gf\mapsto\Res\gf)$\index{s}{conA@$\Con\bA$}\index{s}{Resf@$\Res\gf$} defines a contravariant functor from~$\MIND$\index{s}{Mind@$\MIND$} to the category of all algebraic lattices\index{i}{lattice!algebraic} with maps preserving all meets and all nonempty directed joins.
\end{description}
\end{prop}

\begin{lem}\label{L:Ker2Res}
Let~$\bA$, $\bB$, and $\bC$ be first-order structures, together with morphisms $\gf\colon\bA\to\bB$ and $\gy\colon\bB\to\bC$ in~$\MIND$\index{s}{Mind@$\MIND$}. Then the equation $\Ker(\gy\circ\gf)=(\Res\gf)(\Ker\gy)$\index{s}{Resf@$\Res\gf$}\index{s}{Kerf@$\Ker\gf$} is satisfied.
\end{lem}

\begin{proof}
As $\Ker\gy=(\Res\gy)(\zero_\bC)$\index{s}{Resf@$\Res\gf$}\index{s}{Kerf@$\Ker\gf$} and by the contravariance of the functor~$\Res$,\index{s}{Kerf@$\Ker\gf$}
 \[
 (\Res\gf)(\Ker\gy)=(\Res\gf)\circ(\Res\gy)(\zero_\bC)=
 \Res(\gy\circ\gf)(\zero_\bC)=\Ker(\gy\circ\gf)\,.
 \]
This concludes the proof.
\qed\end{proof}

\begin{defn}\label{D:GQV}
A \emph{\gqv}\index{i}{generalized quasivariety|ii} is a full subcategory~$\cV$ of~$\MIND$\index{s}{Mind@$\MIND$} such that:
\begin{description}
\item[\tui] For every morphism $\gf\colon\bA\to\bB$ in~$\cV$, the quotient $\bA/{\Ker\gf}$\index{s}{Kerf@$\Ker\gf$} is a member of~$\cV$.

\item[\tuii] For any first-order structures~$\bA$ and~$\bB$, if~$\bA$ embeds into~$\bB$, $\Lg(\bA)=\Lg(\bB)$, and $\bB\in\cV$, then $\bA\in\cV$.

\item[\tuiii] Any direct product of a nonempty collection of first-order structures in~$\cV$, all with the same language, belongs to~$\cV$.

\item[\tuiv] The quotient $\bA/\one_\bA$\index{s}{congspone@$\one_\bA$} belongs to~$\cV$, for each $\bA\in\cV$.

\item[\tuv] $\cV$ is closed under directed colimits in~$\MIND$\index{s}{Mind@$\MIND$}.
\end{description}

Then we set\index{s}{conA@$\Con\bA$}\index{s}{conAV@$\ConV\bA$, $\ConV f$|ii}
 \[
 \ConV\bA:=\setm{\bga\in\Con\bA}{\bA/\bga\in\cV}\,,
 \quad\text{for any first-order structure }\bA\,,
 \]
and we call the elements of~$\ConV\bA$\index{s}{conAV@$\ConV\bA$, $\ConV f$} the \emph{$\cV$-congruences}\index{i}{congruencesV@$\cV$-congruences} of~$\bA$. The \jzs\ of all compact elements of that lattice will be denoted by~$\ConcV\bA$\index{s}{compcongVA@$\ConcV\bA$, $\ConcV f$|ii}. We shall usually omit the superscript~$\cV$ in case~$\cV$ is closed under homomorphic images\index{s}{compcon1@$\Conc\bA$, $\Conc f$|ii}.
\end{defn}

The meaning of $\ConV\bA$\index{s}{conAV@$\ConV\bA$, $\ConV f$} that we use here is the same as the meaning of~$\Con_{\mathbf{K}}\cA$ introduced in~\cite[Section~1.4.2]{Gorb}\index{c}{Gorbunov, V.\,A.}.

\begin{exple}\label{Ex:QVar}
A \emph{quasivariety}\index{i}{quasivariety|ii} in a given language~$\scL$ is a class~$\cV$ of first-order structures on~$\scL$ which is closed under substructures, direct products, and directed colimits within the class of all models of~$\scL$ (cf. Gorbunov~\cite{Gorb}\index{c}{Gorbunov, V.\,A.}). Equivalently, $\cV$ is the class of all models, for a given language~$\scL$, that satisfy a given set of formulas each of the form\index{s}{conj@$\conj$}
 \[
 (\forall\overrightarrow{\vx})\Bigl(\conj_{i<k}\vE_i(\overrightarrow{\vx})\Rightarrow\vF(\overrightarrow{\vx})\Bigr)\,,
 \]
for integers $k>0$ and atomic formulas~$\vE_i$, $\vF$.

It is straightforward to verify that any quasivariety is also a \gqv\index{i}{generalized quasivariety}\ (for~(i), observe that~$\bA/{\Ker\gf}$\index{s}{Kerf@$\Ker\gf$} embeds into~$\bB$; for~(iv), observe that~$\bA/\one_\bA$\index{s}{congspone@$\one_\bA$} is the terminal object of~$\cV$, and so it is the product, indexed by the empty set, of a family of objects in~$\cV$.)
\end{exple}

\begin{exple}\label{Ex:1/2nonindexed}
The category~$\MIND$\index{s}{Mind@$\MIND$} is a \gqv\index{i}{generalized quasivariety}. More generally, let~$\scC$ be a class of languages which is closed under directed unions. Then the full subcategory~$\MIND_\scC$\index{s}{Mindsub@$\MIND_\scC$|ii} consisting of all first-order structures~$\bA$ such that $\Lg(\bA)$ belongs to~$\scC$ is a \gqv\index{i}{generalized quasivariety}. Particularly interesting examples are the following:
\begin{description}
\item[\tui] $\scC$ is the class of all languages without relation symbols. Then the objects of~$\MIND_\scC$\index{s}{Mindsub@$\MIND_\scC$} are all the \emph{algebras}\index{i}{algebra!universal} (in the sense of universal algebras).

\item[\tuii] $\scC$ is the class of all languages containing neither relation symbols nor constant symbols, and all whose operation symbols have arity one. Then the objects of~$\MIND_\scC$\index{s}{Mindsub@$\MIND_\scC$} are all the \emph{unary algebras}\index{i}{algebra!unary}.

\item[\tuiii] $\scC$ is the class of all languages containing neither constant symbols nor operation symbols. Then the objects of~$\MIND_\scC$\index{s}{Mindsub@$\MIND_\scC$} are all the \emph{relational systems}.
\end{description}
\end{exple}

Observe that for a \gqv\index{i}{generalized quasivariety}~$\cV$ and a first-order structure~$\bA$, the set $\ConV\bA$\index{s}{conAV@$\ConV\bA$, $\ConV f$} may very well be empty. Because of Definition~\ref{D:GQV}(iv), this cannot occur in case~$\bA\in\cV$. The following result is analogue to \cite[Corollary~1.4.11]{Gorb}\index{c}{Gorbunov, V.\,A.}.

\begin{lem}\label{L:ConGQVAlg}
Let~$\cV$ be a \gqv\index{i}{generalized quasivariety}\ and let~$\bA\in\cV$. Then $\ConV\bA$\index{s}{conAV@$\ConV\bA$, $\ConV f$} is an algebraic subset\index{i}{algebraic subset} of~$\Con\bA$\index{s}{conA@$\Con\bA$}, with the same bounds. In particular, it is an algebraic lattice\index{i}{lattice!algebraic}.
\end{lem}

\begin{proof}
As~$\bA$ belongs to~$\cV$, its zero congruence $\zero_\bA$\index{s}{congspzero@$\zero_\bA$} belongs to $\ConcV\bA$\index{s}{compcongVA@$\ConcV\bA$, $\ConcV f$}. Moreover, it follows from Definition~\ref{D:GQV}(iv) that~$\one_\bA$\index{s}{congspone@$\one_\bA$} belongs to $\ConV\bA$\index{s}{conAV@$\ConV\bA$, $\ConV f$}. If $\famm{\bga_i}{i\in I}$ is a nonempty family of elements in $\ConV\bA$\index{s}{conAV@$\ConV\bA$, $\ConV f$} and $\bga:=\bigwedge\famm{\bga_i}{i\in I}$, then the diagonal map embeds, with the same language, the structure~$\bA/\bga$ into the product $\prod\famm{\bA/\bga_i}{i\in I}$; hence, by~(ii) and~(iii) in Definition~\ref{D:GQV}, $\bA/\bga$ belongs to~$\cV$, and so $\bga\in\ConV\bA$\index{s}{conAV@$\ConV\bA$, $\ConV f$}. Therefore, $\ConV\bA$ is closed under arbitrary meets.

Now let~$I$ be a nonempty directed poset\index{i}{poset!directed} and let $\famm{\bga_i}{i\in I}$ be an isotone family of elements in~$\ConV\bA$\index{s}{conAV@$\ConV\bA$, $\ConV f$}; set $\bga:=\bigvee\famm{\bga_i}{i\in I}$. By using the explicit description of the colimit given in Proposition~\ref{P:DirColimMIND}, it is not hard to verify that $\bA/\bga=\varinjlim_{i\in I}(\bA/\bga_i)$, with the transition morphisms and limiting morphisms being the projection maps (cf. \cite[Proposition~1.4.9]{Gorb}\index{c}{Gorbunov, V.\,A.}). As all the structures $\bA/\bga_i$ belong to~$\cV$ and~$\cV$ is closed under directed colimits in~$\MIND$\index{s}{Mind@$\MIND$}, the structure~$\bA/\bga$ belongs to~$\cV$, and so~$\bga\in \ConV\bA$\index{s}{conAV@$\ConV\bA$, $\ConV f$}.
\qed\end{proof}

In particular, as $\ConV\bA$\index{s}{conAV@$\ConV\bA$, $\ConV f$} is an algebraic subset\index{i}{algebraic subset} of $\Con\bA$\index{s}{conA@$\Con\bA$}, which is an algebraic subset of\index{s}{aris@$\ari(s)$}
$\Pow(A\times A)\times\prod\famm{\Pow(A^{\ari(R)})}{R\in\Rel(\bA)}$, the following notation is well-defined.

\begin{notation}\label{Not:bvNotation}
Let~$\cV$ be a \gqv\index{i}{generalized quasivariety}\ and let~$\bA\in\cV$.
\begin{itemize}
\item For elements $x,y\in A$, we denote by $\bveA{x}{y}$\index{s}{bveA@$\bveA{x}{y}$|ii} the least congruence $\bga\in\ConV\bA$\index{s}{conAV@$\ConV\bA$, $\ConV f$} such that $x\equiv y\pmod{\bga}$.

\item For $R\in\Rel(\bA)$, say of arity~$n$, and $\overrightarrow{x}=(x_1,\dots,x_n)\in A^n$, we denote by $\bvrA{R(x_1,\dots,x_n)}$\index{s}{bvrA@$\bvrA{R(x_1,\dots,x_n)}$, $\bvrA{R\overrightarrow{x}}$|ii}, abbreviated $\bvrA{R\overrightarrow{x}}$, the least congruence $\bga\in\ConV\bA$\index{s}{conAV@$\ConV\bA$, $\ConV f$} such that $R(x_1,\dots,x_n)\pmod{\bga}$.
\end{itemize}
We refer to Notation~\ref{Not:modNotation} for the $\pmod{\bga}$ notation. We shall call the $\cV$-congruences\index{i}{congruencesV@$\cV$-congruences} of the form either $\bveA{x}{y}$ or $\bvrA{R\overrightarrow{x}}$ \emph{principal $\cV$-congruences}\index{i}{congruencesV@$\cV$-congruences}\index{i}{congruencesVp@$\cV$-congruences (principal ${}_{-}$)|ii} of~$\bA$. Again, we shall usually omit the superscript~$\cV$ in case~$\cV$ is closed under homomorphic images.
\end{notation}

\begin{lem}\label{L:ClosUnfRes}
Let $\cV$ be a \gqv\index{i}{generalized quasivariety}\ and let $\gf\colon\bA\to\bB$ be a morphism in~$\cV$. Then the image of $\ConV\bB$\index{s}{conAV@$\ConV\bA$, $\ConV f$} under $\Res\gf$\index{s}{Resf@$\Res\gf$} is contained in $\ConV\bA$\index{s}{conAV@$\ConV\bA$, $\ConV f$}.
\end{lem}

\begin{proof}
Let $\bgb\in\ConV\bB$\index{s}{conAV@$\ConV\bA$, $\ConV f$} and denote by $\gp\colon\bB\onto\bB/\bgb$\index{s}{AtoonB@$f\colon A\onto B$} the canonical projection. Set $\bga:=\Ker(\gp\circ\gf)$\index{s}{Kerf@$\Ker\gf$}. By applying Definition~\ref{D:GQV}(i) to the composite map $\gp\circ\gf$, we obtain that $\bA/\bga$ belongs to~$\cV$, that is, $\bga\in\ConV\bA$\index{s}{conAV@$\ConV\bA$, $\ConV f$}. However, it follows from Lemma~\ref{L:Ker2Res} that $\bga=(\Res\gf)(\bgb)$\index{s}{Resf@$\Res\gf$}.
\qed\end{proof}

Let~$\cV$ be a \gqv\index{i}{generalized quasivariety}. It follows from Lemmas~\ref{L:ConGQVAlg} and~\ref{L:ClosUnfRes} that the assignment $(\bA\mapsto\ConV\bA$\index{s}{conAV@$\ConV\bA$, $\ConV f$}, $\gf\mapsto\Res\gf)$\index{s}{Resf@$\Res\gf$} defines a contravariant functor from~$\cV$ to the category of all algebraic lattices\index{i}{lattice!algebraic} with maps that preserve arbitrary meets and nonempty directed joins.

By the general theory of Galois connections \cite[Section~0.3]{Comp}\index{c}{Gierz, G.}\index{c}{Hofmann, K.\,H.}\index{c}{Keimel, K.}\index{c}{Lawson, J.\,D.}\index{c}{Mislove, M.}\index{c}{Scott, D.\,S.}, for every morphism $\gf\colon\bA\to\bB$ in~$\cV$, there exists a unique map\index{s}{conAV@$\ConV\bA$, $\ConV f$}
 \[
 \ConV\gf\colon\ConV\bA\to\ConV\bB
 \]
such that\index{s}{Resf@$\Res\gf$}
 \[
 (\ConV\gf)(\bga)\leq\bgb\text{ if{f} }\bga\leq(\Res\gf)(\bgb)\,,
 \text{ for all }(\bga,\bgb)\in(\ConV\bA)\times(\ConV\bB)\,.
 \]
As the map $\Res\gf$\index{s}{Resf@$\Res\gf$} preserves arbitrary meets and nonempty directed joins, the map $\ConV\gf$\index{s}{conAV@$\ConV\bA$, $\ConV f$} preserves arbitrary joins and it sends compact elements to compact elements. Hence the assignment $(\bA\mapsto\ConV\bA$, $\gf\mapsto\ConV\gf)$\index{s}{conAV@$\ConV\bA$, $\ConV f$} defines a (covariant) functor from~$\cV$ to the category of all algebraic lattices\index{i}{lattice!algebraic} with compactness-preserving complete \jh s.

\begin{lem}\label{L:ConcVProj}
Let $\bgq$ be a congruence of a first-order structure~$\bA$ and denote by $\gp\colon\bA\onto\bA/\bgq$\index{s}{AtoonB@$f\colon A\onto B$} the canonical projection. Then $(\ConV\gp)(\bga)=\bga\vee\bgq/\bgq$\index{s}{conAV@$\ConV\bA$, $\ConV f$} \pup{where the join $\bga\vee\bgq$ is evaluated in~$\ConV\bA$}\index{s}{conAV@$\ConV\bA$, $\ConV f$}, for each $\bga\in\ConV\bA$.
\end{lem}

\begin{proof}
Set $\bgb':=(\Res\gp)(\bgb/\bgq)$\index{s}{Resf@$\Res\gf$}, for each $\bgb\in(\ConV\bA)\upw\bgq$\index{s}{conAV@$\ConV\bA$, $\ConV f$}.
It suffices to prove that $\bgb=\bgb'$. For $x,y\in A$,
 \begin{align*}
 x\equiv y\pmod{\bgb'}& \Leftrightarrow\gp(x)\equiv\gp(y)\pmod{\bgb/\bgq}\\
 & \Leftrightarrow x\equiv y\pmod{\bgb}\,, 
 \end{align*}
while for each $R\in\Rel(\bA)$ and each $\overrightarrow{x}\in A^{\ari(R)}$\index{s}{aris@$\ari(s)$},
 \begin{align*}
 R\overrightarrow{x}\pmod{\bgb'}& \Leftrightarrow R\gp(\overrightarrow{x})\pmod{\bgb/\bgq}\\
 & \Leftrightarrow \gp(\overrightarrow{x})\in R_{\bgb/\bgq}\\
 & \Leftrightarrow R\overrightarrow{x}\pmod{\bgb}\,.
 \end{align*}
The conclusion follows immediately.
\qed\end{proof}

By the Second Isomorphism Theorem (Lemma~\ref{L:SecIsomThm})\index{i}{Isomorphism Theorem (Second ${}_{-}$)}, it follows that $\ConV(\bA/\bgq)$\index{s}{conAV@$\ConV\bA$, $\ConV f$} is isomorphic to the interval $(\ConV\bA)\upw\bgq$\index{s}{conAV@$\ConV\bA$, $\ConV f$}.

We will often find it more convenient to work with \emph{compact} congruences. For any~$\bA\in\nobreak\cV$, as~$\ConV\bA$\index{s}{conAV@$\ConV\bA$, $\ConV f$} is an algebraic lattice\index{i}{lattice!algebraic}, it is canonically isomorphic to the ideal\index{i}{ideal!of a poset} lattice of the \jzs\ $\ConcV\bA$\index{s}{compcongVA@$\ConcV\bA$, $\ConcV f$} of all compact elements of~$\ConV\bA$\index{s}{conAV@$\ConV\bA$, $\ConV f$}. For a morphism $\gf\colon\bA\to\bB$ in~$\cV$, we shall denote by~$\ConcV\gf$ the restriction of~$\ConV\gf$\index{s}{conAV@$\ConV\bA$, $\ConV f$} from~$\ConcV\bA$ to~$\ConcV\bB$\index{s}{compcongVA@$\ConcV\bA$, $\ConcV f$}. Therefore, the assignment $(\bA\mapsto\ConcV\bA$, $\gf\mapsto\ConcV\gf)$ defines a functor\index{s}{compcongV@$\ConcV$ functor|ii} from~$\cV$ to the category~$\SEM$\index{s}{Sem@$\SEM$|ii} of all \jzs s with \jzh s. We shall also write~$\Conc$\index{s}{compcon@$\Conc$ functor|ii} instead of~$\ConcV$\index{s}{compcongV@$\ConcV$ functor} in case~$\cV$ is closed under homomorphic images.

In the statement of the following lemma we use Notation~\ref{Not:bvNotation}.

\begin{lem}\label{L:ConcvonPpalCong}
The following statements hold, for any morphism $\gf\colon\bA\to\nobreak\bB$ in a \gqv\index{i}{generalized quasivariety}\ $\cV$:
\begin{description}
\item[\tui] The $\cV$-congruences\index{i}{congruencesV@$\cV$-congruences} of~$\bA$ are exactly the joins \pup{in~$\ConV\bA$}\index{s}{conAV@$\ConV\bA$, $\ConV f$} of principal $\cV$-congruences\index{i}{congruencesV@$\cV$-congruences}\index{i}{congruencesVp@$\cV$-congruences (principal ${}_{-}$)} of~$\bA$.

\item[\tuii] The compact $\cV$-congruences\index{i}{congruencesV@$\cV$-congruences} of~$\bA$ are exactly the finite joins \pup{in~$\ConV\bA$}\index{s}{conAV@$\ConV\bA$, $\ConV f$} of principal $\cV$-congruences\index{i}{congruencesV@$\cV$-congruences}\index{i}{congruencesVp@$\cV$-congruences (principal ${}_{-}$)} of~$\bA$.

\item[\tuiii] The equation $(\ConcV\gf)\bigl(\bveA{x}{y}\bigr)=\bveB{\gf(x)}{\gf(y)}$\index{s}{compcongVA@$\ConcV\bA$, $\ConcV f$} is satisfied, for all $x,y\in\nobreak A$.

\item[\tuiv] The equation $(\ConcV\gf)\bigl(\bvrA{R(x_1,\dots,x_n)}\bigr)=\bvrB{R(\gf(x_1),\dots,\gf(x_n))}$ is satisfied, for any $R\in\Rel(\bA)$, say of arity~$n$, and all $x_1,\dots,x_n\in A$.\index{s}{compcongVA@$\ConcV\bA$, $\ConcV f$}
\end{description}
\end{lem}

\begin{proof}
(i). It is obvious that the following equation is satisfied (in $\ConV\bA$\index{s}{conAV@$\ConV\bA$, $\ConV f$}), for any $\bga\in\ConV\bA$\index{s}{conAV@$\ConV\bA$, $\ConV f$}:
 \[
 \bga=\bigvee\famm{\bveA{x}{y}}{(x,y)\in\ga}\vee
 \bigvee\famm{\bvrA{R\overrightarrow{x}}}{R\in\Rel(\bA),\ \overrightarrow{x}\in R_\bga}\,.
 \]
(ii). Let $\bD$ be a nonempty directed subset\index{i}{poset!directed} of $\ConV\bA$\index{s}{conAV@$\ConV\bA$, $\ConV f$}. The join~$\bgq$ of~$\bD$ in~$\Con\bA$\index{s}{conA@$\Con\bA$} is defined by the formulas
 \begin{align}
 \gq&=\bigcup\famm{\ga}{\bga\in\bD}\,,\label{Eq:whatsgq}\\
 R_\bgq&=\bigcup\famm{R_\bga}{\bga\in\bD}\,.\label{Eq:whatsRgq}
 \end{align}
As~$\ConV\bA$\index{s}{conAV@$\ConV\bA$, $\ConV f$} is closed under directed joins (cf. Lemma~\ref{L:ConGQVAlg}), $\bgq$ belongs to~$\ConV\bA$\index{s}{conAV@$\ConV\bA$, $\ConV f$}. Now let $x,y\in A$ such that $\bveA{x}{y}\leq\bgq$. This means that $(x,y)\in\gq$, thus, by~\eqref{Eq:whatsgq}, $(x,y)\in\ga$ for some~$\bga\in\bD$, so $\bveA{x}{y}\leq\bga$. Likewise, using~\eqref{Eq:whatsRgq}, we can prove that for each $R\in\Rel(\bA)$, say of arity~$n$, and each $x_1,\dots,x_n\in A$, $\bvrA{R\overrightarrow{x}}\leq\bgq$ implies that $\bvrA{R\overrightarrow{x}}\leq\bga$ for some $\bga\in\bD$. Therefore, \emph{all principal $\cV$-congruences of~$\bA$ are compact in~$\ConV\bA$}\index{s}{conAV@$\ConV\bA$, $\ConV f$}\index{i}{congruencesV@$\cV$-congruences}\index{i}{congruencesVp@$\cV$-congruences (principal ${}_{-}$)}, and thus so are all their finite joins. The converse follows immediately from~(i).

(iii). Set $\bga:=\bveA{x}{y}$. By the definition of the $\ConcV$\index{s}{compcongV@$\ConcV$ functor} functor, the following equivalences hold for each $\bgb\in\ConV\bB$\index{s}{conAV@$\ConV\bA$, $\ConV f$}:\index{s}{Resf@$\Res\gf$}
 \begin{align*}
 (\ConcV\gf)(\bga)\leq\bgb&\Longleftrightarrow\bga\leq(\Res\gf)(\bgb)\\
 & \Longleftrightarrow\bveA{x}{y}\leq(\Res\gf)(\bgb)\\
 & \Longleftrightarrow x\equiv y\pmod{(\Res\gf)(\bgb)}\\
 & \Longleftrightarrow\gf(x)\equiv\gf(y)\pmod{\bgb}\\
 & \Longleftrightarrow\bveB{\gf(x)}{\gf(y)}\leq\bgb\,; 
 \end{align*}
the desired conclusion follows.

The proof of~(iv) is similar to the proof of~(iii).
\qed\end{proof}

\section[Preservation of directed colimits for congruence lattices]{Preservation of small directed colimits by the relative compact congruence functor}\label{S:PresColimConcV}

The present section will be devoted to the proof of the following result, which does not seem to have appeared anywhere in print yet, even for the special case of algebraic systems\index{i}{algebraic system} as in Gorbunov~\cite{Gorb}\index{c}{Gorbunov, V.\,A.}.

\begin{thm}\label{T:ConcVPresDirColim}
The functor $\ConcV\colon\cV\to\SEM$\index{s}{compcongV@$\ConcV$ functor}\index{s}{Sem@$\SEM$} preserves all small directed colimits, for any \gqv\index{i}{generalized quasivariety}~$\cV$.
\end{thm}

\begin{proof}
We are given a directed colimit cocone of the form~\eqref{Eq:DirLimMIND} in~$\cV$. As~$\cV$ is closed under directed colimits in~$\MIND$\index{s}{Mind@$\MIND$}, \eqref{Eq:DirLimMIND} is also a directed colimit in~$\MIND$\index{s}{Mind@$\MIND$}, so we can take advantage of the explicit description of the colimit given in Proposition~\ref{P:DirColimMIND}.

It follows from Lemma~\ref{L:ConcvonPpalCong}(ii) that every element $\bga\in\ConcV\bA$\index{s}{compcongVA@$\ConcV\bA$, $\ConcV f$} is a finite join of principal $\cV$-congruences\index{i}{congruencesV@$\cV$-congruences}\index{i}{congruencesVp@$\cV$-congruences (principal ${}_{-}$)} of~$\bA$ (cf. Notation~\ref{Not:bvNotation}). As each such principal $\cV$-congruence\index{i}{congruencesVp@$\cV$-congruences (principal ${}_{-}$)} involves only a finite number of elements from~$A$, it follows that~$\bga$ involves only a finite number of parameters from~$A$. As~$A$ is the directed union of all the subsets $\gf_i``(A_i)$, there exists $i\in I$ such that~$\gf_i``(A_i)$ contains all those parameters. Now an immediate use of Lemma~\ref{L:ConcvonPpalCong}(iii,iv) shows that~$\bga$ belongs to the range of~$\ConcV\gf_i$\index{s}{compcongVA@$\ConcV\bA$, $\ConcV f$}. We have proved that $\ConcV\bA$\index{s}{compcongVA@$\ConcV\bA$, $\ConcV f$} is the union of all the ranges of the maps~$\ConcV\gf_i$.

By virtue of the characterization of directed colimits presented in Section~\ref{Su:DirColimFirstOrd} (applied to the category of all \jzs s with \jzh s), it remains to prove that for each $i\in I$ and all $\bga,\bgb\in\ConcV\bA_i$\index{s}{compcongVA@$\ConcV\bA$, $\ConcV f$},
 \begin{equation}\label{Eq:AssumptConcIncl}
 (\ConcV\gf_i)(\bga)\leq(\ConcV\gf_i)(\bgb)
 \end{equation}
implies that there exists $j\in I\upw i$ such that
 \[
 (\ConcV\gf_i^j)(\bga)\leq(\ConcV\gf_i^j)(\bgb)\,.
 \]
As~$\bga$ is a finite join of principal $\cV$-congruences\index{i}{congruencesV@$\cV$-congruences}\index{i}{congruencesVp@$\cV$-congruences (principal ${}_{-}$)} (cf. Lemma~\ref{L:ConcvonPpalCong}), we may assume in turn that~$\bga$ is principal.

We set $\bgq_j:=(\ConcV\gf_i^j)(\bgb)$\index{s}{compcongVA@$\ConcV\bA$, $\ConcV f$} and we denote by $\gp_j\colon\bA_j\onto\bA_j/\bgq_j$\index{s}{AtoonB@$f\colon A\onto B$} the canonical projection, for each $j\in I\upw i$. {}From $\bgq_k=(\ConcV\gf_j^k)(\bgq_j)$\index{s}{compcongVA@$\ConcV\bA$, $\ConcV f$} we get the inequality $\bgq_j\leq(\Res\gf_j^k)(\bgq_k)$\index{s}{Resf@$\Res\gf$}, hence, by using Lemma~\ref{L:Ker2Res}, $\bgq_j\leq\Ker(\gp_k\circ\gf_j^k)$\index{s}{Kerf@$\Ker\gf$}. Therefore, by Lemma~\ref{L:FirstIsomThm}, there exists a unique morphism\linebreak $\gy_j^k\colon\bA_j/\bgq_j\to\bA_k/\bgq_k$ such that $\gp_k\circ\gf_j^k=\gy_j^k\circ\gp_j$. Clearly, the family $\famm{\bA_j/\bgq_j,\gy_j^k}{j\leq k\text{ in }I\upw i}$ is an $(I\upw i)$-indexed diagram in~$\cV$. We form the directed colimit
 \[
 \famm{\bB,\gy_j}{j\in I\upw i}
 =\varinjlim\famm{\bA_j/\bgq_j,\gy_j^k}{j\leq k\text{ in }I\upw i}
 \]
in the category~$\MIND$\index{s}{Mind@$\MIND$}. As~$\cV$ is closed under directed colimits in~$\MIND$\index{s}{Mind@$\MIND$}, the cocone $\famm{\bB,\gy_j}{j\in I\upw i}$ is contained in~$\cV$. Denote by~$\gp\colon\bA\to\bB$ the unique morphism in~$\MIND$\index{s}{Mind@$\MIND$} such that $\gp\circ\gf_j=\gy_j\circ\gp_j$ for each $j\in I\upw i$. The situation is illustrated on Figure~\ref{Fig:Col2Colgq}.

\begin{figure}[htb]
 \[
 \def\labelstyle{\displaystyle}
 \xymatrix{
 \bA_i\ar[rr]^{\gf_i^j}\ar@{->>}[d]^{\gp_i}&&
 \bA_j\ar[rr]^{\gf_j}\ar@{->>}[d]^{\gp_j}&&
 \bA\ar[d]^{\gp}\\
 \bA_i/\bgq_i\ar[rr]^{\gy_i^j}&&\bA_j/\bgq_j\ar[rr]^{\gy_j}&&\bB
 }
 \]
\caption{A commutative diagram in a \gqv~$\cV$}
\label{Fig:Col2Colgq}
\end{figure}

Now we set $\bgq:=\Ker\gp$\index{s}{Kerf@$\Ker\gf$}, and we compute\index{s}{Resf@$\Res\gf$}
 \begin{align*}
 \bgb=\Ker\gp_i&\subseteq\Ker(\gy_i\circ\gp_i)
 &&(\text{use the easy direction of Lemma~\ref{L:FirstIsomThm}})\\
 &=\Ker(\gp\circ\gf_i)&&(\text{cf. Figure~\ref{Fig:Col2Colgq}})\\
 &=(\Res\gf_i)(\Ker\gp)&&(\text{use Lemma~\ref{L:Ker2Res}})\\
 &=(\Res\gf_i)(\bgq)\,,
 \end{align*}
thus $(\ConcV\gf_i)(\bgb)\subseteq\bgq$\index{s}{compcongVA@$\ConcV\bA$, $\ConcV f$}, and thus, by~\eqref{Eq:AssumptConcIncl}, $(\ConcV\gf_i)(\bga)\subseteq\bgq$, and hence, by using part of the calculation above, $\bga\subseteq(\Res\gf_i)(\bgq)=\Ker(\gy_i\circ\gp_i)$\index{s}{Resf@$\Res\gf$}\index{s}{Kerf@$\Ker\gf$}.

As~$\bga$ is principal, it has one of the forms $\bveAi{x}{y}$ (with $x,y\in A_i)$ or $\bvrAi{R\overrightarrow{x}}$ (with $R\in\Rel(\bA_i)$ and $x_1,\dots,x_{\ari(R)}\in A_i$)\index{s}{aris@$\ari(s)$}. In the first case, $\bga\subseteq\Ker(\gy_i\circ\gp_i)$\index{s}{Kerf@$\Ker\gf$} means that $\gy_i\circ\gp_i(x)=\gy_i\circ\gp_i(y)$. Thus, by the description of the directed colimit given in Proposition~\ref{P:DirColimMIND}, there exists $j\in I\upw i$ such that $\gy_i^j\circ\gp_i(x)=\gy_i^j\circ\gp_i(y)$, that is, $\gf_i^j(x)\equiv\gf_i^j(y)\pmod{\bgq_j}$, which means, by the definition of~$\bgq_j$ together with Lemma~\ref{L:ConcvonPpalCong}(iii), that $(\ConcV\gf_i^j)(\bga)\subseteq(\ConcV\gf_i^j)(\bgb)$\index{s}{compcongVA@$\ConcV\bA$, $\ConcV f$}. The proof in the second case, that is, $\bga=\bvrAi{R\overrightarrow{x}}$, is similar.
\qed\end{proof}

\section[Ideal-induced morphisms and projectability witnesses]{Ideal-induced morphisms and projectability witnesses in generalized quasivarieties}\label{S:MINDProjWit}

In view of further monoid-theoretical applications, we shall need to expand our scope slightly beyond the one of \jzs s, and we shall thus introduce the notion of an \emph{ideal-induced homomorphism}\index{i}{homomorphism!ideal-induced} in the context of \emph{\cm s}. Then, in Theorem~\ref{T:GQV2ProjWit}, we shall relate ideal-induced homomorphisms of \jzs s and projectability witnesses\index{i}{projectability witness} with respect to the relative congruence semilattice functor.

We endow every \cm\ $\bM$ with its \emph{algebraic} preordering\index{i}{algebraic preordering|ii}, defined by
 \[
 x\leq y\Longleftrightarrow(\exists z\in M)(x+z=y)\,,\quad\text{for all }x,y\in M\,.
 \]
For \cm s~$\bM$ and~$\bN$ with~$\bN$ \emph{conical}\index{i}{monoid!conical}, we shall say that a monoid homomorphism $\gf\colon\bM\to\bN$ is
\begin{description}
\item[---]\emph{ideal-induced}\index{i}{homomorphism!ideal-induced|ii} if~$\gf$ is surjective and
 \begin{equation}\label{Eq:gfIdInd}
 (\forall x,y\in M)\bigl(\gf(x)=\gf(y)\Rightarrow
 (\exists u,v\in\gf^{-1}\set{0})(x+u=y+v)\bigr)\,;
 \end{equation}
\item[---]\emph{weakly distributive}\index{i}{homomorphism!weakly distributive|ii} if
 \begin{multline*}
 (\forall z\in M)(\forall u,v\in N)\bigl(\gf(z)\leq u+v\Rightarrow\\
 (\exists x,y\in M)(z\leq x+y\text{ and }\gf(x)\leq u\text{ and }
 \gf(y)\leq v)\bigr)\,.
 \end{multline*}
\end{description}
An \emph{o-ideal}\index{i}{ideal!o-|ii} in a \cm\ $\bM$ is a nonempty subset~$I$ of~$M$ such that $x+y\in I$ if{f} $\set{x,y}\subseteq I$, for all $x,y\in M$. (\emph{In particular, in a \jzs, the o-ideals are exactly the ideals\index{i}{ideal!of a poset} in the usual sense}.) Then we denote by~$\bM/I$\index{s}{MonI@$\bM/I$|ii} the quotient of~$\bM$ under the monoid congruence~$\equiv_I$ defined by
 \[
 x\equiv_Iy\Longleftrightarrow(\exists u,v\in I)(x+u=y+v)\,,\quad
 \text{for all }x,y\in M\,.
 \]
Furthermore, we shall write $x/I$ instead of $x/{\equiv_I}$, for any $x\in M$. Obviously, the quotient monoid~$\bM/I$ is conical\index{i}{monoid!conical}.
The ideal-induced\index{i}{homomorphism!ideal-induced} homomorphisms from a \cm~$\bM$ to a conical\index{i}{monoid!conical}~\cm\ are exactly, up to isomorphism, the canonical projections $\bM\onto\bM/I$\index{s}{AtoonB@$f\colon A\onto B$}, for o-ideals~$I$\index{i}{ideal!o-} of~$\bM$. The following result provides a large supply of ideal-induced\index{i}{homomorphism!ideal-induced} \jzh s.

\begin{lem}\label{L:ConcpiIdInd}
Let~$\cV$ be a \gqv\index{i}{generalized quasivariety}, let $\bA,\bB\in\cV$ with the same language, and let $f\colon\bA\onto\bB$\index{s}{AtoonB@$f\colon A\onto B$} be a surjective homomorphism. Then the canonical homomorphism $\ConcV\gp\colon\ConcV\bA\to\ConcV\bB$\index{s}{compcongVA@$\ConcV\bA$, $\ConcV f$} is ideal-induced\index{i}{homomorphism!ideal-induced}.
\end{lem}

\begin{proof}
By the First Isomorphism Theorem (Lemma~\ref{L:FirstIsomThm})\index{i}{Isomorphism Theorem (First ${}_{-}$)}, we may assume that $\bB=\bA/\bgq$, with $f\colon\bA\onto\bA/\bgq$\index{s}{AtoonB@$f\colon A\onto B$} the canonical projection, for some $\bgq\in\ConV\bA$\index{s}{conAV@$\ConV\bA$, $\ConV f$}.
The surjectivity of $\ConcV f$\index{s}{compcongVA@$\ConcV\bA$, $\ConcV f$} follows immediately from the surjectivity of~$f$ together with Lemma~\ref{L:ConcvonPpalCong}. Let $\bga,\bgb\in\ConcV\bA$\index{s}{compcongVA@$\ConcV\bA$, $\ConcV f$} such that $(\ConcV f)(\bga)=(\ConcV f)(\bgb)$. By Lemma~\ref{L:ConcVProj}, this means that $\bga\vee\bgq=\bgb\vee\bgq$, thus, by the compactness of both~$\bga$ and~$\bgb$, there exists a compact $\cV$-congruence\index{i}{congruencesV@$\cV$-congruences} $\bgx\leq\bgq$ such that $\bga\vee\bgx=\bgb\vee\bgx$. As $(\ConcV f)(\bgx)=0$, this proves that the map~$\ConcV f$\index{s}{compcongVA@$\ConcV\bA$, $\ConcV f$} satisfies~\eqref {Eq:gfIdInd}.
\qed\end{proof}

\begin{thm}\label{T:GQV2ProjWit}
Let~$S$ be a \jzs, let~$\cV$ be a \gqv\index{i}{generalized quasivariety}, let $\bA\in\cV$, and let~$S$ be a \jzs. Then every ideal-induced\index{i}{homomorphism!ideal-induced} \jzh\ $\gf\colon\ConcV\bA\to S$\index{s}{compcongVA@$\ConcV\bA$, $\ConcV f$} has a projectability witness\index{i}{projectability witness} with respect to the~$\ConcV$\index{s}{compcongV@$\ConcV$ functor} functor.
\end{thm}

\begin{proof}
We set $\bgq:=\bigvee\famm{\bga\in\ConcV\bA}{\gf(\bga)=0}$\index{s}{compcongVA@$\ConcV\bA$, $\ConcV f$}.  As the join defining~$\bgq$ is nonempty directed, it follows from Lemma~\ref{L:ConGQVAlg} that it may be evaluated as well in~$\Con\bA$\index{s}{conA@$\Con\bA$} as in~$\ConV\bA$\index{s}{conAV@$\ConV\bA$, $\ConV f$}, so $\bgq\in\ConV\bA$\index{s}{conAV@$\ConV\bA$, $\ConV f$}. Furthermore,\index{s}{compcongVA@$\ConcV\bA$, $\ConcV f$}
 \begin{equation}\label{Eq:Char0Kergf}
 \gf(\bga)=0\ \Longleftrightarrow\ \bga\leq\bgq\,,\quad\text{for each }
 \bga\in\ConcV\bA\,.
 \end{equation}
As $\bgq\in\ConV\bA$\index{s}{conAV@$\ConV\bA$, $\ConV f$}, the structure $\oll{\bA}:=\bA/\bgq$ belongs to~$\cV$. The canonical projection $a\colon\bA\onto\bA/\bgq$\index{s}{AtoonB@$f\colon A\onto B$} is obviously an epimorphism. It follows from Lemma~\ref{L:SecIsomThm} that\index{s}{conAV@$\ConV\bA$, $\ConV f$}
 \[
 \ConV\oll{\bA}=\setm{\bga/\bgq}{\bga\in(\ConV\bA)\upw\bgq}\,,
 \]
and thus\index{s}{compcongVA@$\ConcV\bA$, $\ConcV f$}
 \[
 \ConcV\oll{\bA}=\setm{\bga\vee\bgq/\bgq}{\bga\in\ConcV\bA}\,,
 \]
where the joins $\bga\vee\bgq$ are evaluated in $\ConcV\bA$.
Furthermore, for all $\bga,\bgb\in\ConcV\bA$\index{s}{compcongVA@$\ConcV\bA$, $\ConcV f$}, $(\bga\vee\bgq)/\bgq\leq(\bgb\vee\bgq)/\bgq$ if{f} $\bga\leq\bgb\vee\bgq$, if{f} (by the compactness of~$\bga$) there exists $\bgc\leq\bgq$ in~$\ConcV\bA$\index{s}{compcongVA@$\ConcV\bA$, $\ConcV f$} such that $\bga\leq\bgb\vee\bgc$; this implies, by~\eqref{Eq:Char0Kergf}, that $\gf(\bga)\leq\gf(\bgb)$. Conversely, if $\gf(\bga)\leq\gf(\bgb)$, then, as~$\gf$ is ideal-induced\index{i}{homomorphism!ideal-induced}, there exists~$\bgc\in\ConcV\bA$\index{s}{compcongVA@$\ConcV\bA$, $\ConcV f$} such that $\bga\leq\bgb\vee\bgc$ and $\gf(\bgc)=0$; then from~$\gf(\bgc)=0$ it follows that $\bgc\leq\bgq$, thus $(\bga\vee\bgq)/\bgq\leq(\bgb\vee\bgq)/\bgq$. As~$\gf$ is surjective, this makes it possible to define a semilattice isomorphism $\eps\colon\ConcV\oll{\bA}\to S$\index{s}{compcongVA@$\ConcV\bA$, $\ConcV f$} by the rule
 \[
 \eps\bigl(\bga\vee\bgq/\bgq\bigr):=\gf(\bga)\,,\quad
 \text{for each }\bga\in\ConcV\bA\,.
 \]
It follows immediately from Lemma~\ref{L:ConcVProj} that $\gf=\eps\circ(\ConcV a)$\index{s}{compcongVA@$\ConcV\bA$, $\ConcV f$}. In order to prove that $(a,\eps)$ is a projectability witness\index{i}{projectability witness} for $\gf\colon\ConcV\bA\to S$\index{s}{compcongVA@$\ConcV\bA$, $\ConcV f$}, it remains to prove item~(iv) of Definition~\ref{D:ProjFunct}. Let~$f\colon\bA\to\bX$ be a morphism in~$\cV$ and let $\gh\colon\ConcV\oll{\bA}\to\ConcV\bX$\index{s}{compcongVA@$\ConcV\bA$, $\ConcV f$} such that
 \begin{equation}\label{Eq:ConcVfConcVagf}
 \ConcV f=\gh\circ(\ConcV a)\,.
 \end{equation}
For all $x,y\in A$,\index{s}{congspzero@$\zero_\bA$}
 \begin{align*}
 x\equiv y\pmod{\bgq}& \Leftrightarrow a(x)=a(y)\\
 & \Leftrightarrow \bveAb{a(x)}{a(y)}=\zero_{\oll{\bA}}
 &&(\text{because }\oll{\bA}\in\cV)\\
 & \Leftrightarrow (\ConcV a)\bigl(\bveA{x}{y}\bigr)=\zero_{\oll{\bA}}
 &&(\text{use Lemma~\ref{L:ConcvonPpalCong}(iii)})\\
 & \Rightarrow (\ConcV f)\bigl(\bveA{x}{y}\bigr)=\zero_\bX
 &&(\text{use~\eqref{Eq:ConcVfConcVagf}})\\
 & \Leftrightarrow \bveX{f(x)}{f(y)}=\zero_\bX
 &&(\text{use Lemma~\ref{L:ConcvonPpalCong}(iii)})\\
 & \Leftrightarrow f(x)=f(y)
 &&(\text{because }\bX\in\cV)\,,
 \end{align*}
while for each $R\in\Rel(\bA)$ and each $\overrightarrow{x}\in A^{\ari(R)}$\index{s}{aris@$\ari(s)$},
 \begin{align*}
 R\overrightarrow{x}\pmod{\bgq}&\Leftrightarrow a(\overrightarrow{x})\in R^{\oll{\bA}}\\
 & \Leftrightarrow \bvrAb{Ra(\overrightarrow{x})}=\zero_{\oll{\bA}}
 &&(\text{because }\oll{\bA}\in\cV)\\
 & \Leftrightarrow (\ConcV a)\bigl(\bvrA{R\overrightarrow{x}}\bigr)=\zero_{\oll{\bA}}
 &&(\text{use Lemma~\ref{L:ConcvonPpalCong}(iv)})\\
 & \Rightarrow (\ConcV f)\bigl(\bvrA{R\overrightarrow{x}}\bigr)=\zero_\bX
 &&(\text{use~\eqref{Eq:ConcVfConcVagf}})\\
 & \Leftrightarrow \bvrX{Rf(\overrightarrow{x})}=\zero_\bX
 &&(\text{use Lemma~\ref{L:ConcvonPpalCong}(iv)})\\
 & \Leftrightarrow f(\overrightarrow{x})\in R^\bX
 &&(\text{because }\bX\in\cV)\,.
 \end{align*}
This proves that $\bgq\leq\Ker f$\index{s}{Kerf@$\Ker\gf$}. Therefore, by Lemma~\ref{L:FirstIsomThm}, there exists a unique morphism $g\colon\oll{\bA}\to\bX$ in~$\MIND$\index{s}{Mind@$\MIND$} such that $f=g\circ a$. By using~\eqref{Eq:ConcVfConcVagf}, it follows that $\gh\circ(\ConcV a)=\ConcV f=(\ConcV g)\circ(\ConcV a)$\index{s}{compcongVA@$\ConcV\bA$, $\ConcV f$}, thus, as~$\ConcV a$ is surjective, $\gh=\ConcV g$\index{s}{compcongVA@$\ConcV\bA$, $\ConcV f$}.
\qed\end{proof}

\section{An extension of the L\"owenheim-Skolem Theorem}
\label{S:EltaryExt}

The main result of the present section, namely Proposition~\ref{P:ApproxMonIIWD}, will be used for verifying the L\"owenheim-Skolem Condition\index{i}{Lowenheim@L\"owenheim-Skolem Condition} in the proof of Theorem~\ref{T:MindConcLift}. Roughly speaking, it says that in many situations, if $\gf\colon A\Rightarrow B$\index{s}{AtorightarrowB@$f\colon A\Rightarrow B$} is a double arrow\index{i}{double arrow}, then there are enough small substructures~$U$ of~$A$ such that $\gf\res_U\colon U\Rightarrow B$\index{s}{AtorightarrowB@$f\colon A\Rightarrow B$}.

We shall use the following easy model-theoretical lemma.

\begin{lem}\label{L:ModThEltChain}
Let~$\gl$ be an infinite cardinal, let~$\scL$ be a $\gl$-small first-order language, and let~$I$ be a $\gl$-directed\index{i}{poset!directedl@$\gl$-directed} monotone $\gs$-complete\index{i}{poset!monotone $\gs$-complete} poset. Consider a $\gs$-continuous\index{i}{diagram!$\gs$-continuous} directed colimit cocone
 \begin{equation}\label{Eq:DirSystAiMod}
 \famm{\bA,\gf_i}{i\in I}=
 \varinjlim\famm{\bA_i,\gf_i^j}{i\leq j\text{ in }I}\,,
 \end{equation}
of models for~$\scL$ and $\scL$-homomorphisms, with $\card A_i<\gl$ for each $i\in I$. Then the set
 \[
 J:=\setm{i\in I}{\gf_i\text{ is an elementary embedding}}
 \]
is a $\gs$-closed cofinal\index{i}{subset!$\gs$-closed cofinal} subset of~$I$. In particular, $J$ is nonempty.
\end{lem}

\begin{proof}
We set $\ol{I}:=\setm{i\in I}{\gf_i\text{ is an embedding}}$.

\begin{claim}
The set~$\ol{I}$ is $\gs$-closed cofinal\index{i}{subset!$\gs$-closed cofinal} in~$I$.
\end{claim}

\begin{proof}
It is obvious that~$\ol{I}$ is $\gs$-closed in~$I$.

For elements $i,j\in I$, let $i\lessdot j$ hold if $i\leq j$ and for each atomic formula $\vF(\vx_1,\dots,\vx_n)$ of~$\scL$ and all elements $x_1,\dots,x_n\in A_i$, $\bA\models\vF(\gf_i(x_1),\dots,\gf_i(x_n))$ implies that $\bA_j\models\vF(\gf_i^j(x_1),\dots,\gf_i^j(x_n))$. It follows from~\eqref{Eq:DirSystAiMod} that for each $i\in I$, each atomic formula $\vF(\vx_1,\dots,\vx_n)$ of~$\scL$, and all elements $x_1,\dots,x_n\in A_i$, there exists $j\in I\upw i$ such that $\bA\models\vF(\gf_i(x_1),\dots,\gf_i(x_n))$ implies that $\bA_j\models\vF(\gf_i^j(x_1),\dots,\gf_i^j(x_n))$. As both~$\scL$ and~$A_i$ are $\gl$-small, we obtain, using the $\gl$-directedness assumption\index{i}{poset!directedl@$\gl$-directed} on~$I$ and repeating the argument above for all atomic formulas of~$\scL$ and all lists of elements of~$A_i$, that for each $i\in I$ there exists $j\in I$ such that $i\lessdot j$.

Hence, for each $i\in I$ there exists a sequence $\famm{i_m}{m<\go}$ of elements of~$I$ such that $i_0=i$ and $i_m\lessdot i_{m+1}$ for each $m<\go$. Set $j:=\bigvee\famm{i_m}{m<\go}$. Let $\vF(\vx_1,\dots,\vx_n)$ be an atomic formula of~$\scL$ and let $x_1,\dots,x_n\in A_j$ such that $\bA\models\vF(\gf_j(x_1),\dots,\gf_j(x_n))$. As the colimit~\eqref{Eq:DirSystAiMod} is $\gs$-continuous\index{i}{diagram!$\gs$-continuous}, $\bA_j=\varinjlim_{m<\go}\bA_{i_m}$, thus there are $m<\go$ and $y_1,\dots,y_n\in A_{i_m}$ such that $x_s=\gf_{i_m}^j(y_s)$ for each $s\in\set{1,\dots,n}$; hence
$\bA\models\vF(\gf_{i_m}(y_1),\dots,\gf_{i_m}(y_n))$. As $i_m\lessdot i_{m+1}$, we obtain that
 \[
 \bA_{i_{m+1}}\models
 \vF(\gf_{i_m}^{i_{m+1}}(y_1),\dots,\gf_{i_m}^{i_{m+1}}(y_n))\,,
 \]
and hence, applying the homomorphism~$\gf_{i_{m+1}}^j$, we obtain $\bA_j\models\vF(x_1,\dots,x_n)$. This completes the proof that~$j$ belongs to~$\ol{I}$. Therefore, $\ol{I}$ is cofinal in~$I$.
\qed\ Claim\end{proof}

As all $\gf_i$, for $i\in\ol{I}$, are embeddings, we may assume that they are inclusion maps and thus that $\bA=\bigcup\famm{\bA_i}{i\in\ol{I}}$. Now the statement that~$J$ is a $\gs$-closed subset of~$\ol{I}$ follows immediately from the Elementary Chain Theorem \cite[Theorem~3.1.9]{ChKe}\index{c}{Chang, C.\,C.}\index{c}{Keisler, H.\,J.}.

Further, let~$\scL^*$ be a Skolem expansion of~$\scL$ and let~$\bA^*$ be a corresponding Skolem expansion of the model~$\bA$ (cf. \cite[Section~3.3]{ChKe}\index{c}{Chang, C.\,C.}\index{c}{Keisler, H.\,J.}). For each $i\in\ol{I}$, we construct a sequence $\famm{i_m}{m<\go}$ of indices by setting $i_0:=i$ and letting~$i_{m+1}$ be any index in~$\ol{I}\upw i_m$ such that $A_{i_{m+1}}$ contains the Skolem hull of~$A_{i_m}$. The element $j:=\bigvee\famm{i_m}{m<\go}$ belongs to~$\ol{I}$ and~$A_j=\bigcup\famm{A_{i_m}}{m<\go}$ is its own Skolem hull, thus it is an elementary submodel of~$\bA$; so~$j\in J$. This proves that~$J$ is cofinal in~$\ol{I}$, and thus also in~$I$.
\qed\end{proof}

We now present a few simple monoid-theoretical applications of Lemma~\ref{L:ModThEltChain}.

\begin{prop}\label{P:ApproxMonIIWD}
Let~$\gl$ be an infinite cardinal, let~$I$ be a $\gl$-directed\index{i}{poset!directedl@$\gl$-directed} monotone $\gs$-complete\index{i}{poset!monotone $\gs$-complete} poset. 
Let~$\bM$ and~$\bN$ be \cm s and let $\gf\colon\bM\to\nobreak\bN$ be a monoid homomorphism. Consider a $\gs$-continuous\index{i}{diagram!$\gs$-continuous} directed colimit cocone
 \begin{equation}\label{Eq:DirSystMiIIWD}
 \famm{\bM,\gt_i}{i\in I}=
 \varinjlim\famm{\bM_i,\gt_i^j}{i\leq j\text{ in }I}\,,
 \end{equation}
with $\card M_i<\gl$ for each $i\in I$ and $\card N<\gl$.
Then the following statements hold:
\begin{description}
\item[\tui] If $\gf$ is ideal-induced\index{i}{homomorphism!ideal-induced}, then $I_{\mathrm{id}}:=\setm{i\in I}{\gf\circ\gt_i\text{ is ideal-induced}}$ is $\gs$-closed cofinal\index{i}{subset!$\gs$-closed cofinal} in~$I$;

\item[\tuii] If $\gf$ is weakly distributive\index{i}{homomorphism!weakly distributive}, then $I_{\mathrm{wd}}:=\setm{i\in I}{\gf\circ\gt_i\text{ is weakly distributive}}$ is $\gs$-closed cofinal\index{i}{subset!$\gs$-closed cofinal} in~$I$.
\end{description}
\end{prop}

\begin{snote}
It is obvious that if $\gf$ is surjective, then the set
 \[
 I_{\mathrm{surj}}:=\setm{i\in I}{\gf\circ\gt_i\text{ is surjective}}
 \]
is $\gs$-closed cofinal\index{i}{subset!$\gs$-closed cofinal} in~$I$. Indeed, as~$I$ is $\gl$-directed\index{i}{poset!directedl@$\gl$-directed} and~$N$ is $\gl$-small, there exists $i\in I$ such that $\gf\circ\gt_i$ is surjective, and then $\gf\circ\gt_j$ is surjective for each $j\geq i$.
\end{snote}

\begin{proof}
As~\eqref{Eq:DirSystMiIIWD} is a $\gs$-continuous\index{i}{diagram!$\gs$-continuous} directed colimit cocone, it is straightforward to verify that the sets~$I_{\mathrm{id}}$ and~$I_{\mathrm{wd}}$ are both $\gs$-closed in~$I$. Now, in order to apply Lemma~\ref{L:ModThEltChain}, we encode every monoid homomorphism $\gy\colon\bX\to\bY$ (for \cm s~$\bX$ and $\bY$) by a single first-order structure~$\bA_\gy$, in such a way that~$\gy$ being ideal-induced\index{i}{homomorphism!ideal-induced} (resp., weakly distributive)\index{i}{homomorphism!weakly distributive} can be expressed by a first-order sentence. In order to do this, we define a new language~$\scL$ by
 \[
 \scL:=\set{+,0_{\vX},0_{\vY},\vX,\vY,\vR}\,,
 \]
where~$+$ is a binary operation symbol, $\vX$ and~$\vY$ are both (unary) predicate symbols, $\vR$ is a binary relation symbol, and~$0_{\vX}$ and~$0_{\vY}$ are both constant symbols. Define the universe of~$\bA_\gy$ as $(X\times\set{0})\cup(Y\times\set{1})$, interpret~$\vX$ by $X\times\set{0}$, $\vY$ by~$Y\times\set{1}$, and~$\vR$ by the set of all pairs of the form $((x,0),(\gy(x),1))$ for $x\in X$; furthermore, define the addition on~$\bA_\gy$ by setting
 \begin{align*}
 (x_0,0)+(x_1,0)&:=(x_0+x_1,0)\,,&&\text{for all }x_0,x_1\in X\,,\\
 (y_0,1)+(y_1,1)&:=(y_0+y_1,1)\,,&&\text{for all }y_0,y_1\in Y\,,\\
 (x,0)+(y,1)=(y,1)+(x,0)&:=(0_X,0)\,,&&\text{for all }(x,y)\in X\times Y\,.
 \end{align*}
Finally, we interpret $0_{\vX}$ by $(0_X,0)$ and~$0_{\vY}$ by $(0_Y,1)$. Then~$\gy$ being surjective is equivalent to~$\bA_\gy$ satisfying the first-order statement
 \[
 (\forall\vy)\bigl(\vY(\vy)\Rightarrow(\exists\vx)
 (\vX(\vx)\text{ and }\vR(\vx,\vy))\bigr)\,,
 \]
while $\gy$ being ideal-induced\index{i}{homomorphism!ideal-induced} is equivalent to~$\gy$ being surjective together with the first-order statement
 \begin{multline*}
 (\forall\vx_0,\vx_1,\vy)\Bigl(\bigl(\vX(\vx_0)\text{ and }\vX(\vx_1)\text{ and }\vY(\vy)
 \text{ and }\vR(\vx_0,\vy)\text{ and }\vR(\vx_1,\vy)\bigr)\Rightarrow\\
 (\exists\vu_0,\vu_1)\bigl(\vX(\vu_0)\text{ and }\vX(\vu_1)\text{ and }
 \vR(\vu_0,0_{\vY})\text{ and }\vR(\vu_1,0_{\vY})
 \text{ and }\vx_0+\vu_0=\vx_1+\vu_1\bigr)\Bigr)\,.
 \end{multline*}
Likewise, it is straightforward to verify that~$\gy$ being weakly distributive\index{i}{homomorphism!weakly distributive} is also a first-order property of~$\bA_\gy$.

Now assume, for example, that~$\gf$ is ideal-induced\index{i}{homomorphism!ideal-induced}. As $\bA_\gf=\varinjlim_{i\in I}\bA_{\gf\circ\gt_i}$ (with canonical transition morphisms and limiting morphisms), it follows from Lemma~\ref{L:ModThEltChain} that the set~$J$ of all $i\in I$ such that~$\gt_i$ defines an elementary embedding from $\bA_{\gf\circ\gt_i}$ to~$\bA_\gf$ is cofinal in~$I$. By the discussion above, this set contains~$I_{\mathrm{id}}$; hence the latter is also cofinal in~$I$. The argument for~$I_{\mathrm{wd}}$ is similar.
\qed\end{proof}

Of course, it would have been about as easy to prove Proposition~\ref{P:ApproxMonIIWD} directly, in each of the cases~(i) and~(ii). However, we hope that the above approach clearly illustrates the fact that many such statements can be obtained immediately from Lemma~\ref{L:ModThEltChain}.

\section{A diagram version of the Gr\"atzer-Schmidt Theorem}\label{S:GQV2Lard}

In the present section we shall establish a diagram extension of the Gr\"atzer-Schmidt\index{c}{Gr\"atzer, G.}\index{c}{Schmidt, E.\,T.} Theorem, namely Theorem~\ref{T:MindConcLift}.

We remind the reader that~$\SEM$\index{s}{Sem@$\SEM$} denotes the category of all \jzs s with \jzh s.
We also denote by~$\SEM^{\mathrm{surj}}$\index{s}{Sems@$\SEM^{\mathrm{surj}}$|ii} ($\SEM^{\mathrm{idl}}$\index{s}{Semi@$\SEM^{\mathrm{idl}}$|ii}, $\SEM^{\mathrm{wd}}$\index{s}{Semw@$\SEM^{\mathrm{wd}}$|ii}, respectively) the subcategory of~$\SEM$\index{s}{Sem@$\SEM$} whose objects are all \jzs s and whose morphisms are all surjective (ideal-induced\index{i}{homomorphism!ideal-induced}, weakly distributive\index{i}{homomorphism!weakly distributive}, respectively) \jzh s. The proof of the following lemma is a straightforward exercise. (We refer to Definition~\ref{D:ClosDirColim} for subcategories closed under small directed colimits.)

\begin{lem}\label{L:SEMvarClosColim}
The subcategories $\SEM^{\mathrm{surj}}$\index{s}{Sems@$\SEM^{\mathrm{surj}}$}, $\SEM^{\mathrm{idl}}$\index{s}{Semi@$\SEM^{\mathrm{idl}}$}, and $\SEM^{\mathrm{wd}}$\index{s}{Semw@$\SEM^{\mathrm{wd}}$} are closed under all small directed colimits within~$\SEM$\index{s}{Sem@$\SEM$} \pup{cf. Definition~\textup{\ref{D:ClosDirColim}}}.
\end{lem}

An \emph{algebra}\index{i}{algebra!universal} is a first-order structure~$\bA$ such that $\Rel(\bA)=\es$. The Gr\"atzer-Schmidt\index{c}{Gr\"atzer, G.}\index{c}{Schmidt, E.\,T.} Theorem~\cite[Theorem~10]{GrSc62}\index{c}{Gr\"atzer, G.}\index{c}{Schmidt, E.\,T.} states that every \jzs\ is isomorphic to~$\Conc\bA$\index{s}{compcon1@$\Conc\bA$, $\Conc f$}, for some algebra~$\bA$\index{i}{algebra!universal}. As the algebra~$\bA$ has the same congruences as the algebra\index{i}{algebra!universal} consisting of the universe of~$\bA$ endowed with the unary polynomials of~$\bA$, we may assume that~$\bA$ is a unary algebra\index{i}{algebra!unary}. Denote by $\MALG_1$\index{s}{Malg1@$\MALG_1$|ii} the full subcategory of~$\MIND$\index{s}{Mind@$\MIND$} consisting of all unary algebras\index{i}{algebra!unary}. Hence $\MALG_1$\index{s}{Malg1@$\MALG_1$} is a \gqv\index{i}{generalized quasivariety}\ (cf. Example~\ref{Ex:1/2nonindexed}(ii)). We shall call it the category of \emph{unary monotone-indexed algebras}\index{i}{algebra!monotone-indexed unary|ii}. The following result is a diagram extension of the Gr\"atzer-Schmidt\index{c}{Gr\"atzer, G.}\index{c}{Schmidt, E.\,T.} Theorem.

\begin{thm}\label{T:MindConcLift}
Let $P$ be a poset and let~$\bS=\famm{\bS_p,\gs_p^q}{p\leq q\text{ in }P}$ be a $P$-indexed diagram of \jzs s and \jzh s. If either~$P$ is finite or there are cardinals~$\gk$ and~$\gl$, with~$\gl$ regular, such that~$P$ and all~$S_p$, for $p\in P$, are $\gl$-small and the relation $(\gk,{<}\go,\gl)\rightarrow\gl$\index{s}{arr0free@$(\gk,{<}\go,\gl)\rightarrow\gr$} holds, then every $P$-indexed diagram of \jzs s and \jzh s can be lifted\index{i}{diagram!lifted}, with respect to the~$\Conc$\index{s}{compcon@$\Conc$ functor} functor, by some diagram of unary monotone-indexed algebras\index{i}{algebra!monotone-indexed unary}.
\end{thm}

\begin{proof}
Denote by $\ol{P}$ the \jzs\ of all finitely generated\index{i}{finitely generated!lower subset} lower subsets\index{i}{lower subset} of~$P$. As the category~$\SEM$\index{s}{Sem@$\SEM$} has all (not necessarily directed) colimits, we can extend the diagram~$\overrightarrow{\bS}$ to a $\ol{P}$-indexed diagram ${\overrightarrow{\bS}^*}$ in~$\SEM$\index{s}{Sem@$\SEM$} by setting
 \[
 \bS^*_X:=\varinjlim\famm{\bS_p}{p\in X}\,,\quad\text{for each }X\in\ol{P}\,,
 \]
with the canonical transition maps and limiting maps. Furthermore, $\ol{P}$ is finite in case~$P$ is finite, and $\card\ol{P}\leq\card P+\aleph_0$\index{s}{aleph0@$\aleph_{\ga}$}. In particular, the assumptions made for $P$, $\overrightarrow{\bS}$ remain valid for $\ol{P}$, ${\overrightarrow{\bS}^*}$, and so we may assume from the start that~$P$ is a \jzs. Set~$\gl:=\aleph_1$\index{s}{aleph0@$\aleph_{\ga}$} in case~$P$ is finite (in that case we require no relation of the form $(\gk,{<}\go,\gl)\rightarrow\gl$\index{s}{arr0free@$(\gk,{<}\go,\gl)\rightarrow\gr$}). We denote by $\MALG_1^{(\gl)}$\index{s}{Malg1l@$\MALG_1^{(\gl)}$|ii} (resp., $\SEM^{(\gl)}$)\index{s}{Seml@$\SEM^{(\gl)}$|ii} the class of all completely $\gl$-small unary algebras\index{i}{algebra!unary} (resp., the class of all $\gl$-small \jzs s).

\setcounter{nclaim}{0}
\begin{nclaim}\label{Cl:GrSchleftlard}
Denote by~$\Phi$ the identity functor on~$\SEM$\index{s}{Sem@$\SEM$}. Then the quadruple\index{s}{Semi@$\SEM^{\mathrm{idl}}$}
 \[
 (\SEM,\SEM,\SEM^{\mathrm{idl}},\Phi)
 \]
is a left larder\index{i}{larder!left}.
\end{nclaim}

\begin{proof}
The only not completely trivial statement that we need to verify is $(\PROJ(\Phi,\SEM^{\mathrm{idl}}))$\index{s}{Proj@$(\PROJ(\Phi,\cS^\Rightarrow))$}, that is, every extended projection of~$\SEM$\index{s}{Sem@$\SEM$} is ideal-induced\index{i}{homomorphism!ideal-induced}. As every projection in~$\SEM$\index{s}{Sem@$\SEM$} is, trivially, ideal-induced\index{i}{homomorphism!ideal-induced}, this follows immediately from Lemma~\ref{L:SEMvarClosColim}.
\qed\ Claim~\ref{Cl:GrSchleftlard}\end{proof}

\begin{nclaim}\label{Cl:GrSchrightlard}
The $6$-uple $(\MALG_1,\MALG_1^{(\gl)},\SEM,\SEM^{(\gl)},\SEM^{\mathrm{idl}},\Conc)$\index{s}{compcon@$\Conc$ functor}\index{s}{Sem@$\SEM$}\index{s}{Seml@$\SEM^{(\gl)}$}\index{s}{Semi@$\SEM^{\mathrm{idl}}$}\index{s}{Malg1l@$\MALG_1^{(\gl)}$}\index{s}{Malg1@$\MALG_1$} is a projectable right $\gl$-larder\index{i}{larder!right!projectable}.
\end{nclaim}

\begin{proof}
For each $\bA\in\MALG_1^{(\gl)}$\index{s}{Malg1l@$\MALG_1^{(\gl)}$}, the semilattice $\Conc\bA$\index{s}{compcon1@$\Conc\bA$, $\Conc f$} is $\gl$-small, thus, by Proposition~\ref{P:glPresMIND}, it is weakly $\gl$-presented\index{i}{presented!weakly $\gl$-} in~$\SEM$\index{s}{Sem@$\SEM$}. Hence the condition $(\PRES_\gl(\MALG_1^{(\gl)},\Conc))$\index{s}{Pres@$(\PRES_\gl(\cB^\dagger,\Psi))$} is satisfied.

Let~$\bB$ be a unary algebra\index{i}{algebra!unary}, we must verify that $(\LSr_\gl(\bB))$\index{s}{LSr@$(\LSr_\gm(B))$} is satisfied. Let~$\bS$ be a $\gl$-small \jzs, let~$I$ be a $\gl$-small set, and let $\famm{u_i\colon\bU_i\Rightarrow\bB}{i\in I}$\index{s}{AtorightarrowB@$f\colon A\Rightarrow B$} be a family of double arrows\index{i}{double arrow} with all the~$\bU_i$ completely $\gl$-small. As at the end of the proof of Proposition~\ref{P:glPresMIND}, we set $\Omega:=B\cup\Lg(\bB)$ (disjoint union) and we express~$\bB$ as a continuous\index{i}{diagram!continuous} directed colimit\index{s}{Malg1@$\MALG_1$}
 \[
 \bB=\varinjlim_{X\in[\Omega]^{<\gl}}{\bB_X}\quad\text{in }\MALG_1\,,
 \]
with all maps in the cocone being the corresponding inclusion maps. Then applying Theorem~\ref{T:ConcVPresDirColim} to the \gqv\index{i}{generalized quasivariety}\ $\cV:=\MALG_1$\index{s}{Malg1@$\MALG_1$} yields a continuous\index{i}{diagram!continuous} directed colimit\index{s}{compcon1@$\Conc\bA$, $\Conc f$}\index{s}{Sem@$\SEM$}
 \begin{equation}\label{Eq:ConcbBGS}
 \Conc\bB=\varinjlim_{X\in[\Omega]^{<\gl}}{\Conc\bB_X}\quad\text{in }\SEM\,.
 \end{equation}
(Due to an earlier introduced convention, we omit the~$\cV$ superscript in the~$\ConcV$ notation, as, here, $\cV$ is closed under homomorphic images.)
Denote by $\gb_X\colon\bB_X\into\nobreak\bB$\index{s}{AtoinB@$f\colon A\into B$} the inclusion map, for each $X\in[\Omega]^{<\gl}$. It follows from~\eqref{Eq:ConcbBGS} together with Proposition~\ref{P:ApproxMonIIWD} that the set
 \[
 J:=\setm{X\in[\Omega]^{<\gl}}{\gf\circ\Conc\gb_X\text{ is ideal-induced}}
 \]
is $\gs$-closed cofinal\index{i}{subset!$\gs$-closed cofinal} in~$[\Omega]^{<\gl}$. As there exists $X\in[\Omega]^{<\gl}$ such that~$\bB_X$ contains (as well for the universes as for the languages) all images $u_i``(\bU_i)$, it follows that such an~$X$ can be chosen in~$J$. Then~$\gb_X$ is monic, $u_i\utr\gb_X$ for each $i\in I$, and $\gf\circ\Conc\gb_X$\index{s}{compcon1@$\Conc\bA$, $\Conc f$} is ideal-induced\index{i}{homomorphism!ideal-induced}.

The projectability statement follows immediately from Theorem~\ref{T:GQV2ProjWit}.
\qed\ Claim~\ref{Cl:GrSchrightlard}\end{proof}

By the two claims above, it follows from Proposition~\ref{P:LR2Larder} (with $\cA^\dagger:=\SEM^{(\gl)}$\index{s}{Seml@$\SEM^{(\gl)}$}) that the $8$-uple\index{s}{Sem@$\SEM$}\index{s}{Semi@$\SEM^{\mathrm{idl}}$}\index{s}{Malg1l@$\MALG_1^{(\gl)}$}\index{s}{Malg1@$\MALG_1$}
 \[
 (\SEM,\MALG_1,\SEM,\SEM^{(\gl)},\MALG_1^{(\gl)},\SEM^{\mathrm{idl}},
 \Phi,\Conc)
 \]
is a $\gl$-larder\index{i}{larder}\index{s}{compcon@$\Conc$ functor}.

Now we first assume that~$P$ is finite. By Corollary~\ref{C:LiftFinPos}, there are $n<\go$ and a $\gl$-lifter\index{i}{lifter ($\gl$-)} $(X,\bX)$ of~$P$ such that $\card X=\gl^{+n}$. By the Gr\"atzer-Schmidt\index{c}{Gr\"atzer, G.}\index{c}{Schmidt, E.\,T.} Theorem, there exists $\bB\in\MALG_1$\index{s}{Malg1@$\MALG_1$} such that\index{s}{compcon1@$\Conc\bA$, $\Conc f$}\index{s}{FxX@$\xF(X)$}\index{s}{otimAS@$\bA\otimes\overrightarrow{S}$, $\gf\otimes\overrightarrow{S}$}
 \[
 \Conc\bB\cong\xF(X)\otimes\overrightarrow{\bS}\,.
 \]
By CLL\index{i}{Condensate Lifting Lemma (CLL)} (Lemma~\ref{L:CLL}), there exists a $P$-indexed diagram~$\overrightarrow{\bB}$ in~$\MALG_1$\index{s}{Malg1@$\MALG_1$} such that $\Conc\overrightarrow{\bB}\cong\overrightarrow{\bS}$\index{s}{compcon1@$\Conc\bA$, $\Conc f$}; that is, $\overrightarrow{\bS}$ can be lifted\index{i}{diagram!lifted} with respect to the~$\Conc$\index{s}{compcon@$\Conc$ functor} functor.

Finally assume that~$P$ is infinite. We do not have any theorem that ensures the existence of a $\gl$-lifter\index{i}{lifter ($\gl$-)} of~$P$, so instead of using CLL\index{i}{Condensate Lifting Lemma (CLL)} we invoke Corollary~\ref{C:CLLnoLF}, which yields again a morphism $\overrightarrow{\chi}\colon\Conc\overrightarrow{\bB}\Todot\overrightarrow{\bS}$\index{s}{AtoRightarrowdotB@$f\colon\xA\Todot\xB$}\index{s}{compcon1@$\Conc\bA$, $\Conc f$} for some diagram~$\overrightarrow{\bB}$ in~$\MALG_1$\index{s}{Malg1@$\MALG_1$}. We conclude the proof as above, by using projectability.
\qed\end{proof}

By using the comments preceding Corollary~\ref{C:CLLnoLF}, we obtain, in particular, the following extension of the Gr\"atzer-Schmidt\index{c}{Gr\"atzer, G.}\index{c}{Schmidt, E.\,T.} Theorem to diagrams of algebras.

\begin{cor}\label{C:MindConcLift}
Assume that the class of all Erd\H os\index{c}{Erdos@Erd\H os, P.} cardinals\index{i}{Erd\H{o}s cardinal} is proper. Then every poset-indexed diagram of \jzs s and \jzh s can be lifted\index{i}{diagram!lifted}, with respect to the~$\Conc$\index{s}{compcon@$\Conc$ functor} functor, by some diagram of monotone-indexed unary algebras\index{i}{algebra!monotone-indexed unary}.
\end{cor}

A similar proof, with~$\MALG_1$\index{s}{Malg1@$\MALG_1$} replaced by the \gqv\index{i}{generalized quasivariety}\ of all groupoids\index{i}{groupoid}, yields the following result, that extends Lampe's result from~\cite{Lamp82}\index{c}{Lampe, W.\,A.} (and its extension to one arrow in Lampe~\cite{Lamp05}\index{c}{Lampe, W.\,A.}) that every \jzus\ is isomorphic to the compact congruence semilattice of some groupoid\index{i}{groupoid}. This result was first established by the first author for diagrams indexed by finite posets, see Gillibert \cite[Corollary~7.10]{Gill1}\index{c}{Gillibert, P.}.

\begin{prop}\label{P:GpdConcLift}
Every diagram of \jzus s and \jzuh s, indexed by a finite poset, can be lifted\index{i}{diagram!lifted}, with respect to the~$\Conc$\index{s}{compcon@$\Conc$ functor} functor, by some diagram of groupoids\index{i}{groupoid}. Furthermore, if the class of all Erd\H os\index{c}{Erdos@Erd\H os, P.} cardinals\index{i}{Erd\H{o}s cardinal} is proper, then the finiteness assumption on the indexing poset is not needed.
\end{prop}

\begin{remk}\label{Rk:GpdConcLift}
By replacing $\MALG_1$\index{s}{Malg1@$\MALG_1$} and the category of all groupoids\index{i}{groupoid} by their full subcategories of members with $4$-permutable congruence lattices (i.e., $\ga\vee\gb=\ga\gb\ga\gb$ for all congruences~$\ga$ and~$\gb$ of the structure; the definition of \emph{$m$-permutable congruence lattice}\index{i}{permut@$m$-permutable congruence lattice|ii}, for an integer $m\geq2$, is similar), it is furthermore possible to strengthen Theorem~\ref{T:MindConcLift}, Corollary~\ref{C:MindConcLift}, and Proposition~\ref{P:GpdConcLift} by requiring all the unary algebras\index{i}{algebra!unary} (resp., groupoids\index{i}{groupoid}) to have $4$-permutable\index{i}{permut@$m$-permutable congruence lattice} congruence lattices. We need the easily verified fact that every homomorphic image of a congruence $4$-permutable algebra\index{i}{algebra!universal} is congruence $4$-permutable, together with the statement that every congruence $4$-permutable algebra\index{i}{algebra!universal} is a continuous\index{i}{diagram!continuous} directed colimit, in~$\MIND$\index{s}{Mind@$\MIND$}, indexed by some $[\Omega]^{<\gl}$, of $\gl$-small congruence $4$-permutable algebras\index{i}{algebra!universal}. The latter fact is established by proving an easy refinement of the proof of the final statement of Proposition~\ref{P:glPresMIND}. This in turn is used to prove the L\"owenheim-Skolem Condition\index{i}{Lowenheim@L\"owenheim-Skolem Condition} in an analogue of Theorem~\ref{T:MindConcLift} for congruence $4$-permutable\index{i}{permut@$m$-permutable congruence lattice} algebras\index{i}{algebra!universal}. We believe that the reader who beared with us up to now will have no difficulty in supplying the missing details.
\end{remk}

\section[Right $\aleph_0$-larders from first-order structures]{Right $\aleph_0$-larders from congruence-proper quasivarieties}\label{S:Rightal0lardGQV}

One of the caveats in extending Theorem~\ref{T:MindConcLift} to the case where $\gl=\aleph_0$\index{s}{aleph0@$\aleph_{\ga}$} is the impossibility to extend Proposition~\ref{P:ApproxMonIIWD} to that case. Hence, in order to obtain right $\aleph_0$-larders\index{s}{aleph0@$\aleph_{\ga}$}\index{i}{larder!right}, we shall need to strengthen the assumptions over the \gqv\index{i}{generalized quasivariety}~$\cV$. As the applications that we are having in mind deal with structures in a fixed language, we shall deal with \emph{quasivarieties}. This will lead to Theorem~\ref{T:1stordLardCtble}. We prove in Propositions~\ref{P:FinResBdCongPp} and~\ref{P:ResBd2LocFin} that the assumption of that theorem, namely that~$\cV$ be congruence-proper\index{i}{congruence-proper} and locally finite\index{i}{quasivariety!locally finite}, is verified in case~$\cV$ is a finitely generated\index{i}{quasivariety!finitely generated} quasivariety.

\begin{defn}\label{D:ResFin}
We say that a quasivariety~$\cV$ is
\begin{itemize}
\item \emph{congruence-proper}\index{i}{congruence-proper|ii} if $\ConV\bA$\index{s}{conAV@$\ConV\bA$, $\ConV f$} finite implies that~$\bA$ is finite \pup{i.e., it has finite universe}, for each $\bA\in\cV$;

\item \emph{strongly congruence-proper}\index{i}{strongly congruence-proper|ii} if it is congruence-proper and for each finite \jzs~$\bS$ there are only finitely many (up to isomorphism) $\bA\in\nobreak\cV$ such that $\ConV\bA\cong\bS$\index{s}{conAV@$\ConV\bA$, $\ConV f$}.
\end{itemize}
\end{defn}

We denote by~$\cV^\fin$\index{s}{Vf@$\cV^{\mathrm{fin}}$|ii} the class of all finite structures (i.e., structures with finite universe) in a quasivariety~$\cV$. In particular, $\SEM^\fin$\index{s}{Semf@$\SEM^{\mathrm{fin}}$|ii} is the class of all finite \jzs s. As usual, a first-order structure~$\bA$ is \emph{locally finite}\index{i}{algebraic system!locally finite|ii} if the subuniverse of~$\bA$ generated by any finite subset of~$A$ is finite, and a quasivariety~$\cV$ is locally finite\index{i}{quasivariety!locally finite|ii} if every member of~$\bA$ is locally finite.

\begin{thm}\label{T:1stordLardCtble}
Let~$\cV$ be a congruence-proper\index{i}{congruence-proper} and locally finite\index{i}{quasivariety!locally finite} quasivariety on a first-order language~$\scL$ with only finitely many relation symbols. Then the $6$-uple\index{s}{compcongV@$\ConcV$ functor}\index{s}{Semi@$\SEM^{\mathrm{idl}}$}\index{s}{Sem@$\SEM$}\index{s}{Semf@$\SEM^{\mathrm{fin}}$}\index{s}{Vf@$\cV^{\mathrm{fin}}$}
 \[
 (\cV,\cV^\fin,\SEM,\SEM^\fin,\SEM^{\mathrm{idl}},\ConcV)
 \]
is a projectable right $\aleph_0$-larder\index{s}{aleph0@$\aleph_{\ga}$}\index{i}{larder!right}\index{i}{larder!right!projectable}.
\end{thm}

\begin{proof}
The statement $(\PRES_{\aleph_0}(\cV^\fin,\ConcV))$\index{s}{Pres@$(\PRES_\gl(\cB^\dagger,\Psi))$} follows from the finiteness of the collection of all relation symbols of~$\cV$, while the projectability statement follows from Theorem~\ref{T:GQV2ProjWit}.

We must verify $(\LSr_{\aleph_0}(\bB))$\index{s}{LSr@$(\LSr_\gm(B))$}, for~$\bB\in\cV$. Let~$\bS$ be a finite \jzs, let $\gf\colon\ConcV\bB\to\bS$\index{s}{compcongVA@$\ConcV\bA$, $\ConcV f$} be ideal-induced\index{i}{homomorphism!ideal-induced}, let~$n<\go$, let~$\bU_0$, \dots, $\bU_{n-1}$ be finite members of~$\cV$, and let $\famm{u_i\colon\bU_i\to\bB}{i<n}$ be a finite sequence of $\scL$-morphisms (for what follows we shall not require the~$u_i$s be monic). We set
 \[
 \bgq:=\bigvee\famm{\bgb\in\ConcV\bB}{\gf(\bgb)=0}\,.
 \]
As in the proof of Theorem~\ref{T:GQV2ProjWit}, $\bgq$ is a $\cV$-congruence\index{i}{congruencesV@$\cV$-congruences} of~$\bB$ and, setting $\bC:=\bB/\bgq$, there exists a unique isomorphism $\eps\colon\ConcV\bC\to\bS$\index{s}{compcongVA@$\ConcV\bA$, $\ConcV f$} such that $\eps(\bgb\vee\bgq/\bgq)=\gf(\bgb)$ for each $\bgb\in\ConcV\bB$\index{s}{compcongVA@$\ConcV\bA$, $\ConcV f$}. In particular, as $\ConcV\bC$\index{s}{compcongVA@$\ConcV\bA$, $\ConcV f$} is finite and~$\cV$ is congruence-proper\index{i}{congruence-proper}, $C$ is finite, thus so is~$B/\gq$. It follows that there exists a finite subset~$F$ of~$B$ such that
 \begin{equation}\label{Eq:FfillsB}
 B/\gq=\setm{x/\gq}{x\in F}\,.
 \end{equation}
As~$\cV$ is locally finite\index{i}{quasivariety!locally finite}, the subuniverse~$V$ of~$\bB$ generated by the subset\linebreak $F\cup\bigcup\famm{u_i``(U_i)}{i<n}$ is finite. Denote by~$\bV$ the corresponding member of~$\cV$ and by~$v\colon V\into B$\index{s}{AtoinB@$f\colon A\into B$} the inclusion map. In particular, $v$ is monic and $u_i\utr v$ for each $i<n$.

Set $\bgh:=(\Res v)(\bgq)$\index{s}{Resf@$\Res\gf$}, and denote by $\gp\colon\bB\onto\nobreak\bB/\bgq$\index{s}{AtoonB@$f\colon A\onto B$} and $\gr\colon\bV\onto\bV/\bgh$\index{s}{AtoonB@$f\colon A\onto B$} the canonical projections. It follows from Lemma~\ref{L:Ker2Res} that $\Ker(\gp\circ v)=(\Res v)(\Ker\gp)=(\Res v)(\bgq)=\bgh$\index{s}{Resf@$\Res\gf$}\index{s}{Kerf@$\Ker\gf$}, and thus, by the First Isomorphism Theorem (Lemma~\ref{L:FirstIsomThm})\index{i}{Isomorphism Theorem (First ${}_{-}$)}, there exists a unique embedding $\ol{v}\colon\bV/\bgh\into\bB/\bgq$\index{s}{AtoinB@$f\colon A\into B$} such that $\gp\circ v=\ol{v}\circ\gr$, as illustrated on the left hand side of Figure~\ref{Fig:gfConcVwid}.

\begin{figure}[htb]\index{s}{compcongVA@$\ConcV\bA$, $\ConcV f$}
 \[
 \def\labelstyle{\displaystyle}
 \xymatrix{
 \bV\ar@{_(->}[r]^v\ar@{->>}[d]_{\gr} & \bB\ar@{->>}[d]^{\gp} && \ConcV\bV
 \ar[rr]^{\ConcV v}\ar@{=>}[d]_{\ConcV\gr} && \ConcV\bB
 \ar@{=>}[d]_{\ConcV\gp}\ar@{=>}[rd]^{\gf}\\
 \bV/\bgh\ar@{^(->}[r]^{\ol{v}} & \bB/\bgq &&
 \ConcV(\bV/\bgh)\ar[rr]^{\ConcV\ol{v}}&&\ConcV(\bB/\bgq)
 \ar[r]_(.7){\cong}^(.6){\eps} & S
 }
 \]
\caption{Proving that $\gf\circ\ConcV v$ is ideal-induced}
\label{Fig:gfConcVwid}
\end{figure}

Now we consider the commutative diagram represented on the right hand side of Figure~\ref{Fig:gfConcVwid}. It follows from Lemma~\ref{L:ConcpiIdInd} that both morphisms~$\ConcV\gr$ and~$\ConcV\gp$\index{s}{compcongVA@$\ConcV\bA$, $\ConcV f$} are ideal-induced\index{i}{homomorphism!ideal-induced}; we mark ideal-induced\index{i}{homomorphism!ideal-induced} morphisms by double arrows\index{i}{double arrow} on Figure~\ref{Fig:gfConcVwid}. Furthermore, for each $b\in B$, there exists, by~\eqref{Eq:FfillsB}, $x\in F$ such that $x/\gq=b/\gq$. As~$F$ is contained in~$V$, the function~$\ol{v}$ is defined at~$x/\gh$ and $\ol{v}(x/\gh)=v(x)/\gq=x/\gq=b/\gq$, and hence~$\ol{v}$ is surjective. As~$\ol{v}$ is also an embedding, it is an isomorphism. Hence the map $\ConcV\ol{v}$\index{s}{compcongVA@$\ConcV\bA$, $\ConcV f$} is also an isomorphism. As~$\eps$ is also an isomorphism and~$\ConcV\gr$\index{s}{compcongVA@$\ConcV\bA$, $\ConcV f$} is ideal-induced\index{i}{homomorphism!ideal-induced}, it follows that the \jzh 
 \[
 \gf\circ\ConcV v=\eps\circ(\ConcV\ol{v})\circ(\ConcV\gr)
 \]
is ideal-induced\index{i}{homomorphism!ideal-induced}.
\qed\end{proof}

We shall now show a large class of quasivarieties (cf. Example~\ref{Ex:QVar}) that satisfy the assumptions of Theorem~\ref{T:1stordLardCtble}.

\begin{defn}\label{D:FinResBd}
Let~$\cV$ be a quasivariety on a first-order language~$\scL$. A member~$\bA$ of~$\cV$ is \emph{subdirectly irreducible}\index{i}{subdirectly irreducible (within a quasivariety)|ii} (with respect to~$\cV$) if~$\ConV\bA$\index{s}{conAV@$\ConV\bA$, $\ConV f$} has a smallest nonzero element. We say that~$\cV$ is \emph{finitely generated}\index{i}{quasivariety!finitely generated|ii} if it is the smallest quasivariety containing a given finite set of finite structures (not necessarily one finite structure---for example, $\bA$ may not belong to the quasivariety generated by~$\bA\times\bB$).
\end{defn}

\begin{lem}[folklore]\label{L:FinResBd}
A quasivariety~$\cV$ on a first-order language~$\scL$ is finitely generated\index{i}{quasivariety!finitely generated} if{f} it contains only finitely many subdirectly irreducible members \pup{up to isomorphism} and all these structures are finite.
\end{lem}

\begin{proof}
Suppose first that~$\cV$ is generated by a finite set $\setm{\bA_i}{i<n}$ of finite structures. It follows from \cite[Corollary~3.1.6]{Gorb}\index{c}{Gorbunov, V.\,A.} that every relatively subdirectly irreducible member of~$\cV$ embeds into one of the~$\bA_i$. Hence, regardless of the cardinality of~$\scL$, the quasivariety~$\cV$ has only finitely many subdirectly irreducible members up to isomorphism, and these structures are all finite.

Conversely, if~$\cV$ has only finitely many subdirectly irreducible members and all these structures are finite, then it follows from the quasivariety analogue of Birkhoff's\index{c}{Birkhoff, G.} subdirect decomposition Theorem \cite[Theorem~3.1.1]{Gorb}\index{c}{Gorbunov, V.\,A.} that each member of~$\cV$ embeds into a product of those subdirectly irreducibles, and thus~$\cV$ is finitely generated\index{i}{quasivariety!finitely generated}.
\qed\end{proof}

\begin{remk}\label{Rk:FinResBd}
It follows easily from J\'onsson's\index{c}{Jonsson@J\'onsson, B.} Lemma that \emph{Every finitely generated variety of lattices}\index{i}{variety!finitely generated} (or, more generally, every finitely generated congruence-distributive\index{i}{variety!congruence-distributive} variety) \emph{has only finitely many subdirectly irreducible members; that is, it is also finitely generated as a quasivariety}. This situation is quite untypical: while there are many examples of finitely generated quasivarieties, the finitely generated\index{i}{variety!finitely generated} varieties that are also finitely generated\index{i}{quasivariety!finitely generated} as quasivarieties are not so widespread.

Let us for example consider the situation for \emph{groups}. A simple use of the fundamental structure theorem for finite abelian groups shows easily that \emph{Every finitely generated\index{i}{variety!finitely generated} variety of abelian groups is also a finitely generated quasivariety}. On the other hand, consider the \emph{quaternion group}\index{i}{quaternion group} $Q:=\set{1,-1,i,-i,j,-j,k,-k}$ with $i^2=j^2=k^2=ijk=-1$. The center of~$Q$ is $Z:=\set{1,-1}$. Now the subgroup
 \[
 Z_n:=\setm{(x_1,\dots,x_n)\in Z^n}{x_1\cdots x_n=1}
 \]
is normal in~$Q^n$, and it is not hard to verify that $Z^n/{Z_n}$ is the least nontrivial normal subgroup of~$Q^n/{Z_n}$; hence~$Q^n/{Z_n}$ is subdirectly irreducible. It belongs to the variety of groups generated by~$Q$, and it has order $2\cdot 4^n$. In particular, \emph{the variety of groups generated by~$Q$ contains infinitely many finite subdirectly irreducible members}. This example is inspired by the one, given in Neumann \cite[Example~51.33]{Neum}\index{c}{Neumann, H.}, of a monolithic (i.e., subdirectly irreducible), finite, non-critical group.
\end{remk}

The following result extends \cite[Lemma~3.8]{Gill1}\index{c}{Gillibert, P.}, with a similar proof.

\begin{prop}\label{P:FinResBdCongPp}
Let~$\cV$ be a finitely generated\index{i}{quasivariety!finitely generated} quasivariety on a first-order language~$\scL$. Then~$\cV$ is strongly congruence-proper\index{i}{strongly congruence-proper}.
\end{prop}

\begin{proof}
Denote by~$\cK$ a finite set of representatives of the subdirectly irreducible members of~$\cV$ modulo isomorphism.
Let~$\bS$ be a finite lattice and let~$\gi\colon\bS\to\ConV\bA$\index{s}{conAV@$\ConV\bA$, $\ConV f$} be an isomorphism, with $\bA\in\cV$.
We start with a standard argument giving Birkhoff's\index{c}{Birkhoff, G.} subdirect decomposition theorem in quasivarieties, see, for example, \cite[Theorem~3.1.1]{Gorb}\index{c}{Gorbunov, V.\,A.}. It is well-known that any element in an algebraic lattice\index{i}{lattice!algebraic} is a meet of completely \mirr\ elements\index{i}{irreducible!completely meet-} (see, for example, \cite[Theorem~I.4.25]{Comp}\index{c}{Gierz, G.}\index{c}{Hofmann, K.\,H.}\index{c}{Keimel, K.}\index{c}{Lawson, J.\,D.}\index{c}{Mislove, M.}\index{c}{Scott, D.\,S.}). By applying this to the zero element in the lattice~$\bS$, we see that if~$P$ denotes the set of all \mirr\ elements\index{i}{irreducible!meet-} of~$\bS$, the subset $\gi``(P)$ of~$\ConV\bA$\index{s}{conAV@$\ConV\bA$, $\ConV f$} meets to the zero congruence of~$\bA$, thus the diagonal map
 \[
 \bA\to\prod\famm{\bA/\gi(p)}{p\in P}\,,\quad
 x\mapsto\famm{x/{\gi(p)}}{p\in P}
 \]
is an embedding. As all subdirectly irreducible members of~$\cV$ are finite, all structures~$\bA/\gi(p)$ are finite, thus $\bA^*:=\prod\famm{\bA/\gi(p)}{p\in P}$ is finite, and thus $\bA$ is finite. Furthermore, if~$\cA$ denotes the class of all products of the form $\prod\famm{\bS_p}{p\in P}$ where all $\bS_p\in\cK$, $\bA^*$ embeds into some member of~$\cA$, hence so does~$\bA$. As~$\cA$ is finite and each member of~$\cA$ has only finitely many substructures, there are only finitely many possibilities for~$\bA$.
\qed\end{proof}

Recall that a \emph{variety} is a quasivariety closed under homomorphic images.

\begin{prop}[folklore]\label{P:ResBd2LocFin}
Let~$\cV$ be a finitely generated\index{i}{quasivariety!finitely generated}\index{i}{variety!finitely generated} \pup{quasi}variety on a first-order language~$\scL$. Then~$\cV$ is locally finite\index{i}{quasivariety!locally finite}.
\end{prop}

\begin{proof}
Of course, the result for quasivarieties follows from the result for varieties. If a variety~$\cV$ is generated by a structure~$\bA$, then the free $\cV$-object on~$n$ generators is the substructure of~$\bA^{A^n}$ generated by the projections, for each positive integer~$n$ (the usual proof of this fact is not affected by the possible presence of relations); in particular, if~$\bA$ is finite, then this structure is finite.
\qed\end{proof}

Other classes of locally finite\index{i}{quasivariety!locally finite}, strongly congruence-proper\index{i}{strongly congruence-proper} (quasi)varieties will be given in Section~\ref{S:FinGenAlg}.

\section{Relative critical points between quasivarieties}\label{S:CritQVar}

In this section we shall use the results of Section~\ref{S:Rightal0lardGQV} in order to relate, in Theorem~\ref{T:RelCritalephn}, the liftability of finite diagrams of \jzs s, with respect to the relative compact congruence semilattice functor\index{i}{relative!compact congruence semilattice functor}, in quasivarieties~$\cA$ and~$\cB$, with the relative critical point\index{i}{critical point!relative} $\critr(\cA;\cB)$\index{s}{critrAB@$\critr(\cA;\cB)$} introduced in Definition~\ref{D:RelCritPoint}. It will turn out that the latter lies below an aleph of finite index, the latter being smaller than the restricted Kuratowski index\index{i}{Kuratowski!restricted ${}_{-}$ index} (cf. Definition~\ref{D:KurInd0}) of the shape of a diagram liftable\index{i}{diagram!liftable} in~$\cA$ but not in~$\cB$. As a consequence, we obtain, under some finiteness assumptions on~$\cA$ and~$\cB$, that $\critr(\cA;\cB)$\index{s}{critrAB@$\critr(\cA;\cB)$} is always either smaller than~$\aleph_\go$\index{s}{aleph0@$\aleph_{\ga}$} or equal to~$\infty$ (Theorem~\ref{T:DichotCritPt}). Here, the symbol~$\infty$ denotes a ``number'' greater than all cardinal numbers. This extends to quasivarieties of algebraic systems\index{i}{algebraic system} a result obtained earlier by the first author on varieties of algebras\index{i}{algebra!universal} (cf. \cite[Corollary~7.13]{Gill1}\index{c}{Gillibert, P.}).

We call the \emph{relative compact congruence class}\index{i}{relative!compact congruence class|ii} of a \gqv\index{i}{generalized quasivariety}~$\cV$ the class of all \jzs s that are isomorphic to~$\ConcV\bA$\index{s}{compcongVA@$\ConcV\bA$, $\ConcV f$} for some~$\bA\in\cV$. We denote this class by~$\Concr\cV$\index{s}{compcongVcr@$\Concr\cV$|ii}.

\begin{defn}\label{D:RelCritPoint}
The \emph{relative critical point}\index{i}{critical point!relative|ii} between quasivarieties~$\cA$ and~$\cB$ is defined as\index{s}{critrAB@$\critr(\cA;\cB)$|ii}\index{s}{compcongVcr@$\Concr\cV$}
 \[
 \critr(\cA;\cB):=\min\setm{\card S}{\bS\in(\Concr\cA)\setminus(\Concr\cB)}
 \]
if $\Concr\cA\not\subseteq\Concr\cB$\index{s}{compcongVcr@$\Concr\cV$}, and $\infty$ otherwise.
\end{defn}

In particular, in case both~$\cA$ and~$\cB$ are varieties, $\critr(\cA;\cB)$\index{s}{critrAB@$\critr(\cA;\cB)$} is equal to the critical point\index{i}{critical point}~$\crit(\cA;\cB)$\index{s}{critAB@$\crit(\cA;\cB)$} as introduced in \cite[Definition~7.2]{Gill1}\index{c}{Gillibert, P.}.

\begin{thm}\label{T:RelCritalephn}
Let~$\cA$ and~$\cB$ be quasivarieties on \pup{possibly distinct} first-order languages with only finitely many relation symbols such that~$\cB$ is both congruence-proper\index{i}{congruence-proper} and locally finite\index{i}{quasivariety!locally finite}, and let~$P$ be an \ajs\index{i}{almost join-semilattice}. Assume that there exists a $P$-indexed diagram~$\overrightarrow{\bA}=\famm{\bA_p,\ga_p^q}{p\leq q\text{ in }P}$ of finite objects of~$\cA$ such that the diagram $\ConcA\overrightarrow{\bA}$\index{s}{compcongVA@$\ConcV\bA$, $\ConcV f$} has no lifting\index{i}{diagram!lifting}, with respect to~$\ConcB$\index{s}{compcon1@$\Conc\bA$, $\Conc f$}, in~$\cB$. The following statements hold:
\begin{description}
\item[\tui] Let $(X,\bX)$ be an $\aleph_0$-lifter\index{s}{aleph0@$\aleph_{\ga}$}\index{i}{lifter ($\gl$-)} of~$P$. Then $\critr(\cA;\cB)\leq\card X+\aleph_0$\index{s}{aleph0@$\aleph_{\ga}$}\index{s}{critrAB@$\critr(\cA;\cB)$}.

\item[\tuii] If~$P$ is finite nontrivial with zero, then $\critr(\cA;\cB)\leq\aleph_{\kur_0(P)-1}$\index{s}{aleph0@$\aleph_{\ga}$}\index{s}{critrAB@$\critr(\cA;\cB)$}\index{s}{kurP0@$\kur_0(P)$}.
\end{description}
\end{thm}

We refer to Definition~\ref{D:KurInd0} for the definition of the restricted Kuratowski index\index{i}{Kuratowski!restricted ${}_{-}$ index} $\kur_0(P)$\index{s}{kurP0@$\kur_0(P)$}.

\begin{proof}
We set $\cA^\dagger:=\setm{\bA_p}{p\in P}$, $\cB^\dagger:=\setm{\bB\in\cB}{\bB\text{ is finite}}$, $\cS:=\SEM$\index{s}{Sem@$\SEM$}, and we define $\cS^\dagger$ as the class of all finite \jzs s. Furthermore, we define $\cS^\Rightarrow$\index{s}{RightarrowCat@$\cS^\Rightarrow$} as the category of all \jzs s with ideal-induced\index{i}{homomorphism!ideal-induced} homomorphisms. It follows from Theorem~\ref{T:1stordLardCtble} together with the assumptions on~$\cB$ that\index{s}{compcongV@$\ConcV$ functor}\index{s}{aleph0@$\aleph_{\ga}$}
 \begin{equation}\index{i}{larder!right!projectable}\label{Eq:ProjLarderFinQVar}
 (\cB,\cB^\dagger,\cS,\cS^\dagger,\cS^\Rightarrow,\ConcB)
 \text{ is a projectable right $\aleph_0$-larder}.
 \end{equation}
Furthermore, $(\cA,\cS,\cS^\Rightarrow,\ConcA)$\index{s}{compcongV@$\ConcV$ functor} is a left larder\index{i}{larder!left}. (As the empty structure is not allowed in~$\cA$, projections in~$\cA$ are surjective homomorphisms, thus so are extended projections of~$\cA$, and thus, by Lemma~\ref{L:ConcpiIdInd}, $\ConcA$\index{s}{compcongV@$\ConcV$ functor} sends projections of~$\cA$ to ideal-induced\index{i}{homomorphism!ideal-induced} homomorphisms in~$\cS$. That~$\ConcA$\index{s}{compcongV@$\ConcV$ functor} preserves all small directed colimits follows from Theorem~\ref{T:ConcVPresDirColim}.) Therefore, by using Proposition~\ref{P:LR2Larder}, we obtain that $(\cA,\cB,\cS,\cA^\dagger,\cB^\dagger,\cS^\Rightarrow,\ConcA,\ConcB)$\index{s}{compcongV@$\ConcV$ functor} is a projectable $\aleph_0$-larder\index{s}{aleph0@$\aleph_{\ga}$}\index{i}{larder}.

Now suppose that the assumptions of~(i) are satisfied, set $\gk:=\card X+\aleph_0$\index{s}{aleph0@$\aleph_{\ga}$} and $\bA:=\xF(X)\otimes\overrightarrow{\bA}$\index{s}{FxX@$\xF(X)$}\index{s}{otimAS@$\bA\otimes\overrightarrow{S}$, $\gf\otimes\overrightarrow{S}$}. As the language of~$\cA$ has only finitely many relations, it follows from Lemma~\ref{L:ConcvonPpalCong}(i) that~$\ConA\bU$\index{s}{conAV@$\ConV\bA$, $\ConV f$} is finite for each finite $\bU\in\cA$. As~$\bA$ is the colimit of a diagram of at most~$\gk$ finite objects of~$\cA$, we get $\card\ConcA\bA\leq\gk$\index{s}{compcongVA@$\ConcV\bA$, $\ConcV f$}. Hence, in order to conclude the proof, it suffices to prove that~$\ConcA\bA$\index{s}{compcongVA@$\ConcV\bA$, $\ConcV f$} is not isomorphic to $\ConcB\bB$\index{s}{compcongVA@$\ConcV\bA$, $\ConcV f$}, for any object~$\bB$ of~$\cB$.

Suppose otherwise, and let $\chi\colon\ConcB\bB\to\ConcA\bA$\index{s}{compcongVA@$\ConcV\bA$, $\ConcV f$} be an isomorphism. By CLL\index{i}{Condensate Lifting Lemma (CLL)} (Lemma~\ref{L:CLL}), there exists a $P$-indexed diagram~$\overrightarrow{\bB}$ of~$\cB$ such that $\ConcB\overrightarrow{\bB}\cong\ConcA\overrightarrow{\bA}$\index{s}{compcongVA@$\ConcV\bA$, $\ConcV f$}. This contradicts the assumption on~$\overrightarrow{\bA}$.

Now suppose that the assumptions of~(ii) are satisfied, and set $n:=\kur_0(P)$\index{s}{kurP0@$\kur_0(P)$}. As $(\aleph_{n-1},{<}\aleph_0)\leadsto P$\index{s}{aleph0@$\aleph_{\ga}$}\index{s}{arr0x@$(\gk,{<}\gl)\leadsto P$}, it follows from Lemma~\ref{L:Part2Lift} that there exists a norm-covering\index{i}{norm-covering}~$X$ of~$P$ such that~$X$ is a lower finite\index{i}{poset!lower finite} \ajs\index{i}{almost join-semilattice}\ with zero, $\card X=\aleph_{n-1}$\index{s}{aleph0@$\aleph_{\ga}$}, and~$X$, together with the collection of all principal ideals\index{i}{ideal!of a poset}\index{i}{ideal!principal ${}_{-}$, of a poset} of~$X$, is an $\aleph_0$-lifter\index{s}{aleph0@$\aleph_{\ga}$}\index{i}{lifter ($\gl$-)} of~$P$. Now the conclusion follows from~(i) above.
\qed\end{proof}

In order to proceed, we prove a simple compactness result, similar, in formulation and proof, to the one of \cite[Theorem~7.11]{Gill1}\index{c}{Gillibert, P.}. Although this result can be extended to diagrams indexed by structures far more general than posets, we choose, for the sake of simplicity, to formulate it only in the latter context.

\begin{lem}\label{L:CompCatLift}
Let~$\cB$ and~$\cS$ be categories, let~$\Psi\colon\cB\to\cS$ be a functor, let~$P$ be a poset, and let $\overrightarrow{S}=\famm{S_p,\gs_p^q}{p\leq q\text{ in }P}$ be a diagram in~$\cS$. We make the following assumptions:
\begin{description}
\item[\textup{(L1)}] For each $p\in P$, there exists a finite set~$\cB_p$ of objects of~$\cB$ such that $(\forall X\in\Ob\cB)\bigl(\Psi(X)\cong S_p\Rightarrow (\exists Y\in\cB_p)(X\cong Y)\bigr)$.

\item[\textup{(L2)}] For all $p\leq q$ in~$P$, all $X\in\cB_p$, and all $Y\in\cB_q$, there are only finitely many morphisms from~$X$ to~$Y$ in~$\cB$.

\item[\textup{(L3)}] Each $S_p$, for $p\in P$, has only finitely many automorphisms.
\end{description}
If the restriction $\overrightarrow{S}\res_X$ has a lifting\index{i}{diagram!lifting} with respect to~$\Psi$, for each finite subset~$X$ of~$P$, then~$\overrightarrow{S}$ has a lifting\index{i}{diagram!lifting} with respect to~$\Psi$.
\end{lem}

\begin{proof}
For each $X\in[P]^{<\go}$, there are $\overrightarrow{B}_X=\famm{B_{X,p},\gb_{X,p}^{X,q}}{p\leq q\text{ in }X}$ in~$\cB^X$ and a natural equivalence $\famm{\eps_{X,p}}{p\in X}\colon\Psi\overrightarrow{B}_X\todot\overrightarrow{S}\res_X$\index{s}{AtorightarrowdotB@$f\colon\xA\todot\xB$}. It follows from~(L1) that we may assume that $B_{X,p}\in\cB_p$ for each $p\in X$.

Furthermore, we set $Q_X:=[P]^{<\go}\upw X$, and we fix an ultrafilter~$\cU$ on~$[P]^{<\go}$ such that $Q_X\in\cU$ for each $X\in[P]^{<\go}$.

Let $p\in P$. As~$Q_{\set{p}}$ belongs to~$\cU$ and is the disjoint union of all sets
 \[
 \setm{X\in Q_{\set{p}}}{B_{X,p}=B}\,,\quad\text{for }B\in\cB_p\,,
 \]
and as, by~(L1), $\cB_p$ is finite, there exists a unique~$B_p\in\cB_p$ such that
 \[
 R_p:=\setm{X\in Q_{\set{p}}}{B_{X,p}=B_p}\text{ belongs to }\cU\,.
 \]
Let $p\leq q$ in~$P$. As $R_p\cap R_q$ belongs to~$\cU$ and is the disjoint union of all sets
 \[
 \setm{X\in R_p\cap R_q}{\gb_{X,p}^{X,q}=\gb}\,,\quad\text{for }
 \gb\colon B_p\to B_q\text{ in }\cB\,,
 \]
and as, by~(L2), there are only finitely many such~$\gb$, there exists a unique morphism $\gb_p^q\colon B_p\to B_q$ in~$\cB$ such that
 \[
 \setm{X\in R_p\cap R_q}{\gb_{X,p}^{X,q}=\gb_p^q}
 \text{ belongs to }\cU\,.
 \]
Let $p\in P$. As~$R_p$ belongs to~$\cU$ and is the disjoint union of all sets
 \[
 \setm{X\in R_p}{\eps_{X,p}=\eps}\,,\quad\text{for isomorphisms }
 \eps\colon\Psi(B_p)\to S_p\,,
 \]
and as, by~(L3), there are only finitely many such~$\eps$, there exists a unique isomorphism $\eps_p\colon\Psi(B_p)\to S_p$ such that
 \[
 \setm{X\in R_p}{\eps_{X,p}=\eps_p}\text{ belongs to }\cU\,.
 \]
It is straightforward to verify that $\overrightarrow{B}:=\famm{B_p,\gb_p^q}{p\leq q\text{ in }P}$ is a $P$-indexed diagram in~$\cB$ and that $\famm{\eps_p}{p\in P}$ is a natural equivalence from~$\Psi\overrightarrow{B}$ to~$\overrightarrow{S}$.
\qed\end{proof}

Now we can improve the picture by the following dichotomy result, that extends \cite[Corollary~7.14]{Gill1}\index{c}{Gillibert, P.} from varieties to quasivarieties.

\begin{thm}[Dichotomy Theorem]\label{T:DichotCritPt}
\index{i}{Dichotomy Theorem|ii}
Let~$\cA$ and~$\cB$ be locally finite\index{i}{quasivariety!locally finite} quasivarieties on \pup{possibly distinct} first-order languages with only finitely many relation symbols such that~$\cB$ is strongly congruence-proper\index{i}{strongly congruence-proper}. If \linebreak $\Concr\cA\not\subseteq\Concr\cB$\index{s}{compcongVcr@$\Concr\cV$}, then $\critr(\cA;\cB)<\aleph_\go$\index{s}{aleph0@$\aleph_{\ga}$}\index{s}{critrAB@$\critr(\cA;\cB)$}.
\end{thm}

\begin{proof}
Suppose that $\critr(\cA;\cB)\geq\aleph_\go$\index{s}{aleph0@$\aleph_{\ga}$}\index{s}{critrAB@$\critr(\cA;\cB)$}. Given a structure~$\bA\in\cA$, we must find $\bB\in\cB$ such that $\ConcA\bA\cong\ConcB\bB$\index{s}{compcongVA@$\ConcV\bA$, $\ConcV f$}. Pick $a\in A$ and denote by~$P$ the set of all finite substructures of~$\bA$ containing~$a$ as an element. Observe that~$P$ is a \jzs. Set $\bA_p:=p$ and denote by $\ga_p^q$ the inclusion embedding from~$\bA_p$ into~$\bA_q$, for all $p\leq q$ in~$P$. Likewise, denote by~$\ga_p$ the inclusion embedding from~$\bA_p$ into~$\bA$. Set $\overrightarrow{\bA}:=\famm{\bA_p,\ga_p^q}{p\leq q\text{ in }P}$. It follows from the local finiteness of~$\cA$ that
 \begin{equation}\label{Eq:bApDirLimFinite}
 \famm{\bA,\ga_p}{p\in P}=\varinjlim\overrightarrow{\bA}\quad\text{in }\cA\,.
 \end{equation}
For each finite \jz-subsemilattice~$Q$ of~$P$, the restricted Kuratowski index\index{i}{Kuratowski!restricted ${}_{-}$ index}~$\kur_0(Q)$\index{s}{kurP0@$\kur_0(P)$} of~$Q$ is a non-negative integer, hence it follows from Theorem~\ref{T:RelCritalephn} that the diagram $\ConcA(\overrightarrow{\bA}\res_Q)$\index{s}{compcongVA@$\ConcV\bA$, $\ConcV f$} can be lifted\index{i}{diagram!lifted} with respect to the functor~$\ConcB$. As~$\cB$ is strongly congruence-proper\index{i}{strongly congruence-proper}, it follows from Lemma~\ref{L:CompCatLift} that the full diagram $\ConcA\overrightarrow{\bA}$\index{s}{compcongVA@$\ConcV\bA$, $\ConcV f$} can be lifted\index{i}{diagram!lifted} with respect to~$\ConcB$\index{s}{compcongV@$\ConcV$ functor}, that is, there are a $P$-indexed diagram~$\overrightarrow{\bB}$ of~$\cB$ and a natural equivalence $\ConcA\overrightarrow{\bA}\todot\ConcB\overrightarrow{\bB}$\index{s}{AtorightarrowdotB@$f\colon\xA\todot\xB$}\index{s}{compcongVA@$\ConcV\bA$, $\ConcV f$}. (\emph{Observe that this compactness argument requires the strong congruence-properness\index{i}{strongly congruence-proper} assumption on~$\cB$}.) Therefore, taking~$\bB:=\varinjlim\overrightarrow{\bB}$, we obtain, using~\eqref{Eq:bApDirLimFinite} together with Theorem~\ref{T:ConcVPresDirColim} (preservation of directed colimits by the~$\ConcA$ and~$\ConcB$ functors)\index{s}{compcongV@$\ConcV$ functor}, that $\ConcA\bA\cong\ConcB\bB$\index{s}{compcongVA@$\ConcV\bA$, $\ConcV f$}.
\qed\end{proof}

\begin{remk}\label{Rk:CongPperLocFin}
As stated in Propositions~\ref{P:FinResBdCongPp} and~\ref{P:ResBd2LocFin}, the assumptions on~$\cB$ in the statements of both Theorems~\ref{T:RelCritalephn} and~\ref{T:DichotCritPt}, that is, being both strongly congruence-proper\index{i}{strongly congruence-proper} and locally finite\index{i}{quasivariety!locally finite}, are satisfied in case~$\cB$ is a finitely generated\index{i}{quasivariety!finitely generated} quasivariety.
\end{remk}

\begin{remk}\label{Rk:RelCritalephn}
The proof of Theorem~\ref{T:RelCritalephn} shows that we can weaken the assumptions of that theorem, by requiring $\card\ConcA\bA\leq\aleph_{\kur_0(P)-1}$\index{s}{aleph0@$\aleph_{\ga}$}\index{s}{compcongVA@$\ConcV\bA$, $\ConcV f$}\index{s}{kurP0@$\kur_0(P)$}, for each nonempty finite product~$\bA$ of structures of the form~$\bA_p$ (indeed, if this holds, then the inequality $\card\ConcA\bigl(\xF(X)\otimes\overrightarrow{\bA}\bigr)\leq\aleph_{\kur_0(P)-1}$\index{s}{aleph0@$\aleph_{\ga}$}\index{s}{compcongVA@$\ConcV\bA$, $\ConcV f$}\index{s}{FxX@$\xF(X)$}\index{s}{otimAS@$\bA\otimes\overrightarrow{S}$, $\gf\otimes\overrightarrow{S}$}\index{s}{kurP0@$\kur_0(P)$} holds as well). In that case we no longer need the finiteness assumption about the relational part of~$\cA$. This holds, in particular, in case~$\ConcA\bA$\index{s}{compcongVA@$\ConcV\bA$, $\ConcV f$} is finite for each nonempty finite product~$\bA$ of structures of the form~$\bA_p$. This holds, for example, in case~$\cA$ is a variety of modular\index{i}{lattice!modular} lattices in which every finitely generated member has finite length.

A similar remark applies to Theorem~\ref{T:DichotCritPt}, with a slightly more awkward formulation. The assumption on~$\cA$ now states that each $\bA\in\cA$ is a directed colimit of a diagram, indexed by some \jzs, of structures in~$\cA$ any nonempty finite product of which has a finite congruence lattice (relatively to~$\cA$). In particular, this holds again in case~$\cA$ is a quasivariety of modular\index{i}{lattice!modular} lattices in which every finitely generated member has finite length.
\end{remk}

We conclude this section by the following analogue of \cite[Corollary~7.12]{Gill1}\index{c}{Gillibert, P.} for quasivarieties. Observe the different placement of the finiteness assumption: instead of dealing with finite semilattices that can be lifted\index{i}{diagram!lifted} from~$\cA$, we are dealing with the more restricted class of the $\ConcA\bA$\index{s}{compcongVA@$\ConcV\bA$, $\ConcV f$} for finite~$\bA\in\cA$.

\begin{cor}\label{C:RelCritalephn}
Let~$\cA$ and~$\cB$ be locally finite\index{i}{quasivariety!locally finite} quasivarieties on \pup{possibly distinct} first-order languages with only finitely many relation symbols such that~$\cB$ is strongly congruence-proper\index{i}{strongly congruence-proper}. The following statements are equivalent:
\begin{description}
\item[\tui] $\critr(\cA;\cB)>\aleph_0$\index{s}{aleph0@$\aleph_{\ga}$}\index{s}{critrAB@$\critr(\cA;\cB)$}.

\item[\tuii] For every diagram $\overrightarrow{\bA}$ of finite members of~$\cA$, indexed by a tree, there exists a diagram~$\overrightarrow{\bB}$ of~$\cB$ such that $\ConcA\overrightarrow{\bA}\cong\ConcB\overrightarrow{\bB}$.

\item[\tuiii] For every diagram $\overrightarrow{\bA}$ of finite members of~$\cA$, indexed by a finite chain, there exists a diagram~$\overrightarrow{\bB}$ of~$\cB$ such that $\ConcA\overrightarrow{\bA}\cong\ConcB\overrightarrow{\bB}$.
\end{description}
\end{cor}

\begin{proof}
(i)$\Rightarrow$(ii). Assume that~(i) holds. By using again Lemma~\ref{L:CompCatLift}, it suffices to prove that for every diagram $\overrightarrow{\bA}$ of finite members of~$\cA$, indexed by a \emph{finite} tree~$P$, there exists a diagram~$\overrightarrow{\bB}$ of~$\cB$ such that $\ConcA\overrightarrow{\bA}\cong\ConcB\overrightarrow{\bB}$. By Proposition~\ref{P:ktighttrees}, $P$ has an $\aleph_0$-lifter\index{s}{aleph0@$\aleph_{\ga}$}\index{i}{lifter ($\gl$-)} $(X,\bX)$ with~$X$ at most countable. The conclusion follows immediately from Theorem~\ref{T:RelCritalephn}(i).

(ii)$\Rightarrow$(iii) is trivial.

(iii)$\Rightarrow$(i). It follows from Lemma~\ref{L:CompCatLift} that every diagram of finite \jzs s, indexed by the chain~$\go$ of all natural numbers, which has a lifting\index{i}{diagram!lifting} in~$\cA$, also has a lifting\index{i}{diagram!lifting} in~$\cB$.

Now let~$\bS$ be an at most countable \jzs\ in $\Concr\cA$. There exists $\bA\in\cA$ such that $\bS\cong\ConcA\bA$\index{s}{compcongVA@$\ConcV\bA$, $\ConcV f$}. We claim that~$\bA$ can be taken at most countable. We use the argument of the second part of the proof of Proposition~\ref{P:glPresMIND}. Fix $a\in A$ and denote by~$\bA_X$ the substructure of~$\bA$ generated by~$X\cup\set{a}$, for each $X\subseteq A$. By using the description of directed colimits given in Section~\ref{Su:DirColimFirstOrd}, we obtain the directed colimit representation\index{s}{aleph0@$\aleph_{\ga}$}
 \[
 \bA=\varinjlim_{X\in[A]^{\les\aleph_0}}{\bA_X}\quad\text{in }\cA\,,
 \]
with all maps in the cocone being the corresponding inclusion maps. By applying Theorem~\ref{T:ConcVPresDirColim}\index{s}{compcongV@$\ConcV$ functor}, we obtain\index{s}{aleph0@$\aleph_{\ga}$}
 \begin{equation}\label{Eq:ConcAbAal0}
 \ConcA\bA=\varinjlim_{X\in[A]^{\les\aleph_0}}{\ConcA\bA_X}
 \quad\text{in }\SEM\,.
 \end{equation}
As~$\bA$ is locally finite\index{i}{algebraic system!locally finite}, each~$A_X$ is at most countable. As the language of~$\bA$ has only finitely many relation symbols, it follows that each $\ConcA\bA_X$ is at most countable. Denote by $\ga_X\colon\bA_X\into\bA$\index{s}{AtoinB@$f\colon A\into B$} the inclusion map, for each $X\in[A]^{\les\aleph_0}$\index{s}{aleph0@$\aleph_{\ga}$}. By applying Lemma~\ref{L:ModThEltChain} (with $\gl:=\aleph_1$\index{s}{aleph0@$\aleph_{\ga}$}, $I:=[A]^{\les\aleph_0}$\index{s}{aleph0@$\aleph_{\ga}$}, and $\scL:=\set{\vee,\wedge,0}$) to~\eqref{Eq:ConcAbAal0}, we obtain that the set\index{s}{aleph0@$\aleph_{\ga}$}
 \[
 J:=\setm{X\in[A]^{\les\aleph_0}}{\ConcA\ga_X\text{ is an embedding}}
 \]
is cofinal in $[A]^{\les\aleph_0}$\index{s}{aleph0@$\aleph_{\ga}$}. On the other hand, as $\ConcA\bA\cong\bS$ is at most countable, there exists $X\in[A]^{\les\aleph_0}$\index{s}{aleph0@$\aleph_{\ga}$} such that $\ConcA\ga_X$ is surjective. There exists $Y\in J$ containing~$X$. As~$\ConcA\ga_Y$\index{s}{compcongVA@$\ConcV\bA$, $\ConcV f$} is an isomorphism, it follows that $\bS\cong\Conc\bA_Y$, which completes the proof of our claim.

{}From now on until the end of the proof we fix $\bA\in\cA$ at most countable such that $\ConcA\bA\cong\bS$. We must prove that $\ConcA\bA$ is isomorphic to $\ConcB\bB$ for some $\bB\in\cB$. We argue as at the end of the proof of \cite[Corollary~7.12]{Gill1}\index{c}{Gillibert, P.}: as~$\bA$ is at most countable and locally finite\index{i}{algebraic system!locally finite}, it is the directed union of a chain $\famm{\bA_n}{n<\go}$ of finite substructures, giving canonically an $\go$-indexed diagram~$\overrightarrow{\bA}$. Let an $\go$-indexed diagram~$\overrightarrow{\bB}$ in~$\cB$ lift, with respect to the functor~$\ConcB$, the diagram $\ConcA\overrightarrow{\bA}$. Setting $\bB:=\varinjlim\overrightarrow{\bB}$, it follows from Theorem~\ref{T:ConcVPresDirColim} (preservation of directed colimits by the~$\ConcA$ and~$\ConcB$ functors)\index{s}{compcongV@$\ConcV$ functor} that $\ConcA\bA\cong\ConcB\bB$\index{s}{compcongVA@$\ConcV\bA$, $\ConcV f$}.
\qed\end{proof}

\section[Finitely generated varieties of algebras]{Strong congruence-properness of certain finitely generated varieties of algebras}\label{S:FinGenAlg}
In view of Propositions~\ref{P:FinResBdCongPp} and~\ref{P:ResBd2LocFin}, it would have looked easier to formulate the Dichotomy Theorem (Theorem~\ref{T:DichotCritPt})\index{i}{Dichotomy Theorem} under the more restrictive assumption where~$\cB$ is a \emph{finitely generated quasivariety}\index{i}{quasivariety!finitely generated}, arguing that there are no ``natural'' examples of (quasi)varieties that are both strongly congruence-proper\index{i}{strongly congruence-proper} and locally finite\index{i}{quasivariety!locally finite} without being finitely generated\index{i}{quasivariety!finitely generated}. This impression (that the Dichotomy Theorem\index{i}{Dichotomy Theorem} is designed for finitely generated\index{i}{quasivariety!finitely generated} quasivarieties only) might be reinforced by the comments in Remark~\ref{Rk:FinResBd}, showing an example of a finitely generated\index{i}{variety!finitely generated} variety of groups which is not finitely generated\index{i}{quasivariety!finitely generated} as a quasivariety.

We shall now show some directions suggesting that this is not so. In what follows we shall focus on \emph{varieties of algebras of finite type}\index{i}{algebra!universal}: there are only finitely many operation (or constant) symbols and no relation symbols. A variety~$\cV$ of algebras\index{i}{algebra!universal} is \emph{congruence-modular}\index{i}{variety!congruence-modular|ii} if the congruence lattice of every member of~$\cV$ is modular\index{i}{lattice!modular}. For examples, varieties of \emph{groups} (or even \emph{loops}) or of \emph{modules} are congruence-modular. By using their commutator theory\index{i}{commutator theory} for congruence-modular varieties, Freese and McKenzie prove in \cite[Theorem~10.16]{FrMK87}\index{c}{Freese, R.}\index{c}{McKenzie, R.\,N.} the following remarkable result:

\begin{thm}\label{T:CMStrCongPp}
Let~$\bA$ be a finite algebra\index{i}{algebra!universal} such that the variety $\Var(\bA)$\index{s}{VarA@$\Var(\bA)$|ii} generated by~$\bA$ is congruence-modular\index{i}{variety!congruence-modular}, and let $\bB\in\Var(\bA)$. If $\Con\bB$\index{s}{conA@$\Con\bA$} has finite length~$n$, then $\card B\leq(\card A)^n$.
\end{thm}

It follows immediately that \emph{Every finitely generated congruence-modular variety of finite type is strongly congruence-proper}\index{i}{variety!finitely generated}\index{i}{variety!congruence-modular}\index{i}{strongly congruence-proper}. Even for varieties of groups, this result is not trivial, in particular in view of the example in Remark~\ref{Rk:FinResBd}. In Hobby and McKenzie \cite[Theorem~14.6]{HoMK88}\index{c}{Hobby, D.}\index{c}{McKenzie, R.\,N.}, Theorem~\ref{T:CMStrCongPp} is extended to the case where $\Var(\bA)$ \emph{omits the tame congruence theory types~$\mathbf{1}$ and~$\mathbf{5}$}\index{i}{tame congruence theory}. This holds, in particular, in case there is a nontrivial lattice identity satisfied by the congruence lattices of all members of~$\Var(\bA)$, cf. \cite[Theorem~9.18]{HoMK88}\index{c}{Hobby, D.}\index{c}{McKenzie, R.\,N.}. Therefore we obtain the following consequence of that result together with the Dichotomy Theorem\index{i}{Dichotomy Theorem} (Theorem~\ref{T:DichotCritPt}).

\begin{thm}\label{T:DichotNonType15}
Let~$\cA$ be a locally finite\index{i}{quasivariety!locally finite} quasivariety on a first-order language with only finitely many relation symbols and let~$\cB$ be a finitely generated\index{i}{variety!finitely generated} variety of algebras\index{i}{algebra!universal} with finite similarity type, omitting types~$\mathbf{1}$ and~$\mathbf{5}$\index{i}{tame congruence theory} \pup{this holds in case $\cB$ satisfies a nontrivial congruence identity}. If $\Concr\cA\not\subseteq\Concr\cB$\index{s}{compcongVcr@$\Concr\cV$}, then $\critr(\cA;\cB)<\aleph_\go$\index{s}{aleph0@$\aleph_{\ga}$}\index{s}{critrAB@$\critr(\cA;\cB)$}.
\end{thm}

In particular, in case both~$\cA$ and~$\cB$ are finitely generated\index{i}{variety!finitely generated} varieties of either lattices, groups, loops, or modules (the latter on finite rings), $\Concr\cA\not\subseteq\Concr\cB$\index{s}{compcongVcr@$\Concr\cV$} implies that $\critr(\cA;\cB)<\aleph_\go$\index{s}{aleph0@$\aleph_{\ga}$}\index{s}{critrAB@$\critr(\cA;\cB)$}.

An example of a finitely generated\index{i}{variety!finitely generated} variety of algebras\index{i}{algebra!universal} with arbitrarily large (finite or infinite) simple algebras\index{i}{algebra!universal} (and thus non congruence-proper\index{i}{congruence-proper}), attributed to C. Shallon\index{c}{Shallon, C.}, is given in \cite[Exercise~14.9(4)]{HoMK88}\index{c}{Hobby, D.}\index{c}{McKenzie, R.\,N.}.

\section{A potential use of larders on non-regular cardinals}\label{S:NonRegLard}
Although we do not know of any ``concrete'' problem solved by $\gl$-larders\index{i}{larder} for non-regular~$\gl$, we shall outline in the present section a potential use which does not seem to follow from CLL\index{i}{Condensate Lifting Lemma (CLL)} applied to larders indexed by regular cardinals.

For a quasivariety~$\cV$ and an infinite cardinal~$\gl$, we denote by~$\cV^{(\gl)}$\index{s}{Vl@$\cV^{(\gl)}$, $\cV$ quasivariety|ii} the class of all $\gl$-small members of~$\cV$. We start with the following larderhood result.

\begin{thm}\label{T:Algorightlard}
Let~$\gl$ be an uncountable cardinal and let~$\cV$ be a quasivariety on a $\gl$-small first-order language~$\scL$\index{s}{compcongV@$\ConcV$ functor}. Then the $6$-uple\index{s}{Vl@$\cV^{(\gl)}$, $\cV$ quasivariety}\index{s}{Semi@$\SEM^{\mathrm{idl}}$}\index{s}{Seml@$\SEM^{(\gl)}$}\index{s}{Sem@$\SEM$}
 \[
 (\cV,\cV^{(\gl)},\SEM,\SEM^{(\gl)},\SEM^{\mathrm{idl}},\ConcV)
 \]
is a projectable right $\gl$-larder\index{i}{larder!right!projectable}.
\end{thm}

\begin{proof}
Similar to the proof of Claim~\ref{Cl:GrSchrightlard} within the proof of Theorem~\ref{T:MindConcLift}.
For each $\bA\in\cV^{(\gl)}$\index{s}{Vl@$\cV^{(\gl)}$, $\cV$ quasivariety}, the semilattice $\ConcV\bA$\index{s}{compcongVA@$\ConcV\bA$, $\ConcV f$} is $\gl$-small, thus, by Proposition~\ref{P:glPresMIND}, it is weakly $\gl$-presented\index{i}{presented!weakly $\gl$-} in~$\SEM$\index{s}{Sem@$\SEM$}. Hence the statement $(\PRES_{\gl}(\cV^{(\gl)},\ConcV))$\index{s}{Pres@$(\PRES_\gl(\cB^\dagger,\Psi))$} is satisfied.

Let~$\bB\in\cV$, we must verify that $(\LSr_{\gl}(\bB))$\index{s}{LSr@$(\LSr_\gm(B))$} is satisfied. Let~$\bS$ be a $\gl$-small \jzs, let~$I$ be a $\cf(\gl)$-small set, and let $\famm{u_i\colon\bU_i\Rightarrow\bB}{i\in I}$\index{s}{AtorightarrowB@$f\colon A\Rightarrow B$} be an $I$-indexed family of double arrows\index{i}{double arrow} with all the~$\bU_i$ being $\gl$-small. There exists a successor (thus regular uncountable) cardinal~$\gm\leq\gl$ such that~$\scL$, $S$, and all~$U_i$ are $\gm$-small. As at the end of the proof of Proposition~\ref{P:glPresMIND}, we express~$\bB$ as a continuous\index{i}{diagram!continuous} directed colimit
 \[
 \bB=\varinjlim_{X\in[B]^{<\gm}}{\bB_X}\quad\text{in }\cV
 \]
(here~$\bB_X$ is simply the $\cV$-substructure of~$\bB$ generated by~$X$ together with some fixed element of~$B$),
with all maps in the cocone being the corresponding inclusion maps. Then applying Theorem~\ref{T:ConcVPresDirColim} to~$\cV$ yields a continuous\index{i}{diagram!continuous} directed colimit\index{s}{compcongVA@$\ConcV\bA$, $\ConcV f$}\index{s}{Sem@$\SEM$}
 \begin{equation}\label{Eq:ConcbBAlgo}
 \ConcV\bB=\varinjlim_{X\in[B]^{<\gm}}{\ConcV\bB_X}\quad\text{in }\SEM\,.
 \end{equation}
Denote by $\gb_X\colon\bB_X\into\bB$\index{s}{AtoinB@$f\colon A\into B$} the inclusion map, for each $X\in[B]^{<\gm}$. It follows from~\eqref{Eq:ConcbBAlgo} together with Proposition~\ref{P:ApproxMonIIWD} that the set
 \[
 J:=\setm{X\in[B]^{<\gm}}{\gf\circ\ConcV\gb_X\text{ is ideal-induced}}
 \]
is $\gs$-closed cofinal\index{i}{subset!$\gs$-closed cofinal} in~$[B]^{<\gm}$. As there exists $X\in[B]^{<\gm}$ such that~$B_X$ contains all images $u_i``(U_i)$, it follows that such an~$X$ can be chosen in~$J$. Then~$\gb_X$ is monic, $u_i\utr\gb_X$ for each $i\in I$, and $\gf\circ\Conc\gb_X$ is ideal-induced\index{i}{homomorphism!ideal-induced}.

The projectability statement follows immediately from Theorem~\ref{T:GQV2ProjWit}.
\qed\end{proof}

By applying Theorem~\ref{T:Algorightlard} to the case $\gl:=\aleph_\go$\index{s}{aleph0@$\aleph_{\ga}$}, we obtain the following.

\begin{cor}\label{C:AlgoConcLift}
Let~$\cV$ be a quasivariety on an $\aleph_\go$-small\index{s}{aleph0@$\aleph_{\ga}$} first-order language~$\scL$ and suppose that there exists a diagram $\overrightarrow{\bS}=\famm{S_m,\gs_m^n}{m\leq n<\go}$, indexed by the chain~$\go$ of all natural numbers, of $\aleph_\go$-small\index{s}{aleph0@$\aleph_{\ga}$} \jzs s, that cannot be lifted\index{i}{diagram!lifted}, with respect to the~$\ConcV$\index{s}{compcongV@$\ConcV$ functor} functor, by any diagram in~$\cV$. Then there exists a \jzs\ with at most~$\aleph_\go$\index{s}{aleph0@$\aleph_{\ga}$} elements that is not isomorphic to $\ConcV\bB$\index{s}{compcongVA@$\ConcV\bA$, $\ConcV f$} for any $\bB\in\cV$.
\end{cor}

\begin{proof}
Our proof will follow the lines of the one of Theorem~\ref{T:MindConcLift}.

\begin{claim}
Denote by~$\Phi$ the identity functor on~$\SEM$\index{s}{Sem@$\SEM$}. Then the quadruple\index{s}{Semi@$\SEM^{\mathrm{idl}}$}
 \[
 (\SEM,\SEM,\SEM^{\mathrm{idl}},\Phi)
 \]
is a left larder\index{i}{larder!left}.
\end{claim}

\begin{proof}
As the proof of Claim~\ref{Cl:GrSchleftlard} within the proof of Theorem~\ref{T:MindConcLift}.
\qed\ Claim\end{proof}

By the claim above together with Theorem~\ref{T:Algorightlard} and Proposition~\ref{P:LR2Larder} (with $\cA^\dagger:=\SEM^{(\aleph_\go)}$)\index{s}{Seml@$\SEM^{(\gl)}$}, we obtain that the $8$-uple\index{s}{compcongV@$\ConcV$ functor}\index{s}{Vl@$\cV^{(\gl)}$, $\cV$ quasivariety}\index{s}{Semi@$\SEM^{\mathrm{idl}}$}\index{s}{Seml@$\SEM^{(\gl)}$}\index{s}{Sem@$\SEM$}
 \[
 (\SEM,\cV,\SEM,\SEM^{(\aleph_\go)},\cV^{(\aleph_\go)},\SEM^{\mathrm{idl}},
 \Phi,\ConcV)
 \]
is an $\aleph_\go$-larder\index{s}{aleph0@$\aleph_{\ga}$}\index{i}{larder}.

By applying Proposition~\ref{P:ktighttrees} to $T:=\go$, we obtain an $\aleph_\go$-lifter\index{s}{aleph0@$\aleph_{\ga}$}~$(X,\bX)$\index{i}{lifter ($\gl$-)} of~$\go$ such that~$X$ is a lower finite\index{i}{poset!lower finite} \ajs\index{i}{almost join-semilattice}\ with zero and $\card X=\aleph_\go$\index{s}{aleph0@$\aleph_{\ga}$}. The \jzs\ $\bS:=\xF(X)\otimes\overrightarrow{\bS}$\index{s}{FxX@$\xF(X)$}\index{s}{otimAS@$\bA\otimes\overrightarrow{S}$, $\gf\otimes\overrightarrow{S}$} has at most $\aleph_\go$\index{s}{aleph0@$\aleph_{\ga}$} elements. If~$\bS$ is isomorphic to $\ConcV\bB$\index{s}{compcongVA@$\ConcV\bA$, $\ConcV f$} for some $\bB\in\cV$, then, by CLL\index{i}{Condensate Lifting Lemma (CLL)} (Lemma~\ref{L:CLL}), there exists an $\go$-indexed diagram~$\overrightarrow{\bB}$ in~$\cV$ such that $\Conc\overrightarrow{\bB}\cong\overrightarrow{\bS}$, a contradiction.
\qed\end{proof}

\chapter[Congruence-preserving extensions]{Congruence-permutable, congruence-preserving extensions of lattices}\label{Ch:CongPres}

\textbf{Abstract.} This chapter is intended to illustrate how to use CLL\index{i}{Condensate Lifting Lemma (CLL)} for solving an open problem of lattice theory, although the statement of that problem does not involve lifting\index{i}{diagram!lifting} diagrams with respect to functors. We set
 \[
 \ga\gb:=\setm{(x,z)\in X\times X}{(\exists y\in X)((x,y)\in\ga\text{ and }
 (y,z)\in\gb}\,,
 \]
for any binary relations~$\ga$ and~$\gb$ on a set~$X$, and we say that an algebra~$\bA$ is \emph{congruence-permutable}\index{i}{algebra!congruence-permutable|ii} if $\ga\gb=\gb\ga$ for any congruences~$\ga$ and~$\gb$ of~$\bA$. The problem in question is:
 \begin{quote}\normalsize\em
 Does every lattice of cardinality~$\aleph_1$\index{s}{aleph0@$\aleph_{\ga}$} have a congruence-permutable\index{i}{algebra!congruence-permutable}, congruence-preserving extension\index{i}{congruence-preserving extension}? 
 \end{quote}
This problem is part of Problem~4 in the survey paper T\r{u}ma and Wehrung~\cite{CLPSurv}\index{c}{Tuma@T\r{u}ma, J.}\index{c}{Wehrung, F.} but it was certainly known before.
Due to earlier work by Plo\v{s}\v{c}ica, T\r{u}ma, and Wehrung~\cite{PTW}\index{c}{Plo\v{s}\v{c}ica, M.}\index{c}{Tuma@T\r{u}ma, J.}\index{c}{Wehrung, F.}, the corresponding negative result for free lattices on~$\aleph_2$\index{s}{aleph0@$\aleph_{\ga}$} generators was already known. The countable case is still open, although it is proved in Gr\"atzer, Lakser, and Wehrung~\cite{GLWe}\index{c}{Gr\"atzer, G.}\index{c}{Lakser, H.}\index{c}{Wehrung, F.} that every countable, locally finite\index{i}{lattice!locally finite} (or even ``locally congruence-finite''\index{i}{lattice!locally congruence-finite}) lattice has a relatively complemented\index{i}{lattice!relatively complemented} (thus congruence-permutable\index{i}{algebra!congruence-permutable}) congruence-preserving extension\index{i}{congruence-preserving extension}. The finite case is solved in Tischendorf~\cite{Tisch}\index{c}{Tischendorf, M.}, where it is proved that every finite lattice embeds congruence-preservingly into some finite atomistic (thus congruence-permutable\index{i}{algebra!congruence-permutable}) lattice. This result is improved in Gr\"atzer and Schmidt~\cite{GrSc99}\index{c}{Gr\"atzer, G.}\index{c}{Schmidt, E.\,T.}, where the authors prove that every finite lattice embeds congruence-preservingly into some sectionally complemented\index{i}{lattice!sectionally complemented} finite lattice.

Most of Chapter~\ref{Ch:CongPres} will consist of checking one after another the various assumptions that need to be satisfied in order to be able to use CLL\index{i}{Condensate Lifting Lemma (CLL)}. Most of these verifications are elementary. In that sense (i.e., considering the verification of the assumptions underlying CLL\index{i}{Condensate Lifting Lemma (CLL)} as tedious but elementary), the hard core of the solution to the problem above consists of the unliftable\index{i}{unliftable diagram} family of squares\index{i}{square (shape of a diagram)} presented in Lemma~\ref{L:UnliftSqMetr}.

\section{The category of semilattice-metric spaces}\label{S:MetrSp}

In this section we shall introduce the category that will play the role of the~$\cS$ of the statement of CLL\index{i}{Condensate Lifting Lemma (CLL)} (Lemma~\ref{L:CLL}) and state its good behavior in terms of directed colimits and finitely presented\index{i}{presented!finitely} objects. We remind the reader that \emph{semilattice-valued distances}\index{i}{distance!semilattice-valued} have been introduced in Example~\ref{Ex:RTWe}.

\begin{defn}\label{D:MetrSp}
For a set~$A$, a \jzs\ $S$, and an $S$-valued distance\index{i}{distance!semilattice-valued} $\gd\colon A\times A\to S$, we shall say that the triple $\bA:=(A,\gd,S)$ is a \emph{semilattice-metric space}\index{i}{semilattice-metric!space|ii}, and we shall often write $\gd_\bA:=\gd$, $\tA:=S$. We say that~$\bA$ is \emph{finite} if both~$A$ and~$\tA$ are finite.

For semilattice-metric spaces\index{i}{semilattice-metric!space}~$\bA$ and~$\bB$, a \emph{morphism} from~$\bA$ to~$\bB$ is a pair $\xf=(f,\tf)$, where $f\colon A\to B$, $\tf\colon\tA\to\tB$ is a \jzh, and the equation $\gd_\bB(f(x),f(y))=\tf\bigl(\gd_\bA(x,y)\bigr)$ is satisfied for all $x,y\in A$. The composition of morphisms is defined by $\xg\circ\xf:=(g\circ f,\tg\circ\tf)$.

We denote by $\METR$\index{s}{Metr@$\METR$|ii} the category of all semilattice-metric\index{i}{semilattice-metric!space} spaces.
\end{defn}

The following lemma implies that both forgetful functors from~$\METR$\index{s}{Metr@$\METR$} to~$\SET$\index{s}{Set@\textbf{Set}} and~$\SEM$\index{s}{Sem@$\SEM$} preserve all small directed colimits. Its proof, although somewhat tedious, is straightforward, and we shall omit it.

\begin{lem}\label{L:FgtMetrPres}
Let $\famm{\bA_i,\xf_i^j}{i\leq j\text{ in }I}$ be a directed poset-indexed diagram in~$\METR$\index{s}{Metr@$\METR$}. We form the colimits\index{s}{Set@\textbf{Set}}\index{s}{Sem@$\SEM$}
 \begin{align*}
 \famm{A,f_i}{i\in I}&=\varinjlim\famm{A_i,f_i^j}{i\leq j\text{ in }I}\quad
 \text{in }\SET\,,\\
 \famm{\tA,\tf_i}{i\in I}&=\varinjlim\famm{\tA_i,\tf_i^j}{i\leq j\text{ in }I}\quad
 \text{in }\SEM\,. 
 \end{align*}
Then there exists a unique $\tA$-valued distance\index{i}{distance!semilattice-valued}~$\gd$ on~$A$ such that $\xf_i:=(f_i,\tf_i)$ is a morphism from~$\bA_i$ to~$\bA:=(A,\gd,\tA)$ for each $i\in I$. Furthermore,\index{s}{Metr@$\METR$}
 \[
 \famm{\bA,\xf_i}{i\in I}=\varinjlim\famm{\bA_i,\xf_i^j}{i\leq j\text{ in }I}\quad
 \text{in }\METR\,.
 \]
\end{lem}

The proof of the following result goes along the lines of the proof of Proposition~\ref{P:glPresMIND}.

\begin{prop}\label{P:FPMetr}
Every semilattice-metric space\index{i}{semilattice-metric!space}~$\bA$ is a monomorphic\index{i}{monomorphic colimit} directed colimit of a diagram of finite semilattice-metric spaces. Furthermore, $\bA$ is finitely presented\index{i}{presented!finitely} in~$\METR$\index{s}{Metr@$\METR$} if{f} it is finite.
\end{prop}

\begin{proof}
Denote by~$I$ the set of all pairs $i=(X,\tX)$ such that~$X$ is a finite subset of~$A$, $\tX$ is a finite \jz-subsemilattice of~$\tA$, and $\setm{\gd_\bA(x,y)}{x,y\in X}\subseteq\tX$; order~$I$ by componentwise containment. Then set $\bX_i:=(X,\gd_\bA\res_{X\times X},\tX)$, and denote by~$\gx_i$ (resp., $\tilde{\gx}_i$) the inclusion map from~$X_i$ into~$A$ (resp., from~$\tX_i$ into~$\tA$). The pair $\bgx_i:=(\gx_i,\tilde{\gx}_i)$ is a monomorphism from~$\bX_i$ to~$\bA$ in~$\METR$\index{s}{Metr@$\METR$}. Define likewise a morphism $\bgx_i^j\colon\bX_i\to\bX_j$ in~$\METR$\index{s}{Metr@$\METR$}, for $i\leq j$ in~$I$. As every finitely generated \jzs\ is finite, it is straightforward to verify the statement
 \begin{equation}\label{Eq:MetrDirLimFin}
 \famm{\bA,\bgx_i}{i\in I}=\varinjlim\famm{\bX_i,\bgx_i^j}{i\leq j\text{ in }I}\,.
 \end{equation}
The first part of Proposition~\ref{P:FPMetr} follows.

The proof that~$\bA$ finite implies~$\bA$ finitely presented\index{i}{presented!finitely} is similar to the proof of the first part of Proposition~\ref{P:glPresMIND}. Conversely, suppose that~$\bA$ is finitely presented\index{i}{presented!finitely}. As~\eqref{Eq:MetrDirLimFin} is a monomorphic\index{i}{monomorphic colimit} directed colimit, we infer, as at the end of proof of Corollary~\ref{C:FinPres=Comp}, that~$\bA$ is isomorphic to one of the~$\bX_i$, thus~$\bA$ is finite.
\qed\end{proof}

\section{The category of all semilattice-metric covers}\label{S:MetrCov}

In this section we shall introduce the category that will play the role of~$\cB$ and the functor that will play the role of~$\Psi$ in the statement of CLL\index{i}{Condensate Lifting Lemma (CLL)} (Lemma~\ref{L:CLL}), and state their good behavior in terms of directed colimits and finitely presented\index{i}{presented!finitely} objects. Our~$\Psi$ is defined as a forgetful functor. The right $\aleph_0$-larder\index{s}{aleph0@$\aleph_{\ga}$}\index{i}{larder!right} in question will be stated in Proposition~\ref{P:MetrLarder}.

\goodbreak
\begin{defn}\label{D:MetrCov}
A \emph{semilattice-metric cover}\index{i}{semilattice-metric!cover|ii} is a quadruple $\bA:=(A^*,A,\gd,S)$, where
\begin{description}
\item[\tui] $A$ is a set and $A^*$ is a subset of $A$;

\item[\tuii] $S$ is a \jzs;

\item[\tuiii] $\gd$ is an $S$-valued distance\index{i}{distance!semilattice-valued} on~$A$;

\item[\tuiv] (\emph{Parallelogram Rule}\index{i}{Parallelogram Rule|ii}) For all $x,y,z\in A^*$, there exists $t\in A$ such that
 \[
 \gd(x,t)\leq\gd(y,z)\text{ and }\gd(t,z)\leq\gd(x,y)\,.
 \]
\end{description}
We shall often write $\gd_\bA:=\gd$ and $\tA:=S$. Observe that the triple $(A,\gd_\bA,\tA)$ is then a semilattice-metric space\index{i}{semilattice-metric!space}. We say that~$\bA$ is \emph{finite} if both~$A$ and~$\tA$ are finite.

For semilattice-metric covers\index{i}{semilattice-metric!cover}~$\bA$ and~$\bB$, a \emph{morphism} from~$\bA$ to~$\bB$ is a morphism $\xf\colon(A,\gd_\bA,\tA)\to(B,\gd_\bB,\tB)$ in~$\METR$\index{s}{Metr@$\METR$} such that $f``(A^*)\subseteq B^*$.

We denote by $\METR^*$\index{s}{Metrs@$\METR^*$|ii} the category of all semilattice-metric covers\index{i}{semilattice-metric!cover}.
\end{defn}
 
Our next result shows how to construct a right $\aleph_0$-larder\index{s}{aleph0@$\aleph_{\ga}$}\index{i}{larder!right} from semilattice-valued spaces and covers.

\begin{prop}\label{P:MetrLarder}
Denote by $\METR^{\mathrm{fin}}$\index{s}{Metrf@$\METR^{\mathrm{fin}}$|ii} \pup{resp., $\METR^{*\fin}$}\index{s}{Metrsf@$\METR^{*\mathrm{fin}}$|ii} the class of all finite objects in~$\METR$\index{s}{Metr@$\METR$} \pup{resp., $\METR^*$\index{s}{Metrs@$\METR^*$}}, by $\METR^\Rightarrow$\index{s}{METRarr@$\METR^\Rightarrow$|ii} the subcategory of~$\METR$\index{s}{Metr@$\METR$} consisting of all morphisms~$(f,\tf)$ with~$f$ surjective, and by $\Psi\colon\METR^*\to\METR$\index{s}{Metrs@$\METR^*$} the forgetful functor $\bA\mapsto\bA^\flat:=(A^*,\gd_\bA\res_{A^*\times A^*},\tA)$\index{s}{Aiiflat@$\bA^\flat$, $f^\flat$|ii}. Then the octuple $(\METR^*,\METR^{*\fin},\METR,\METR^{\mathrm{fin}},\METR^\Rightarrow,\Psi)$\index{s}{Metrf@$\METR^{\mathrm{fin}}$}\index{s}{Metr@$\METR$}\index{s}{Metrsf@$\METR^{*\mathrm{fin}}$} is a right $\aleph_0$-larder\index{s}{aleph0@$\aleph_{\ga}$}\index{i}{larder!right}.
\end{prop}

\begin{proof}
The statement $(\PRES_{\aleph_0}(\METR^{*\fin},\Psi))$\index{s}{Pres@$(\PRES_\gl(\cB^\dagger,\Psi))$} follows from the second part of Proposition~\ref{P:FPMetr}.

It remains to check $(\LSr_{\aleph_0}(\bB))$\index{s}{LSr@$(\LSr_\gm(B))$}, for each semilattice-metric cover\index{i}{semilattice-metric!cover}~$\bB$. Let~$\bA$ be a finite semilattice-metric space\index{i}{semilattice-metric!space}. The triple $\Psi(\bB)=(B^*,\gd_\bB\res_{B^*\times B^*},\tB)$ is a semilattice-metric space\index{i}{semilattice-metric!space}. Let $\bgf\colon\Psi(\bB)\Rightarrow\bA$\index{s}{AtorightarrowB@$f\colon A\Rightarrow B$} be a morphism in $\METR^\Rightarrow$\index{s}{METRarr@$\METR^\Rightarrow$}; hence $\gf\colon B^*\to A$ is surjective. Let $n<\go$ and let $\xu_i\colon\bU_i\to\bB$, for $i<n$, be objects in the comma category~$\METR^{*\fin}\dnw\bB$\index{s}{Metrsf@$\METR^{*\mathrm{fin}}$}. This means that each~$\bU_i$ is a finite semilattice-metric cover\index{i}{semilattice-metric!cover} while $\xu_i\colon\bU_i\to\bB$ in~$\METR^*$\index{s}{Metrs@$\METR^*$}.

As $\gf``(B^*)=A$ and both~$A$ and $u_i``(U_i^*)$ are finite, with $u_i``(U_i^*)\subseteq B^*$ (for all $i<n$), there exists a finite subset~$V^*$ of~$B^*$ containing $\bigcup\famm{u_i``(U_i^*)}{i<n}$ such that $\gf``(V^*)=A$. As~$\bB$ is a semilattice-metric cover\index{i}{semilattice-metric!cover} and both~$V^*$ and $\bigcup\famm{u_i``(U_i)}{i<n}$ are finite subsets of~$B$, with $V^*\subseteq B^*$, there exists a finite subset~$V$ of~$B$ containing $V^*\cup\bigcup\famm{u_i``(U_i)}{i<n}$ such that
 \begin{equation}\label{Eq:VV*MetrCov}
 (\forall x,y,z\in V^*)(\exists t\in V)\bigl(\gd_\bB(x,t)\leq\gd_\bB(y,z)\text{ and }
 \gd_\bB(t,z)\leq\gd_\bB(x,y)\bigr)\,.
 \end{equation}
As both $\bigcup\famm{\tu_i``(\tU_i)}{i<n}$ and~$V$ are finite, there exists a finite \jz-sub\-semi\-lat\-tice~$\tV$ of~$\tB$ such that
 \[
 \bigcup\famm{\tu_i``(\tU_i)}{i<n}\cup\setm{\gd_\bB(x,y)}{x,y\in V}\subseteq\tV\,.
 \]
Denote by~$\gd_\bV$ the restriction of~$\gd_\bB$ from~$V\times V$ to~$\tV$. It follows from~\eqref{Eq:VV*MetrCov} that the quadruple $\bV:=(V^*,V,\gd_\bV,\tV)$ is a semilattice-metric cover\index{i}{semilattice-metric!cover}. Further, denote by $\xf_i\colon\bU_i\to\bV$ the restriction of~$\xu_i$ from~$\bU_i$ to~$\bV$, for each $i<n$, by~$v$ (resp., $\tv$) the inclusion map from~$V$ into~$B$ (resp., from~$\tV$ into~$\tB$), and set $\xv:=(v,\tv)$. As~$v$ and~$\tv$ are both one-to-one, $\xv$ is monic. Obviously, $\xu_i=\xv\circ\xf_i$, so $\xf_i\colon\xu_i\to\xv$ is an arrow in the comma category~$\METR^{*\fin}\dnw\bB$\index{s}{Metrsf@$\METR^{*\mathrm{fin}}$}. Furthermore, from $\gf``(V^*)=A$ it follows that $\bgf\circ\Psi(\xv)\in\METR^\Rightarrow$\index{s}{METRarr@$\METR^\Rightarrow$}.
\qed\end{proof}

\section{A family of unliftable squares of semilattice-metric spaces}\label{S:UnliftMetr}

In this section we shall introduce a family of diagrams in~$\METR$\index{s}{Metr@$\METR$}, indexed by the finite poset $\set{0,1}^2$, which are unliftable\index{i}{unliftable diagram} with respect to the functor~$\Psi$ defined in Proposition~\ref{P:MetrLarder}.

We shall say that a \emph{square}\index{i}{square (shape of a diagram)|ii} from a category~$\cC$ is a $\set{0,1}^2$-indexed diagram in~$\cC$, that is, an object of the category~$\cC^{\set{0,1}^2}$. A typical square in the category~$\METR$\index{s}{Metr@$\METR$} is represented on Figure~\ref{Fig:sqinMetr}.

\begin{figure}[htb]\index{i}{square (shape of a diagram)}
 \[
 \def\labelstyle{\displaystyle}
 \xymatrix{
 & \bA & \\
 \bA_1\ar[ur]^{\xg_1} & & \bA_2\ar[ul]_{\xg_2}\\
 & \bA_0\ar[ul]^{\xf_1}\ar[ur]_{\xf_2} &
 }
 \]
\caption{A square in the category~$\METR$}
\label{Fig:sqinMetr}
\end{figure}
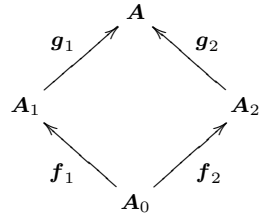

Say that elements $x$ and $y$ in a poset with zero are \emph{orthogonal}\index{i}{orthogonal|ii} if there is no nonzero element~$z$ such that $z\leq x$ and $z\leq y$.

The following lemma is a special case of a more general statement, however it will be sufficient for our purposes.

\begin{lem}\label{L:UnliftSqMetr}
Let $\overrightarrow{\bA}$ be a square\index{i}{square (shape of a diagram)} in $\METR$\index{s}{Metr@$\METR$} labeled as on Figure~\textup{\ref{Fig:sqinMetr}}, with elements $0,1\in A_0$, $a_i\in A_i$ for $i\in\set{0,1,2}$, and $\ga,\gb\in\tA_0$, satisfying the following conditions:
\begin{description}
\item[\tui] All maps $f_1$, $f_2$, $g_1$, $g_2$ are inclusion maps.

\item[\tuii] The elements $\tf_i(\ga)$ and $\tf_i(\gb)$ are orthogonal\index{i}{orthogonal} in~$\tA_i$, for each $i\in\set{1,2}$.

\item[\tuiii] $\gd_{\bA_0}(0,a_0)\leq\ga$, while $\gd_{\bA_i}(a_i,1)\leq\tf_i(\ga)$ for each $i\in\set{1,2}$.

\item[\tuiv] $\gd_{\bA_0}(a_0,1)\leq\gb$, while $\gd_{\bA_i}(0,a_i)\leq\tf_i(\gb)$ for each $i\in\set{1,2}$.
\end{description}

If there exists a square\index{i}{square (shape of a diagram)}~$\overrightarrow{\bB}$ of semilattice-metric covers\index{i}{semilattice-metric!cover} together with a morphism $\overrightarrow{\bchi}\colon\Psi\overrightarrow{\bB}\Todot\overrightarrow{\bA}$\index{s}{AtoRightarrowdotB@$f\colon\xA\Todot\xB$} in $(\METR^\Rightarrow)^{\set{0,1}^2}$\index{s}{METRarr@$\METR^\Rightarrow$}, then $\gd_\bA(a_1,a_2)=0$.
\end{lem}

We illustrate the assumptions of Lemma~\ref{L:UnliftSqMetr} on Figure~\ref{Fig:UnliftDiag}. Unlike what is suggested on the top part of that illustration, it will turn out that the liftability assumption implies that~$a_1$ and~$a_2$ are identified by the distance\index{i}{distance!semilattice-valued} function on~$\bA$.

\begin{figure}[htb]\index{i}{square (shape of a diagram)}
\includegraphics{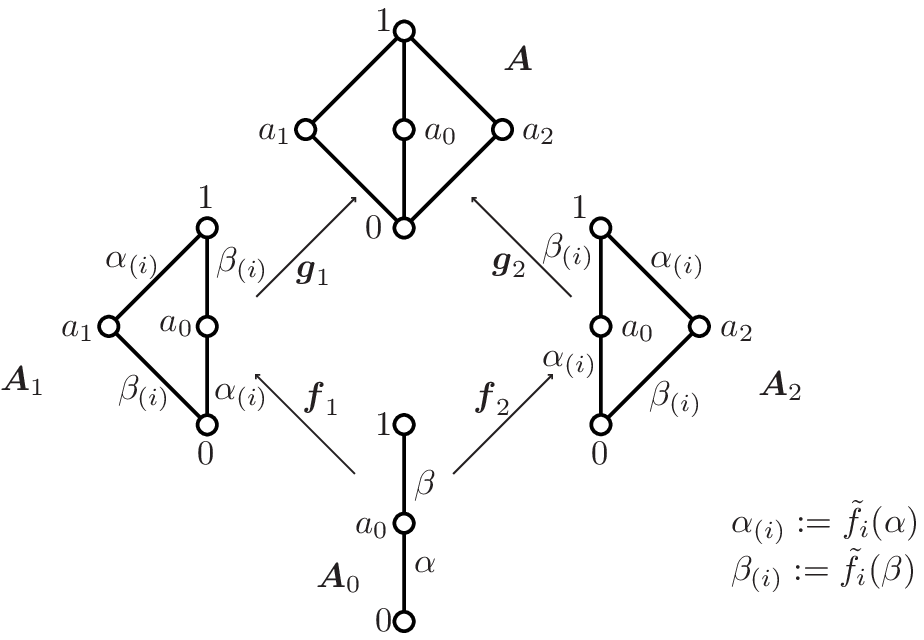}
\caption{A square in $\METR$ unliftable with respect to $\Psi$ and $\METR^\Rightarrow$}\label{Fig:UnliftDiag}
\end{figure}

\begin{proof}[Proof of Lemma~\ref{L:UnliftSqMetr}]
Represent~$\overrightarrow{\bB}$ as on the left hand side of Figure~\ref{Fig:sqinMetr*}, and set $\xw:=\xv_1\circ\xu_1=\xv_2\circ\xu_2$. Hence $\Psi\overrightarrow{\bB}$ is represented on the right hand side of Figure~\ref{Fig:sqinMetr*} (recall the flat notation, $\Psi(\bX)=\bX^\flat$\index{s}{Aiiflat@$\bA^\flat$, $f^\flat$}).

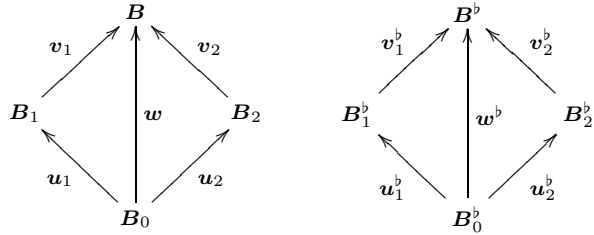
\begin{figure}[htb]\index{i}{square (shape of a diagram)}
 \[
 \def\labelstyle{\displaystyle}
 \xymatrix{
 & \bB & & & \bB^\flat &\\
 \bB_1\ar[ur]^{\xv_1} & & \bB_2\ar[ul]_{\xv_2} &
 \bB_1^\flat\ar[ur]^{\xv_1^\flat} & & \bB_2^\flat\ar[ul]_{\xv_2^\flat}\\
 & \bB_0\ar[ul]^{\xu_1}\ar[ur]_{\xu_2}\ar[uu]_{\xw} & & &
 \bB_0^\flat\ar[ul]^{\xu_1^\flat}\ar[ur]_{\xu_2^\flat}\ar[uu]_{\xw^\flat} &
 }
 \]
\caption{Squares in the categories $\METR^*$ and $\METR$}
\label{Fig:sqinMetr*}
\end{figure}

A morphism $\overrightarrow{\bchi}\colon\Psi\overrightarrow{\bB}\Todot\overrightarrow{\bA}$\index{s}{AtoRightarrowdotB@$f\colon\xA\Todot\xB$} in $(\METR^\Rightarrow)^{\set{0,1}^2}$\index{s}{METRarr@$\METR^\Rightarrow$} consists of a collection of double arrows\index{i}{double arrow} $\bchi_i\colon\bB_i^\flat\Rightarrow\bA_i$\index{s}{AtorightarrowB@$f\colon A\Rightarrow B$}\index{s}{Aiiflat@$\bA^\flat$, $f^\flat$}, for $i\in\set{0,1,2}$, and $\bchi\colon\bB^\flat\Rightarrow\bA$\index{s}{AtorightarrowB@$f\colon A\Rightarrow B$}, all in $\METR^\Rightarrow$\index{s}{METRarr@$\METR^\Rightarrow$}, such that the diagram represented in Figure~\ref{Fig:CubeMetr} commutes. Actually our proof will not require the top arrow~$\bchi\colon\bB^\flat\to\bA$\index{s}{Aiiflat@$\bA^\flat$, $f^\flat$} to be a double arrow.

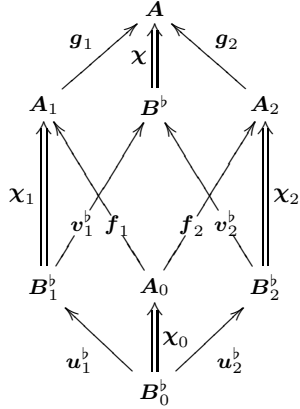
\begin{figure}[htb]\index{i}{double arrow}\index{i}{cube (shape of a diagram)}
 \[
 \def\labelstyle{\displaystyle}
 \xymatrix{
 & \bA & \\
 \bA_1\ar[ur]^{\xg_1} & \bB^\flat\ar@{=>}[u]^{\bchi} & \bA_2\ar[ul]_{\xg_2}\\
 \\
 \bB_1^\flat\ar@{=>}[uu]^{\bchi_1}\ar[uur]|-(.3){\xv_1^\flat} &
 \bA_0\ar[uul]|-(.3){\xf_1}\ar[uur]|-(.3){\xf_2} &
 \bB_2^\flat\ar[uul]|-(.3){\xv_2^\flat}\ar@{=>}[uu]_{\bchi_2}\\
 &\bB_0^\flat\ar[ul]^{\xu_1^\flat}\ar@{=>}[u]_{\bchi_0}\ar[ur]_{\xu_2^\flat} &
 }
 \]
\caption{A cube of semilattice-metric spaces}
\label{Fig:CubeMetr}
\end{figure}
As $\chi_0$ is surjective, there are $\dzero,\done,\da_0\in B_0^*$ such that $\chi_0(\dzero)=0$, $\chi_0(\done)=1$, and $\chi_0(\da_0)=a_0$. For each $i\in\set{1,2}$, as~$\chi_i$ is surjective, there exists $\da_i\in B_i^*$ such that $\chi_i(\da_i)=a_i$. As~$\bB_0$ is a semilattice-metric cover\index{i}{semilattice-metric!cover}, it follows from the Parallelogram Rule\index{i}{Parallelogram Rule} that there exists $\dx\in B_0$ such that
 \begin{equation}\label{Eq:Pptiesdx}
 \gd_{\bB_0}(\dzero,\dx)\leq\gd_{\bB_0}(\da_0,\done)\text{ and }
 \gd_{\bB_0}(\dx,\done)\leq\gd_{\bB_0}(\dzero,\da_0)\,.
 \end{equation}
Let $i\in\set{1,2}$. As $\chi_iu_i(\dzero)=f_i\chi_0(\dzero)=f_i(0)=0$, we get
 \begin{align}
 \tilde{\chi}_i\gd_{\bB_i}(\da_i,u_i(\dzero))&=\gd_{\bA_i}(\chi_i(\da_i),\chi_iu_i(\dzero))
&&(\text{because }\set{\da_i,u_i(\dzero)}\subseteq B_i^*)\notag\\
 &=\gd_{\bA_i}(a_i,0)\notag\\
 &\leq\tf_i(\gb)\,.\label{Eq:1/2chileqcong}
 \end{align}
On the other hand,
 \begin{align}
 \tilde{\chi}_i\tu_i\gd_{\bB_0}(\da_0,\done)&=\tf_i\tilde{\chi}_0\gd_{\bB_0}(\da_0,\done)
 \notag\\
 &=\tf_i\gd_{\bA_0}(a_0,1)\notag&&(\text{because }\set{\da_0,\done}\subseteq B_0^*)\\
 &\leq\tf_i(\gb)\,.\label{Eq:2/2chileqcong}
 \end{align}
As the following inequalities hold (we use~\eqref{Eq:Pptiesdx}),
 \[
 \gd_{\bB_i}(u_i(\dzero),u_i(\dx))=\tu_i\bigl(\gd_{\bB_0}(\dzero,\dx)\bigr)
 \leq\tu_i\bigl(\gd_{\bB_0}(\da_0,\done)\bigr)\,,
 \]
we obtain the inequalities
 \begin{align}
 \gd_{\bB_i}(\da_i,u_i(\dx))
 &\leq\gd_{\bB_i}(\da_i,u_i(\dzero))\vee\gd_{\bB_i}(u_i(\dzero),u_i(\dx))\notag\\
 &\leq\gd_{\bB_i}(\da_i,u_i(\dzero))\vee\tu_i\bigl(\gd_{\bB_0}(\da_0,\done)\bigr)\,.
 \label{Eq:dBiaiuileqdB*ai0}
 \end{align}
Therefore, by applying the \jzh\ $\tilde{\chi}_i$ to~\eqref{Eq:dBiaiuileqdB*ai0} and by using~\eqref{Eq:1/2chileqcong} and~\eqref{Eq:2/2chileqcong}, we obtain the inequality
 \begin{equation}\label{Eq:chiidBiaiuix}
 \tilde{\chi}_i\gd_{\bB_i}(\da_i,u_i(\dx))\leq\tf_i(\gb)\,.
 \end{equation}
A similar proof, exchanging the roles of the elements~$0$ and~$1$ of~$A_0$ and of the elements~$\ga$ and~$\gb$ of~$\tA_0$, in the argument above, leads to the inequality
 \begin{equation}\label{Eq:chiidBiaiuix'}
 \tilde{\chi}_i\gd_{\bB_i}(\da_i,u_i(\dx))\leq\tf_i(\ga)\,.
 \end{equation}
As $\tf_i(\ga)$ and $\tf_i(\gb)$ are orthogonal\index{i}{orthogonal} in $\tA_i$, it follows from~\eqref{Eq:chiidBiaiuix} and~\eqref{Eq:chiidBiaiuix'} that the following equation is satisfied:
 \begin{equation}\label{Eq:chiidBiaiuix0}
 \tilde{\chi}_i\gd_{\bB_i}(\da_i,u_i(\dx))=0\,.
 \end{equation}
By applying $\tg_i$ to the equation~\eqref{Eq:chiidBiaiuix0} and using the equation $\tg_i\circ\tilde{\chi}_i=\tilde{\chi}\circ\tv_i$, we obtain the equation
 \[
 \tilde{\chi}\tv_i\gd_{\bB_i}(\da_i,u_i(\dx))=0\,,
 \]
that is (cf. Figure~\ref{Fig:sqinMetr*} for the definition of~$\xw$),
 \[
 \tilde{\chi}\gd_{\bB}(v_i(\da_i),w(\dx))=0\,.
 \]
This holds for each $i\in\set{1,2}$, hence, by using the triangular inequality,
 \[
 \tilde{\chi}\gd_{\bB}(v_1(\da_1),v_2(\da_2))=0\,.
 \]
As both elements $v_1(\da_1)$ and $v_2(\da_2)$ belong to~$B^*$, we obtain the equation
 \[
 \gd_{\bA}(\chi v_1(\da_1),\chi v_2(\da_2))=0\,.
 \]
As $(\chi\circ v_i)(\da_i)=(g_i\circ\chi_i)(\da_i)=g_i(a_i)=a_i$, for each $i\in\set{1,2}$, the conclusion of the lemma follows.
\qed\end{proof}

\section{A left larder involving algebras and semilattice-metric spaces}\label{S:MalgMetrLard}

The original problem that we wish to solve involves congruence-preserving extensions\index{i}{congruence-preserving extension} of \emph{algebras}\index{i}{algebra!universal}, that is, first-order structures without relation symbols (cf. Example~\ref{Ex:1/2nonindexed}). In the present section, we shall associate functorially, to every algebra~$\bA$\index{i}{algebra!universal}, a semilattice-metric space\index{i}{semilattice-metric!space}~$\bA^\natural$\index{s}{Aiinatural@$\bA^\natural$, $f^\natural$}, thus creating a left larder\index{i}{larder!left} (cf. Proposition~\ref{P:MalgMetrLeftLard}). Furthermore, the assignment $(\bA\mapsto\bA^\natural)$ turns quotients of algebras\index{i}{algebra!universal} to quotients of semilattice-metric spaces\index{i}{semilattice-metric!space} (Proposition~\ref{P:QuotMalgMetr}).

\begin{notation}\label{Not:MAlg}
We denote by $\MALG$\index{s}{Malg@$\MALG$|ii} the full subcategory of~$\MIND$\index{s}{Mind@$\MIND$} whose objects are all the \emph{algebras}\index{i}{algebra!universal}, that is, the first-order structures without relation symbols. Furthermore, for an algebra\index{i}{algebra!universal}~$\bA$, we set~$\bA^\natural:=(A,\gd_\bA,\Conc\bA)$\index{s}{compcon1@$\Conc\bA$, $\Conc f$}\index{s}{Aiinatural@$\bA^\natural$, $f^\natural$|ii}, where~$\gd_\bA(x,y)$ is defined as the congruence of~$\bA$ generated by the pair~$(x,y)$, for all $x,y\in A$. That is, using the notation of Section~\ref{S:ConFunct}, $\gd_\bA(x,y):=\left\Vert x=y\right\Vert_{\bA}$.

For a morphism $f\colon\bA\to\bB$ in~$\MALG$\index{s}{Malg@$\MALG$}, we set $f^\natural:=(f,\Conc f)$\index{s}{Aiinatural@$\bA^\natural$, $f^\natural$}, where $\Conc f$ is the natural \jzh\ from~$\Conc\bA$\index{s}{compcon1@$\Conc\bA$, $\Conc f$} to~$\Conc\bB$\index{s}{compcon1@$\Conc\bA$, $\Conc f$} (cf. Section~\ref{S:ConFunct}).
\end{notation}

\begin{lem}\label{L:DirMrpdMAlg}
The category~$\MALG$\index{s}{Malg@$\MALG$} has arbitrary small products \pup{indexed by nonempty sets}.
\end{lem}

\begin{proof}
Let $I$ be a nonempty set and let $\famm{\bA_i}{i\in I}$ be an $I$-indexed family of algebras\index{i}{algebra!universal}. We form the intersection $\scL:=\bigcap\famm{\Lg(\bA_i)}{i\in I}$, the cartesian product $A:=\prod\famm{A_i}{i\in I}$, and for each $f\in\scL$, say of arity~$n$, we define $f^\bA\colon A^n\to A$ the usual way, that is,
 \[
 f^\bA\bigl((x_i^1\mid i\in I),\dots,\famm{x_i^n}{i\in I}\bigr):=
 \famm{f^{\bA_i}(x_i^1,\dots,x_i^n)}{i\in I}\,,
 \]
for all elements $\famm{x_i^k}{i\in I}$ (for $1\leq k\leq n$) in~$A$. The interpretations of the constants are defined similarly. This defines a first-order structure~$\bA:=\bigl(A,\famm{f^\bA}{f\in\scL}\bigr)$. It is straightforward to verify that~$\bA$, together with the canonical projections $\bA\onto\bA_i$\index{s}{AtoonB@$f\colon A\onto B$}, for $i\in I$, is the product of the family $\famm{\bA_i}{i\in I}$ in~$\MALG$\index{s}{Malg@$\MALG$}.
\qed\end{proof}

\begin{prop}\label{P:MalgMetrLeftLard}
Let $\Phi\colon\MALG\to\METR$\index{s}{Malg@$\MALG$}\index{s}{Metr@$\METR$} be the $\bX\mapsto\bX^\natural$\index{s}{Aiinatural@$\bA^\natural$, $f^\natural$} functor introduced in Notation~\textup{\ref{Not:MAlg}}. Then the quadruple $(\MALG,\METR,\METR^\Rightarrow,\Phi)$\index{s}{Malg@$\MALG$}\index{s}{Metr@$\METR$} is a left larder\index{i}{larder!left}.
\end{prop}

\begin{proof}
The condition $(\CLOS(\MALG))$\index{s}{ClosA@$(\CLOS(\cA))$} follows from the closure of $\MALG$\index{s}{Malg@$\MALG$} under directed colimits within~$\MIND$\index{s}{Mind@$\MIND$} (this follows trivially from Proposition~\ref{P:DirColimMIND}). The condition $(\PROD(\MALG))$\index{s}{ProdA@$(\PROD(\cA))$} follows from Lemma~\ref{L:DirMrpdMAlg}. The condition $(\CONT(\Phi))$\index{s}{Cont@$(\CONT(\Phi))$} follows easily from the description of the directed colimits given in Proposition~\ref{P:DirColimMIND} and Lemma~\ref{L:FgtMetrPres}, together with the preservation of directed colimits by the~$\Conc$ functor (Theorem~\ref{T:ConcVPresDirColim}).

The projection $X\times Y\onto X$\index{s}{AtoonB@$f\colon A\onto B$} is surjective, for all nonempty sets~$X$ and~$Y$. It follows easily that the map~$f$ is surjective for each extended projection $\xf\colon\bX\to\bY$ in~$\MALG$\index{s}{Malg@$\MALG$}. The condition~$(\PROJ(\Phi,\METR^\Rightarrow))$\index{s}{Proj@$(\PROJ(\Phi,\cS^\Rightarrow))$} follows.
\qed\end{proof}

Our next lemma will make it possible to define \emph{quotients} of semilattice-metric spaces\index{i}{semilattice-metric!space}. The proof is straightforward and we shall omit it. We refer to Section~\ref{S:MINDProjWit} for the definition of the quotient of a commutative monoid by an o-ideal\index{i}{ideal!o-}.

\begin{lem}\label{L:QuotMetr}
Let $\bA$ be a semilattice-metric space\index{i}{semilattice-metric!space} and let $I$ be an ideal of~$\tA$\index{i}{ideal!of a poset}. Then the binary relation
 \[
 \ga:=\setm{(x,y)\in A\times A}{\gd_\bA(x,y)\in I}
 \]
is an equivalence relation on~$A$, and there exists a unique $\tA/I$-valued distance\index{i}{distance!semilattice-valued}~$\gd$ on~$A/\ga$ such that
 \[
 \gd(x/\ga,y/\ga)=\gd_\bA(x,y)/I\,,\quad\text{for all }x,y\in A\,.
 \]
Denote by~$\bA/I$\index{s}{AoverImetr@$\bA/I$ ($\bA\in\METR$, $I$ ideal of~$\tA$)|ii} the semilattice-metric space\index{i}{semilattice-metric!space}~$(A/\ga,\gd,\tA/I)$. Then the pair $(\gp,\tilde{\gp})$, where~$\gp$ \pup{resp., $\tilde{\gp}$}denotes the canonical projection from~$A$ onto~$A/\ga$ \pup{resp., from~$\tA$ onto~$\tA/I$}, is a double arrow\index{i}{double arrow} from~$\bA$ onto~$\bA/I$ in~$\METR$\index{s}{Metr@$\METR$}.
\end{lem}

In the context of Lemma~\ref{L:QuotMetr}, we shall say that the semilattice-metric space\index{i}{semilattice-metric!space}~$\bA/I$ is the \emph{quotient} of~$\bA$ by the ideal\index{i}{ideal!of a poset}~$I$, and that the pair $(\gp,\tilde{\gp})$ is the \emph{canonical projection} from~$\bA$ onto~$\bA/I$.

Quotients of semilattice-metric spaces\index{i}{semilattice-metric!space} and quotients of algebras\index{i}{algebra!universal} are related by the following consequence of Lemma~\ref{L:ConcVProj}.

\begin{prop}\label{P:QuotMalgMetr}
Let $\ga$ be a congruence of an algebra\index{i}{algebra!universal}~$\bA$, and set $I:=(\Conc\bA)\dnw\ga$\index{s}{compcon1@$\Conc\bA$, $\Conc f$}. Then the semilattice-metric spaces\index{i}{semilattice-metric!space} $(\bA/\ga)^\natural$ and $\bA^\natural/I$\index{s}{Aiinatural@$\bA^\natural$, $f^\natural$} are isomorphic.
\end{prop}

\begin{proof}
It follows from the definition of the ${}^\natural$ operator that\index{s}{Aiinatural@$\bA^\natural$, $f^\natural$}
 \[
 (\bA/\ga)^\natural=\bigl(A/\ga,\gd_{\bA/\ga},\Conc(\bA/\ga)\bigr)\,,
 \]
where, by using the Second Isomorphism Theorem (Lemma~\ref{L:SecIsomThm})\index{i}{Isomorphism Theorem (Second ${}_{-}$)},
 \[
 \gd_{\bA/\ga}(x/\ga,y/\ga)=\Vert x/\ga=y/\ga\Vert_{\bA/\ga}=
 \bigl(\ga\vee\Vert x=y\Vert_\bA\bigr)/\ga=\bigl(\ga\vee\gd_\bA(x,y)\bigr)/\ga\,,
 \]
for any $x,y\in A$. On the other hand, the equivalence~$\ga'$ on~$A$ determined, as in Lemma~\ref{L:QuotMetr}, by the ideal\index{i}{ideal!of a poset}~$I$ and the distance\index{i}{distance!semilattice-valued}~$\gd_\bA$ satisfies
 \[
 (x,y)\in\ga'\Leftrightarrow\gd_\bA(x,y)\in I\Leftrightarrow\Vert x=y\Vert_\bA\leq\ga
 \Leftrightarrow(x,y)\in\ga\,,
 \]
for any $x,y\in A$; thus $\ga=\ga'$. Therefore\index{s}{compcon1@$\Conc\bA$, $\Conc f$},
 \[
 \bA^\natural/I=\bigl(A/\ga,\gd,(\Conc\bA)/I\bigr)\,,
 \]
where\index{s}{Aiinatural@$\bA^\natural$, $f^\natural$}
 \[
 \gd(x/\ga,y/\ga):=\gd_\bA(x,y)/I\,,\quad\text{for all }x,y\in A\,.
 \]
Now it follows from Lemma~\ref{L:ConcVProj} together with the Second Isomorphism Theorem (Theorem~\ref{L:SecIsomThm})\index{i}{Isomorphism Theorem (Second ${}_{-}$)} that the assignment $(\gx/I\mapsto(\ga\vee\gx)/\ga)$ defines an isomorphism $\gf\colon(\Conc\bA)/I\to\Conc(\bA/\ga)$\index{s}{compcon1@$\Conc\bA$, $\Conc f$}. As, by definition, $\gf\bigl(\gd_\bA(x,y)/I\bigr)=\bigl(\ga\vee\gd_\bA(x,y)\bigr)/\ga$ for any $x,y\in A$, the pair $(\id_{A/\ga},\gf)$ is an isomorphism from~$\bA^\natural/I$ onto~$(\bA/\ga)^\natural$\index{s}{Aiinatural@$\bA^\natural$, $f^\natural$}. 
\qed\end{proof}

\section{CPCP-retracts and CPCP-extensions}\label{S:CPCP}

In this section we finally solve the problem stated at the beginning of Chapter~\ref{Ch:CongPres}, in the negative (Corollary~\ref{C:NoCPCP}). In fact, we prove a stronger negative statement (Theorem~\ref{T:NoCPCP}). This statement involves the notion of a CPCP-retract\index{i}{CPCP-retract}.

\begin{defn}\label{D:CPCPRetr}
A semilattice-metric space\index{i}{semilattice-metric!space}~$\bA$ is a \emph{CPCP-retract}\index{i}{CPCP-retract|ii} if there are a semilattice-metric cover\index{i}{semilattice-metric!cover}~$\bB$ and a double arrow\index{i}{double arrow} $\bchi\colon\bB^\flat\Rightarrow\bA$\index{s}{AtorightarrowB@$f\colon A\Rightarrow B$}\index{s}{Aiiflat@$\bA^\flat$, $f^\flat$}.
\end{defn}

\begin{prop}\label{P:QuotCPCP}
Let~$\bA$ be a semilattice-metric space\index{i}{semilattice-metric!space} and let~$I$ be an ideal\index{i}{ideal!of a poset} of~$\tA$. If~$\bA$ is a CPCP-retract\index{i}{CPCP-retract}, then so is~$\bA/I$.
\end{prop}

\begin{proof}
The canonical projection $\bgp\colon\bA\to\bA/I$ is a double arrow\index{i}{double arrow} in~$\METR$\index{s}{Metr@$\METR$}. It follows immediately that if~$\bB$ is a semilattice-metric cover\index{i}{semilattice-metric!cover} and $\bchi\colon\bB^\flat\Rightarrow\bA$\index{s}{AtorightarrowB@$f\colon A\Rightarrow B$}, then $\bgp\circ\bchi\colon\bB^\flat\Rightarrow\bA/I$\index{s}{Aiiflat@$\bA^\flat$, $f^\flat$}.
\qed\end{proof}

\begin{defn}\label{D:CPCPext}
An algebra\index{i}{algebra!universal}~$\bB$ is a \emph{CP-extension}\index{i}{CP-extension|ii} of an algebra\index{i}{algebra!universal}~$\bA$ if $A\subseteq B$, $\Lg(\bA)\subseteq\Lg(\bB)$, the inclusion mapping $e\colon A\into B$\index{s}{AtoinB@$f\colon A\into B$} is a morphism in~$\MALG$\index{s}{Malg@$\MALG$}, and~$\Con e$ is an isomorphism from~$\Con\bA$\index{s}{conA@$\Con\bA$} onto~$\Con\bB$\index{s}{conA@$\Con\bA$}. In other words, every congruence of~$\bA$ extends to a unique congruence of~$\bB$.

We say that a CP-extension\index{i}{CP-extension}~$\bB$ of~$\bA$ is a \emph{CPCP-extension}\index{i}{CPCP-extension|ii} of~$\bA$ if for all $x,y,z\in A$ there exists $t\in B$ such that $(x,t)\in\left\Vert y=z\right\Vert_\bB$ and $(t,z)\in\left\Vert x=y\right\Vert_\bB$.
\end{defn}

In particular, if~$\bB$ is a CP-extension\index{i}{CP-extension} of~$\bA$ and any two congruences of~$\bB$ permute (i.e., $\ga\gb=\gb\ga$ for all $\ga,\gb\in\Con\bB$\index{s}{conA@$\Con\bA$}), then~$\bB$ is a CPCP-extension\index{i}{CPCP-extension} of~$\bA$.

\begin{prop}\label{P:aroundCP}
If an algebra\index{i}{algebra!universal}~$\bA$ has a CPCP-extension\index{i}{CPCP-extension}, then the sem\-i\-lat\-tice-metric space\index{i}{semilattice-metric!space}~$\bA^\natural$\index{s}{Aiinatural@$\bA^\natural$, $f^\natural$} is a CPCP-retract\index{i}{CPCP-retract}.
\end{prop}

\begin{proof}
Let~$\bB$ be a CPCP-extension\index{i}{CPCP-extension} of~$\bA$ and denote by $e\colon A\into B$\index{s}{AtoinB@$f\colon A\into B$} the inclusion mapping. Set $\gd_\bB(x,y):=\left\Vert x=y\right\Vert_\bB$, for all $x,y\in B$. Then $\bC:=(A,B,\gd_\bB,\Conc\bB)$\index{s}{compcon1@$\Conc\bA$, $\Conc f$} is a semilattice-metric cover\index{i}{semilattice-metric!cover}. By assumption, $\Conc e$ is an isomorphism from~$\Conc\bA$ onto~$\Conc\bB$\index{s}{compcon1@$\Conc\bA$, $\Conc f$}. Then $(\id_A,(\Conc e)^{-1})\colon\bC^\flat\Rightarrow\bA^\natural$\index{s}{AtorightarrowB@$f\colon A\Rightarrow B$}\index{s}{Aiinatural@$\bA^\natural$, $f^\natural$}\index{s}{Aiiflat@$\bA^\flat$, $f^\flat$}.
\qed\end{proof}

Now we reach the main theorem of Chapter~\ref{Ch:CongPres}.

\begin{thm}\label{T:NoCPCP}
Let $\cV$ be a nondistributive\index{i}{distributive!non-${}_{-}$ variety} variety of lattices and let~$\bF$ be a free bounded lattice on at least~$\aleph_1$\index{s}{aleph0@$\aleph_{\ga}$} generators within~$\cV$. Then~$\bF^\natural$\index{s}{Aiinatural@$\bA^\natural$, $f^\natural$} is not a CPCP-retract\index{i}{CPCP-retract}. Consequently, $\bF$ has no CPCP-extension\index{i}{CPCP-extension}.
\end{thm}

Consequently, $\bF$ has no congruence-preserving\index{i}{congruence-preserving extension}, congruence-permutable\index{i}{algebra!congruence-permutable} extension.

\begin{proof}
It suffices to prove that there exists a bounded lattice~$\bL$ of cardinality at most~$\aleph_1$\index{s}{aleph0@$\aleph_{\ga}$} within~$\cV$ such that~$\bL^\natural$\index{s}{Aiinatural@$\bA^\natural$, $f^\natural$} is not a CPCP-retract\index{i}{CPCP-retract}. Indeed, $\bL$ is a quotient of~$\bF$, thus, by Proposition~\ref{P:QuotMalgMetr}, $\bL^\natural$\index{s}{Aiinatural@$\bA^\natural$, $f^\natural$} is a quotient of~$\bF^\natural$\index{s}{Aiinatural@$\bA^\natural$, $f^\natural$}. Hence, by Proposition~\ref{P:QuotCPCP}, if~$\bF^\natural$\index{s}{Aiinatural@$\bA^\natural$, $f^\natural$} were a CPCP-retract\index{i}{CPCP-retract}, then so would be~$\bL^\natural$\index{s}{Aiinatural@$\bA^\natural$, $f^\natural$}. This proves our claim.

As~$\cV$ is a nondistributive\index{i}{distributive!non-${}_{-}$ variety} lattice variety, either the five-element nonmodular\index{i}{lattice!modular} lattice~$\bN_5$ or the five-element modular\index{i}{lattice!modular} nondistributive lattice~$\bM_3$ belong to~$\cV$. Label the elements of~$\bM_3$ and~$\bN_5$ as on Figure~\ref{Fig:M3N5}.

\begin{figure}[htb]
\includegraphics{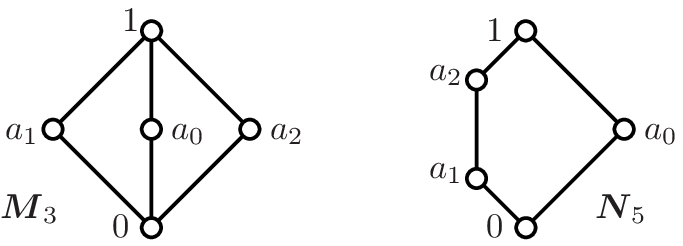}
\caption{Labeling the lattices $\bM_3$ and $\bN_5$}\label{Fig:M3N5}
\end{figure}

Now we shall construct a square\index{i}{square (shape of a diagram)} of lattices $\bA_0$, $\bA_1$, $\bA_2$, $\bA$ with inclusion maps $f_i\colon\bA_0\into\bA_i$\index{s}{AtoinB@$f\colon A\into B$}, $g_i\colon\bA_i\into\bA$\index{s}{AtoinB@$f\colon A\into B$}, as follows. If~$\bM_3$ belongs to~$\cV$, let $\overrightarrow{\bA}$ be the diagram on the left hand side of Figure~\ref{Fig:M3N5diagram}. If~$\bM_3$ does not belong to~$\cV$, let $\overrightarrow{\bA}$ be the diagram on the right hand side of Figure~\ref{Fig:M3N5diagram}. Observe that~$\overrightarrow{\bA}$ is always a square\index{i}{square (shape of a diagram)} of finite lattices in~$\cV$.

\begin{figure}[htb]\index{i}{square (shape of a diagram)}
\includegraphics{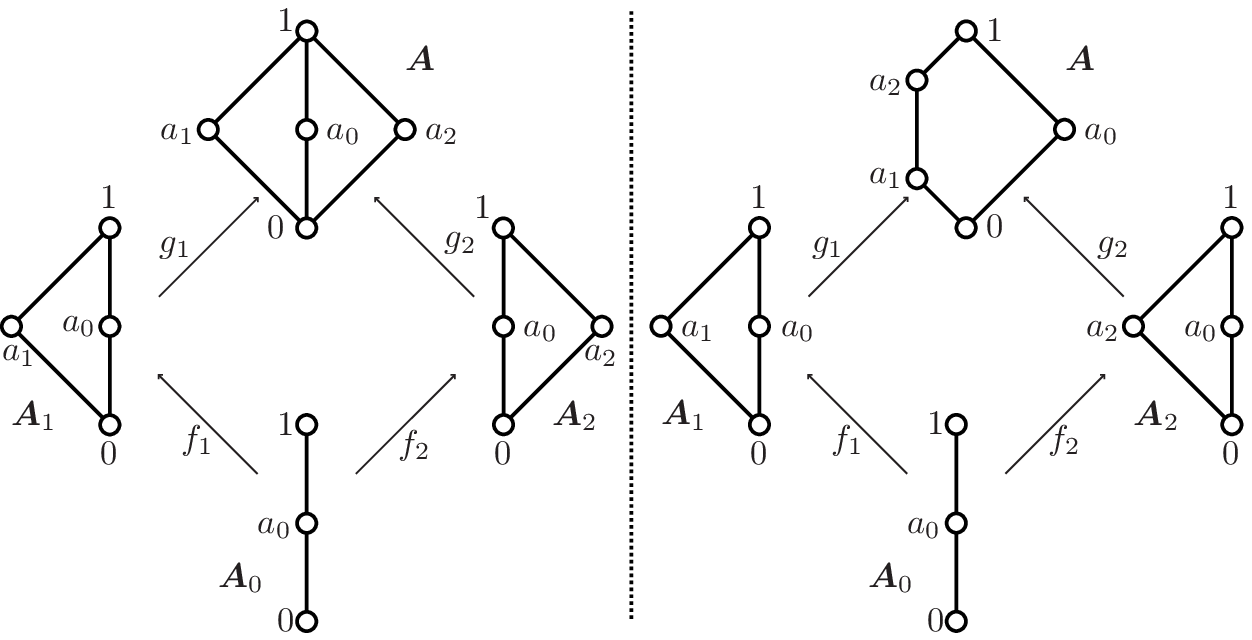}
\caption{Two squares of finite lattices}\label{Fig:M3N5diagram}
\end{figure}

It follows from the comments following Definition~\ref{D:KurInd} that $\kur(\set{0,1}^2)\leq 2$\index{s}{kurP@$\kur(P)$} (in fact $\kur(\set{0,1}^2)=2$\index{s}{kurP@$\kur(P)$}), thus the relation $(\aleph_1,{<}\aleph_0)\leadsto\set{0,1}^2$\index{s}{aleph0@$\aleph_{\ga}$}\index{s}{arr0x@$(\gk,{<}\gl)\leadsto P$} holds. As this argument relies on results from our paper~\cite{GiWe1}\index{c}{Gillibert, P.}\index{c}{Wehrung, F.}, we shall present a direct proof.

Let $F\colon[\go_1]^{<\go}\to[\go_1]^{<\go}$ be an isotone map. It follows easily from Kuratowski's characterization of~$\aleph_1$\index{s}{aleph0@$\aleph_{\ga}$} given in Kuratowski~\cite{Kura51}\index{c}{Kuratowski, C.} (see also \cite[Theorem~46.1]{EHMR}\index{c}{Erd\H{o}s, P.}\index{c}{Hajnal, A.}\index{c}{Mate@M\'at\'e, A.}\index{c}{Rado, R.}) that there are distinct $\ga_0,\ga_1\in\go_1\setminus F(\set{0})$ such that $\ga_0\notin F(\set{0,\ga_1})$ and $\ga_1\notin F(\set{0,\ga_0})$. Pick $\ga\in\go_1\setminus\bigl(\set{\ga_0,\ga_1}\cup F(\set{\ga_0})\cup F(\set{\ga_1})\bigr)$. Then the map
 \[
 f\colon\set{0,1}^2\to\go_1\,,\quad(0,0)\mapsto 0\,,\ (1,0)\mapsto\ga_0\,,\ 
 (0,1)\mapsto\ga_1\,,\ (1,1)\mapsto\ga
 \]
is one-to-one, and
 \[
 F(f``(\set{0,1}^2\dnw p))\cap f``(\set{0,1}^2\dnw q)\subseteq f``(\set{0,1}^2\dnw p)\,,
 \]
for all $p,q\in\set{0,1}^2$ such that $p\leq q$. This completes the proof of our claim.

Hence, by Lemma~\ref{L:Part2Lift}, the square\index{i}{square (shape of a diagram)}~$\set{0,1}^2$ has an $\aleph_0$-lifter\index{s}{aleph0@$\aleph_{\ga}$}~$(X,\bX)$\index{i}{lifter ($\gl$-)} where~$X$ has cardinality~$\aleph_1$\index{s}{aleph0@$\aleph_{\ga}$}. As~$\overrightarrow{\bA}$ is a diagram of bounded lattices and $0,1$-lattice embeddings in~$\cV$, the condensate\index{i}{condensate} $\bL:=\xF(X)\otimes\overrightarrow{\bA}$\index{s}{FxX@$\xF(X)$}\index{s}{otimAS@$\bA\otimes\overrightarrow{S}$, $\gf\otimes\overrightarrow{S}$} is a bounded lattice in~$\cV$. It has cardinality at most $\aleph_1$\index{s}{aleph0@$\aleph_{\ga}$}---one can prove that it is exactly~$\aleph_1$\index{s}{aleph0@$\aleph_{\ga}$} but this will not matter here.

We apply CLL\index{i}{Condensate Lifting Lemma (CLL)} (Lemma~\ref{L:CLL}) to the $\aleph_0$-larder\index{s}{aleph0@$\aleph_{\ga}$}\index{i}{larder} given by Propositions~\ref{P:MetrLarder} and~\ref{P:MalgMetrLeftLard} (via the trivial Proposition~\ref{P:LR2Larder}), namely\index{s}{Malg@$\MALG$}\index{s}{Metr@$\METR$}\index{s}{Metrsf@$\METR^{*\mathrm{fin}}$}\index{s}{Metrs@$\METR^*$}
 \[
 (\MALG,\METR^*,\METR,\MALG^\fin,\METR^{*\fin},
 \METR^\Rightarrow,\Phi,\Psi)\,,
 \]
where $\MALG^\fin$\index{s}{Malgf@$\MALG^{\mathrm{fin}}$|ii} denotes the class of all finite algebras\index{i}{algebra!universal} and~$\Phi$ and~$\Psi$ are the functors introduced in Sections~\ref{S:MalgMetrLard} and~\ref{S:MetrCov}, respectively. Suppose that~$\bL^\natural$\index{s}{Aiinatural@$\bA^\natural$, $f^\natural$} is a CPCP-retract\index{i}{CPCP-retract}. By definition, this means that there are a semilattice-metric cover\index{i}{semilattice-metric!cover}~$\bB$ and a double arrow\index{i}{double arrow} $\bchi\colon\Psi(\bB)\Rightarrow\Phi(\xF(X)\otimes\overrightarrow{\bA})$\index{s}{AtorightarrowB@$f\colon A\Rightarrow B$}\index{s}{FxX@$\xF(X)$}\index{s}{otimAS@$\bA\otimes\overrightarrow{S}$, $\gf\otimes\overrightarrow{S}$}. Now it follows from CLL\index{i}{Condensate Lifting Lemma (CLL)} (Lemma~\ref{L:CLL}) that there are a square\index{i}{square (shape of a diagram)}~$\overrightarrow{\bB}$ from~$\METR^*$\index{s}{Metrs@$\METR^*$} and a double arrow\index{i}{double arrow} $\overrightarrow{\bchi}\colon\Psi\overrightarrow{\bB}\Rightarrow\Phi\overrightarrow{\bA}$\index{s}{AtorightarrowB@$f\colon A\Rightarrow B$}.

However, $\Phi\overrightarrow{\bA}$ is a diagram of semilattice-metric spaces\index{i}{semilattice-metric!space} of the sort described in the assumptions of Lemma~\ref{L:UnliftSqMetr}, with $\ga:=\left\Vert0=a_0\right\Vert_{\bA_0}$ and $\gb:=\left\Vert a_0=1\right\Vert_{\bA_0}$: for example, both $\xf_i=\Conc f_i$ are isomorphisms, with $\Conc\bA_0\cong\Conc\bA_1\cong\Conc\bA_2\cong\set{0,1}^2$\index{s}{compcon1@$\Conc\bA$, $\Conc f$}, and~$\ga$ and~$\gb$ correspond, through those isomorphisms, to the atoms~$(1,0)$ and~$(0,1)$ of the square\index{i}{square (shape of a diagram)}~$\set{0,1}^2$, which are orthogonal\index{i}{orthogonal} in~$\set{0,1}^2$.

On the other hand, $\gd_\bA(a_1,a_2)\neq0$ in both cases: a contradiction.
\qed\end{proof}

As the free bounded lattice on~$\aleph_1$\index{s}{aleph0@$\aleph_{\ga}$} generators within~$\cV$ is a homomorphic image of the free lattice on~$\aleph_1$\index{s}{aleph0@$\aleph_{\ga}$} generators within~$\cV$, the result of Theorem~\ref{T:NoCPCP} applies to the latter lattice as well. This comment also applies to the following immediate corollary.

\begin{cor}\label{C:NoCPCP}
Let $\cV$ be a nondistributive\index{i}{distributive!non-${}_{-}$ variety} variety of lattices. Then the free lattice \pup{resp. free bounded lattice} on~$\aleph_1$\index{s}{aleph0@$\aleph_{\ga}$} generators within~$\cV$ has no congruence-permutable\index{i}{algebra!congruence-permutable}, congruence-preserving extension\index{i}{congruence-preserving extension}.
\end{cor}

\begin{remk}\label{Rk:NoCPCP}
It is not hard to verify that the class of all lattices with $3$-permutable\index{i}{permut@$m$-permutable congruence lattice} congruences is closed under directed colimits and finite direct products (the verification for finite direct products requires congruence-distributivity). As $\bA_0$, $\bA_1$, $\bA_2$, and~$\bA$ are all bounded lattices with $3$-permutable\index{i}{permut@$m$-permutable congruence lattice} congruences while the~$f_i$ and~$g_i$ are $0,1$-lattice homomorphisms, it follows that the condensate $\bL:=\xF(X)\otimes\overrightarrow{\bA}$\index{s}{FxX@$\xF(X)$}\index{s}{otimAS@$\bA\otimes\overrightarrow{S}$, $\gf\otimes\overrightarrow{S}$} of the proof of Theorem~\ref{T:NoCPCP} is also a bounded lattice with $3$-permutable\index{i}{permut@$m$-permutable congruence lattice} congruences. Nevertheless, this lattice has no CPCP-extension.
\end{remk}

\chapter{Larders from von Neumann regular rings}\label{Ch:RegRngLard}

\textbf{Abstract.} The assignment that sends a regular ring\index{i}{regular ring}~$R$ to its lattice of all principal right ideals\index{i}{ideal (in a ring)!right} can be naturally extended to a functor, denoted by~$\LL$\index{s}{LLR@$\LL(R)$, $\LL(f)$} (cf. Section~\ref{Su:PartFunctSol}). An earlier occurrence of a condensate-like\index{i}{condensate} construction is provided by the proof in Wehrung~\cite[Theorem~9.3]{CXCoord}. This construction turns the non-liftability of a certain $0,1$-lattice endomorphism from~$\bM_\go$\index{s}{Momeg@$\bM_\omega$} (cf. Example~\ref{Ex:Momega}) to a non-coordinatizable\index{i}{lattice!coordinatizable}, $2$-distributive\index{i}{lattice!$2$-distributive} complemented\index{i}{lattice!complemented} modular\index{i}{lattice!modular} lattice, of cardinality~$\aleph_1$\index{s}{aleph0@$\aleph_{\ga}$}, with a spanning\index{i}{spanning (sublattice)}~$\bM_\go$\index{s}{Momeg@$\bM_\omega$}. Thus the idea to adapt the functor~$\LL$\index{s}{LL@$\LL$ functor} to our larder\index{i}{larder} context is natural. The present chapter is designed for this goal. In addition, it will pave the categorical way for solving, in the second author's paper~\cite{Banasch2}\index{c}{Wehrung, F.}, a 1962 problem by J\'onsson\index{c}{Jonsson@J\'onsson, B.}.

\section{Ideals of regular rings and of lattices}\label{S:IdRegRings}
In the present section we shall establish a few basic facts about regular rings\index{i}{regular ring} without unit and their ideals, some of which, although they are well-known in the unital case, have not always, to our knowledge, appeared anywhere in print in the non-unital case.

All our rings will be associative but not necessarily unital. Likewise, our ring homomorphisms will not necessarily preserve the unit even if it exists. An element~$b$ in a ring~$R$ is a \emph{quasi-inverse}\index{i}{quasi-inverse|ii} of an element~$a$ if $a=aba$. We say that~$R$ is (von~Neumann) \emph{regular}\index{i}{regular ring|ii} if every element of~$R$ has a quasi-inverse in~$R$. A reference for unital regular rings\index{i}{regular ring} is Goodearl's monograph~\cite{Good91}\index{c}{Goodearl, K.\,R.}.

We start with the following well-known fact, which is contained, in the regular case, in \cite[Section~VI.4]{Maed}\index{c}{Maeda, F.}. However, the result is valid for general rings. We include a proof for convenience.

\begin{lem}[folklore]\label{L:ayxbxy}
Let~$a$ and~$b$ be idempotent elements in a ring~$R$. Then the following are equivalent:
\begin{description}
\item[\tui] $aR$ and~$bR$ are isomorphic as \pup{non necessarily unital} right $R$-modules.

\item[\tuii] There are mutually quasi-inverse\index{i}{quasi-inverse} elements $x,y\in R$ such that $a=yx$ and $b=xy$.

\item[\tuiii] There are $x,y\in R$ such that $a=yx$ and $b=xy$.
\end{description}
\end{lem}

\begin{proof}
(i)$\Rightarrow$(ii). Let $f\colon aR\to bR$ be an isomorphism of right $R$-modules. The element $x:=f(a)$ belongs to~$bR$ and $x=f(aa)=f(a)a=xa$ belongs to~$Ra$, so $x\in bRa$. Likewise, the element $y:=f^{-1}(b)$ belongs to~$aRb$. Furthermore,
 \[
 a=f^{-1}f(a)=f^{-1}(x)=f^{-1}(bx)=f^{-1}(b)x=yx\,,
 \]
likewise $b=xy$. Finally, $xyx=bx=x$ (we use $xy=b$ and $x\in bR$), likewise $yxy=y$.

(ii)$\Rightarrow$(iii) is trivial. Finally, assume that~(iii) holds. For each $t\in aR$, the element $xt=xat=xyxt=bxt$ belongs to $bR$. It follows that the assignment $(t\mapsto xt)$ defines a homomorphism $f\colon aR\to bR$ of right $R$-modules. Likewise, the assignment $(t\mapsto yt)$ defines a homomorphism $g\colon bR\to aR$ of right $R$-modules. {}From $gf(a)=yxa=a^2=a$ it follows that $g\circ f=\id_{aR}$; likewise $f\circ g=\id_{bR}$, so~$f$ and~$g$ are mutually inverse.
\qed\end{proof}

\begin{cor}\label{C:Iidrng2st}
Let~$I$ be a two-sided ideal\index{i}{ideal (in a ring)!two-sided} of a ring~$R$ and let $a$ and~$b$ be idempotent elements of~$R$ such that~$aR$ and~$bR$ are isomorphic as right $R$-modules. Then $a\in I$ if{f} $b\in I$.
\end{cor}

\begin{proof}
It follows from Lemma~\ref{L:ayxbxy} that there are mutually quasi-inverse\index{i}{quasi-inverse} elements $x,y\in R$ such that $a=yx$ and $b=xy$. If $a\in I$, then $y=yxy=ay$ also belongs to~$I$, hence $b=xy$ belongs to~$I$.
\qed\end{proof}

For a regular ring\index{i}{regular ring} $R$, we set\index{s}{LLR@$\LL(R)$, $\LL(f)$|ii}
 \[
 \LL(R):=\setm{xR}{x\in R}=\setm{xR}{x\in R,\ x^2=x}\,.
 \]
Although our rings may not be unital, it is the case that every element~$a$ in a regular ring\index{i}{regular ring}~$R$ belongs to $aR$: indeed, if~$b$ is a quasi-inverse\index{i}{quasi-inverse} of~$a$, then $a=aba\in aR$.

The following result is proved in Fryer and Halperin \cite[Section~3.2]{FrHa56}\index{c}{Fryer, K.\,D.}\index{c}{Halperin, I.}.

\begin{prop}\label{P:+capinL(R)}
Let $R$ be a regular ring\index{i}{regular ring} and let $a$, $b$ be idempotent elements in~$R$.
Furthermore, let $u$ be a quasi-inverse\index{i}{quasi-inverse} of $b-ab$. Then the following
statements hold:
\begin{description}
\item[\tui] Put $c:=(b-ab)u$. Then $aR+bR=(a+c)R$.

\item[\tuii] Suppose that $b^2=b$ and put $d:=u(b-ab)$. Then
$aR\cap bR=(b-bd)R$.

\item[\tuiii] If $aR\subseteq bR$, then $bR=aR\oplus(b-ab)R$.
\end{description}
Consequently, $\LL(R)$\index{s}{LLR@$\LL(R)$, $\LL(f)$}, partially ordered by containment, is a sectionally complemented\index{i}{lattice!sectionally complemented} sublattice of the lattice of all right ideals\index{i}{ideal (in a ring)!right} of~$R$. In particular, it is modular\index{i}{lattice!modular}.
\end{prop}

A lattice is \emph{coordinatizable}\index{i}{lattice!coordinatizable|ii} if it is isomorphic to~$\LL(R)$\index{s}{LLR@$\LL(R)$, $\LL(f)$} for some regular ring\index{i}{regular ring}~$R$. Hence every coordinatizable lattice is sectionally complemented\index{i}{lattice!sectionally complemented} and
modular\index{i}{lattice!modular}. The converse is false, the smallest counterexample being the nine-element lattice of length two~$\bM_7$.

The following consequence of Proposition~\ref{P:+capinL(R)} is observed in Micol's thesis~\cite{Micol}\index{c}{Micol, F.}.

\begin{cor}\label{C:Lfunctor}
Let $R$ and $S$ be regular rings\index{i}{regular ring} and let $f\colon R\to S$ be a ring
homomorphism. Then there exists a unique map
$\LL(f)\colon\LL(R)\to\LL(S)$\index{s}{LLR@$\LL(R)$, $\LL(f)$|ii} such that $\LL(f)(xR)=f(x)S$ for each $x\in R$, and~$\LL(f)$ is a $0$-lattice homomorphism.

Furthermore, the assignment $(R\mapsto\LL(R)$, $f\mapsto\LL(f))$\index{s}{LL@$\LL$ functor|ii} defines a functor from the category~$\REG$\index{s}{Reg@$\REG$|ii} of all regular rings\index{i}{regular ring} and ring homomorphisms to the category~$\SCML$\index{s}{Scml@$\SCML$|ii} of all sectionally complemented\index{i}{lattice!sectionally complemented} modular\index{i}{lattice!modular} lattices and $0$-lattice homomorphisms.
\end{cor}

The following useful result is folklore.

\begin{lem}\label{L:Persp2Iso}
Let $R$ be a \pup{not necessarily unital} ring, let $M$ be a right $R$-module, and let~$A$ and~$B$ be submodules of~$M$. If~$A$ and~$B$ are perspective in the submodule lattice~$\Sub M$ of all submodules of~$M$, then~$A$ and~$B$ are isomorphic. If $A\cap B=\set{0}$, then the converse holds.
\end{lem}

\begin{proof}
If $C$ is a submodule of~$M$ such that~$A\oplus C=B\oplus C$, then the map~$f\colon A\to B$ that to each $a\in A$ associates the unique element of $(a+C)\cap B$ is an isomorphism. Conversely, if $A\cap B=\set{0}$ and $f\colon A\to B$ is an isomorphism, then, setting $C:=\setm{x-f(x)}{x\in A}$, we obtain that $A\oplus C=B\oplus C$.
\qed\end{proof}

Observe, in the second part of the proof above, that if~$A$ and~$B$ are finitely generated, then so is~$C$. This will be used (for unital~$R$) in case~$M=R_R$, that is, $R$ viewed as a right module over itself.

The following result is observed in the unital case in \cite[Lemma~4.2]{WeURP}\index{c}{Wehrung, F.}, however the argument presented there works in general.

\begin{lem}\label{L:NeuIdIsCl}
Let $R$ be a regular ring\index{i}{regular ring}. Then an ideal\index{i}{ideal!of a poset}~$\bI$ of~$\LL(R)$\index{s}{LLR@$\LL(R)$, $\LL(f)$} is neutral\index{i}{ideal!neutral} if{f} it is closed under isomorphism, that is, $X\in\bI$ and $X\cong Y$ \pup{as right $R$-modules} implies that $Y\in\bI$, for all $X,Y\in\LL(R)$.
\end{lem}

\begin{proof}
If~$\bI$ is closed under isomorphism, then it is closed under perspectivity (cf. Lemma~\ref{L:Persp2Iso}), thus, as~$\LL(R)$\index{s}{LLR@$\LL(R)$, $\LL(f)$} is sectionally complemented\index{i}{lattice!sectionally complemented} modular\index{i}{lattice!modular}, $\bI$ is neutral\index{i}{ideal!neutral} (cf. Section~\ref{Su:Posets}).

Conversely, assume that~$\bI$ is neutral\index{i}{ideal!neutral} and let $X,Y\in\LL(R)$\index{s}{LLR@$\LL(R)$, $\LL(f)$} such that $X\cong Y$ and $X\in\bI$. It follows from Proposition~\ref{P:+capinL(R)} that there exists $Y'\in\LL(R)$ such that $(X\cap Y)\oplus Y'=Y$. As~$X$ belongs to~$\bI$, so does~$X\cap Y$. As $X\cong Y$ and~$Y'$ is a right ideal\index{i}{ideal (in a ring)!right} of~$Y$, there exists a right ideal\index{i}{ideal (in a ring)!right}~$X'$ of~$X$ such that $X'\cong Y'$. As $X'\cap Y'\subseteq X\cap Y\cap Y'=\set{0}$, it follows from Lemma~\ref{L:Persp2Iso} that $X'\sim Y'$. {}From $X'\subseteq X$ and $X\in\bI$ it follows that $X'\in\bI$, but $X'\sim Y'$ and~$\bI$ is a neutral\index{i}{ideal!neutral} ideal of~$\LL(R)$\index{s}{LLR@$\LL(R)$, $\LL(f)$}, thus $Y'\in\bI$. Therefore, $Y=(X\cap Y)+Y'$ belongs to~$\bI$.
\qed\end{proof}

The following result follows from the proof of \cite[Lemma~2]{FaUt63}\index{c}{Faith, C.}\index{c}{Utumi, Y.}. It enables to reduce many problems about regular rings\index{i}{regular ring} to the unital case. We include a proof for convenience.

\begin{lem}[Faith and Utumi]\label{L:FaithUt}
Let~$R$ be a regular ring\index{i}{regular ring}. Then every finite subset~$X$ of~$R$ is contained into~$eRe$, for some idempotent $e\in R$.
\end{lem}

\begin{proof}
It follows from Proposition~\ref{P:+capinL(R)}(i), applied to the opposite ring of~$R$, that there exists an idempotent $f\in R$ such that $RX\subseteq Rf$. It follows from Proposition~\ref{P:+capinL(R)}(i), applied to~$R$, that there exists an idempotent $g\in R$ such that $XR+fR\subseteq gR$. Set $e:=f+g-fg$. Then $e^2=e$ while $fe=f$ and $eg=g$, so $X\subseteq Rf=Rfe\subseteq Re$ and $X\subseteq gR=egR\subseteq eR$, and so $X\subseteq eRe$.
\qed\end{proof}

The following result is observed, in the unital case, in the proof of \cite[Theorem~4.3]{WeURP}\index{c}{Wehrung, F.}. The non-unital case requires a non-trivial use of Lemma~\ref{L:FaithUt}, that we present here.

\begin{lem}\label{L:xRyxRNeutr}
Let~$R$ be a regular ring\index{i}{regular ring} and let~$\bI$ be a neutral\index{i}{ideal!neutral} ideal of~$\LL(R)$\index{s}{LLR@$\LL(R)$, $\LL(f)$}. Then $xR\in\bI$ implies that $yxR\in\bI$, for all $x,y\in R$.
\end{lem}

\begin{proof}
Let $y'$ be a quasi-inverse\index{i}{quasi-inverse} of~$y$. It follows from Lemma~\ref{L:FaithUt} that there exists an idempotent element~$e$ of~$R$ such that $\set{x,y}\subseteq eRe$. It follows from Proposition~\ref{P:+capinL(R)} that $Y:=xR\cap(e-y'y)R$ belongs to~$\LL(R)$. As $Y\subseteq xR$ and~$\LL(R)$ is sectionally complemented\index{i}{lattice!sectionally complemented}, there exists $Z\in\LL(R)$\index{s}{LLR@$\LL(R)$, $\LL(f)$} such that $Y\oplus Z=xR$. Denote by $f\colon xR\onto yxR$\index{s}{AtoonB@$f\colon A\onto B$} the left multiplication by~$y$, and let~$t$ be an element of $\ker f:=f^{-1}\set{0}$\index{s}{kerf@$\ker f$|ii}. Then $(e-y'y)t=et-y'yt=et-0=t$, thus $t\in Y$. Conversely, let $t\in Y$. There exists $t'\in R$ such that $t=(e-y'y)t'$, so $yt=y(e-y'y)t'=(ye-yy'y)t'=(y-y)t'=0$, thus $t\in\ker f$. Consequently, $Y=\ker f$, and so~$f$ induces an isomorphism from~$Z$ onto~$yxR$. {}From $Z\subseteq xR$ it follows that $Z\in\bI$, and thus, by Lemma~\ref{L:NeuIdIsCl}, $yxR\in\bI$.
\qed\end{proof}

The following result is stated, in the unital case, in \cite[Theorem~4.3]{WeURP}\index{c}{Wehrung, F.}. The proof presented there, taking for granted the results of Corollary~\ref{C:Iidrng2st}, Lemma~\ref{L:NeuIdIsCl}, and Lemma~\ref{L:xRyxRNeutr}, trivially extends to the non-unital case. We denote by~$\Id R$\index{s}{IdR@$\Id R$, $R$ ring|ii} the lattice of all two-sided ideals\index{i}{ideal (in a ring)!two-sided} of a ring~$R$, and by $\NId L$\index{s}{IdN@$\NId L$, $L$ lattice|ii} the lattice of all neutral\index{i}{ideal!neutral} ideals of a lattice~$L$ (that it is indeed a lattice follows from \cite[Theorem~III.2.9]{GLT2}, see also \cite[Theorem~259]{LTF}\index{c}{Gr\"atzer, G.}).

\begin{prop}\label{P:NIdequiv2Id}
Let $R$ be a regular ring\index{i}{regular ring}. Then one can define mutually inverse lattice isomorphisms $\gf\colon\NId\LL(R)\to\Id R$\index{s}{LLR@$\LL(R)$, $\LL(f)$}\index{s}{IdR@$\Id R$, $R$ ring}\index{s}{IdN@$\NId L$, $L$ lattice} and $\gy\colon\Id R\to\NId\LL(R)$\index{s}{LLR@$\LL(R)$, $\LL(f)$} by the following rule:
 \begin{align*}
 \gf(\bI)&:=\setm{x\in R}{xR\in\bI}\,,&&\text{for each }\bI\in\NId\LL(R)\,;\\
 \gy(I)&:=\LL(R)\dnw I\,,&&\text{for each }I\in\Id R\,. 
 \end{align*}
\end{prop}

\begin{lem}\label{L:L(R/I)}
Let $R$ be a regular ring\index{i}{regular ring} and let~$\bI$ be a neutral\index{i}{ideal!neutral} ideal of~$\LL(R)$\index{s}{LLR@$\LL(R)$, $\LL(f)$}. Set
$I:=\setm{x\in R}{xR\in\bI}$. Then~$I$ is a two-sided ideal\index{i}{ideal (in a ring)!two-sided} of~$R$ and there exists a unique map $\gy\colon\LL(R)/\bI\to\LL(R/I)$\index{s}{LLR@$\LL(R)$, $\LL(f)$} such that
 \begin{equation}\label{Eq:FactL/I}
 \gy(xR/\bI)=(x+I)(R/I)\,,\quad\text{for each }x\in R\,,
 \end{equation}
and $\gy$ is a lattice isomorphism.
\end{lem}

\begin{proof}
It follows from Proposition~\ref{P:NIdequiv2Id} that~$I$ is a two-sided ideal\index{i}{ideal (in a ring)!two-sided} of~$R$. Denote by $p\colon R\onto R/I$\index{s}{AtoonB@$f\colon A\onto B$} and $\gp\colon\LL(R)\onto\LL(R)/\bI$\index{s}{LLR@$\LL(R)$, $\LL(f)$}\index{s}{AtoonB@$f\colon A\onto B$} the respective canonical projections. Then the required condition~\eqref{Eq:FactL/I} is equivalent to $\gy\circ\gp=\LL(p)$\index{s}{LLR@$\LL(R)$, $\LL(f)$}. The uniqueness statement on~$\gy$ follows from the surjectivity of~$\gp$. As~$\LL(R)$\index{s}{LLR@$\LL(R)$, $\LL(f)$} is sectionally complemented\index{i}{lattice!sectionally complemented}, every congruence of~$\LL(R)$\index{s}{LLR@$\LL(R)$, $\LL(f)$} is determined by the congruence class of zero, which is a neutral\index{i}{ideal!neutral} ideal; hence, in order to prove that there exists a map~$\gy$ satisfying~\eqref{Eq:FactL/I} and that this map is a lattice embedding, it suffices to prove that $\gp^{-1}\set{0}=\LL(p)^{-1}\set{0}$\index{s}{LLR@$\LL(R)$, $\LL(f)$}. For each $x\in R$, $\gp(xR)=0$ if{f} $xR\in\bI$, if{f} $x\in I$, if{f} $(x+I)(R/I)=0$, that is, $\LL(p)(xR)=0$, which proves our claim. Finally, the surjectivity of~$\gy$ follows from the surjectivity of~$\LL(p)$ together with the equation $\gy\circ\gp=\LL(p)$\index{s}{LLR@$\LL(R)$, $\LL(f)$}.
\qed\end{proof}

The following result is folklore. We include a proof for convenience.

\begin{prop}\label{P:LLpresDirLim}
The functor~$\LL$\index{s}{LL@$\LL$ functor} preserves all small directed colimits.
\end{prop}

\begin{proof}
Let $\famm{R,f_i}{i\in I}=\varinjlim\famm{R_i,f_i^j}{i\leq j\text{ in }I}$ be a directed colimit cocone in~$\REG$\index{s}{Reg@$\REG$}. It follows from the results of Section~\ref{Su:DirColimFirstOrd} that the following statements hold:
\begin{align}
R&=\bigcup\famm{\rng f_i}{i\in I}\,,\label{Eq:RascupRifi}\\
\ker f_i&=\bigcup\famm{\ker f_i^j}{j\in I\upw i}\,,&&\text{for each }i\in I\,.\label{Eq:KerfiascupKerfifi}
\end{align}
It follows from~\eqref{Eq:RascupRifi} that $\LL(R)=\bigcup\famm{\rng \LL(f_i)}{i\in I}$\index{s}{LLR@$\LL(R)$, $\LL(f)$}. Further, for each $i\in I$, as the lattice~$\LL(R_i)$\index{s}{LLR@$\LL(R)$, $\LL(f)$} is sectionally complemented\index{i}{lattice!sectionally complemented}, every congruence of~$\LL(R_i)$\index{s}{LLR@$\LL(R)$, $\LL(f)$} is determined by the congruence class of zero (cf. \cite[Section~III.3]{GLT2}, see also \cite[Theorem~272]{LTF}\index{c}{Gr\"atzer, G.}), thus, in order to verify the equation\index{s}{LLR@$\LL(R)$, $\LL(f)$}\index{s}{Kerf@$\Ker\gf$}
 \[
 \Ker\LL(f_i)=\bigcup\famm{\Ker\LL(f_i^j)}{j\in I\upw i}
 \]
(where we set $\Ker g:=\setm{(x,y)\in D\times D}{g(x)=g(y)}$, for each function~$g$ with domain~$D$)\index{s}{Kerf@$\Ker\gf$}, it is sufficient to verify the equation\index{s}{LLR@$\LL(R)$, $\LL(f)$}
 \[
 \LL(f_i)^{-1}\set{0}=\bigcup\famm{\LL(f_i^j)^{-1}\set{0}}{j\in I\upw i}\,.
 \]
However, this follows immediately from~\eqref{Eq:KerfiascupKerfifi}. By the results of Section~\ref{Su:DirColimFirstOrd}, we obtain that\index{s}{LLR@$\LL(R)$, $\LL(f)$}
 \begin{equation*}
 \famm{\LL(R),\LL(f_i)}{i\in I}=\varinjlim\famm{\LL(R_i),\LL(f_i^j)}{i\leq j\text{ in }I}\,.
 \end{equation*}
This concludes the proof.
\qed\end{proof}

\section{Right larders from regular rings}\label{S:RRLarder}
In the present section we shall construct many right larders\index{i}{larder!right} from categories of regular rings\index{i}{regular ring} satisfying a few additional closure properties (cf. Theorem~\ref{T:RR2Lard}). All the larders considered in this section will be projectable, and our first lemma deals with projectability witnesses\index{i}{projectability witness}.

\begin{lem}\label{L:RRProjWit}
Let $R$ be a regular ring\index{i}{regular ring} and let~$L$ be a sectionally complemented\index{i}{lattice!sectionally complemented} modular\index{i}{lattice!modular} lattice. Then every surjective $0$-lattice homomorphism $\gf\colon\LL(R)\onto L$\index{s}{LLR@$\LL(R)$, $\LL(f)$}\index{s}{AtoonB@$f\colon A\onto B$} has a projectability witness\index{i}{projectability witness} of the form $(a,\eps)$, where~$a$ is a surjective ring homomorphism with domain~$R$.
\end{lem}

\begin{proof}
As the ideal $\bI:=\gf^{-1}\set{0}$ is neutral\index{i}{ideal!neutral} in~$\LL(R)$ and~$\LL(R)$ is sectionally complemented\index{i}{lattice!sectionally complemented} modular\index{i}{lattice!modular}, $\gf$ induces an isomorphism $\gb\colon\LL(R)/\bI\onto L$\index{s}{AtoonB@$f\colon A\onto B$}. By Proposition~\ref{P:NIdequiv2Id}, $I:=\setm{x\in R}{xR\in\bI}$ is a two-sided ideal\index{i}{ideal (in a ring)!two-sided} of~$R$. Denote by $\gp\colon\LL(R)\onto\LL(R)/\bI$\index{s}{AtoonB@$f\colon A\onto B$} and $a\colon R\onto R/I$\index{s}{AtoonB@$f\colon A\onto B$} the respective canonical projections. By Lemma~\ref{L:L(R/I)}, there exists a unique lattice isomorphism $\ga\colon\LL(R/I)\onto\LL(R)/\bI$\index{s}{AtoonB@$f\colon A\onto B$} such that $\gp=\ga\circ\LL(a)$. Hence $\eps:=\gb\circ\ga$ is an isomorphism from~$\LL(R/I)$\index{s}{LLR@$\LL(R)$, $\LL(f)$} onto~$L$, and $\gf=\eps\circ\LL(a)$\index{s}{LLR@$\LL(R)$, $\LL(f)$}.

In order to prove that $(a,\eps)$ is a projectability witness\index{i}{projectability witness} for $\gf\colon\LL(R)\onto L$\index{s}{LLR@$\LL(R)$, $\LL(f)$}\index{s}{AtoonB@$f\colon A\onto B$}, it remains to prove that for every regular ring\index{i}{regular ring}~$X$, every ring homomorphism $f\colon R\to\nobreak X$, and every $0$-lattice homomorphism $\gh\colon\LL(R/I)\to\LL(X)$\index{s}{LLR@$\LL(R)$, $\LL(f)$} such that $\LL(f)=\gh\circ\LL(a)$, there exists a ring homomorphism $g\colon R/I\to X$ such that $f=g\circ a$ and $\gh=\LL(g)$. The assumption $\LL(f)=\gh\circ\LL(a)$ means that
 \[
 \gh\bigl((x+I)(R/I)\bigr)=f(x)X\,,\qquad\text{for each }x\in R\,.
 \]
In particular, $\ker a=I\subseteq\ker f$, thus there exists a unique ring homomorphism $g\colon R/I\to X$ such that $f=g\circ a$. It follows that
 \[
 \gh\circ\LL(a)=\LL(f)=\LL(g)\circ\LL(a)\,,
 \]
but~$\LL(a)$ is surjective, thus $\gh=\LL(g)$\index{s}{LLR@$\LL(R)$, $\LL(f)$}.
\qed\end{proof}

In Theorem~\ref{T:RR2Lard}, for a class~$\cC$ of structures and an infinite cardinal~$\gl$, we denote by~$\cC^{(\gl)}$\index{s}{Cl@$\cC^{(\gl)}$|ii} the class of all members of~$\cC$ with $\gl$-small universe. Our next result will provide us with a large class of right $\gl$-larders\index{i}{larder!right}.

\begin{thm}\label{T:RR2Lard}
Let~$\cR$ be a full subcategory of the category~$\REG$\index{s}{Reg@$\REG$} of regular rings\index{i}{regular ring} and ring homomorphisms and let~$\gl$ be an infinite cardinal. We assume the following:
\begin{description}
\item[\tui] $\cR$ is closed under small directed colimits and homomorphic images.

\item[\tuii] Every $\gl$-small subset~$X$ of a member~$R$ of~$\cR$ is contained in some $\gl$-small member of~$\cR$ contained in~$R$.
\end{description}
Denote by $\SCML^\onto$\index{s}{Scmlo@$\SCML^\onto$|ii}\index{s}{AtoonB@$f\colon A\onto B$} the subcategory of~$\SCML$\index{s}{Scml@$\SCML$} with the same objects and whose arrows are the surjective lattice homomorphisms. Then the $6$-uple\index{s}{LL@$\LL$ functor}\index{s}{Scml@$\SCML$}\index{s}{Scmll@$\SCML^{(\gl)}$}\index{s}{AtoonB@$f\colon A\onto B$}
 \[
 (\cR,\cR^{(\gl)},\SCML,\SCML^{(\gl)},\SCML^\onto,\LL)
 \]
is a projectable right $\gl$-larder\index{i}{larder!right!projectable}.
\end{thm}

\begin{proof}
For each $R\in\cR^{(\gl)}$, the lattice~$\LL(R)$ is $\gl$-small, thus, by Proposition~\ref{P:glPresMIND} applied within the category of all lattices with zero with $0$-lattice homomorphisms, $\LL(R)$\index{s}{LLR@$\LL(R)$, $\LL(f)$} is weakly $\gl$-presented\index{i}{presented!weakly $\gl$-}. This completes the proof of~$(\PRES_\gl(\cR^{(\gl)},\LL))$\index{s}{Pres@$(\PRES_\gl(\cB^\dagger,\Psi))$}. The projectability statement follows trivially from Lemma~\ref{L:RRProjWit} together with the closure of~$\cR$ under homomorphic images.

It remains to verify $(\LSr_{\cf(\gl)}(R))$\index{s}{LSr@$(\LSr_\gm(B))$}, for each object~$R\in\cR$. Let $L$ be a $\gl$-small sectionally complemented\index{i}{lattice!sectionally complemented} modular\index{i}{lattice!modular} lattice, let~$I$ be a $\cf(\gl)$-small set, let $\famm{u_i\colon U_i\to R}{i\in I}$ be an $I$-indexed family of ring homomorphisms with all $\card U_i<\gl$, and let $\gf\colon\LL(R)\onto L$\index{s}{LLR@$\LL(R)$, $\LL(f)$}\index{s}{AtoonB@$f\colon A\onto B$} be a surjective $0$-lattice homomorphism. We shall construct a monomorphism $v\colon V\mono R$\index{s}{AtomonoB@$f\colon A\mono B$}, with a $\gl$-small object~$V$ of~$\cR$, with all $u_i\utr v$ and $\gf\circ\LL(v)$\index{s}{LLR@$\LL(R)$, $\LL(f)$} surjective. The first paragraph of the proof of Lemma~\ref{L:RRProjWit} shows that we may assume that $L=\LL(R/J)$\index{s}{LLR@$\LL(R)$, $\LL(f)$}, for some two-sided ideal\index{i}{ideal (in a ring)!two-sided}~$J$ of~$R$, and $\gf=\LL(p)$\index{s}{LLR@$\LL(R)$, $\LL(f)$}, where~$p\colon R\onto R/J$\index{s}{AtoonB@$f\colon A\onto B$} denotes the canonical projection. {}From $\card I<\cf(\gl)$ and all $\card U_i<\gl$ it follows that the set $U:=\bigcup\famm{u_i``(U_i)}{i\in I}$ is $\gl$-small. As\index{s}{LLR@$\LL(R)$, $\LL(f)$}
 \[
 \LL(R/J)=\setm{(x+J)(R/J)}{x\in R}
 \]
is also $\gl$-small, there exists a $\gl$-small subset~$V$ of~$R$ containing $U$ such that\index{s}{LLR@$\LL(R)$, $\LL(f)$}
 \begin{equation}\label{Eq:L(X/I)full}
 \LL(R/J)=\setm{(x+J)(R/J)}{x\in V}\,.
 \end{equation}
By assumption, there exists a $\gl$-small $R'\in\cR$ contained in~$R$ such that $V\subseteq R'$. Denote by~$v\colon R'\into R$\index{s}{AtoinB@$f\colon A\into B$} the inclusion map. {}From $u_i``(U_i)\subseteq R'$ it follows that $u_i\utr v$, for each $i\in I$. As the range of~$v$ contains~$V$, the range of $\gf\circ\LL(v)$\index{s}{LLR@$\LL(R)$, $\LL(f)$} contains all elements of the form $(x+J)(R/J)$, where $x\in V$, and thus, by~\eqref{Eq:L(X/I)full}, $\gf\circ\LL(v)$\index{s}{LLR@$\LL(R)$, $\LL(f)$} is surjective.
\qed\end{proof}

The assumptions of Theorem~\ref{T:RR2Lard}, about the category~$\cR$, are satisfied in two noteworthy cases:

\begin{itemize}
\item $\cR$ is the category of all regular rings\index{i}{regular ring}. We obtain that the $6$-uple\index{s}{LL@$\LL$ functor}\index{s}{Reg@$\REG$}\index{s}{Regl@$\REG^{(\gl)}$}\index{s}{Scml@$\SCML$}\index{s}{Scmlo@$\SCML^\onto$}\index{s}{Scmll@$\SCML^{(\gl)}$}\index{s}{AtoonB@$f\colon A\onto B$}
 \[
 (\REG,\REG^{(\gl)},\SCML,\SCML^{(\gl)},\SCML^\onto,\LL)
 \]
is a projectable right $\gl$-larder\index{i}{larder!right!projectable} for each uncountable cardinal~$\gl$. This is used, for $\gl:=\aleph_1$\index{s}{aleph0@$\aleph_{\ga}$}, in the second author's paper~\cite{Banasch2}\index{c}{Wehrung, F.}, to solve a 1962 problem due to J\'onsson\index{c}{Jonsson@J\'onsson, B.}, by finding a non-coordinatizable\index{i}{lattice!coordinatizable} sectionally complemented\index{i}{lattice!sectionally complemented} modular\index{i}{lattice!modular} lattice, of cardinality~$\aleph_1$\index{s}{aleph0@$\aleph_{\ga}$}, with a large J\'onsson\index{c}{Jonsson@J\'onsson, B.} four-frame.

\item $\cR$ is the category~$\mathbf{LocMat}_\FF$\index{s}{LocMat@$\mathbf{LocMat}_\FF$|ii} of all rings that are \emph{locally matricial} over some field~$\FF$. By definition, a ring is matricial over~$\FF$ if it is a finite direct product of full matrix rings over~$\FF$, and locally matricial over~$\FF$ if it is a directed colimit of matricial rings over~$\FF$. Then Theorem~\ref{T:RR2Lard} yields that the $6$-uple\index{s}{LL@$\LL$ functor}\index{s}{Scml@$\SCML$}\index{s}{Scmlo@$\SCML^\onto$}\index{s}{Scmll@$\SCML^{(\gl)}$}\index{s}{LocMatgl@$\mathbf{LocMat}^{(\gl)}_\FF$|ii}\index{s}{AtoonB@$f\colon A\onto B$}
 \[
 (\mathbf{LocMat}_\FF,\mathbf{LocMat}_\FF^{(\gl)},\SCML,\SCML^{(\gl)},
 \SCML^\onto,\LL)
 \]
is a projectable right $\gl$-larder\index{i}{larder!right!projectable}, for every infinite cardinal~$\gl$ such that $\cf(\gl)>\card\FF$ (in particular, in case~$\FF$ is finite, every infinite cardinal~$\gl$ works).
\end{itemize}

\chapter{Discussion}\label{Ch:Discussion}

The discussion undertaken, in Chapter~\ref{Ch:RegRngLard}, about the functor~$\LL$\index{s}{LL@$\LL$ functor} on regular rings\index{i}{regular ring}, can be mimicked for the functor~$\VV$\index{s}{VV@$\VV$ functor} (\emph{nonstable K-theory}\index{i}{nonstable K- (or K$_0$-) theory}) introduced in Example~\ref{Ex:nsKth}, restricted to (von Neumann) regular rings\index{i}{regular ring}. It is a fundamental open problem in the theory of regular rings\index{i}{regular ring} whether every conical\index{i}{monoid!conical} refinement\index{i}{monoid!refinement} monoid, of cardinality at most~$\aleph_1$\index{s}{aleph0@$\aleph_{\ga}$}, is isomorphic to~$\VV(R)$\index{s}{VV@$\VV$ functor} for some regular ring\index{i}{regular ring}~$R$, cf. Goodearl~\cite{Good95}\index{c}{Goodearl, K.\,R.}, Ara~\cite{Ara08}\index{c}{Ara, P.}. Due to counterexamples developed in Wehrung~\cite{NonMeas}\index{c}{Wehrung, F.}, the situation is hopeless in cardinality~$\aleph_2$\index{s}{aleph0@$\aleph_{\ga}$} or above. Recent advances on those matters, for more general classes of rings such as \emph{exchange rings} but also for \emph{C*-algebras}, can be found in Wehrung~\cite{VLift}\index{c}{Wehrung, F.}.

Now let us discuss some open problems. In our mind, the most fundamental open problem raised by the statement of CLL\index{i}{Condensate Lifting Lemma (CLL)} is the extension of the class of posets~$P$, or even categories that are not posets, for which a weak form of CLL\index{i}{Condensate Lifting Lemma (CLL)} could hold. For example, we establish in Corollary~\ref{C:CLLnoLF} a weak form of CLL\index{i}{Condensate Lifting Lemma (CLL)}, valid for any \ajs\index{i}{almost join-semilattice}\ (assuming large enough cardinals), regardless of the existence of a lifter\index{i}{lifter ($\gl$-)}.

\begin{problem}\label{Pb:CLLBowTie}
Is there a weak form of CLL\index{i}{Condensate Lifting Lemma (CLL)} that would work for an arbitrary poset, or even just the rightmost poset of Figure~\ref{Fig:posets}, page~\pageref{Fig:posets}?
\end{problem}

This problem is related to the representation problem of distributive algebraic lattices\index{i}{lattice!algebraic}\index{i}{lattice!distributive} as congruence lattices of members of a congruence-distributive\index{i}{variety!congruence-distributive} variety---for example, now that lattices are ruled out~\cite{CLP}\index{c}{Wehrung, F.}, \emph{majority algebras}\index{i}{algebra!majority}. Is every distributive algebraic lattice\index{i}{lattice!algebraic}\index{i}{lattice!distributive} isomorphic to the congruence lattice of some majority algebra\index{i}{algebra!majority}? If this were the case, then, because of the results of~\cite{Bowtie}\index{c}{Tuma@T\r{u}ma, J.}\index{c}{Wehrung, F.}, the somewhat loosely formulated Problem~\ref{Pb:CLLBowTie} could have no reasonable positive answer.

In our next two problems, we ask whether the assumptions of local finiteness and strong congruence-properness\index{i}{strongly congruence-proper} can be removed from the statements of both Theorem~\ref{T:RelCritalephn} and Theorem~\ref{T:DichotCritPt}. Positive results in that direction can be found in the first author's paper~\cite{Gill3}\index{c}{Gillibert, P.}.

\begin{problem}\label{Pb:RelCritalephn}
Let~$\cA$ and~$\cB$ be quasivarieties in finite \pup{possibly different} languages, and let~$P$ be a nontrivial finite \ajs\index{i}{almost join-semilattice}\ with zero. Prove that if there exists a $P$-indexed diagram~$\overrightarrow{\bA}=\famm{\bA_p,\ga_p^q}{p\leq q\text{ in }P}$ of finite members of~$\cA$ such that $\ConcA\overrightarrow{\bA}$\index{s}{compcongVA@$\ConcV\bA$, $\ConcV f$} has no lifting\index{i}{diagram!lifting}, with respect to~$\ConcB$, in~$\cB$, then $\critr(\cA;\cB)\leq\aleph_{\kur_0(P)-1}$\index{s}{aleph0@$\aleph_{\ga}$}\index{s}{critrAB@$\critr(\cA;\cB)$}\index{s}{kurP0@$\kur_0(P)$}.
\end{problem}

The following problem asks for an ambitious generalization of Theorem~\ref{T:DichotCritPt} (the Dichotomy Theorem)\index{i}{Dichotomy Theorem}: it asks not only for relaxing the assumptions on the quasivarieties~$\cA$ and~$\cB$, but also for an improvement of the cardinality bound. It can also be viewed as an (ultimate?) recasting of the Critical Point Conjecture formulated in T\r{u}ma\index{c}{Tuma@T\r{u}ma, J.} and Wehrung\index{c}{Wehrung, F.}~\cite{CLPSurv}. For a partial positive solution, see Gillibert~\cite{Gill5}\index{c}{Gillibert, P.}.

\begin{problem}\label{Pb:DichotCritPt}
Let~$\cA$ and~$\cB$ be quasivarieties on finite \pup{possibly different} languages. Prove that if $\Concr\cA\not\subseteq\Concr\cB$\index{s}{compcongVcr@$\Concr\cV$}, then $\critr(\cA;\cB)\leq\aleph_2$\index{s}{aleph0@$\aleph_{\ga}$}\index{s}{critrAB@$\critr(\cA;\cB)$}.
\end{problem}

Our next problem is formulated in the same context as Problems~\ref{Pb:RelCritalephn} and~\ref{Pb:DichotCritPt}.

\begin{problem}\label{Pb:AbsoluteCP}
Is there a recursive algorithm that, given finitely generated\index{i}{quasivariety!finitely generated} quasivarieties~$\cA$ and~$\cB$ in finite languages, outputs a code for their relative critical point\index{i}{critical point!relative}~$\critr(\cA;\cB)$\index{s}{critrAB@$\critr(\cA;\cB)$} (e.g., the pair $(0,n)$ if the critical point\index{i}{critical point} is~$n$ and the pair $(1,n)$ if the critical point\index{i}{critical point} is~$\aleph_n$\index{s}{aleph0@$\aleph_{\ga}$})?
\end{problem}

To illustrate the extent of our ignorance with respect to Problem~\ref{Pb:AbsoluteCP}, we do not even know whether there are finitely generated\index{i}{variety!finitely generated} lattice varieties~$\cA$ and~$\cB$ such that $\crit(\cA;\cB)$\index{s}{critAB@$\crit(\cA;\cB)$} is equal to~$\aleph_1$\index{s}{aleph0@$\aleph_{\ga}$} in one universe of set theory but to~$\aleph_2$\index{s}{aleph0@$\aleph_{\ga}$} in another one. We do not even know whether the critical point\index{i}{critical point} between two finitely generated\index{i}{variety!finitely generated} lattice varieties is \emph{absolute} (in the set-theoretical sense)!

Our next problem asks for a functorial extension of the Pudl\'ak-T\r{u}ma Theorem~\cite{PuTu80}\index{c}{Pudl\'ak, P.}\index{c}{Tuma@T\r{u}ma, J.}, which states that every finite lattice embeds into some finite partition lattice. It involves the formulation of lattice embeddings into partition lattices \emph{via} the semilattice-valued distances\index{i}{distance!semilattice-valued} introduced in J\'onsson~\cite{Jons53}\index{c}{Jonsson@J\'onsson, B.}. We remind the reader that~$\METR$\index{s}{Metr@$\METR$} denotes the category of all semilattice-metric spaces\index{i}{semilattice-metric!space} (cf. Definition~\ref{D:MetrSp}) and $\METR^\fin$\index{s}{Metrf@$\METR^{\mathrm{fin}}$} denotes the full subcategory of all finite members of~$\METR$\index{s}{Metr@$\METR$}. Further, we denote by $\SEM^\fin$\index{s}{Semf@$\SEM^{\mathrm{fin}}$} (resp., $\SEM^{\mathrm{fin,inj}}$\index{s}{Sems@$\SEM^{\mathrm{fin,inj}}$|ii}) the category of all  finite \jzs s with \jzh s (resp., finite \jzs s with \jze s) and by $\Pi\colon\METR^\fin\to\SEM^\fin$\index{s}{Semf@$\SEM^{\mathrm{fin}}$} the forgetful functor.

\begin{problem}\label{Pb:DiagrPuTu}
Does there exist a functor $\Gamma\colon\SEM^{\mathrm{fin,inj}}\to\METR^\fin$\index{s}{Semfi@$\SEM^{\mathrm{fin,inj}}$} such that~$\gd_{\Gamma(S)}$ is a surjective V-distance, for each finite \jzs~$S$, and $\Pi\circ\Gamma$ is isomorphic to the identity?
\end{problem}

In his paper~\cite{Plos08}\index{c}{Plo\v{s}\v{c}ica, M.}, Plo\v{s}\v{c}ica extends the methods introduced in \cite{CLP}\index{c}{Wehrung, F.} (for solving CLP\index{i}{Congruence Lattice Problem (CLP)}) and~\cite{Ruzi08}\index{c}{Ruzicka@R\r{u}\v{z}i\v{c}ka, P.} (for improving the cardinality bound) to prove that if~$F$ denotes the free object, in the variety of all bounded lattices generated by~$\bM_3$, on~$\aleph_2$\index{s}{aleph0@$\aleph_{\ga}$} generators, then for every positive integer~$m$ and every lattice~$L$ with $m$-permutable\index{i}{permut@$m$-permutable congruence lattice} congruences, the congruence lattices of~$F$ and of~$L$ are not isomorphic. We ask whether there is a similar strengthening of Corollary~\ref{C:NoCPCP}:

\begin{problem}\label{Pb:NoCPCP}
Denote by~$F$ the free object on~$\aleph_1$\index{s}{aleph0@$\aleph_{\ga}$} generators in a nondistributive\index{i}{distributive!non-${}_{-}$ variety} variety~$\cV$ of (bounded) lattices and let~$m$ be a positive integer. Prove that~$F$ has no congruence-preserving extension\index{i}{congruence-preserving extension} to a lattice with $m$-permutable congruences\index{i}{permut@$m$-permutable congruence lattice}.
\end{problem}

The answer to Problem~\ref{Pb:NoCPCP} is known to be positive in case~$\cV$ contains as a member either~$\bM_3$ or a few of the successors of~$\bN_5$ with respect to the variety order (cf. Gillibert~\cite{Gill4}), but it is not known, for example, in case~$\cV$ is the variety generated by~$\bN_5$.

Our next problem asks for a nearly as strong as possible diagram extension of the Gr\"atzer-Schmidt\index{c}{Gr\"atzer, G.}\index{c}{Schmidt, E.\,T.} Theorem~\cite{GrSc62}\index{c}{Gr\"atzer, G.}\index{c}{Schmidt, E.\,T.}. A partial positive answer is given in Theorem~\ref{T:MindConcLift}.

\begin{problem}\label{Pb:GS}
Does there exist a functor~$\Gamma$, from \jzs s with \jze s, to the category~$\MIND$\index{s}{Mind@$\MIND$} of all monotone-indexed structures, such that $\Conc\circ\Gamma$ is isomorphic to the identity?
\end{problem}

\begin{problem}\label{Pb:gooplift}
Let~$\gl$ be an infinite cardinal. Is every $\gl$-liftable\index{i}{liftable!$\gl$-${}_{-}$ poset} poset well-founded\index{i}{poset!well-founded}?
\end{problem}

By Lemma~\ref{L:NonWFgo+1}, Problem~\ref{Pb:gooplift} amounts to determining whether $(\go+1)^{\op}$\index{s}{omega1op@$(\omega+1)^{\op}$} has a $\gl$-lifter\index{i}{lifter ($\gl$-)}. By Corollary~\ref{C:LiftnotWFrestr}, this does not hold for all values of~$\gl$, in particular in case $\gl=(2^{\aleph_0})^+$\index{s}{aleph0@$\aleph_{\ga}$}.

For further comments about Problem~\ref{Pb:gooplift}, see Remark~\ref{Rk:CLLnoLF}.

A related problem is the following.

\begin{problem}\label{Pb:GenLifter}
Let~$P$ be a poset and let~$\gl$ be an infinite cardinal. If~$P$ has a $\gl$-lifter\index{i}{lifter ($\gl$-)}, does it have a $\gl$-lifter\index{i}{lifter ($\gl$-)} $(X,\bX)$ such that~$X$ is a lower finite\index{i}{poset!lower finite} \ajs\index{i}{almost join-semilattice}\ and~$\bX$ is the collection of all extreme\index{i}{ideal!extreme} ideals of~$X$?
\end{problem}

If~$P$ is well-founded\index{i}{poset!well-founded}, then so is the set of all extreme\index{i}{ideal!extreme} ideals of~$X$ (partially ordered by containment). In particular, if Problems~\ref{Pb:gooplift} and~\ref{Pb:GenLifter} both have a positive answer, then every $\gl$-liftable\index{i}{liftable!$\gl$-${}_{-}$ poset} poset is well-founded and it has a $\gl$-lifter\index{i}{lifter ($\gl$-)} $(X,\bX)$ with~$X$ a lower finite\index{i}{poset!lower finite} \ajs\index{i}{almost join-semilattice}\ and~$\bX$, being the collection of all extreme\index{i}{ideal!extreme} ideals of~$X$, is also well-founded\index{i}{poset!well-founded}.

Open problems related to those on our list can be found in various items from our bibliography, for example \cite{Gill1,Gill2, GiWe1, CLPSurv, Bowtie, Ultra, RetrLift, CXCoord,CLP}\index{c}{Gillibert, P.}\index{c}{Tuma@T\r{u}ma, J.}\index{c}{Wehrung, F.}.

\printindex{s}{Symbol Index}
\printindex{i}{Subject Index}
\printindex{c}{Author index}

\listoffigures

\end{document}